\documentclass[11pt,letterpaper]{article}
\usepackage[T1]{fontenc}
\usepackage{amsmath,amsthm,amssymb,mathtools}
\usepackage{mathrsfs}
\usepackage{nicefrac}
\usepackage[margin=1.05in]{geometry}
\usepackage{microtype}
\usepackage{xcolor}
\usepackage{tikz}
\definecolor{paperlink}{RGB}{0,0,114}
\usepackage{hyperref}
\hypersetup{colorlinks=true,linktoc=all,linkcolor=paperlink,citecolor=paperlink,urlcolor=paperlink,pdfstartview=Fit}
\usepackage{enumitem}
\setlist[itemize]{leftmargin=2em,itemsep=0.25em,topsep=0.35em}
\setlist[enumerate]{leftmargin=2.25em,itemsep=0.25em,topsep=0.35em}
\usepackage{fancyhdr}
\numberwithin{equation}{section}
\newtheorem{theorem}{Theorem}[section]
\newtheorem{proposition}[theorem]{Proposition}
\newtheorem{lemma}[theorem]{Lemma}

\newtheorem{theoremA}{Theorem}

\newtheorem{corollaryA}[theoremA]{Corollary}
\theoremstyle{definition}
\newtheorem{definition}[theorem]{Definition}

\newcommand{\R}{\mathbb{R}}
\newcommand{\Z}{\mathbb{Z}}
\newcommand{\N}{\mathbb{N}}
\newcommand{\Rd}{\mathbb{R}^d}
\newcommand{\Zd}{\mathbb{Z}^d}
\renewcommand{\a}{\mathbf{a}}
\newcommand{\bfA}{\mathbf{A}}
\newcommand{\eps}{\varepsilon}
\newcommand{\nf}{\nicefrac}
\newcommand{\qand}{\quad\mbox{and}\quad}
\newcommand{\cu}{\square}
\newcommand*{\tr}{\mathrm{trace\,}}
\newcommand*{\supp}{\mathrm{supp\,}}
\DeclareMathOperator{\dist}{dist}

\DeclareMathOperator*{\esssup}{ess\,sup}
\DeclareMathOperator*{\essinf}{ess\,inf}
\DeclarePairedDelimiter{\norm}{\lVert}{\rVert}
\renewcommand{\subset}{\subseteq}
\DeclareMathAlphabet{\mathmybb}{U}{bbold}{m}{n}
\newcommand{\indc}{{\mathmybb{1}}}
\newcommand{\Ener}{\mathcal E}
\newcommand{\Hzero}[1]{H^1_{#1,0}}
\newcommand{\Csol}{\mathcal C_{\mathrm{sol}}}
\newcommand{\Csub}{\mathcal C_{\mathrm{sub}}}
\def\Xint#1{\mathchoice
	{\XXint\displaystyle\textstyle{#1}}%
	{\XXint\textstyle\scriptstyle{#1}}%
	{\XXint\scriptstyle\scriptscriptstyle{#1}}%
	{\XXint\scriptscriptstyle\scriptscriptstyle{#1}}%
	\!\int}
\def\XXint#1#2#3{{\setbox0=\hbox{$#1{#2#3}{\int}$}
		\vcenter{\hbox{$#2#3$}}\kern-.5\wd0}}
\def\fint{\Xint-}
\makeatletter
\newcommand{\avsum}{\mathop{\mathpalette\avsuminner\relax}\displaylimits}
\newcommand\avsuminner[2]{%
	{\sbox0{$\m@th#1\sum$}%
		\vphantom{\usebox0}%
		\ooalign{%
			\hidewidth
			\smash{\,\rule[.23em]{8.8pt}{1.1pt} \relax}%
			\hidewidth\cr
			$\m@th#1\sum$\cr
		}%
	}%
}
\makeatother

\title{Coarse ellipticity and De Giorgi--Nash--Moser theory in the optimal range}

\author{S. Armstrong
\thanks{CNRS \& Laboratoire Jacques-Louis Lions, Sorbonne Universit\'e and Courant Institute of Mathematical Sciences, New York University. {\footnotesize\href{mailto:scottnarmstrong@gmail.com}{scottnarmstrong@gmail.com}.}}
\and
B. Avelin
\thanks{Department of Mathematics, Uppsala University. {\footnotesize\href{mailto:benny.avelin@math.uu.se}{benny.avelin@math.uu.se}.}}
\and
T. Kuusi
\thanks{Department of Mathematics and Statistics, University of Helsinki. {\footnotesize\href{mailto:tuomo.kuusi@helsinki.fi}{tuomo.kuusi@helsinki.fi}.}}
\and
A. Turpeinen
\thanks{Department of Mathematics and Statistics, University of Helsinki. {\footnotesize\href{mailto:aatu.turpeinen@helsinki.fi}{aatu.turpeinen@helsinki.fi}.}}
}
\date{}

\begin{document}
\maketitle

\begin{abstract}
We extend the theory of De Giorgi--Nash--Moser to elliptic equations with symmetric coefficients~$\mathbf{a}(x)$ which are possibly degenerate and unbounded but satisfy a coarse ellipticity condition. 
We prove local upper bounds for weak subsolutions, a weak Harnack inequality for nonnegative supersolutions, and a Harnack inequality for nonnegative solutions. 
The coarse ellipticity hypothesis requires spatial moments of coarse-grained matrices, suitably discounted and summed across scales, to be finite. In particular, it holds if, for some~$\alpha,\beta\geq0$ and~$1<p,q<\infty$,
\begin{equation*}
	\mathbf{a}\in W^{-\alpha,p}\cap L^1\,,\qquad\mathbf{a}^{-1}\in W^{-\beta,q}\cap L^1\qquad\mbox{and}\qquad \frac{\alpha+\beta}{2}+\frac{d-1}{2}\biggl(\frac1p+\frac1q\biggr)< 1\,.
\end{equation*}
We show that the coefficient range is sharp, including its boundary, in every dimension~$d\geq3$, for~$\alpha=\beta=0$ and for the corresponding Besov-type quasi-norms of negative order when~$\alpha,\beta>0$. For~$\alpha=\beta=0$, it corresponds to the results of Bella and Sch\"affner~\cite{Bella.Schaffner.regularity}. 
\end{abstract}

\tableofcontents

\section{Introduction}
\label{s.introduction}

\subsection{Background and overview}

Let~$d\geq3$ and let~$\cu_0\coloneqq(-\nf12,\nf12)^d$. We study weak solutions of the elliptic equation
\begin{equation}
	-\nabla\cdot\a\nabla u=0\qquad\mbox{in }\cu_0\,,
	\label{e.subsolution.equation}
\end{equation}
where the coefficient field~$\a$ is measurable, symmetric and positive definite almost everywhere. We assume neither a positive lower bound nor an upper bound on~$\a$. The only qualitative assumption used is the condition~$\tr\a,\tr\a^{-1}\in L^1(\cu_0)$. This ensures the weighted energy space of Section~\ref{s.weighted.spaces} is well defined, but none of our regularity estimates depend on these integrals.

\smallskip

If~$\a$ is uniformly elliptic, the theory of De Giorgi, Nash and Moser~\cite{DeGiorgi,Nash,Moser1,Moser2} shows that subsolutions are locally bounded above and that nonnegative solutions satisfy the Harnack inequality, with constants depending only on~$d$ and on the ratio of the ellipticity constants. In this paper we consider weaker assumptions on~$\a$ which imply these estimates.

\paragraph{Integrability.} The classical answer is given in terms of integrability. After the work of Murthy and Stampacchia~\cite{Murthy.Stampacchia} on degenerate elliptic operators, Trudinger~\cite{Trudinger,Trudinger.1973} proved local boundedness and the Harnack inequality when~$\a\in L^p$ and~$\a^{-1}\in L^q$ with~$\nf1p+\nf1q<\nf2d$. Bella and Sch\"affner~\cite{Bella.Schaffner.regularity} enlarged the range to
\begin{equation}
	\frac1p+\frac1q<\frac2{d-1}\,,
	\label{e.intro.bs.range}
\end{equation}
and this range cannot be enlarged. Bella and Sch\"affner construct scalar coefficients admitting unbounded weak subsolutions when~$\nf1p+\nf1q>\nf2{(d-1)}$ in every dimension~$d\geq3$, and also at equality for~$d\geq4$~\cite{Bella.Schaffner.plaplace}. For~$d\geq4$, the examples of Franchi, Serapioni and Serra Cassano~\cite{FSS} give unbounded weak solutions when~$\nf1p+\nf1q>\nf2{(d-1)}$ (see the discussion in~\cite{Bella.Schaffner.regularity}). Our Theorem~\ref{t.sharpness} below constructs in every dimension~$d\geq3$ a scalar coefficient field satisfying the integrability condition, even at the endpoint, and a weak solution~$u$ with unbounded essential supremum and finite positive essential infimum on~$\frac12\cu_0$. Thus both local boundedness and the Harnack inequality fail.

\smallskip

The improvement from~$d$ to~$d-1$ comes from the choice of the cutoff function. The Caccioppoli inequality for a nonnegative subsolution~$v$ is obtained by testing the weak formulation with~$\varphi^2v$ for a cutoff function~$\varphi$, applying the Cauchy--Schwarz inequality and controlling the right-hand side term~$\int v^2\nabla\varphi\cdot\a\nabla\varphi$. Bella and Sch\"affner minimize this term over~$\varphi$, so the chosen cutoff depends on~$v$. They then estimate the resulting term through integrals over spheres using the Sobolev inequality in dimension~$d-1$. The underlying idea of dimensional reduction through slicing had already been used by Struwe~\cite{Struwe} and Frehse and R\r{u}\v{z}i\v{c}ka~\cite{Frehse.Ruzicka} in the study of the Navier--Stokes equations, and by Kontovourkis~\cite{Kontovourkis} for elliptic equations with divergence-free drift.

\paragraph{Coarse ellipticity.} Integrability measures the size of~$\a$ pointwise. 
Large or small conductivities confined to small inclusions may change the energy of solutions by only a bounded factor, even when the integrability bounds deteriorate.
In~\cite{AADKM}, three of the authors, together with De Filippis and Mingione, replaced integrability by \emph{coarse ellipticity}, a notion introduced in~\cite{AK.HC} for the homogenization of high-contrast random media. On each triadic subcube of~$\cu_0$, the condition uses one energy bound for the mean flux and another for the mean gradient of admissible functions. At each scale it takes the largest constants over the cubes and combines them across scales with discount factors. The coarse ellipticity condition is implied by negative Sobolev regularity of~$\a$ and~$\a^{-1}$. It allows coefficient fields with~$\a\notin L^{1+\delta}$ for every~$\delta>0$, and it holds with uniform constants for the scalar fields~$(1+\mu_n)I$, where the~$\mu_n$ are regularizations of a fractal measure or of a Gaussian multiplicative chaos, whose~$L^{1+\delta}$ norms are unbounded. Under the coarse ellipticity condition,~\cite{AADKM} proves local boundedness and the Harnack inequality with the constant~$\exp(C\Theta^{\nf12})$, in terms of a coarse ellipticity ratio~$\Theta$. However, when restricted to coefficient fields with~$\a\in L^p$ and~$\a^{-1}\in L^q$, these results reach only Trudinger's range~$\nf1p+\nf1q<\nf2d$, and the negative Sobolev criterion of~\cite{AADKM} has the same dimension~$d$ in its admissible range. When estimated from integrability, taking the maximum over the~$3^{kd}$ cubes of size~$3^{-k}$ introduces a factor~$3^{kd/p}$, and this is where the full dimension enters.

\paragraph{Coarse ellipticity with optimal exponents.} 
We introduce a notion of coarse ellipticity, which uses~$\ell_p$ rather than~$\ell_\infty$ norms for the coarse-grained matrices in triadic simplices on each scale; this is implied by the~$L^p$- and negative-Sobolev conditions in the optimal range. This hypothesis assumes bounds on two coarse ellipticity constants. Each constant is defined using a discounted sum of spatial moments of one of the two coarse-grained matrices on triadic simplices (rather than cubes). In each simplex, the coarse-grained matrices control the mean gradients and fluxes of every solution through its energy. They are exactly the coarse-grained matrices of~\cite{AK.HC} for symmetric coefficient fields, and they differ from those of~\cite{AADKM} in an important way: both are defined through solutions alone, whereas the upper matrix of~\cite{AADKM} is a supremum over a cone of functions such as truncations and powers of solutions. The admissible ranges of the integrability and negative-regularity conditions that imply the hypothesis involve~$d-1$ where those of~\cite{AADKM} involved~$d$. 

\smallskip

We explicitly track the dependence of our estimates on the coarse ellipticity ratio, and our estimates are sharp in this dependence. In particular, we obtain local upper bounds for subsolutions with an optimal polynomial dependence on the ratio~$\Theta$ of the two coarse ellipticity constants, and the weak Harnack and Harnack inequalities with the constant~$\exp(C\sqrt\Theta)$ that has optimal growth in~$\Theta$ (Theorems~\ref{t.cg.local.boundedness} and~\ref{t.harnack} and Proposition~\ref{p.sharpness.polynomial}). The weak Harnack inequality holds at the endpoint exponent~$r^*/2$ of Theorem~\ref{t.harnack}, and Proposition~\ref{p.sharpness.weak.harnack} shows that this exponent cannot be increased for any fixed admissible choice of~$d,p,q,s,t$. A coefficient criterion and its sharpness are as follows.
\begin{itemize}
\item The hypothesis holds, and so do local boundedness and the Harnack inequality (Corollary~\ref{c.sobolev.coefficients}), if for some~$\alpha,\beta\geq0$ and~$1<p,q<\infty$
\begin{equation}
	\a\in W^{-\alpha,p}\cap L^1\,,\qquad\a^{-1}\in W^{-\beta,q}\cap L^1\qquad\mbox{and}\qquad\frac{\alpha+\beta}{2}+\frac{d-1}{2}\biggl(\frac1p+\frac1q\biggr)<1\,.
	\label{e.goodrange}
\end{equation}
For~$\alpha=\beta=0$ this is the whole range~\eqref{e.intro.bs.range} for~$\a\in L^p$ and~$\a^{-1}\in L^q$: the coarse-grained theory recovers the range of Bella and Sch\"affner for symmetric coefficients. The Harnack constant is
\begin{equation*}
	\exp\Bigl(C\bigl(\norm{\a}_{L^p(\cu_0)}\norm{\a^{-1}}_{L^q(\cu_0)}\bigr)^{\nf12}\Bigr)\,.
\end{equation*}
This dependence on the coefficient norms was previously known only in Trudinger's range~\cite{AADKM}. For~$\alpha,\beta>0$, the local boundedness or Harnack inequality under negative regularity of~$\a$ and~$\a^{-1}$ in the full range of~\eqref{e.goodrange} is new. Here~$W^{-\alpha,p}$ is normed by duality with~$W^{\alpha,p'}(\cu_0)$, as in~\eqref{e.negative.sobolev.norm}, and not with~$W^{\alpha,p'}_0(\cu_0)$. The hypothesis also holds, in the same range with~$\alpha,\beta>0$, if the Besov-type quasi-norms of Theorem~\ref{t.coefficient.conditions}\ref{c.negative.regularity} of orders~$-\alpha$ and~$-\beta$ are finite.
\item For finite exponents, both the integrability range~\eqref{e.intro.bs.range} and the corresponding range for the Besov-type quasi-norms of Theorem~\ref{t.coefficient.conditions}\ref{c.negative.regularity} are sharp, including at the endpoint, in every dimension~$d\geq3$, already for scalar coefficient fields (Theorem~\ref{t.sharpness}).
\end{itemize}

\paragraph{Defining coarse-grained matrices with solutions versus subsolutions.} That both coarse-grained matrices are defined through solutions alone is essential for the applications of the hypothesis. The renormalization argument of~\cite{AK.HC} for random coefficient fields with high contrast bounds exactly these matrices across scales, and no other quantities. We have no technique to control analogous quantities defined through subsolutions, such as the suprema over cones in~\cite{AADKM}, in this renormalization scheme or in many other applications. Our conclusions nevertheless concern subsolutions and supersolutions, so the proofs must control them by quantities computed from solutions. For the lower matrix this is automatic, since the supremum defining it is the same over all functions of finite energy (Proposition~\ref{p.lower.matrix}). For the upper matrix it is not, and the test functions of Section~\ref{ss.intro.ideas} are built so that the coefficient field enters them only through harmonic replacements on simplices, whose energies are given by the upper matrices of the simplices.

\subsection{Main results}
\label{ss.intro.results}

\paragraph{Coarse-grained matrices.} The weighted spaces~$H^1_\a$ and~$\Hzero{\a}$ and the classes~$\Csol$ of solutions and~$\Csub$ of subsolutions are defined in Section~\ref{s.weighted.spaces}. Let~$U\subseteq\Rd$ be a bounded connected Lipschitz domain with~$\tr\a,\tr\a^{-1}\in L^1(U)$, as in~\eqref{e.weighted.hypotheses}. The \emph{coarse-grained matrix}~$\a(U)$ is the symmetric matrix whose quadratic form is the normalized energy of the~$\a$-harmonic function with affine boundary values:
\begin{equation}
	e\cdot\a(U)e=\min\biggl\{\fint_U\nabla w\cdot\a\nabla w:w\in H^1_\a(U),\ w-e\cdot x\in\Hzero{\a}(U)\biggr\}\qquad(e\in\Rd)\,.
	\label{e.upper.matrix.definition}
\end{equation}
The second coarse-grained matrix~$\a_*(U)$ is defined through its inverse, the symmetric matrix with quadratic form
\begin{equation}
	e\cdot\a_*^{-1}(U)e=\sup_{w\in\Csol(U)}\fint_U\bigl(-\nabla w\cdot\a\nabla w+2e\cdot\nabla w\bigr)\qquad(e\in\Rd)\,.
	\label{e.lower.matrix.definition}
\end{equation}
Section~\ref{s.harmonic.matrices} shows that the right sides of~\eqref{e.upper.matrix.definition} and~\eqref{e.lower.matrix.definition} are quadratic forms of symmetric positive definite matrices and records their properties. The hypothesis uses these matrices on the simplices~$\triangle$ of the triadic triangulations~$\mathscr T_k$ of~$\cu_0$ in~\eqref{e.simplex.family}. The triangulation~$\mathscr T_k$ splits each triadic subcube of side length~$3^{-k}$ into~$d!$ simplices, so it has~$d!\,3^{kd}$ simplices of diameter comparable to~$3^{-k}$.

\paragraph{Coarse ellipticity constants.} For~$s,t>0$ and~$1<p,q<\infty$, the coarse ellipticity assumption bounds the coarse ellipticity constants of Definition~\ref{d.cg.constants}, which are spatial moments of these matrices. For each~$k$ we take the~$L^p$ mean of the operator norms~$|\a(\triangle)|$ over the simplices of~$\mathscr T_k$, where~$\avsum$ denotes the arithmetic mean, and we sum their square roots over~$k$ with the discount~$3^{-ks}$:
\begin{equation*}
	\Lambda_{s,1,p}(\cu_0)=\biggl((1-3^{-s})\sum_{k\geq0}3^{-ks}\biggl(\avsum_{\triangle\in\mathscr T_k}|\a(\triangle)|^p\biggr)^{\nf1{2p}}\biggr)^2\,.
\end{equation*}
The lower constant is defined in the same way:
\begin{equation*}
	\lambda_{t,1,q}(\cu_0)=\biggl((1-3^{-t})\sum_{k\geq0}3^{-kt}\biggl(\avsum_{\triangle\in\mathscr T_k}|\a_*^{-1}(\triangle)|^q\biggr)^{\nf1{2q}}\biggr)^{-2}\,.
\end{equation*}
The middle index records the summability exponent over scales;
our regularity estimates use the value one for this index. At each scale, the spatial averages weight the simplices by their volume, and the factors~$3^{-ks}$ and~$3^{-kt}$ discount finer scales. A constant matrix~$\bfA$ has~$\Lambda_{s,1,p}(\cu_0)=|\bfA|$ and~$\lambda_{t,1,q}(\cu_0)=|\bfA^{-1}|^{-1}$. When~$\Lambda_{s,1,p}(\cu_0)<\infty$ and~$\lambda_{t,1,q}(\cu_0)>0$, we call the ratio
\begin{equation}
	\Theta\coloneqq\frac{\Lambda_{s,1,p}(\cu_0)}{\lambda_{t,1,q}(\cu_0)}
	\label{e.spatial.moment.ratio}
\end{equation}
the \emph{coarse ellipticity ratio}. The comparison inequality~\eqref{e.moment.comparison} gives~$\Theta\geq1$. Proposition~\ref{p.cubical.simplicial.equivalence} shows that replacing the simplices by triadic cubes gives equivalent coarse ellipticity constants when~$s<\nf12(1-1/p)$ and~$t<\nf12(1-1/q)$.

\paragraph{Local boundedness.} We use the exponents
\begin{equation}
	\sigma\coloneqq s+\frac{d-1}{2p}\,,\qquad \sigma_*\coloneqq t+\frac{d-1}{2q}\,,\qquad \theta\coloneqq1-\sigma-\sigma_*\,,\qquad r\coloneqq\frac{2q}{q+1}\in(1,2)\,.
	\label{e.intro.parameters}
\end{equation}
The exponent~$\sigma$ describes the growth in~$k$ allowed by~$\Lambda_{s,1,p}(\cu_0)$ for~$|\a(\triangle)|^{\nf12}$ on the simplices~$\triangle\in\mathscr T_k$ near the cube boundaries~$\partial(\tau\cu_0)$. The scale discount gives~$3^{ks}$, while selecting the largest contribution among the roughly~$3^{k(d-1)}$ simplices near a boundary gives~$3^{k(d-1)/(2p)}$. The lower constant~$\lambda_{t,1,q}(\cu_0)$ controls fractional boundary traces for almost every~$\tau$. Converting their~$L^r$ bounds into~$L^2$ bounds at scale~$3^{-k}$ leaves the exponent~$1-\sigma_*$, where~$r$ is the integrability of the controlled gradient averages. The condition~$\sigma+\sigma_*<1$, or~$\theta>0$, allows the decay from the boundary trace estimate to outweigh the growth of the upper matrices. The bound of Theorem~\ref{t.cg.local.boundedness} carries the power~$\nf{(d-1)}{(4\theta)}$ of~$\Theta$. Section~\ref{ss.intro.ideas} explains how these exponents enter the proof.

\begin{theoremA}[Local boundedness above]
\label{t.cg.local.boundedness}
Let~$d\geq3$, let~$1<p,q<\infty$ and~$s,t>0$ satisfy~$\theta>0$, and suppose that~$\Lambda_{s,1,p}(\cu_0)<\infty$ and~$\lambda_{t,1,q}(\cu_0)>0$. Every subsolution~$u\in\Csub(\cu_0)$ of~\eqref{e.subsolution.equation} is locally bounded above in~$\cu_0$. More precisely, there are~$C<\infty$ and~$\gamma>0$, depending only on~$d,p,q,s,t$, such that the positive part~$u_+=\max\{u,0\}$ satisfies, for~$\nf12\leq\rho_1<\rho_2\leq1$,
\begin{equation}
	\norm{u_+}_{L^\infty(\rho_1\cu_0)}\leq C(\rho_2-\rho_1)^{-\gamma}\Theta^{\nf{(d-1)}{4\theta}}\norm{u_+}_{L^2(\rho_2\cu_0)}\,,
	\label{e.cg.local.boundedness}
\end{equation}
and the right side is finite when~$\rho_2<1$.
\end{theoremA}

The power~$\nf{(d-1)}{(4\theta)}$ of~$\Theta$ is optimal for every fixed admissible choice of~$d,p,q,s,t$: Proposition~\ref{p.sharpness.polynomial} gives uniformly elliptic coefficients and solutions that change sign, for which~$\Theta$ is arbitrarily large and the ratio of the two sides of~\eqref{e.cg.local.boundedness} tends to infinity if the power is replaced by any smaller number, even for fixed~$\rho_1$ and~$\rho_2$. The corresponding bound of~\cite{AADKM} had the power~$\nf d{(4(1-s-t))}$ of its ellipticity ratio in place of~$\nf{(d-1)}{(4\theta)}$; here~$1-s-t$ is the value of~$\theta$ for~$p=q=\infty$, since the constants of~\cite{AADKM} take maxima over positions. The numerator~$d-1$ is the dimension of the cube boundaries on which the proof takes place. A subsolution need not be square integrable up to~$\partial\cu_0$,
so the right side of~\eqref{e.cg.local.boundedness} may be
infinite when~$\rho_2=1$.
Every subsolution belongs to~$L^r(\cu_0)$, and interpolation
gives analogous bounds in terms of~$L^\eta$ norms for~$0<\eta<2$, with the power~$\nf{(d-1)}{(2\eta\theta)}$
of~$\Theta$.

\begin{corollaryA}[Bounds in~$L^\eta$]
\label{c.lr.bound}
Under the hypotheses of Theorem~\ref{t.cg.local.boundedness}, for every~$0<\eta<2$ there is~$C<\infty$, depending only on~$\eta$ and~$d,p,q,s,t$, such that every subsolution~$u\in\Csub(\cu_0)$ satisfies, with~$\gamma$ as in Theorem~\ref{t.cg.local.boundedness} and for~$\nf12\leq\rho_1<\rho_2\leq1$,
\begin{equation}
	\norm{u_+}_{L^\infty(\rho_1\cu_0)}\leq C(\rho_2-\rho_1)^{-\nf{2\gamma}\eta}\Theta^{\nf{(d-1)}{2\eta\theta}}\norm{u_+}_{L^\eta(\rho_2\cu_0)}\,,
	\label{e.lr.bound}
\end{equation}
and the right side is finite when~$\rho_2<1$.
\end{corollaryA}

The same examples show that the power~$\nf{(d-1)}{(2\eta\theta)}$ in~\eqref{e.lr.bound} is optimal for every fixed~$0<\eta<2$.

\paragraph{Harnack inequalities.} A supersolution is a function~$u$ with~$-u\in\Csub(\cu_0)$. We derive the weak Harnack inequality for nonnegative supersolutions
and the Harnack inequality for nonnegative solutions in terms of
the coarse ellipticity constants.

\begin{theoremA}[Weak Harnack and Harnack inequalities]
\label{t.harnack}
Under the hypotheses of Theorem~\ref{t.cg.local.boundedness}, let~$r^*\coloneqq\nf{dr}{(d-(1-t)r)}$, the Sobolev exponent of~$W^{1-t,r}$. For every~$0<\eta\leq\nf{r^*}2$ there is~$C<\infty$, depending only on~$\eta$ and~$d,p,q,s,t$, such that every nonnegative supersolution~$u$ satisfies
\begin{equation}
	\biggl(\fint_{\frac58\cu_0}u^\eta\biggr)^{\nf1\eta}\leq\exp\bigl(C\sqrt\Theta\bigr)\essinf_{\frac12\cu_0}u\,,
	\label{e.weak.harnack}
\end{equation}
and there is~$C<\infty$, depending only on~$d,p,q,s,t$, such that every nonnegative solution~$u\in\Csol(\cu_0)$ satisfies
\begin{equation}
	\esssup_{\frac12\cu_0}u\leq\exp\bigl(C\sqrt\Theta\bigr)\essinf_{\frac12\cu_0}u\,.
	\label{e.harnack}
\end{equation}
\end{theoremA}

The range~$0<\eta\leq\nf{r^*}2$ of exponents in~\eqref{e.weak.harnack} does not depend on~$\Theta$. For each~$\eta$, the constant is~$\exp(C\sqrt\Theta)$, with~$C$ depending on~$\eta$ but not on~$\Theta$. The endpoint~$\eta=\nf{r^*}2$ uses the summability of the scale contributions to~$\lambda_{t,1,q}(\cu_0)$ (Section~\ref{ss.harnack.endpoint}). For each fixed~$d,p,q,s,t$, Proposition~\ref{p.sharpness.weak.harnack} gives a radial coefficient field satisfying the hypotheses for which~\eqref{e.weak.harnack} fails for every~$\eta>\nf{r^*}2$.

For uniformly elliptic coefficients the hypotheses hold for every~$1<p,q<\infty$ and~$s,t>0$ satisfying~$\theta>0$. The union of the ranges~$0<\eta\leq\nf{r^*}2$ over these parameters is~$0<\eta<\nf d{(d-2)}$, the range of the weak Harnack inequality of Moser and Trudinger~\cite{Moser2,Trudinger.1967}. The fundamental solution~$|x|^{2-d}$ shows that this range is sharp. Since the condition is open and~$t>0$, the classical endpoint~$\nf d{(d-2)}$ is not reached.

Trudinger~\cite{Trudinger,Trudinger.1973} carried Moser's method over to~$\a\in L^p$ and~$\a^{-1}\in L^q$ and proved the weak Harnack inequality for~$\nf1p+\nf1q<\nf2d$. Bella and Sch\"affner~\cite{Bella.Schaffner.regularity} proved it in the range~\eqref{e.intro.bs.range} for every exponent below the value of~$\nf{r^*}2$ at~$t=0$. Theorem~\ref{t.harnack} reaches the same exponents through Theorem~\ref{t.coefficient.conditions}\ref{c.classical.moments}, which allows every~$t>0$.

The growth of the Harnack constant in~$\Theta$ cannot be improved, even for constant coefficients. For~$\Theta\geq1$, the example of~\cite{AADKM} takes the constant coefficient field and positive solution
\begin{equation*}
	\a\coloneqq\operatorname{diag}(1,\Theta,1,\ldots,1)\qquad\mbox{and}\qquad u(x)\coloneqq\exp(\sqrt\Theta\,x_1)\cos x_2\,,
\end{equation*}
whose coarse ellipticity ratio is~$\Theta$. For this solution,
\begin{equation*}
	\esssup_{\frac12\cu_0}u\geq\exp(\nf{\sqrt\Theta}2)\essinf_{\frac12\cu_0}u\,.
\end{equation*}

For uniformly elliptic coefficients the constant~$\exp(C\sqrt\Theta)$ is that of Bombieri and Giusti~\cite{Bombieri.Giusti}. In dimension~$d\geq3$, Guerra~\cite{Guerra} constructed, for measurable coefficients with eigenvalues between~$K^{-1}$ and~$K$, solutions whose H\"older exponent is only~$\exp(-cK)$; up to the value of~$c$, this is the exponent that follows from the Harnack inequality of Bombieri and Giusti. Such coefficients have~$\Theta\leq K^2$ on every cube, so a Harnack constant~$\exp(C\Theta^\gamma)$ with~$\gamma<\nf12$ would contradict his examples.

\paragraph{Conditions on the coefficients.} Averages of~$\a$ and~$\a^{-1}$ dominate~$\a(\triangle)$ and~$\a_*^{-1}(\triangle)$ on every simplex, so classical moments of the coefficients control the coarse ellipticity constants with constant one. Only the averages of~$\a$ and~$\a^{-1}$ over the simplices enter this comparison, so integrability can be replaced by negative regularity, which we measure by a Besov-type quasi-norm built from averages over triadic cubes, as in~\cite{AK.HC,AADKM}. Recall that~$\mathscr T_k$ is the triangulation of~$\cu_0$ obtained by splitting each triadic subcube of side length~$3^{-k}$, that is, each cube~$z+\cu_{-k}$ with~$z\in3^{-k}\Zd\cap\cu_0$, into its~$d!$ coordinate-order simplices; the precise definition is~\eqref{e.simplex.family}. For~$s>0$,~$1\leq p\leq\infty$ and a matrix field~$\mathbf b$ with~$|\mathbf b|\in L^1(\cu_0)$, let
\begin{equation}
	\norm{\mathbf b}_{\mathring{\underline B}^{-2s}_{p,\nf12}(\cu_0)}\coloneqq\biggl(\sum_{k\geq0}3^{-ks}\biggl(\avsum_{z\in3^{-k}\Zd\cap\cu_0}|(\mathbf b)_{z+\cu_{-k}}|^p\biggr)^{\nf1{2p}}\biggr)^2\,.
	\label{e.negative.regularity.norm}
\end{equation}
For~$p=\infty$ the mean over~$z$ is replaced by the maximum. Here and below, the~$L^p$ norm of a matrix field is that of its operator norm. The expression in~\eqref{e.negative.regularity.norm} is a Besov-type quasi-norm based on cube averages, with order~$-2s$ and fine index~$\nf12$, as in~\cite[Section~2.4]{AK.HC} and~\cite[Section~2.2]{AADKM}. It discounts finer scales and sums square roots, as does~$\Lambda_{s,1,p}(\cu_0)$. For positive semidefinite fields, such as~$\a$ and~$\a^{-1}$, it is comparable, up to factors depending only on~$d$,~$p$ and~$s$, to the Besov quasi-norm~$B^{-2s}_{p,\nf12}(\Rd)$ of the extension of~$\mathbf b$ by zero. Proposition~\ref{p.besov.averages} and the heat characterization of~\cite[Definition~6.1 and Theorem~6.7]{Kerkyacharian.Petrushev} give this comparison in Appendix~\ref{a.besov}.

\begin{theoremA}[Conditions on the coefficients]
\label{t.coefficient.conditions}
Let~$1<p,q\leq\infty$ and~$s,t>0$.
\begin{enumerate}[label=\textup{(\roman*)},wide]
\item\label{c.classical.moments} If~$|\a|\in L^p(\cu_0)$ and~$|\a^{-1}|\in L^q(\cu_0)$, where~$|\cdot|$ is the operator norm, then
\begin{equation}
	\Lambda_{s,1,p}(\cu_0)\leq\norm{\a}_{L^p(\cu_0)}\qand \lambda_{t,1,q}(\cu_0)^{-1}\leq\norm{\a^{-1}}_{L^q(\cu_0)}\,,
	\label{e.classical.series}
\end{equation}
and hence~$\Theta\leq\norm{\a}_{L^p(\cu_0)}\norm{\a^{-1}}_{L^q(\cu_0)}$.
\item\label{c.negative.regularity} If~$\norm{\a}_{\mathring{\underline B}^{-2s}_{p,\nf12}(\cu_0)}$ and~$\norm{\a^{-1}}_{\mathring{\underline B}^{-2t}_{q,\nf12}(\cu_0)}$ are finite, then
\begin{equation}
	\Lambda_{s,1,p}(\cu_0)\leq d!\,(1-3^{-s})^2\norm{\a}_{\mathring{\underline B}^{-2s}_{p,\nf12}(\cu_0)}\qand\lambda_{t,1,q}(\cu_0)^{-1}\leq d!\,(1-3^{-t})^2\norm{\a^{-1}}_{\mathring{\underline B}^{-2t}_{q,\nf12}(\cu_0)}\,,
	\label{e.negative.regularity.series}
\end{equation}
and hence~$\Theta\leq(d!)^2\norm{\a}_{\mathring{\underline B}^{-2s}_{p,\nf12}(\cu_0)}\norm{\a^{-1}}_{\mathring{\underline B}^{-2t}_{q,\nf12}(\cu_0)}$.
\item\label{c.negative.sobolev} If~$0<\eps\leq2\min\{s,t\}$ and~$\norm{\a}_{W^{-2s+\eps,p}(\cu_0)}$ and~$\norm{\a^{-1}}_{W^{-2t+\eps,q}(\cu_0)}$ are finite, where these are the norms~\eqref{e.negative.sobolev.norm}, dual to fractional Sobolev norms and tested on bounded functions, then there is~$C<\infty$, depending only on~$d,p,q,s,t,\eps$, such that
\begin{equation}
	\Lambda_{s,1,p}(\cu_0)\leq C\norm{\a}_{W^{-2s+\eps,p}(\cu_0)}\qand\lambda_{t,1,q}(\cu_0)^{-1}\leq C\norm{\a^{-1}}_{W^{-2t+\eps,q}(\cu_0)}\,,
	\label{e.negative.sobolev.series}
\end{equation}
and hence~$\Theta\leq C^2\norm{\a}_{W^{-2s+\eps,p}(\cu_0)}\norm{\a^{-1}}_{W^{-2t+\eps,q}(\cu_0)}$.
\end{enumerate}
\end{theoremA}

Under the hypotheses of each part of Theorem~\ref{t.coefficient.conditions},~$\Lambda_{s,1,p}(\cu_0)<\infty$ and~$\lambda_{t,1,q}(\cu_0)>0$. Hence Theorems~\ref{t.cg.local.boundedness} and~\ref{t.harnack} and Corollary~\ref{c.lr.bound} hold, with the stated bound on~$\Theta$, whenever~$d\geq3$ and the indices can be chosen with
\begin{equation}
	\theta\coloneqq1-s-t-\frac{d-1}{2}\biggl(\frac1p+\frac1q\biggr)>0\,.
	\label{e.negative.regularity.range}
\end{equation}
At each scale, a finite~$L^p$ mean over~$\mathscr T_k$ is at most the maximum. Thus~$\Lambda_{s,1,p}(\cu_0)$ is bounded above by its~$p=\infty$ version, while~$\lambda_{t,1,q}(\cu_0)$ is bounded below by its~$q=\infty$ version. The strict condition~\eqref{e.negative.regularity.range} then permits nearby finite exponents. For part~\ref{c.classical.moments}, whose bounds hold for all~$s,t>0$, indices with~\eqref{e.negative.regularity.range} exist exactly when
\begin{equation}
	\frac1p+\frac1q<\frac{2}{d-1}\,,
	\label{e.classical.moment.range}
\end{equation}
since~$\theta<1-\frac{d-1}2(\nf1p+\nf1q)$ for all~$s,t>0$. Conversely, under~\eqref{e.classical.moment.range}, choose
\begin{equation*}
	s=t=\frac13\biggl(1-\frac{d-1}2\Bigl(\frac1p+\frac1q\Bigr)\biggr)\,,\qquad \theta=s>0\,.
\end{equation*}
With this choice the constants depend only on~$d$,~$p$ and~$q$ (and on~$\eta$ in Corollary~\ref{c.lr.bound} and in~\eqref{e.weak.harnack}), and the Harnack constant of Theorem~\ref{t.harnack} is~$\exp(C(\norm{\a}_{L^p(\cu_0)}\norm{\a^{-1}}_{L^q(\cu_0)})^{\nf12})$.

\smallskip

For part~\ref{c.negative.regularity} the condition~\eqref{e.negative.regularity.range} applies to the indices of the quasi-norms themselves, since the quasi-norms use the same scale discounts as the coarse ellipticity constants and likewise sum square roots over the scales; this is why the fine index is~$\nf12$.

\smallskip

For finite~$p$ and~$q$, part~\ref{c.negative.sobolev} gives Corollary~\ref{c.sobolev.coefficients}, proved in Section~\ref{s.classical.moments}. Compared with the negative Sobolev criterion of~\cite{AADKM}, the corollary replaces~$d$ by~$d-1$ and removes the restrictions~$\alpha<1-\nf1p$ and~$\beta<1-\nf1q$.

\begin{corollaryA}[Coefficients in negative Sobolev spaces]
\label{c.sobolev.coefficients}
Let~$d\geq3$,~$1<p,q<\infty$ and~$\alpha,\beta\geq0$, and set
\begin{equation*}
	\theta\coloneqq1-\frac{\alpha+\beta}2-\frac{d-1}2\Bigl(\frac1p+\frac1q\Bigr)\,.
\end{equation*}
If~$\theta>0$ and~$\a\in W^{-\alpha,p}(\cu_0)\cap L^1(\cu_0)$ and~$\a^{-1}\in W^{-\beta,q}(\cu_0)\cap L^1(\cu_0)$, with the norms~\eqref{e.negative.sobolev.norm}, then there are~$C<\infty$ and~$\gamma>0$, depending only on~$d,p,q,\alpha,\beta$, such that every subsolution~$u\in\Csub(\cu_0)$ satisfies, for~$\nf12\leq\rho_1<\rho_2\leq1$,
\begin{equation}
	\norm{u_+}_{L^\infty(\rho_1\cu_0)}\leq C(\rho_2-\rho_1)^{-\gamma}\bigl(\norm{\a}_{W^{-\alpha,p}(\cu_0)}\norm{\a^{-1}}_{W^{-\beta,q}(\cu_0)}\bigr)^{\nf{(d-1)}{2\theta}}\norm{u_+}_{L^2(\rho_2\cu_0)}\,,
	\label{e.sobolev.local.boundedness}
\end{equation}
with a finite right side when~$\rho_2<1$, and every nonnegative solution~$u\in\Csol(\cu_0)$ satisfies
\begin{equation}
	\esssup_{\frac12\cu_0}u\leq\exp\Bigl(C\bigl(\norm{\a}_{W^{-\alpha,p}(\cu_0)}\norm{\a^{-1}}_{W^{-\beta,q}(\cu_0)}\bigr)^{\nf12}\Bigr)\essinf_{\frac12\cu_0}u\,.
	\label{e.sobolev.harnack}
\end{equation}

\end{corollaryA}

\paragraph{Sharpness.} For~$1<p,q<\infty$, the range of Corollary~\ref{c.sobolev.coefficients} cannot be relaxed, even at its boundary, if the negative Sobolev norms of orders~$-\alpha$ and~$-\beta$ are replaced by the quasi-norms~\eqref{e.negative.regularity.norm} of the same orders, with Lebesgue norms at a zero order. For~$\alpha>0$, the quasi-norm of order~$-\alpha$ is~\eqref{e.negative.regularity.norm} with~$s=\nf\alpha2$. In terms of Theorem~\ref{t.coefficient.conditions}, neither~\eqref{e.negative.regularity.range} for part~\ref{c.negative.regularity} nor~\eqref{e.classical.moment.range} for part~\ref{c.classical.moments} can be relaxed, even at the boundary.

\begin{theoremA}[Sharpness of the range]
\label{t.sharpness}
Let~$d\geq3$,~$1<p,q<\infty$ and~$\alpha,\beta\geq0$ satisfy
\begin{equation*}
	\theta_0\coloneqq1-\frac{\alpha+\beta}2-\frac{d-1}{2}\biggl(\frac1p+\frac1q\biggr)\leq0\,.
\end{equation*}
Writing~$x=(x_1,y)$ with~$y\in\R^{d-1}$, there are a measurable function~$a:(-\nf12,\nf12)^{d-1}\to(0,\infty)$ and a solution~$u\in\Csol(\cu_0)$ for~$\a(x)\coloneqq a(y)I$ such that~$\tr\a,\tr\a^{-1}\in L^1(\cu_0)$ and~$u\geq1$ almost everywhere. The coefficient and its inverse satisfy
\begin{equation*}
\begin{aligned}
	\norm{\a}_{\mathring{\underline B}^{-\alpha}_{p,\nf12}(\cu_0)}&<\infty\quad\text{if }\alpha>0\,, & \norm{\a}_{L^p(\cu_0)}&<\infty\quad\text{if }\alpha=0\,,\\
	\norm{\a^{-1}}_{\mathring{\underline B}^{-\beta}_{q,\nf12}(\cu_0)}&<\infty\quad\text{if }\beta>0\,, & \norm{\a^{-1}}_{L^q(\cu_0)}&<\infty\quad\text{if }\beta=0\,.
\end{aligned}
\end{equation*}
Moreover,
\begin{equation}
	\esssup_{\frac12\cu_0}u=\infty\qquad\mbox{and}\qquad 0<\essinf_{\frac12\cu_0}u<\infty\,.
	\label{e.sharpness}
\end{equation}
In fact,~$u$ is essentially unbounded in every neighborhood of every point of~$\{(x_1,0):|x_1|<\nf12\}$.
\end{theoremA}

By~\eqref{e.sharpness}, both local boundedness and the Harnack inequality fail. The coefficient is built from disjoint thin cylinders parallel to the~$x_1$-axis, whose axes accumulate on the line~$\{y=0\}$; each consists of a conducting inner cylinder surrounded by an insulating annulus. An inner cylinder with a large conductivity inside an annulus with a small one supports a large nonnegative subsolution of small energy. A cylinder of radius~$\eps$ occupies a fraction of order~$\eps^{d-1}$ of the simplices of size~$\eps$, so the averages over simplices and cubes see it only through its codimension~$d-1$. An inner conductivity of order~$\eps^{-\kappa}$ and an annular one of order~$\eps^{2-\kappa}$ then make contributions bounded independently of~$\eps$ if~$\kappa\leq\nf{(d-1)}p+\alpha$ and~$2-\kappa\leq\nf{(d-1)}q+\beta$. Such a~$\kappa\in(0,2)$ exists exactly when~$\theta_0\leq0$, and this is why the threshold of Theorem~\ref{t.sharpness} is exactly the boundary of the range of Corollary~\ref{c.sobolev.coefficients}. Since the two conductivities have ratio of order~$\eps^2$, the growth rate of the subsolution along the axis does not depend on the radius, so the radii can be chosen so small that the norm~\eqref{e.weighted.norm} of the sum of these subsolutions is finite. Harmonic replacement then gives a solution above that sum, which is unbounded near the line. For~$\alpha=\beta=0$, Theorem~\ref{t.sharpness} gives scalar coefficients with~$\a\in L^p$,~$\a^{-1}\in L^q$ and~$\nf1p+\nf1q\geq\nf2{(d-1)}$, and unbounded solutions, in every dimension~$d\geq3$. The proof is in Section~\ref{s.sharpness}.

\subsection{Relation to previous work}
\label{ss.intro.discussion}

\paragraph{Comparison with~\cite{AADKM}.} The present hypothesis differs from that of~\cite{AADKM} in three ways. First, the upper matrix is computed from solutions alone. In~\cite{AADKM} the upper matrix of a cube is a supremum over a cone of functions adapted to the argument, such as truncations and powers of solutions, and such a supremum need not be a quadratic form; here the Dirichlet energy with affine boundary values suffices, even for bounding subsolutions. Second, the matrices are attached to simplices rather than cubes, because the construction of the test functions begins with a piecewise affine extension on simplices. For the matrices used here, Proposition~\ref{p.cubical.simplicial.equivalence} shows that these two geometries give equivalent constants in the range~\eqref{e.cubical.simplicial.range}. Third, the matrices enter through spatial moments rather than maxima over positions. The variational classes and the normalizations still differ, so this geometric comparison alone does not establish an implication between the two hypotheses in general. The conclusions have the same form, with~$d-1$ in place of~$d$ in the admissible ranges of the coefficient conditions and in the power of the coarse ellipticity ratio in the supremum bound, and with the same growth~$\exp(C\sqrt\Theta)$ of the Harnack constant. The arguments are different: the Caccioppoli inequalities of~\cite{AADKM} are coarse-grained versions of the classical ones, whereas our arguments construct the test function depending on the subsolution and the coefficient field. On the other hand,~\cite{AADKM} covers dimension two, and its examples are not analyzed here.

\paragraph{Comparison with the result of Bella and Sch\"affner.} The results of~\cite{Bella.Schaffner.regularity} hold for possibly nonsymmetric coefficients, in dimension two and for infinite exponents, in solution spaces normed with~$\int|\a|u^2$ in place of the squared mean; Section~\ref{s.classical.moments} records the precise hypotheses and shows that, in the symmetric case with~$d\geq3$, Theorem~\ref{t.coefficient.conditions}\ref{c.classical.moments} reaches their range. Both proofs use the Sobolev inequality in dimension~$d-1$
and test functions that depend on the subsolution.

\paragraph{Random media and scale.} Moment conditions of the type~\eqref{e.intro.bs.range} are natural for random walks and diffusions among random degenerate conductances, where Harnack inequalities under moment conditions enter invariance principles and local limit theorems~\cite{Biskup,Andres.Deuschel.Slowik.2015,Andres.Deuschel.Slowik.2016,Chiarini.Deuschel.2016,Chiarini.Deuschel.2015,Bella.Schaffner.RCM,Bella.Schaffner.parabolic}. The coarse ellipticity condition was designed for random coefficient fields with high contrast~\cite{AK.HC,AK.survey}, and Lau~\cite{Lau} proved qualitative homogenization under it.

Like those of~\cite{Trudinger,Bella.Schaffner.regularity,AADKM}, our estimates are at unit scale. The coarse ellipticity ratio of smaller cubes may grow as the cubes shrink, so iterating the estimates does not give H\"older continuity; this is different from weighted theories such as that of Fabes, Kenig and Serapioni~\cite{FKS}, in which a scale-invariant weighted Poincar\'e inequality is assumed.

% \paragraph{Limitations.} We treat symmetric coefficients in dimension~$d\geq3$, and not nonsymmetric coefficients or dimension two. The constants deteriorate as~$\theta\downarrow0$. We do not know whether the boundary counterexamples of Theorem~\ref{t.sharpness} extend to~$p=\infty$ or~$q=\infty$. The optimality concerns only the power of~$\Theta$; we do not determine the optimal powers of~$\rho_2-\rho_1$, nor those of the coefficient norms and quasi-norms in Theorem~\ref{t.coefficient.conditions} and Corollary~\ref{c.sobolev.coefficients}, where~$\Theta$ is only bounded from above.

\subsection{Ideas of the proof}
\label{ss.intro.ideas}

The main ingredient in our regularity theory is a Caccioppoli inequality with constants depending only on the coarse ellipticity parameters. We construct a test function that agrees with an approximation of the subsolution inside a cube and extends its boundary values piecewise harmonically outside. Testing bounds the energy by an exterior integral, which we control through the fractional Sobolev seminorm and the~$L^2$ norm of the boundary trace. The cube on which this argument is performed is chosen with a good radius, at which the trace norms, the energy in boundary layers and the coarse-grained matrix norms near the boundary are controlled simultaneously. The Sobolev inequality on the~$(d-1)$-dimensional cube boundary then gives the superlinear gain needed for the De Giorgi iteration proving local boundedness. We outline this argument first, then explain how power and logarithmic estimates give Harnack inequalities and how a separate potential argument reaches the weak Harnack endpoint.

\paragraph{The Caccioppoli inequality.} Fix two radii~$\nf12\leq\rho_1<\rho_2\leq1$ and let~$v$ be a nonnegative subsolution in~$\cu_0$. Using the classical test function~$\varphi^2v$, where~$\varphi\in C^\infty_c(\rho_2\cu_0)$ equals one on~$\rho_1\cu_0$, produces the term~$\int v^2\nabla\varphi\cdot\a\nabla\varphi$, which is not controlled by our coarse ellipticity constants. Although the matrices~$\a(\triangle)$ can be used to construct a cutoff function of controlled energy, this is not enough to control the energy with weight~$v^2$. As in Bella and Sch\"affner's approach, the test function needs to depend on~$v$. Indeed, we construct a test function that agrees with a nonnegative smooth approximation~$v_i$ of~$v$ inside a cube~$\tau\cu_0$, where~$\tau\in(\rho_1,\rho_2)$ is to be chosen later, and extends its boundary values outside.

For Lipschitz boundary data~$f$, we construct a continuous piecewise affine extension~$L_hf$ on the Whitney simplices of the exterior of~$\tau\cu_0$ (see Figure~\ref{f.whitney.vertices}). Its free-vertex values are boundary averages multiplied by a cutoff, and its hanging-vertex values are determined by interpolation. The averaging patches have size comparable to the adjacent simplices, and the cutoff confines the support to~$(\tau+3h)\cu_0$, for a width~$h$ chosen later. We then replace~$L_hf$ on each simplex~$\triangle$ by the~$\a$-harmonic function with the same boundary values. If~$e$ is the affine gradient on~$\triangle$, this replacement has energy exactly~$|\triangle|\,e\cdot\a(\triangle)e$. This is how a matrix defined through solutions enters an estimate for a subsolution. Thus the energy of the resulting piecewise harmonic extension~$H_hf$ is expressed through the upper matrices and differences of boundary averages (Section~\ref{s.whitney.extension}). With~$f=v_i|_{\partial(\tau\cu_0)}$, the function equal to~$v_i$ inside the cube and to~$H_hf$ outside is an admissible nonnegative test function. Testing the subsolution inequality gives
\begin{equation*}
    \int_{\tau\cu_0}\nabla v_i\cdot\a\nabla v
    \leq
    -\int_{\cu_0\setminus\tau\overline{\cu}_0}
        \nabla H_hf\cdot\a\nabla v
    \,.
\end{equation*}
The left-hand side converges to~$\Ener_{\tau\cu_0}(v)$ as~$i\to\infty$. It remains to bound the exterior integral on the right.

\paragraph{Fractional regularity and traces.} Every~$w\in H^1_\a(\cu_0)$ belongs to the fractional Sobolev space~$W^{1-t,r}(\cu_0)$, where~$r=\nf{2q}{(q+1)}$, and satisfies
\begin{equation*}
    [w]_{W^{1-t,r}(\cu_0)}
    \leq C\lambda_{t,1,q}(\cu_0)^{-\nf12}
    \Ener_{\cu_0}(w)^{\nf12}
    \,.
\end{equation*}
This continuous embedding (Proposition~\ref{p.lower.fractional}) is proved by reconstructing~$w$ from its gradient averages (Lemma~\ref{l.fractional.reconstruction}), which are controlled by the energy (Lemma~\ref{l.lower.averages}). A slicing argument (Lemma~\ref{l.fractional.slicing}) then places the trace of~$w$ in~$W^{1-t,r}(\partial(\tau\cu_0))$ for almost every~$\tau$. These restrictions retain the order of regularity~$1-t$ because their norms are estimated on average over the family of cube boundaries.

\paragraph{Choosing a good radius.} The extension~$H_hf$ uses the upper matrices only on simplices near~$\partial(\tau\cu_0)$. Let~$A_k(\tau)$ be the largest~$|\a(\triangle)|$ among the simplices of~$\mathscr T_k$ within distance~$C3^{-k}$ of this boundary. Each simplex lies this close to the boundary only for radii in an interval of length~$C3^{-k}$. Consequently, integrating~$A_k(\tau)^p$ over~$\tau$ gives at most~$C3^{k(d-1)}$ times the spatial mean of~$|\a(\triangle)|^p$. This accounts for the boundary dimension~$d-1$ in the exponent~$\sigma=s+\nf{d-1}{(2p)}$ and bounds the discounted sum
\begin{equation*}
    S(\tau)\coloneqq\sum_{k\geq0}3^{-k\sigma}A_k(\tau)^{\nf12}
    \,.
\end{equation*}
Its~$L^{2p}$ norm in the radius is controlled by~$\Lambda_{s,1,p}(\cu_0)^{\nf12}$.

Chebyshev's inequality, the averaged trace estimates and a maximal function estimate for the energy give a radius~$\tau\in(\rho_1,\rho_2)$ at which~$S(\tau)$, the fractional norm of the trace and the energy in thin boundary layers are controlled simultaneously (Proposition~\ref{p.good.radius}). At the same radius, a subsequence of the approximations~$v_i$ and all its truncations have convergent traces.

\paragraph{Estimating the exterior integral.} In estimating the exterior integral, we apply Cauchy--Schwarz separately on the Whitney simplices of each size~$3^{-j}$. Write~$D_v(\tau)$ for the maximal function of the boundary energy defined in~\eqref{e.boundary.energy.maximal}. The energy of~$v$ on these simplices is at most~$C3^{-j}D_v(\tau)$, and the estimate for the extension also contains the boundary-layer volume factor~$3^{-j}$. After taking square roots, these factors cancel the gradient scale~$3^j$ of the affine extension.

The remaining factors measure the regularity of the boundary values and the growth of the upper matrices. Comparing boundary averages in~$L^2$ using their~$L^r$ control costs~$3^{j(d-1)/(2q)}$. Together with the fractional order~$1-t$, this leaves the exponent~$1-\sigma_*$, where~$\sigma_*=t+\nf{d-1}{(2q)}$. The weight in~$S(\tau)$ allows the square roots of the upper matrices to grow like~$3^{j\sigma}$. Their product is
\begin{equation*}
    3^{j\sigma}3^{-j(1-\sigma_*)}=3^{-j\theta}
    \,,\qquad \theta=1-\sigma-\sigma_*
    \,.
\end{equation*}
Only sizes~$3^{-j}<h$ occur in the support of the extension, so~$\theta>0$ gives a factor~$h^\theta$. Lemma~\ref{l.exterior.integral} makes this balance precise:
\begin{equation*}
    \biggl|\int_{\cu_0\setminus\tau\overline{\cu}_0}
        \nabla H_hf\cdot\a\nabla v\biggr|
    \leq CS(\tau)D_v(\tau)^{\nf12}
    \Bigl(
        h^\theta[f]_{W^{1-t,r}(\partial(\tau\cu_0))}
        +h^{-\sigma}\norm{f}_{L^2(\partial(\tau\cu_0))}
    \Bigr)
    \,.
\end{equation*}
The second term comes from cutting off the extension at width~$h$.

We apply this estimate to the traces of~$v_i$. As~$i\to\infty$, the interior pairing converges to the energy of~$v$ in~$\tau\cu_0$, and we pass to the limit in the estimates for the exterior integrals without extending the trace of~$v$ itself. Combining the estimates at the good radius produces an energy term on~$\rho_2\cu_0$ with coefficient~$C(\rho_2-\rho_1)^{-\gamma_1}\Theta^{\nf12}h^\theta$, while the remaining terms are controlled by the~$L^2$ norm of~$v$. Choosing~$h$ to make this coefficient small and iterating over radii gives the Caccioppoli inequality of Proposition~\ref{p.cg.caccioppoli}. Thus~$\theta>0$ expresses exactly the balance between trace regularity and matrix growth needed to absorb the energy on~$\rho_2\cu_0$. \paragraph{Local boundedness and the power of~$\Theta$.} A refinement of these arguments is needed to obtain local boundedness. Let~$u$ be a subsolution in~$\cu_0$. To compare two levels~$0\leq a<b$, write~$w_a=(u-a)_+$ and~$w_b=(u-b)_+$. We apply Proposition~\ref{p.good.radius.energy} with~$v=w_a$ and~$k=b-a$. This bounds the energy of~$w_b$ through its boundary trace, using the boundary-layer energy of~$w_a$. Since~$w_b=(w_a-(b-a))_+$, its fractional seminorm is at most that of~$w_a$. Its~$L^2$ norm gains an additional factor from the difference between the levels:
\begin{equation*}
    \norm{w_b}_{L^2(\partial(\tau\cu_0))}^2
    \leq
    (b-a)^{2-r_\partial^*}
    \norm{w_a}_{L^{r_\partial^*}(\partial(\tau\cu_0))}^{r_\partial^*}
    \,.
\end{equation*}
The boundary Sobolev exponent satisfies~$r_\partial^*>2$, so increasing~$b-a$ makes this term smaller.

The good radius estimates control the fractional norm of the trace by both the energy and the~$L^r$ norm of~$w_a$. We therefore track the quantity
\begin{equation*}
    Y_a(\rho)\coloneqq
    \Bigl(
        \norm{w_a}_{L^r(\rho\cu_0)}^2
        +\lambda_{t,1,q}(\cu_0)^{-1}
        \Ener_{\rho\cu_0}(w_a)
    \Bigr)^{\nf12}
    \,,
\end{equation*}
as in~\eqref{e.two.level.quantity}. The boundary estimate controls the energy of~$w_b$, while the fractional Sobolev inequality on~$\rho_1\cu_0$, with~$r^*>r$, controls its~$L^r$ norm. Choosing the extension width~$h$ makes the seminorm contribution small. Combining the boundary~$L^2$ estimate and the interior~$L^r$ estimate gives the recurrence of Proposition~\ref{p.two.level.recurrence}.

Apart from factors depending on the gap between the radii, the boundary contribution to this recurrence is
\begin{equation*}
    \Theta^{\nf{1-\sigma_*}\theta}
    \left(\frac{Y_a(\rho_2)}{b-a}\right)^{r_\partial^*-2}
    Y_a(\rho_2)^2
    \,.
\end{equation*}
Thus we choose the next level so that~$b-a$ is a large multiple of~$\Theta^\mu Y_a(\rho_2)$, where
\begin{equation*}
    \mu
    =\frac{1-\sigma_*}{\theta(r_\partial^*-2)}
    =\frac{d-3+2\sigma_*}{4\theta}
    \,.
\end{equation*}
This choice cancels the power of~$\Theta$ in the boundary term and also makes the interior term small. The level increments include additional factors accounting for the shrinking gaps between radii. The resulting values of~$Y$ decrease geometrically, and the level increments have a finite sum. If~$a_\infty$ is the limit of the levels, then~$(u-a_\infty)_+=0$ on~$\rho_1\cu_0$, which bounds the supremum (Proposition~\ref{p.energy.to.sup}).

On fixed nested cubes, this argument and the Caccioppoli inequality give, respectively,
\begin{equation*}
	\norm{u_+}_{L^\infty(\frac12\cu_0)}\leq C\Theta^\mu Y_0(\tfrac34)\qquad\mbox{and}\qquad Y_0(\tfrac34)\leq C\Theta^{\nf{1-\sigma_*}{2\theta}}\norm{u_+}_{L^2(\cu_0)}\,.
\end{equation*}
The total power of~$\Theta$ is therefore
\begin{equation*}
	\mu+\frac{1-\sigma_*}{2\theta}=\frac{d-1}{4\theta}\,,
\end{equation*}
as in Theorem~\ref{t.cg.local.boundedness}.

\paragraph{Harnack inequalities.}
Let~$u$ be a nonnegative supersolution and set~$w=u+\eps$
with~$\eps>0$.
For~$m<\nf12$ with~$m\neq0$, we apply the piecewise harmonic
test function construction to approximations and truncations
of~$w^m$.
In the power-testing inequality~\eqref{e.power.testing},
the interior terms bound the energy with coefficient~$1-2m$,
while the exterior pairing is multiplied by~$m$.
Taking its absolute value and dividing by~$1-2m$ therefore
gives the factor
\begin{equation*}
    c_m\coloneqq\frac{|m|}{1-2m}
    \,.
\end{equation*}
Estimating the exterior integral gives the Caccioppoli inequality
for powers (Proposition~\ref{p.power.caccioppoli}), which bounds
the energy of~$w^m$ on an interior cube by the square of its~$L^r$ norm on a larger cube.
The fractional Sobolev inequality then bounds its~$L^{r^*}$ norm
on a smaller cube.
For~$m>0$, this compares moments of~$w$ of orders~$mr$ and~$mr^*$;
for~$m<0$, it compares moments of~$w^{-1}$ of orders~$|m|r$
and~$|m|r^*$.
Thus each application increases the moment order by the
factor~$\nf{r^*}r>1$.

To connect the positive and negative moments, we first divide the Caccioppoli inequality for powers by~$m^2$ and let~$m\uparrow0$. This bounds the energy of~$\log w$ by~$C\Lambda_{s,1,p}(\cu_0)$. Fractional reconstruction consequently bounds its deviation from its mean on~$\frac78\cu_0$ in~$L^r$ by~$C\sqrt\Theta$. We choose~$p_*=\nf c{\sqrt\Theta}$, with~$c>0$ sufficiently small, so that the corresponding deviation of~$\log(w^{p_*})$ is bounded independently of~$\Theta$. At these small exponents the factor~$1+c_m^2\Theta$ in the moment estimates is also bounded. These estimates and the logarithmic bound allow us to apply Bombieri's lemma, Lemma~\ref{l.bombieri}, to obtain
\begin{equation*}
    \fint_{\frac34\cu_0}w^{p_*}
    \fint_{\frac34\cu_0}w^{-p_*}
    \leq C
    \,.
\end{equation*}

This crossover estimate bounds the small positive moment by the reciprocal of the small negative moment. Negative-power iteration bounds that reciprocal by the infimum of~$w$ on a smaller cube. Positive-power iteration raises the positive moment to the fixed order~$\nf r4$. Combining these three estimates gives the weak Harnack inequality at~$\nf r4$, with constant~$\exp(C/p_*)=\exp(C\sqrt\Theta)$. For solutions, Corollary~\ref{c.lr.bound} bounds the supremum by this~$L^{\nf r4}$ norm. Its polynomial factor in~$\Theta$ is absorbed into the exponential, giving the Harnack inequality. All bounds are uniform in~$\eps$, which is then sent to zero (Section~\ref{s.harnack}).

\paragraph{The endpoint weak Harnack inequality.} The Caccioppoli inequality for powers requires~$m<\nf12$, and~$c_m$ diverges as~$m\uparrow\nf12$. It therefore does not give the endpoint~$\eta=\nf{r^*}2$. For this endpoint we separate a supersolution into a potential and a solution. Let~$\nu$ be the restriction of the nonnegative measure~$-\nabla\cdot\a\nabla u$ to~$\frac34\cu_0$, and let~$v$ be its zero-boundary potential in~$\cu_0$. This means that~$v$ vanishes on~$\partial\cu_0$ and satisfies~$-\nabla\cdot\a\nabla v=\nu$ in~$\cu_0$. Comparison gives~$0\leq v\leq u$, and~$u-v$ is a nonnegative solution in~$\frac34\cu_0$. The Harnack inequality controls this remainder, so it remains to estimate~$v$.

The potential estimate of Proposition~\ref{p.endpoint.potential} gives
\begin{equation*}
    \norm{v}_{L^{\nf{r^*}2}(\cu_0)}
    \leq
    C\lambda_{t,1,q}(\cu_0)^{-1}\nu(\cu_0)
    \,.
\end{equation*}
Lemma~\ref{l.dirichlet.reconstruction} decomposes~$v$ scale by scale, and the resulting distribution estimate gives this bound. The summability of the discounted scale terms in~$\lambda_{t,1,q}(\cu_0)$ allows the distribution bounds to be integrated at the endpoint exponent. The capacitary argument of Lemma~\ref{l.source.mass} gives
\begin{equation*}
    \nu(\cu_0)
    \leq C\Lambda_{s,1,p}(\cu_0)
    \exp(C\sqrt\Theta)
    \essinf_{\frac12\cu_0}u
    \,.
\end{equation*}
Combining this with the potential estimate and absorbing the factor~$\Theta$ into the exponential bounds the~$L^{r^*/2}$ norm of~$v$ by~$C\exp(C\sqrt\Theta)\essinf_{\frac12\cu_0}u$. Since~$u-v$ is a nonnegative solution in~$\frac34\cu_0$, the Harnack inequality also bounds its supremum on~$\frac58\cu_0$ by~$\exp(C\sqrt\Theta)\essinf_{\frac12\cu_0}u$. Writing~$u=v+(u-v)$ and combining these two bounds gives the endpoint weak Harnack inequality. All smaller exponents follow by H\"older's inequality (Section~\ref{ss.harnack.endpoint}).

\subsection{Lean formalization}
\label{ss.lean.formalization}

A formalization of this paper in the Lean~4 proof assistant is available at
\begin{center}
\url{https://github.com/scottnarmstrong/CoarseDeGiorgi}\,.
\end{center}
All results of this paper are formalized, within the scope described here. The coarse ellipticity constants are formalized with scale summability~$\mathsf m=1$, the case used in the regularity estimates. In particular, Proposition~\ref{p.cubical.simplicial.equivalence} and Appendix~\ref{a.cubical.simplicial} are formalized for~$\mathsf m=1$ and finite spatial exponents~$1<p,q<\infty$. The formal versions of Proposition~\ref{p.weighted.calculus} and Lemma~\ref{l.weighted.testing} are stated for bounded convex domains, which include the cubes and simplices on which we use them. The results we cite from the literature are either available in \texttt{mathlib} or proved in the formalization. The formal proofs follow the arguments given here, except that the reconstruction estimate of Lemma~\ref{l.fractional.reconstruction} is proved by a different argument that avoids the Fourier transform. They build on \texttt{mathlib} and the \texttt{CoarseGraining} library and contain no unproved placeholders or custom axioms. The main formal statements are also restated using only \texttt{mathlib} definitions, with separate Lean-checked proofs connecting them to the development.

\smallskip

The formalization was written by AI agents under the close supervision of the authors. The development was orchestrated by Claude Opus~5.5, with Lean code written by subagents, primarily GPT-6-Luna and GPT-6-Sol, and checked by Opus~5.5 subagents.

\subsection{Organization}

Sections~\ref{s.harmonic.matrices} and~\ref{s.triadic.simplices} introduce the weighted space, coarse-grained matrices, triangulations and ellipticity constants. Sections~\ref{s.fractional.reconstruction}--\ref{s.good.radius.energy} develop the fractional estimates, radius selection and extension used in the proof outline. Section~\ref{s.local.boundedness} proves local boundedness and Section~\ref{s.harnack} proves the Harnack inequalities. Section~\ref{s.classical.moments} gives the coefficient criteria, and Section~\ref{s.sharpness} proves sharpness of the coefficient range, the powers of~$\Theta$ and the weak Harnack exponent. Appendix~\ref{a.weighted.proofs} contains the weighted-space proofs. Appendix~\ref{a.besov} compares simplex, cube and Gaussian averages, while Appendix~\ref{a.cubical.simplicial} compares cubical and simplicial constants. Appendix~\ref{a.bombieri} proves Bombieri's lemma.

\subsection{Notation}
\label{ss.notation}

Functions are real-valued. On vectors,~$|\cdot|$ is the Euclidean norm. The operator norm of a matrix~$\bfA$ is~$|\bfA|\coloneqq\sup_{|e|=1}|\bfA e|$, and matrix inequalities are inequalities of quadratic forms;~$I$ is the identity matrix. We write~$\indc_A$ for the indicator of a set~$A$ and~$\wedge$ for the pointwise minimum, and for a function or a number~$u$ we write~$u_+\coloneqq\max\{u,0\}$.

\smallskip

Lebesgue measure is written~$|U|$, and integrals over cube boundaries and their subsets use Hausdorff measure~$\mathcal H^{d-1}$. Integrals and~$L^p$ norms are unnormalized. The average of~$f$ over a set~$U$ of positive finite measure is~$\fint_Uf\coloneqq|U|^{-1}\int_Uf$, also written~$(f)_U$, and averages of vector and matrix fields are taken componentwise; over a subset of a cube boundary, the average~$\fint f\,d\mathcal H^{d-1}$ is normalized by its Hausdorff measure. For a nonempty finite family,~$\avsum$ denotes the arithmetic mean. For a measurable~$f$ on~$U$, we use the following expressions, the second when~$U$ is a cube, with~$0<\eta<1$ and~$b>0$:
\begin{equation*}
	\norm{f}_{L^\eta(U)}\coloneqq\Bigl(\int_U|f|^\eta\Bigr)^{1/\eta}\,,\qquad
	\norm{f}_{\underline L^b(U)}\coloneqq\Bigl(\fint_U|f|^b\Bigr)^{1/b}\,.
\end{equation*}
The first expression is not a norm, and no triangle inequality is used for it. The second, used in Section~\ref{s.harnack}, is not a norm when~$b<1$. Essential suprema and infima are written~$\esssup$ and~$\essinf$, and~$\norm{f}_{L^\infty(U)}$ is the essential supremum of~$|f|$. Norms may be infinite; an inequality whose right side is infinite holds trivially.

\smallskip

For~$n\in\Z$ let~$\cu_n\coloneqq3^n\cu_0$. The triadic cubes of side length~$3^n$ are~$z+\cu_n$ with~$z\in3^n\Zd$. Cubes and simplices are open, and bars denote closures. We write~$\ell(D)$ for the side length of a cube~$D$ and~$5\overline D$ for its fivefold dilation about its center, and~$|x|_\infty\coloneqq\max_i|x_i|$; distances marked~$\infty$ use this norm, the others are Euclidean. For~$\rho>0$, the cube~$\rho\cu_0$ has side length~$\rho$. The simplices and triangulations on which the coarse-grained matrices are evaluated are defined in~\eqref{e.simplex.def} and~\eqref{e.simplex.family}.

\smallskip

The space~$W^{1,1}(U)$ is the ordinary Sobolev space, and~$W^{1,1}_0(U)$ is the closure of~$C^\infty_c(U)$ in it. The weighted spaces~$H^1_\a$ and~$\Hzero{\a}$ and the classes~$\Csol$ and~$\Csub$ are defined in Section~\ref{s.weighted.spaces}. For an open~$U\subseteq\Rd$, an order~$0<\alpha<1$ and an integrability~$1\leq\xi<\infty$, the ordinary fractional seminorm is
\begin{equation}
	[w]_{W^{\alpha,\xi}(U)}\coloneqq\biggl(\int_U\int_U\frac{|w(x)-w(y)|^\xi}{|x-y|^{d+\alpha\xi}}\,dx\,dy\biggr)^{\nf1\xi}\,,
	\label{e.fractional.seminorm}
\end{equation}
The corresponding norm is
\begin{equation*}
	\norm{w}_{W^{\alpha,\xi}(U)}\coloneqq\Bigl(\norm{w}_{L^\xi(U)}^\xi+[w]_{W^{\alpha,\xi}(U)}^\xi\Bigr)^{\nf1\xi}\,.
\end{equation*}
The seminorm makes sense, possibly infinite, for every measurable function. For fixed~$\alpha$, the fractional Sobolev exponents of~$W^{\alpha,\xi}$ on~$d$-dimensional domains and on~$(d-1)$-dimensional boundaries are
\begin{equation*}
	\xi^*\coloneqq\frac{d\xi}{d-\alpha\xi}\quad(\alpha\xi<d)\qquad\mbox{and}\qquad\xi_\partial^*\coloneqq\frac{(d-1)\xi}{d-1-\alpha\xi}\quad(\alpha\xi<d-1)\,,
\end{equation*}
For every order~$\beta\geq0$, written~$\beta=m+\alpha$ with an integer~$m\geq0$ and~$0\leq\alpha<1$, we set
\begin{equation*}
	\norm{w}_{W^{\beta,\xi}(U)}\coloneqq\biggl(\sum_{j=0}^m\norm{\nabla^jw}_{L^\xi(U)}^\xi+[\nabla^mw]_{W^{\alpha,\xi}(U)}^\xi\biggr)^{\nf1\xi}\,,
\end{equation*}
where~$\nabla^jw$ is the array of the weak partial derivatives of order~$j$ and the seminorm is omitted if~$\alpha=0$; for~$0<\beta<1$ this is the norm above, and~$W^{0,\xi}(U)=L^\xi(U)$. For~$1<p\leq\infty$, with~$p'\coloneqq\nf p{(p-1)}$ and~$p'=1$ if~$p=\infty$, the negative Sobolev norm of a matrix field~$\mathbf b$ with~$|\mathbf b|\in L^1(U)$ is its dual norm tested on bounded functions,
\begin{equation}
	\norm{\mathbf b}_{W^{-\beta,p}(U)}\coloneqq\sup\biggl\{\biggl|\int_Ug\,\mathbf b\biggr|:g\in L^\infty(U),\ \norm{g}_{W^{\beta,p'}(U)}\leq1\biggr\}\,.
	\label{e.negative.sobolev.norm}
\end{equation}
On cube boundaries we use the seminorm~\eqref{e.boundary.fractional.seminorm}, whose double integral includes pairs of points on different faces; restrictions to cube boundaries are used for almost every radius, as in Lemma~\ref{l.fractional.slicing}, or at a good radius, as in Proposition~\ref{p.good.radius}, and the trace of a function on a cube boundary is denoted by the same symbol as the function.

\smallskip

Constants~$C$ may change from line to line, and their dependence is stated at the start of each section, in the statement where they occur, or, for Sections~\ref{s.good.radii}--\ref{s.local.boundedness}, at the end of Section~\ref{s.triadic.simplices}.

\section{Weighted spaces and coarse-grained matrices}
\label{s.harmonic.matrices}

\subsection{Weighted spaces}
\label{s.weighted.spaces}

Since~$\a$ has neither a positive lower bound nor an upper bound, subsolutions are taken in an energy space weighted by~$\a$ itself rather than in~$H^1$. This subsection defines that space and the classes of solutions and subsolutions.

\smallskip

Throughout this section~$U\subseteq\Rd$ is a bounded connected Lipschitz domain, and the coefficient field is measurable, symmetric and positive definite almost everywhere, with
\begin{equation}
	\tr\a,\ \tr\a^{-1}\in L^1(U)\,.
	\label{e.weighted.hypotheses}
\end{equation}

\smallskip

The energy of~$u\in C^\infty(U)$ is~$\Ener_U(u)\coloneqq\int_U\nabla u\cdot\a\nabla u\in[0,\infty]$. The space~$H^1_\a(U)$ is the completion of the class~$\{u\in C^\infty(U)\cap L^1(U):\Ener_U(u)<\infty\}$ of \emph{smooth functions of finite energy}, which contains the constants, under the norm
\begin{equation}
	\norm{u}_{H^1_\a(U)}\coloneqq\bigl(|(u)_U|^2+\Ener_U(u)\bigr)^{\nf12}\,.
	\label{e.weighted.norm}
\end{equation}
The mean fixes the additive constant. Proposition~\ref{p.weighted.calculus} shows that elements of~$H^1_\a(U)$ belong to~$W^{1,1}(U)$, and Proposition~\ref{p.lower.fractional} gives~$L^r(\cu_0)$ integrability. The subspace~$\Hzero{\a}(U)$ is the closure of~$C^\infty_c(U)$. For measurable~$A\subseteq U$ we write~$\Ener_A(u)\coloneqq\int_A\nabla u\cdot\a\nabla u$. The solutions~$\Csol(U)$ and subsolutions~$\Csub(U)$ are the~$u\in H^1_\a(U)$ satisfying, respectively,
\begin{equation}
	\int_U\nabla\varphi\cdot\a\nabla u=0\quad\mbox{for all }\varphi\in C^\infty_c(U)\,,\qquad
	\int_U\nabla\varphi\cdot\a\nabla u\leq0\quad\mbox{for all }0\leq\varphi\in C^\infty_c(U)\,.
	\label{e.weighted.subsolution}
\end{equation}
Subsolutions carry no sign restriction, and~$\Csol(U)\subseteq\Csub(U)$. A \emph{supersolution} is a~$u$ with~$-u\in\Csub(U)$.

\smallskip

Appendix~\ref{a.weighted.proofs} establishes the calculus and testing properties used below. Proposition~\ref{p.weighted.calculus} treats zero boundary values, restriction, extension by zero, chain rules, truncations, bounded smooth approximation, and products with Lipschitz functions. Lemma~\ref{l.weighted.testing} permits every nonnegative element of~$\Hzero{\a}(U)$ as a test function in~\eqref{e.weighted.subsolution}. We use these results for harmonic replacement and for the test functions in Sections~\ref{s.whitney.extension}--\ref{s.harnack}.

\subsection{Harmonic replacement and the coarse-grained matrices}

In this subsection we introduce harmonic replacement and the two coarse-grained matrices of~\eqref{e.upper.matrix.definition} and~\eqref{e.lower.matrix.definition}. Harmonic replacement minimizes energy under prescribed boundary values. The matrix~$\a(U)$ measures this energy for affine boundary data and is used in the piecewise harmonic extension of Section~\ref{s.whitney.extension}. Bounds involving~$\a_*^{-1}(U)$ control mean gradients and lead to the fractional regularity of Section~\ref{s.fractional.reconstruction}. Both matrices are bounded by averages of~$\a$ and~$\a^{-1}$, as used in Section~\ref{s.classical.moments}.

% Source for the next paragraph: AK.HC = HIgh-Contrast-Elliptic-Homogenization/highcontrast.tex (SHA-256
% e2ea7fad0bc603c2832219bec19fdadb8e2c2a9d06c3e99304fd2542e05990b1), Section 2.1 (ss.Sob: e.Omega.U.def, e.H1s,
% closure of C^infty_c, e.H1s.Dirichlet, e.H1a.Neumann) and Section 2.3 (ss.bfA.def: e.J.def, e.J.mat, e.J.mat.star,
% e.J.P0.Dirichlet, e.ssk.bounds, e.subadditivity); arXiv:2405.10732v3 (2.1)-(2.2), (2.9), (2.11)-(2.12), (2.14),
% (2.21), (2.24), as reported in writing/vs9/structure-review-20260928. Symmetric specialization (k=0): a(U)=s(U)=b(U),
% a_*(U)=s_*(U); e.a(U)e=2J(U,-e,0) and e.a_*^{-1}(U)e=2J(U,0,e). Bounds by averages and ordering: e.ssk.bounds;
% subadditivity: the diagonal blocks of e.subadditivity. HC e.energymaps.nonsymm (2.36) is for solutions only;
% e.lower.all.functions and the range bound of p.harmonic.replacement are proved here. HC's L^1_loc(closure U)
% is L^1(U) for bounded U. Space comparison: Appendix a.weighted.proofs, after the proof of p.weighted.calculus.
For symmetric coefficient fields, harmonic replacement reduces to the Dirichlet problem of~\cite[Section~2.1]{AK.HC}. The matrices~$\a(U)$ and~$\a_*(U)$ are the symmetric coarse-grained matrices of~\cite[Section~2.3]{AK.HC} (see also~\cite{AKM.Book}); the antisymmetric counterpart vanishes in this case. That section states their definitions by suprema over solutions, the minimum with affine boundary values, bounds by averages, the ordering~\eqref{e.matrix.ordering}, and subadditivity for finite partitions.

The space of~\cite[Section~2.1]{AK.HC} uses the norm~$(\norm{u}_{L^1(U)}^2+\Ener_U(u))^{\nf12}$. Appendix~\ref{a.weighted.proofs} shows that this norm and~\eqref{e.weighted.norm} give the same completion, zero-boundary subspace, and solution class. It proves Propositions~\ref{p.harmonic.replacement},~\ref{p.upper.dirichlet} and~\ref{p.lower.matrix} in this space. For~$g\in H^1_\a(U)\cap L^\infty(U)$, the proof of harmonic replacement establishes the maximum principle~$\essinf_Ug\leq H_Ug\leq\esssup_Ug$. The appendix also proves the equality of the suprema over solutions and all of~$H^1_\a(U)$ in~\eqref{e.lower.all.functions}, and the subadditivity bounds~\eqref{e.upper.subadditivity} and~\eqref{e.lower.cube.aggregation} for countable as well as finite partitions.

\smallskip

For~$g\in H^1_\a(U)$, a harmonic representative of~$g$ is an element of~$\Csol(U)\cap(g+\Hzero{\a}(U))$; when it is unique we denote it by~$H_Ug$.

\begin{proposition}[Harmonic replacement]
\label{p.harmonic.replacement}
For every~$g\in H^1_\a(U)$, the harmonic representative~$H_Ug$ exists, is unique, and satisfies
\begin{equation}
	\int_U\nabla\varphi\cdot\a\nabla H_Ug=0\qquad\mbox{for every }\varphi\in\Hzero{\a}(U)\,.
	\label{e.harmonic.orthogonality}
\end{equation}
For every~$w\in g+\Hzero{\a}(U)$,
\begin{equation}
	\Ener_U(w)=\Ener_U(H_Ug)+\Ener_U(w-H_Ug)\,.
	\label{e.harmonic.minimization}
\end{equation}
The map~$g\mapsto H_Ug$ is linear, and~$H_U(g+c)=H_Ug+c$ for constants~$c$. If~$g\in L^\infty(U)$, then~$\essinf_Ug\leq H_Ug\leq\esssup_Ug$ almost everywhere in~$U$.
\end{proposition}

Proposition~\ref{p.harmonic.replacement} is proved in Appendix~\ref{a.weighted.proofs}.

\paragraph{The coarse-grained matrix~$\a(U)$.} For~$e\in\Rd$ let~$\ell_e(x)\coloneqq e\cdot x$, a bounded smooth function of finite energy since~$\tr\a\in L^1(U)$. The matrix~$\a(U)$ of~\eqref{e.upper.matrix.definition} is also given by a concave maximization, which makes sense over any class of functions. For~$\mathcal C\in\{\Csol,\Csub\}$, let
\begin{equation}
	\a(U,e;\mathcal C)\coloneqq\sup_{w\in\mathcal C(U)}\fint_U\bigl(-\nabla w\cdot\a\nabla w+2e\cdot\a\nabla w\bigr)\,.
	\label{e.upper.directional.definition}
\end{equation}
Proposition~\ref{p.upper.dirichlet} shows that the minimum in~\eqref{e.upper.matrix.definition} is attained and equals~$\a(U,e;\Csol)$. As a function of~$e$, this is the quadratic form~$e\cdot\a(U)e$. For each direction~$e$, the subsolution supremum~$\a(U,e;\Csub)$ is at least~$\a(U,e;\Csol)$ by~\eqref{e.upper.coefficient.bound}. The hypothesis uses the solution-based matrix~$\a(U)$, including when the conclusion concerns subsolutions (compare Section~\ref{ss.intro.discussion}).

\begin{proposition}[The coarse-grained matrix~$\a(U)$]
\label{p.upper.dirichlet}
The minimum in~\eqref{e.upper.matrix.definition} is attained, and its value is a positive definite quadratic form in~$e$, so that~$\a(U)$ is symmetric and positive definite. For every~$e\in\Rd$ and~$c\in\R$,
\begin{equation}
	e\cdot\a(U)e=\a(U,e;\Csol)=\frac{\Ener_U(H_U(\ell_e+c))}{|U|}\,.
	\label{e.upper.dirichlet}
\end{equation}
The two directional quantities are ordered and dominated by the average of the coefficient,
\begin{equation}
	0\leq\a(U,e;\Csol)\leq\a(U,e;\Csub)\leq e\cdot(\a)_Ue\,,
	\label{e.upper.coefficient.bound}
\end{equation}
and consequently
\begin{equation}
	0<|\a(U)|\leq|(\a)_U|<\infty\,.
	\label{e.upper.norm.bound}
\end{equation}
\end{proposition}

Proposition~\ref{p.upper.dirichlet} is proved in Appendix~\ref{a.weighted.proofs}.

\paragraph{The coarse-grained matrix~$\a_*(U)$.} The definition~\eqref{e.lower.matrix.definition} of~$\a_*^{-1}(U)$ pairs~$e$ with the gradient rather than the flux. Proposition~\ref{p.lower.matrix} proves that it is a positive definite quadratic form and that its supremum also equals the supremum over~$H^1_\a(U)$. Consequently,~$\a_*^{-1}(U)$ bounds the mean gradient of every~$w\in H^1_\a(U)$. Section~\ref{s.fractional.reconstruction} uses this bound on the simplices of the triangulations.

\begin{proposition}[The coarse-grained matrix~$\a_*(U)$]
\label{p.lower.matrix}
The right side of~\eqref{e.lower.matrix.definition} is a positive definite quadratic form in~$e$. The supremum is attained by a mean-zero solution and equals
\begin{equation}
	e\cdot\a_*^{-1}(U)e=\sup_{w\in H^1_\a(U)}\fint_U\bigl(-\nabla w\cdot\a\nabla w+2e\cdot\nabla w\bigr)\,.
	\label{e.lower.all.functions}
\end{equation}
For every~$w\in H^1_\a(U)$ and~$e\in\Rd$,
\begin{equation}
	\bigl|e\cdot(\nabla w)_U\bigr|^2\leq\bigl(e\cdot\a_*^{-1}(U)e\bigr)\frac{\Ener_U(w)}{|U|}\,,\qquad\mbox{in particular}\qquad |(\nabla w)_U|^2\leq|\a_*^{-1}(U)|\frac{\Ener_U(w)}{|U|}\,.
	\label{e.lower.mean.gradient}
\end{equation}
The matrix~$\a_*^{-1}(U)$ is dominated by the average of the inverse coefficient,
\begin{equation}
	\a_*^{-1}(U)\leq(\a^{-1})_U\,,
	\label{e.lower.coefficient.bound}
\end{equation}
and consequently
\begin{equation}
	0<|\a_*^{-1}(U)|\leq|(\a^{-1})_U|<\infty\,.
	\label{e.lower.norm.bound}
\end{equation}
\end{proposition}

For~$w\in H^1_\a(U)$, the difference~$H_Uw-w$ belongs to~$\Hzero{\a}(U)$, so~$(\nabla H_Uw)_U=(\nabla w)_U$ by~\eqref{e.weighted.coercivity}, and~$\Ener_U(H_Uw)\leq\Ener_U(w)$ by~\eqref{e.harmonic.minimization}. Replacing~$w$ by the solution~$H_Uw$ therefore does not decrease the functional in~\eqref{e.lower.all.functions}, so the supremum over~$H^1_\a(U)$ is at most the one over~$\Csol(U)$ in~\eqref{e.lower.matrix.definition}; the reverse inequality holds because~$\Csol(U)\subseteq H^1_\a(U)$. Inserting~$cw$ in~\eqref{e.lower.all.functions} and optimizing over~$c\in\R$ gives~\eqref{e.lower.mean.gradient}. The proof of Proposition~\ref{p.lower.matrix} in Appendix~\ref{a.weighted.proofs} obtains the quadratic form and the mean-zero maximizer from the Riesz representation of~$w\mapsto e\cdot(\nabla w)_U$.

\smallskip

All results of this section apply to the interior of a nondegenerate simplex, which is a bounded convex, hence Lipschitz, domain~\cite[Chapter 3, p.~27]{Guermond.traces}; Section~\ref{s.triadic.simplices} uses them on the simplices of triangulations. When~$\a$ is given on a larger domain, Proposition~\ref{p.weighted.calculus}\ref{i.weighted.domains} supplies restriction and zero extension. The mean-gradient bounds~\eqref{e.lower.mean.gradient}, which Section~\ref{s.fractional.reconstruction} uses quantitatively, have constant one and do not involve~$\int\tr\a^{\pm1}$.

\paragraph{Ordering.} The two matrices are ordered,
\begin{equation}
	\a_*(U)\leq\a(U)\,,
	\label{e.matrix.ordering}
\end{equation}
equivalently~$\a(U)^{-1}\leq\a_*^{-1}(U)$. Indeed, by~\eqref{e.weighted.coercivity} and~\eqref{e.upper.dirichlet}, the function~$H_U\ell_\xi$ has mean gradient~$\xi$ and normalized energy~$\xi\cdot\a(U)\xi$. Hence~\eqref{e.lower.all.functions} gives, for every~$e\in\Rd$,
\begin{equation*}
	e\cdot\a_*^{-1}(U)e
	\geq\sup_{\xi\in\Rd}\bigl(2e\cdot\xi-\xi\cdot\a(U)\xi\bigr)
	=e\cdot\a(U)^{-1}e\,.
\end{equation*}

\paragraph{Subadditivity.} Let~$\{U_i\}_{i\in I}$ be a finite or countable family of pairwise disjoint bounded connected Lipschitz subdomains of~$U$ whose union covers~$U$ up to a null set. Then
\begin{equation}
	\a(U)\leq\sum_{i\in I}\frac{|U_i|}{|U|}\a(U_i)\,,
	\label{e.upper.subadditivity}
\end{equation}
and
\begin{equation}
	\a_*^{-1}(U)\leq\sum_{i\in I}\frac{|U_i|}{|U|}\a_*^{-1}(U_i)\,.
	\label{e.lower.cube.aggregation}
\end{equation}
Both matrix series converge absolutely, and the bounds are proved in Appendix~\ref{a.weighted.proofs}.

\section{Triangulations and coarse ellipticity constants}
\label{s.triadic.simplices}

This section fixes the simplices on which the coarse-grained matrices are evaluated, defines the coarse ellipticity constants and condition, and collects the exponents used below. We use two families of simplices: the nested triadic triangulations~$\mathscr T_k$ of~$\cu_0$, for~$k\geq0$, and the Whitney simplices outside a cube, obtained from a Whitney decomposition. The size of a Whitney simplex is comparable to its distance from the cube, so the extension of Section~\ref{s.whitney.extension} uses matrices at the scale on which it varies. Whitney simplices of size~$3^{-j}$ inside~$\cu_0$ belong to~$\mathscr T_j$; the sums defining~$\Lambda_{s,1,p}(\cu_0)$ include~$\a(\triangle)$ for these~$\triangle$.

\paragraph{Cubes and simplices.} We use the conventions for cubes of Section~\ref{ss.notation}. Geometric constants~$c,C$ depend only on~$d$. For a permutation~$\pi$ of~$\{1,\ldots,d\}$, the standard simplex is
\begin{equation}
	\triangle_n^\pi(z)\coloneqq z+3^n\bigl\{x\in\Rd:-\frac12<x_{\pi(1)}<\cdots<x_{\pi(d)}<\frac12\bigr\}\,.
	\label{e.simplex.def}
\end{equation}
These~$d!$ simplices partition~$z+\cu_n$ up to faces, each has volume~$3^{nd}/d!$, and they differ only by coordinate permutations and dilation, so their diameter-to-inradius ratios are bounded by a dimensional constant. For each~$k\geq0$ we use the full triangulation of~$\cu_0$,
\begin{equation}
	\mathscr T_k\coloneqq\bigl\{\triangle_{-k}^{\pi}(z):z\in 3^{-k}\Zd\cap\cu_0,\ \pi\text{ a permutation}\bigr\}\,,
	\label{e.simplex.family}
\end{equation}
with~$d!\,3^{kd}$ simplices of volume~$3^{-kd}/d!$, whose closures cover~$\overline{\cu}_0$. The arithmetic mean is therefore a volume average:
\begin{equation}
	\avsum_{\triangle\in\mathscr T_k}b_\triangle\coloneqq\frac{1}{d!\,3^{kd}}\sum_{\triangle\in\mathscr T_k}b_\triangle=\sum_{\triangle\in\mathscr T_k}|\triangle|b_\triangle\,.
	\label{e.partition.average}
\end{equation}
For~$U\subseteq\cu_0$ let~$\mathscr T_k(U)\coloneqq\{\triangle\in\mathscr T_k:\triangle\subset U\}$, the simplices of~$\mathscr T_k$ contained in~$U$. If~$U$ is the interior of a union of closed triadic cubes of side length at least~$3^{-k}$, then~$\mathscr T_k(U)$ tiles~$U$ up to faces. By~\eqref{e.partition.average}, for nonnegative numbers~$b_\triangle$,
\begin{equation}
	\sum_{\triangle\in\mathscr T_k(U)}|\triangle|\,b_\triangle\leq\avsum_{\triangle\in\mathscr T_k}b_\triangle\,.
	\label{e.partition.restricted}
\end{equation}
The triangulations are nested and meet face to face. By a volume comparison, we have, for~$N\geq1$ and~$x\in\Rd$,
\begin{equation}
	\#\bigl\{\triangle\in\mathscr T_k:\overline\triangle\cap B(x,N3^{-k})\neq\varnothing\bigr\}\leq C(1+N)^d\,.
	\label{e.simplex.count}
\end{equation}

\smallskip

\paragraph{Whitney cubes.} For~$\nf12\leq\tau<1$, we cover the exterior~$\Rd\setminus\tau\overline{\cu}_0$ of the cube~$\tau\cu_0$ by Whitney cubes: we select the maximal triadic cubes~$D$, of all side lengths~$3^n$,~$n\in\Z$, with~$5\overline D\cap\tau\overline{\cu}_0=\varnothing$.

\begin{lemma}[Whitney cubes]
\label{l.whitney.cubes}
The selected cubes have disjoint interiors, their closures cover~$\Rd\setminus\tau\overline{\cu}_0$, and they are locally finite there. For every selected~$D$,
\begin{equation}
	2\ell(D)<\dist_\infty(\overline D,\tau\overline{\cu}_0)\leq\dist_\infty(x,\tau\overline{\cu}_0)\leq9\ell(D)\qquad(x\in\overline D)\,.
	\label{e.whitney.distance}
\end{equation}
\end{lemma}

\begin{proof}
An exterior point lies in the closure of arbitrarily small admissible cubes, and ancestors of side length greater than~$|x|_\infty+1$ are not admissible, so a maximal admissible ancestor exists; triadic cubes are nested or have disjoint interiors. Fivefold enlargement reduces each positive coordinate gap to~$\tau\overline{\cu}_0$ by~$2\ell(D)$, which gives the lower bound. The parent, of side length~$3\ell(D)$, is not admissible, so its distance to~$\tau\overline{\cu}_0$ is at most~$6\ell(D)$, and points of its closure are within~$3\ell(D)$ of each other; this gives the upper bound. On a compact exterior set,~\eqref{e.whitney.distance} bounds the sizes of the cubes meeting it from above and below, which gives local finiteness.
\end{proof}

\smallskip

\paragraph{Whitney simplices.}
We subdivide each selected cube once into its~$3^d$ children and triangulate them by~\eqref{e.simplex.def}. The resulting simplices are the \emph{Whitney simplices} of~$\tau\overline{\cu}_0$. A Whitney simplex has \emph{size}~$3^{-j}$ if its selected cube~$D$ has side length~$3^{1-j}$, so that
\begin{equation}
	\ell(D)=3^{1-j}\qquad\mbox{and}\qquad |\triangle|=\frac{3^{-jd}}{d!}\,.
	\label{e.whitney.simplex.size}
\end{equation}
For~$j\geq0$, the Whitney simplices of size~$3^{-j}$ contained in~$\cu_0$ belong to~$\mathscr T_j$. Whitney simplices of different sizes need not meet face to face.

\smallskip

\begin{figure}[t]
\centering
\begin{tikzpicture}[x=4cm,y=4cm,line join=round]
% cubes of side length 1/3 below and 1 above, and the Whitney simplices in their children
\foreach \ya/\yb/\m in {0.4/1/{1/9},1/1.75/{1/3}}{
	\begin{scope}
	\clip (0,\ya) rectangle (3,\yb);
	\draw[black!45,line width=0.25pt,step=\m] (0,0) grid (3,2);
	\foreach \k in {-30,...,30} \draw[black!45,line width=0.25pt] (\k*\m,0) -- ++(2,2);
	\draw[line width=0.7pt,step=3*\m] (0,0) grid (3,2);
	\end{scope}}
% free and hanging vertices where the cube size changes
\foreach \k in {0,...,9} \fill (\k/3,1) circle (1.8pt);
\foreach \k in {0,...,8} \foreach \t in {1,2} \filldraw[fill=white,line width=0.5pt] (\k/3+\t/9,1) circle (1.8pt);
% a hanging vertex z and the edge from which g(z) is interpolated
\draw[red!75!black,line width=2.2pt] (4/3,1) -- (5/3,1);
\fill (4/3,1) circle (2pt) (5/3,1) circle (2pt);
\filldraw[fill=white,line width=0.6pt] (13/9,1) circle (2.1pt);
\node[above=3pt,fill=white,inner sep=1pt] at (13/9,1) {$z$};
\node[below=3pt,fill=white,inner sep=1pt,text=red!75!black] at (1.5,1) {$g(z)$};
\end{tikzpicture}
\caption{Whitney cubes (thick lines) and Whitney simplices (thin lines) in dimension two, where cubes of side length~$\ell$ meet cubes of side length~$3\ell$. Filled circles are free vertices and open circles are hanging vertices. At a hanging vertex~$z$, the function~$g$ of Lemma~\ref{l.whitney.interpolation} is interpolated from its values at the endpoints of the red edge.}
\label{f.whitney.vertices}
\end{figure}
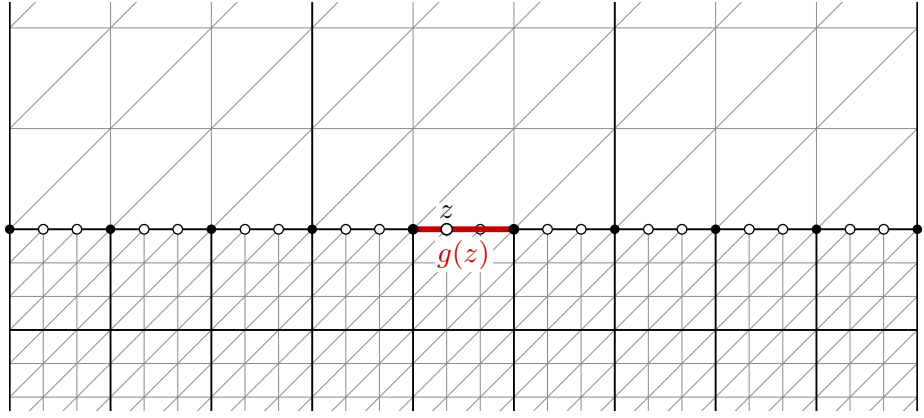

\paragraph{Free and hanging vertices.} For~$n\in\Z$, the simplices~$\triangle_n^\pi(y)$,~$y\in3^n\Zd$, form the triangulation of~$\Rd$ of \emph{mesh}~$3^n$. Its restriction to~$\cu_m$,~$m\geq n$, is a dilation of~$\mathscr T_{m-n}$, so these triangulations are nested and meet face to face. The Whitney simplices in a selected cube~$D$ are the simplices of mesh~$\ell(D)/3$ in~$D$. A vertex of a Whitney simplex is \emph{free} if it is a vertex of every Whitney simplex whose closure contains it, and \emph{hanging} otherwise; see Figure~\ref{f.whitney.vertices}. Section~\ref{s.whitney.extension} prescribes values at the free vertices and extends them by the following lemma.

\begin{lemma}[Interpolation on Whitney simplices]
\label{l.whitney.interpolation}
Given real values at the free vertices, there is a unique continuous function~$g$ on~$\Rd\setminus\tau\overline{\cu}_0$, affine on every Whitney simplex, attaining these values. On the closure~$\overline D$ of each selected cube~$D$,~$g$ lies between the smallest and largest prescribed values at free vertices in~$\overline D$, and
\begin{equation}
	\norm{\nabla g}_{L^\infty(D)}\leq\frac{C(d)}{\ell(D)}\max\bigl\{|g(z)-g(z')|:z,z'\in\overline D\mbox{ free vertices}\bigr\}\,.
	\label{e.interpolation.gradient}
\end{equation}
\end{lemma}

\begin{proof}
If selected cubes~$D$ and~$E$ have intersecting closures, then~\eqref{e.whitney.distance} gives~$2\ell(D)<9\ell(E)$. Since their side lengths are powers of three,
\begin{equation}
	\ell(D)\leq3\ell(E)\qquad\mbox{if }\overline D\cap\overline E\neq\varnothing\,.
	\label{e.whitney.neighbors}
\end{equation}
We use the nesting and face-to-face properties of the standard triangulations. In particular, a vertex of a mesh is a vertex of every closed simplex of that mesh containing it, and remains a vertex on finer meshes.

\smallskip

Let~$z$ be a hanging vertex and~$\ell$ the smallest side length of a selected cube whose closure contains~$z$. The selected cubes whose closures contain~$z$ have side lengths~$\ell$ and~$3\ell$, by~\eqref{e.whitney.neighbors}, and both occur: otherwise~$z$ would be free. Let~$F_z$ be the unique face of mesh~$\ell$ whose relative interior contains~$z$. The closure of each of these cubes is a union of closed simplices of mesh~$\ell$, so it contains~$F_z$. Every vertex~$v$ of~$F_z$ therefore belongs to the closures of cubes of both side lengths. Any selected cube whose closure contains~$v$ touches both of these cubes, so its side length is~$\ell$ or~$3\ell$, by~\eqref{e.whitney.neighbors}. Its Whitney mesh is consequently~$\nf\ell3$ or~$\ell$. Since~$v$ is a vertex of mesh~$\ell$, it is free.

\smallskip

At each hanging vertex~$z$, assign~$g(z)$ by affine interpolation of the prescribed values at the vertices of~$F_z$. On each selected cube~$D$, interpolate all its Whitney vertex values to obtain a continuous piecewise affine function~$g_D$ on~$\overline D$.

\smallskip

Let~$D$ and~$E$ be selected cubes whose closures meet, and let~$I\coloneqq\overline D\cap\overline E$. We check that~$g_D=g_E$ on~$I$. If~$\ell(D)=\ell(E)$, the two cubes carry the same triangulation and the same vertex values on~$I$. Otherwise, by~\eqref{e.whitney.neighbors} we may assume~$\ell(D)=\ell$ and~$\ell(E)=3\ell$. Every selected cube whose closure meets~$I$ has side length~$\ell$ or~$3\ell$, by~\eqref{e.whitney.neighbors}. Thus the vertices of mesh~$\ell$ in~$I$ are free, and every other vertex of mesh~$\nf\ell3$ in~$I$ is hanging, with~$F_z$ a face of mesh~$\ell$ in~$I$. The assigned fine vertex values are therefore precisely those of~$g_E$. The intersection~$I$ is a union of faces of mesh~$\ell$, on each of which~$g_E$ is affine. Interpolating these affine traces on the finer nested triangulation reproduces them, so~$g_D=g_E$ on~$I$. This argument applies to intersections of every dimension.

\smallskip

The cube closures form a locally finite cover of the exterior, by Lemma~\ref{l.whitney.cubes}, so the functions~$g_D$ define a continuous function~$g$ there. Uniqueness follows because any continuous function affine on every Whitney simplex is affine on~$F_z$, which is a face of a Whitney simplex in a largest selected cube whose closure contains~$z$. Its value at each hanging vertex is therefore forced by the prescribed free values.

\smallskip

Finally,~$F_z\subset\overline D$ whenever~$z\in\overline D$. Hence all vertex values used in~$g_D$, and then all its values, are convex combinations of prescribed values at free vertices in~$\overline D$. The Whitney simplices in~$D$ are translates of the standard simplices dilated by~$\frac{\ell(D)}3$. Their shape regularity bounds the gradient of an affine function by~$\frac{C(d)}{\ell(D)}$ times its largest vertex difference. This convex-combination property controls the difference by the maximum in~\eqref{e.interpolation.gradient}, proving that estimate.
\end{proof}

\subsection{Coarse ellipticity constants}

The constants below use square roots of matrix moments because the estimates are based on quadratic energy. The lower matrix bounds mean gradients through~\eqref{e.lower.mean.gradient}, while the upper matrix enters the exterior energy pairing after Cauchy--Schwarz (Section~\ref{s.good.radius.energy}). These bounds are used in the fractional reconstruction and the piecewise harmonic extension, respectively.

\begin{definition}[Coarse ellipticity constants]
\label{d.cg.constants}
Let~$s>0$ and~$1\leq\mathsf m,p\leq\infty$. The upper and lower coarse ellipticity constants of~$\a$ on~$\cu_0$, with discount~$s$, scale summability~$\mathsf m$ and spatial integrability~$p$, are, for finite~$\mathsf m$ and~$p$,
\begin{equation}
	\Lambda_{s,\mathsf m,p}(\cu_0)\coloneqq\biggl((1-3^{-s\mathsf m})\sum_{k\geq0}3^{-s\mathsf mk}\biggl(\avsum_{\triangle\in\mathscr T_k}|\a(\triangle)|^p\biggr)^{\nf{\mathsf m}{2p}}\biggr)^{\nf2{\mathsf m}}\,,
	\label{e.upper.spatial.moment}
\end{equation}
and
\begin{equation}
	\lambda_{s,\mathsf m,p}(\cu_0)\coloneqq\biggl((1-3^{-s\mathsf m})\sum_{k\geq0}3^{-s\mathsf mk}\biggl(\avsum_{\triangle\in\mathscr T_k}|\a_*^{-1}(\triangle)|^p\biggr)^{\nf{\mathsf m}{2p}}\biggr)^{-\nf2{\mathsf m}}\,.
	\label{e.lower.spatial.moment}
\end{equation}
For~$p=\infty$, the spatial mean~$(\avsum_{\triangle\in\mathscr T_k}|\cdot|^p)^{1/p}$ is replaced by the maximum over~$\triangle\in\mathscr T_k$. For~$\mathsf m=\infty$, the normalized sum over scales is replaced by the supremum:
\begin{equation*}
	\Lambda_{s,\infty,p}(\cu_0)^{\nf12}\coloneqq\sup_{k\geq0}3^{-ks}\biggl(\avsum_{\triangle\in\mathscr T_k}|\a(\triangle)|^p\biggr)^{\nf1{2p}}
\end{equation*}
and
\begin{equation*}
	\lambda_{s,\infty,p}(\cu_0)^{-\nf12}\coloneqq\sup_{k\geq0}3^{-ks}\biggl(\avsum_{\triangle\in\mathscr T_k}|\a_*^{-1}(\triangle)|^p\biggr)^{\nf1{2p}}\,.
\end{equation*}
The values lie in~$(0,\infty]$ and~$[0,\infty)$, respectively.
\end{definition}

The lower constant never exceeds the upper one, whatever the indices: for~$s,t>0$ and~$1\leq\mathsf m_1,p_1,\mathsf m_2,p_2\leq\infty$,~$\lambda_{t,\mathsf m_2,p_2}(\cu_0)\leq\Lambda_{s,\mathsf m_1,p_1}(\cu_0)$, in the extended sense if a constant degenerates, with equality for~$\a=cI$,~$c>0$. This is proved below, after~\eqref{e.moment.comparison}.

\smallskip

The regularity estimates use only scale summability one. With the discount~$t$ and the spatial integrability~$q$ for the lower constant, the definitions~\eqref{e.upper.spatial.moment} and~\eqref{e.lower.spatial.moment} then read
\begin{equation}
	\Lambda_{s,1,p}(\cu_0)^{\nf12}=(1-3^{-s})\sum_{k\geq0}3^{-ks}\biggl(\avsum_{\triangle\in\mathscr T_k}|\a(\triangle)|^p\biggr)^{\nf1{2p}}\,,
	\label{e.upper.half.series}
\end{equation}
and
\begin{equation}
	\lambda_{t,1,q}(\cu_0)^{-\nf12}=(1-3^{-t})\sum_{k\geq0}3^{-kt}\biggl(\avsum_{\triangle\in\mathscr T_k}|\a_*^{-1}(\triangle)|^q\biggr)^{\nf1{2q}}\,.
	\label{e.lower.half.series}
\end{equation}
For a constant matrix~$\a=\bfA$, every affine function is harmonic, so~$\a(\triangle)=\bfA$, and the affine solution~$x\mapsto\bfA^{-1}e\cdot x$ attains the bound~\eqref{e.lower.coefficient.bound}, so~$\a_*^{-1}(\triangle)=\bfA^{-1}$. Since~$(1-3^{-s\mathsf m})\sum_k3^{-s\mathsf mk}=1$ for finite~$\mathsf m$, and the supremum over~$k$ is attained at~$k=0$ for~$\mathsf m=\infty$, every coarse ellipticity constant of a constant matrix equals~$|\bfA|$ or~$|\bfA^{-1}|^{-1}$.

\paragraph{Comparison of the two constants.} We prove the ordering stated after Definition~\ref{d.cg.constants}. Let~$s,t>0$ and~$1\leq \mathsf m_1,p_1,\mathsf m_2,p_2\leq\infty$. For every~$k\geq0$, the subadditivity bounds~\eqref{e.upper.subadditivity} and~\eqref{e.lower.cube.aggregation} for the partition~$\mathscr T_k$ of~$\cu_0$, monotonicity and the triangle inequality for the operator norm,~\eqref{e.partition.average} and the power-mean inequality give
\begin{equation*}
	|\a(\cu_0)|\leq\biggl(\avsum_{\triangle\in\mathscr T_k}|\a(\triangle)|^{p_1}\biggr)^{\nf1{p_1}}\qquad\mbox{and}\qquad|\a_*^{-1}(\cu_0)|\leq\biggl(\avsum_{\triangle\in\mathscr T_k}|\a_*^{-1}(\triangle)|^{p_2}\biggr)^{\nf1{p_2}}\,,
\end{equation*}
with maxima in place of means for infinite exponents. The scale weights in~\eqref{e.upper.spatial.moment} and~\eqref{e.lower.spatial.moment} sum to one, and for infinite scale summability the term~$k=0$ suffices, so~$|\a(\cu_0)|\leq\Lambda_{s,\mathsf m_1,p_1}(\cu_0)$ and~$|\a_*^{-1}(\cu_0)|\leq\lambda_{t,\mathsf m_2,p_2}(\cu_0)^{-1}$. With the ordering~\eqref{e.matrix.ordering}, for finite positive constants,
\begin{equation}
	\Lambda_{s,\mathsf m_1,p_1}(\cu_0)^{-1}I\leq\a(\cu_0)^{-1}\leq\a_*^{-1}(\cu_0)\leq\lambda_{t,\mathsf m_2,p_2}(\cu_0)^{-1}I\,.
	\label{e.moment.comparison}
\end{equation}
Hence~$\lambda_{t,\mathsf m_2,p_2}(\cu_0)\leq\Lambda_{s,\mathsf m_1,p_1}(\cu_0)$ when both constants are finite and positive, and trivially in the extended sense otherwise; equality for~$\a=cI$ follows from the constant case above. In particular~$\Theta\geq1$.

\paragraph{Comparison with cubes.} The simplices are adapted to the piecewise affine functions used in the proof. To compare the hypothesis with one computed on cubes, let
\begin{equation}
	\mathscr Q_k\coloneqq\{z+\cu_{-k}:z\in3^{-k}\Zd\cap\cu_0\}\qquad(k\geq0)\,,
	\label{e.cubical.family}
\end{equation}
and define~$\widetilde\Lambda_{s,\mathsf m,p}(\cu_0)$ and~$\widetilde\lambda_{t,\mathsf m,q}(\cu_0)$ by replacing~$\mathscr T_k$ with~$\mathscr Q_k$ in Definition~\ref{d.cg.constants}, with the same conventions at infinite exponents. Each cube has volume~$3^{-kd}$, so the arithmetic mean over~$\mathscr Q_k$ is again a volume average. The following proposition compares the two choices of cells without changing the discount or the spatial exponent.

\begin{proposition}[Comparison with cubes]
\label{p.cubical.simplicial.equivalence}
Assume~\eqref{e.weighted.hypotheses} on~$\cu_0$. Let~$1\leq\mathsf m\leq\infty$,~$1<p,q\leq\infty$, and suppose that
\begin{equation}
	0<s<\frac12\biggl(1-\frac1p\biggr) \qquad\mbox{and}\qquad 0<t<\frac12\biggl(1-\frac1q\biggr)\,.
	\label{e.cubical.simplicial.range}
\end{equation}
There exists~$C=C(d,s,t,p,q)<\infty$, independent of~$\mathsf m$, such that
\begin{equation}
	\widetilde\Lambda_{s,\mathsf m,p}(\cu_0) \leq\Lambda_{s,\mathsf m,p}(\cu_0) \leq C\widetilde\Lambda_{s,\mathsf m,p}(\cu_0)\,,
	\label{e.Lambda.cube.simplex.equivalence}
\end{equation}
and
\begin{equation}
	C^{-1}\widetilde\lambda_{t,\mathsf m,q}(\cu_0) \leq\lambda_{t,\mathsf m,q}(\cu_0) \leq\widetilde\lambda_{t,\mathsf m,q}(\cu_0)\,.
	\label{e.lambda.cube.simplex.equivalence}
\end{equation}
The inequalities with constant one hold for all~$s,t>0$ and~$1\leq p,q\leq\infty$.
\end{proposition}

Proposition~\ref{p.cubical.simplicial.equivalence} is proved in Appendix~\ref{a.cubical.simplicial}. The restriction~\eqref{e.cubical.simplicial.range} reduces to~$s,t<\nf12$ when~$p=q=\infty$. For finite spatial exponents, H\"older's inequality on the Whitney layers gives the additional terms~$1/p$ and~$1/q$.

\paragraph{The coarse ellipticity condition.} The hypothesis of Theorem~\ref{t.cg.local.boundedness}, which we call the \emph{coarse ellipticity condition}, is that~$1<p,q<\infty$ and
\begin{equation}
	s,t>0\,,\qquad s+t+\frac{d-1}{2}\biggl(\frac1p+\frac1q\biggr)<1\,,\qquad \Lambda_{s,1,p}(\cu_0)<\infty\,,\qquad \lambda_{t,1,q}(\cu_0)>0\,.
	\label{e.spatial.moment.range}
\end{equation}
The second condition says that the exponent~$\theta$ of~\eqref{e.intro.parameters} is positive. Under the coarse ellipticity condition, the coarse ellipticity ratio~$\Theta$ of~\eqref{e.spatial.moment.ratio} is finite.

If, in addition,~\eqref{e.cubical.simplicial.range} holds, then~\eqref{e.spatial.moment.range} is equivalent to the condition obtained by replacing~$\Lambda$ and~$\lambda$ with~$\widetilde\Lambda$ and~$\widetilde\lambda$. In this range, Proposition~\ref{p.cubical.simplicial.equivalence} gives
\begin{equation}
	\widetilde\Theta\coloneqq\frac{\widetilde\Lambda_{s,1,p}(\cu_0)}{\widetilde\lambda_{t,1,q}(\cu_0)}\leq\Theta\leq C\widetilde\Theta\,,
	\label{e.cubical.simplicial.ratio}
\end{equation}
where~$C$ depends only on~$d,p,q,s,t$. Thus Theorems~\ref{t.cg.local.boundedness} and~\ref{t.harnack} and Corollary~\ref{c.lr.bound} hold with~$\widetilde\Theta$ in place of~$\Theta$, after changing their constants. The restrictions~\eqref{e.cubical.simplicial.range} are additional to~\eqref{e.spatial.moment.range}; the latter does not imply them.

\subsection{Parameters}

The exponents~$\sigma$,~$\sigma_*$ and~$r$ of~\eqref{e.intro.parameters} satisfy
\begin{equation}
	\frac{d-1}{2p}=\sigma-s\,,\qquad (d-1)\Bigl(\frac1r-\frac12\Bigr)=\frac{d-1}{2q}=\sigma_*-t\,,\qquad (1-t)-\frac{d-1}{2q}=1-\sigma_*\,.
	\label{e.reconstruction.parameters}
\end{equation}
The trace of a subsolution on a cube boundary will be controlled in~$W^{1-t,r}$. With~$\alpha=1-t$ fixed, the fractional Sobolev exponents of Section~\ref{ss.notation} for~$\xi=r$ are
\begin{equation}
	r^*=\frac{dr}{d-(1-t)r}\,,\qquad r_\partial^*=\frac{(d-1)r}{d-1-(1-t)r}=\frac{2(d-1)}{d-3+2\sigma_*}\,.
	\label{e.sobolev.exponents}
\end{equation}
The last identity follows from~$\nf1{r_\partial^*}=\nf1r-\nf{(1-t)}{(d-1)}=\nf12-\nf{(1-\sigma_*)}{(d-1)}$, by~\eqref{e.reconstruction.parameters}. From~\eqref{e.spatial.moment.range},
\begin{equation}
	0<\sigma\,,\qquad 0<t<\sigma_*<1-\sigma\,,\qquad 1<r<2\,,\qquad r^*>r\,,\qquad r_\partial^*>2\,.
	\label{e.parameter.facts}
\end{equation}
We have~$\sigma\geq s>0$,~$\sigma_*-t=\nf{(d-1)}{(2q)}>0$, and~$1-\sigma-\sigma_*=\theta>0$. Also~$1<r<2$ because~$1<q<\infty$. Since~$d\geq3$ and~$\sigma_*>0$,
\begin{equation*}
	d-1-(1-t)r=\frac r2(d-3+2\sigma_*)>0\,.
\end{equation*}
Thus both exponents in~\eqref{e.sobolev.exponents} are finite, and~$r^*>r$. The formula for~$r_\partial^*$ gives~$r_\partial^*>2$ because~$\sigma_*<1$. It is the Sobolev exponent of~$W^{1-\sigma_*,2}$ in dimension~$d-1$.

The exponents~$r^*$ and~$r_\partial^*$ provide the superlinear gains of the De Giorgi iteration in Section~\ref{s.local.boundedness}: on the boundary of a cube with a good radius through Lemma~\ref{l.critical.trace.embedding}, and on the smaller cube in the two-level recurrence through the fractional Sobolev inequality~\eqref{e.ordinary.fractional.sobolev}. The recurrence uses that~$\nf{2r^*}r$ and~$r_\partial^*$ exceed two, and the moment iteration of Section~\ref{s.harnack} uses~$r^*>r$; the exponent~$r^*$ itself may be smaller than two.

\smallskip

In the estimate of Section~\ref{s.good.radius.energy}, the Whitney simplices of size~$3^{-j}$ contribute two factors besides their volume and the gradient scale. The trace in~$W^{1-t,r}$ gains~$3^{-j(1-\sigma_*)}$ over boundary patches of size~$3^{-j}$, after the loss~$3^{j(d-1)(\frac1r-\frac12)}=3^{j(\sigma_*-t)}$ from comparing its averages in~$L^r$ rather than~$L^2$ on patches of dimension~$d-1$ (Section~\ref{s.whitney.extension}). The matrices lose~$3^{j\sigma}$: the square roots of~$|\a(\triangle)|$ may grow like~$3^{js}$, as the discount in~$\Lambda_{s,1,p}(\cu_0)$ allows, and taking their maximum near a cube boundary, of dimension~$d-1$, introduces the factor~$3^{j(d-1)/(2p)}$ (Section~\ref{s.good.radii}); the series~\eqref{e.boundary.maxima.series} of these maxima carries the compensating weight~$3^{-j\sigma}$. Multiplying these factors gives
\begin{equation*}
	3^{j\sigma}3^{-j(1-\sigma_*)}=3^{-j\theta}\,.
\end{equation*}
Since~$\theta>0$, the Whitney simplices of each size~$3^{-j}<h$ contribute at most~$h^\theta$ in Section~\ref{s.good.radius.energy}.

\smallskip

The exponents~$\gamma_1,\gamma_2,\ldots$ in the estimates below are the powers of~$(\rho_2-\rho_1)^{-1}$, the inverse of the gap between two radii; they are positive, depend only on~$d,p,q,s,t$, are defined where they arise, and matter only through their finiteness. They arise because the good radius is chosen in an interval of length comparable to~$\rho_2-\rho_1$ and the localization works at that scale. Unless stated otherwise, constants in Sections~\ref{s.good.radii}--\ref{s.local.boundedness} depend only on~$d,p,q,s,t$.

\section{Embedding into fractional Sobolev spaces}
\label{s.fractional.reconstruction}

The lower constant~$\lambda_{t,1,q}(\cu_0)$ controls fractional regularity through averages of mean gradients. Proposition~\ref{p.lower.fractional} shows that every~$w\in H^1_\a(Q)$ belongs to~$W^{1-t,r}(Q)$ for the cubes~$Q$ specified there, with~$r=\nf{2q}{(q+1)}$ and seminorm bounded by~$C\lambda_{t,1,q}(\cu_0)^{-\nf12}\Ener_Q(w)^{\nf12}$. Section~\ref{s.good.radii} restricts this regularity to cube boundaries to control the trace used in the extension of Section~\ref{s.whitney.extension}. Sections~\ref{s.local.boundedness} and~\ref{s.harnack} also use this fractional estimate on cubes.

The exponent~$r<2$ comes from Lemma~\ref{l.lower.averages}. At each scale the energies of the cubes sum to the energy of their union, whereas~$\lambda_{t,1,q}(\cu_0)$ controls the~$q$-th spatial moment of~$\a_*^{-1}(\triangle)$. H\"older's inequality therefore bounds the mean gradients in~$L^r$, where~$\nf r2+\nf r{(2q)}=1$.

\smallskip

For~$n\in\N_0$,~$z\in3^{-n-1}\Zd$, the cube~$Q=z+\cu_{-n}$,~$k\geq n$ and~$f\in L^1(Q)$, set
\begin{equation*}
	P_k(f;Q)\coloneqq\sum_{y\in3^{-k}\Zd\cap\cu_{-n}}(f)_{z+y+\cu_{-k}}\indc_{z+y+\cu_{-k}}\,.
\end{equation*}
The~$3^{d(k-n)}$ cubes~$z+y+\cu_{-k}$ tile~$Q$ up to faces, and~$P_n(f;Q)=(f)_Q$. Equalities hold almost everywhere. For~$k\geq n+1$ the cubes~$z+y+\cu_{-k}$ are triadic; if~$Q\subseteq\cu_0$, each of them is tiled by the~$d!$ simplices of~$\mathscr T_k(z+y+\cu_{-k})$, and~$\mathscr T_k(Q)$ tiles~$Q$. In a proof in which~$Q$ is fixed, we write~$P_kf$ for~$P_k(f;Q)$ and say so at its beginning.

\smallskip

If~$w\in W^{1,1}(Q)$, each~$P_k(\nabla w;Q)$ is a finite step function. The first lemma bounds it in~$L^r$ by the matrices~$\a_*^{-1}(\triangle)$,~$\triangle\in\mathscr T_k(Q)$.

\begin{lemma}[The matrices~$\a_*$ on cubes]
\label{l.lower.averages}
Let~$n\in\N_0$,~$z\in3^{-n-1}\Zd$ and~$Q=z+\cu_{-n}\subset\cu_0$. Assume~\eqref{e.weighted.hypotheses} on~$\cu_0$, and let~$r$ be as in~\eqref{e.intro.parameters}. For every~$w\in H^1_\a(Q)$ and~$k\geq n+1$,
\begin{equation}
	\norm{P_k(\nabla w;Q)}_{L^r(Q)}\leq\biggl(\sum_{\triangle\in\mathscr T_k(Q)}|\triangle|\,|\a_*^{-1}(\triangle)|^q\biggr)^{\nf1{2q}}\Ener_Q(w)^{\nf12}\,.
	\label{e.lower.spatial.holder}
\end{equation}
\end{lemma}

\begin{proof}
Write~$P_k=P_k(\cdot\,;Q)$, and let~$D$ range over the cubes~$z+y+\cu_{-k}$,~$y\in3^{-k}\Zd\cap\cu_{-n}$. By~\eqref{e.lower.cube.aggregation} for the partition~$\mathscr T_k(D)$ of each~$D$, monotonicity and the triangle inequality for the operator norm, and Jensen's inequality,
\begin{equation*}
	\sum_D|D|\,|\a_*^{-1}(D)|^q\leq\sum_D\sum_{\triangle\in\mathscr T_k(D)}|\triangle|\,|\a_*^{-1}(\triangle)|^q=\sum_{\triangle\in\mathscr T_k(Q)}|\triangle|\,|\a_*^{-1}(\triangle)|^q\,.
\end{equation*}
Since~$\nf{r}{2}+\nf{r}{(2q)}=1$, H\"older's inequality over~$D$ and~\eqref{e.lower.mean.gradient} give
\begin{align*}
	\norm{P_k\nabla w}_{L^r(Q)}^r=\sum_D|D|\,|(\nabla w)_D|^r&\leq\sum_D\bigl(|D|\,|\a_*^{-1}(D)|^q\bigr)^{\nf{r}{2q}}\Ener_D(w)^{\nf r2}\\
	&\leq\biggl(\sum_{\triangle\in\mathscr T_k(Q)}|\triangle|\,|\a_*^{-1}(\triangle)|^q\biggr)^{\nf r{2q}}\Ener_Q(w)^{\nf r2}\,.\qedhere
\end{align*}
\end{proof}

The next estimate does not involve the coefficients. It recovers the fractional differences of~$w$ from the triadic averages of its gradient.

\begin{lemma}[Reconstruction from triadic gradient averages]
\label{l.fractional.reconstruction}
Let~$n\in\N_0$,~$z\in3^{-n-1}\Zd$ and~$Q=z+\cu_{-n}$. For~$0<\alpha<1$ and~$1<\xi<\infty$ there exists~$C(d,\alpha,\xi)<\infty$ such that every~$w\in W^{1,1}(Q)\cap L^\xi(Q)$ satisfies
\begin{equation}
	[w]_{W^{\alpha,\xi}(Q)}\leq C\sum_{k\geq n}3^{-k(1-\alpha)}\norm{P_k(\nabla w;Q)}_{L^\xi(Q)}\,.
	\label{e.fractional.reconstruction}
\end{equation}
\end{lemma}

% Neumann display form adapted from HC, highcontrast-arxiv-v3.tex, e.H1a.Neumann (1427--1439), and ASD, algsuperdiff.tex (14563--14574); only the display form is imported.
\begin{proof}
We may assume that the sum on the right is finite, and we write~$P_k=P_k(\cdot\,;Q)$. We first estimate the Neumann solutions associated with successive differences of the gradient averages, then estimate their fractional seminorms and sum them. The constants below are independent of~$n$,~$k$ and~$Q$.

\smallskip

\emph{Step 1.} Neumann estimates. Define
\begin{equation*}
	G_n\coloneqq P_n\nabla w\,,\qquad G_k\coloneqq(P_k-P_{k-1})\nabla w\quad(k>n)\,.
\end{equation*}
For~$k\geq n$, let~$u_k\in H^1(Q)$ be the weak solution, normalized by~$(u_k)_Q=0$, of
\begin{equation*}
	\left\{
	\begin{aligned}
		&\Delta u_k=\nabla\cdot G_k &&\mbox{in }Q\,,\\
		&\mathbf{n}\cdot(\nabla u_k-G_k)=0 &&\mbox{on }\partial Q\,.
	\end{aligned}
	\right.
\end{equation*}
Here~$\mathbf{n}$ is the outer unit normal to~$Q$, and the boundary condition is understood weakly. Given~$\phi\in L^2(Q)\cap L^{\xi'}(Q)$, where~$\xi'=\xi/(\xi-1)$, let~$v\in H^1(Q)$ be the weak solution, normalized by~$(v)_Q=0$, of
\begin{equation*}
	\left\{
	\begin{aligned}
		&-\Delta v=\phi-(\phi)_Q &&\mbox{in }Q\,,\\
		&\mathbf{n}\cdot\nabla v=0 &&\mbox{on }\partial Q\,.
	\end{aligned}
	\right.
\end{equation*}
Classical estimates for the Neumann problem give
\begin{equation*}
	\norm{\nabla u_k}_{L^\xi(Q)}\leq C(d,\xi)\norm{G_k}_{L^\xi(Q)}\,,\qquad \norm{\nabla^2v}_{L^{\xi'}(Q)}\leq C(d,\xi)\norm{\phi}_{L^{\xi'}(Q)}\,.
\end{equation*}
Indeed, these inequalities follow by reflection from the classical Calder\'on--Zygmund estimates for the Laplacian.

\smallskip

For~$k>n$, the field~$G_k$ has zero average on every parent cube~$D=z+y+\cu_{1-k}$ with~$y\in3^{1-k}\Zd\cap\cu_{-n}$. The equations of~$u_k$ and~$v$, the zero averages of~$u_k$ and~$G_k$, the Poincar\'e and H\"older inequalities on the parent cubes, and the Neumann estimate for~$v$ on~$Q$ give
\begin{align*}
	\biggl|\int_Q u_k\phi\biggr|&=\biggl|\int_Q G_k\cdot\nabla v\biggr|=\biggl|\sum_D\int_DG_k\cdot\bigl(\nabla v-(\nabla v)_D\bigr)\biggr|\notag\\
	&\leq C3^{1-k}\norm{G_k}_{L^\xi(Q)}\norm{\nabla^2v}_{L^{\xi'}(Q)}\leq C3^{-k}\norm{G_k}_{L^\xi(Q)}\norm{\phi}_{L^{\xi'}(Q)}\,.
\end{align*}
Since~$L^2(Q)\cap L^{\xi'}(Q)$ is dense in~$L^{\xi'}(Q)$, duality gives~$\norm{u_k}_{L^\xi(Q)}\leq C(d,\xi)3^{-k}\norm{G_k}_{L^\xi(Q)}$ for~$k>n$. For~$k=n$, the solution is~$u_n(x)=G_n\cdot(x-(x)_Q)$, where~$(x)_Q$ denotes the mean of the coordinate vector on~$Q$. Since~$Q$ has side length~$3^{-n}$, the same bound holds for~$k=n$.

\smallskip

\emph{Step 2.} Fractional interpolation. Convexity of~$Q$ and the fundamental theorem of calculus give, for~$h\in\Rd$,
\begin{equation*}
	\norm{u_k(\cdot+h)-u_k}_{L^\xi(Q\cap(Q-h))}\leq\min\bigl\{2\norm{u_k}_{L^\xi(Q)},\,|h|\norm{\nabla u_k}_{L^\xi(Q)}\bigr\}\,.
\end{equation*}
Changing variables~$y=x+h$ in the seminorm, splitting at~$|h|=\varrho$ for~$\varrho>0$, and using polar coordinates, we obtain
\begin{align*}
	[u_k]_{W^{\alpha,\xi}(Q)}^\xi&=\int_{\Rd}\frac{\norm{u_k(\cdot+h)-u_k}_{L^\xi(Q\cap(Q-h))}^\xi}{|h|^{d+\alpha\xi}}\,dh\notag\\
	&\leq C(d,\xi)\biggl(\norm{\nabla u_k}_{L^\xi(Q)}^\xi\int_0^\varrho t^{\xi(1-\alpha)-1}\,dt+\norm{u_k}_{L^\xi(Q)}^\xi\int_\varrho^\infty t^{-\alpha\xi-1}\,dt\biggr)\,.
\end{align*}
Evaluating the two radial integrals yields
\begin{equation*}
	[u_k]_{W^{\alpha,\xi}(Q)}^\xi\leq C(d,\xi)\biggl(\frac{\varrho^{\xi(1-\alpha)}}{\xi(1-\alpha)}\norm{\nabla u_k}_{L^\xi(Q)}^\xi+\frac{\varrho^{-\alpha\xi}}{\alpha\xi}\norm{u_k}_{L^\xi(Q)}^\xi\biggr)\,.
\end{equation*}
If either norm is zero, the seminorm vanishes. Otherwise choose~$\varrho=\norm{u_k}_{L^\xi(Q)}/\norm{\nabla u_k}_{L^\xi(Q)}$ and take the~$\xi$-th root to obtain
\begin{equation*}
	[u_k]_{W^{\alpha,\xi}(Q)}\leq C(d,\alpha,\xi)\norm{u_k}_{L^\xi(Q)}^{1-\alpha}\norm{\nabla u_k}_{L^\xi(Q)}^\alpha\,.
\end{equation*}
Combining this interpolation inequality with the bounds from Step~1 gives, for~$k\geq n$,
\begin{equation*}
	[u_k]_{W^{\alpha,\xi}(Q)}
	\leq C(d,\alpha,\xi)3^{-k(1-\alpha)}\norm{G_k}_{L^\xi(Q)}\,.
\end{equation*}

\smallskip

\emph{Step 3.} Summation and identification of the limit. For~$k>n$, we have~$\norm{G_k}_{L^\xi(Q)}\leq\norm{P_k\nabla w}_{L^\xi(Q)}+\norm{P_{k-1}\nabla w}_{L^\xi(Q)}$. Reindexing the second term in the weighted sum contributes the factor~$3^{-(1-\alpha)}$, and hence
\begin{equation*}
	\sum_{k=n}^\infty[u_k]_{W^{\alpha,\xi}(Q)}\leq C\sum_{k=n}^\infty3^{-k(1-\alpha)}\norm{G_k}_{L^\xi(Q)}\leq C\sum_{k=n}^\infty3^{-k(1-\alpha)}\norm{P_k\nabla w}_{L^\xi(Q)}\,.
\end{equation*}
The~$L^\xi$ estimate of Step~1 and~$3^{-k}\leq3^{-n\alpha}3^{-k(1-\alpha)}$ for~$k\geq n$ also give~$\sum_{k=n}^\infty\norm{u_k}_{L^\xi(Q)}<\infty$. Thus~$S_N\coloneqq\sum_{k=n}^N u_k$ converges in~$W^{\alpha,\xi}(Q)$.

\smallskip

To identify the limit of~$S_N$, take~$\phi\in C_c^\infty(Q)$ and its dual Neumann solution~$v$ from Step~1. The reflected datum is smooth, so~$\nabla v$ is bounded. Since~$P_N\nabla w\to\nabla w$ in~$L^1(Q)$, the weak equations and integration by parts give
\begin{equation*}
	\int_Q S_N\phi=\int_Q P_N\nabla w\cdot\nabla v\longrightarrow\int_Q\nabla w\cdot\nabla v=\int_Q\bigl(w-(w)_Q\bigr)\phi\,.
\end{equation*}
It follows that~$S_N\to w-(w)_Q$ in~$W^{\alpha,\xi}(Q)$. The preceding sum of seminorms proves the assertion.
\end{proof}

With~$\alpha=1-t$ and~$\xi=r$, the two lemmas combine into the fractional estimate, first on each cube in terms of the matrices~$\a_*^{-1}(\triangle)$ of its simplices, and then in terms of~$\lambda_{t,1,q}(\cu_0)$. Since~$H^1_\a(Q)$ is a completion, we prove it for smooth functions and pass to the limit.

\begin{proposition}[Embedding into~$W^{1-t,r}$]
\label{p.lower.fractional}
Assume~\eqref{e.weighted.hypotheses} on~$\cu_0$ and the coarse ellipticity condition~\eqref{e.spatial.moment.range}, and let~$r$ be as in~\eqref{e.intro.parameters}. There exists~$C(d,q,t)<\infty$ such that, for every~$n\in\N_0$,~$z\in3^{-n-1}\Zd$ and~$Q=z+\cu_{-n}\subset\cu_0$, the embedding~$H^1_\a(Q)\subset W^{1-t,r}(Q)$ is continuous, and every~$w\in H^1_\a(Q)$ satisfies
\begin{equation}
	[w]_{W^{1-t,r}(Q)}\leq C\Ener_Q(w)^{\nf12}\sum_{k>n}3^{-kt}\biggl(\sum_{\triangle\in\mathscr T_k(Q)}|\triangle|\,|\a_*^{-1}(\triangle)|^q\biggr)^{\nf1{2q}}\,.
	\label{e.lower.fractional.local}
\end{equation}
Consequently,
\begin{equation}
	[w]_{W^{1-t,r}(Q)}\leq C\,\lambda_{t,1,q}(\cu_0)^{-\nf12}\,\Ener_Q(w)^{\nf12}\qquad(w\in H^1_\a(Q))\,.
	\label{e.lower.fractional}
\end{equation}
In particular, every~$w\in H^1_\a(\cu_0)$ belongs to~$L^r(\cu_0)$.
\end{proposition}

\begin{proof}
Suppose first~$w$ is a bounded smooth function of finite energy~$\Ener_Q(w)$, and write~$P_k=P_k(\cdot\,;Q)$. Since~$P_n\nabla w$ is the average of~$P_{n+1}\nabla w$ over~$Q$, Jensen's inequality gives~$\norm{P_n\nabla w}_{L^r(Q)}\leq\norm{P_{n+1}\nabla w}_{L^r(Q)}$, so the term~$k=n$ of~\eqref{e.fractional.reconstruction} is at most~$3^t$ times the term~$k=n+1$. Lemma~\ref{l.fractional.reconstruction} with~$\alpha=1-t$ and~$\xi=r$, so that~$1-\alpha=t$, and~\eqref{e.lower.spatial.holder} of Lemma~\ref{l.lower.averages} give the first inequality below, and~\eqref{e.partition.restricted} for~$U=Q$ and~\eqref{e.lower.half.series} give the second:
\begin{align}
	[w]_{W^{1-t,r}(Q)}&\leq C\Ener_Q(w)^{\nf12}\sum_{k>n}3^{-kt}\biggl(\sum_{\triangle\in\mathscr T_k(Q)}|\triangle|\,|\a_*^{-1}(\triangle)|^q\biggr)^{\nf1{2q}}\notag\\
	&\leq\frac{C}{1-3^{-t}}\lambda_{t,1,q}(\cu_0)^{-\nf12}\Ener_Q(w)^{\nf12}\,.
	\label{e.lower.fractional.smooth}
\end{align}
The~$W^{1-t,r}$ seminorm also controls the oscillation in~$Q$: Jensen's inequality gives, for any measurable~$v$ on~$Q$, whose side length is~$3^{-n}$,
\begin{equation}
	\norm{v}_{L^r(Q)}\leq C3^{-n(1-t)}[v]_{W^{1-t,r}(Q)}+|Q|^{\nf1r}|(v)_Q|\,.
	\label{e.fractional.mean.norm}
\end{equation}
Take bounded smooth~$w_i\to w$ in~$H^1_\a(Q)$ (Proposition~\ref{p.weighted.calculus}\ref{i.weighted.approximation}). By~\eqref{e.lower.fractional.smooth} and~\eqref{e.fractional.mean.norm} applied to differences,~$(w_i)$ is Cauchy in~$W^{1-t,r}(Q)$; its limit agrees with the~$L^1$ limit~$w$, so that~$w_i\to w$ in~$W^{1-t,r}(Q)$. Passing to the limit in~\eqref{e.lower.fractional.smooth}, whose right sides depend on~$w$ only through~$\Ener_Q(w)$, proves~\eqref{e.lower.fractional.local} and~\eqref{e.lower.fractional}, and combining this with~\eqref{e.fractional.mean.norm} gives continuity of the embedding. Taking~$Q=\cu_0$ gives the last assertion.
\end{proof}

\section{Cubes with good radii}
\label{s.good.radii}

To build the piecewise harmonic extension, we choose a radius~$\tau$ for which the trace of~$v$ has controlled~$W^{1-t,r}$ and~$L^r$ norms, with an~$L^2$ bound when needed. We also control the energy in thin layers around~$\partial(\tau\cu_0)$ and the matrices~$\a(\triangle)$ on nearby simplices. Slicing and the cubical coarea formula give the trace and energy bounds on average over~$\tau$. Each simplex of~$\mathscr T_k$ lies near~$\partial(\tau\cu_0)$ only for radii in an interval of length~$C3^{-k}$, giving the boundary-dimensional estimate for the matrix maxima (Lemma~\ref{l.radius.averages}). Proposition~\ref{p.good.radius} selects one radius where these bounds and convergence of the approximation traces hold together. The selection applies to nonnegative functions in~$H^1_\a(\cu_0)$.

\smallskip

Throughout this section~\eqref{e.weighted.hypotheses} and the coarse ellipticity condition~\eqref{e.spatial.moment.range} hold on~$\cu_0$. For radii~$\nf12\leq\rho_1<\rho_2\leq1$ we set
\begin{equation}
	\delta\coloneqq\rho_2-\rho_1\,,\qquad J\coloneqq\bigl(\rho_1+\tfrac\delta4,\rho_1+\tfrac\delta2\bigr)\,.
	\label{e.selection.parameters}
\end{equation}
Surface integrals use Hausdorff measure~$\mathcal H^{d-1}$. For~$0<\tau\leq1$, an order~$0<\alpha<1$ and an integrability~$1\leq\xi<\infty$, the seminorm on~$\partial(\tau\cu_0)$ is
\begin{equation}
	[f]_{W^{\alpha,\xi}(\partial(\tau\cu_0))}\coloneqq\biggl(\int_{\partial(\tau\cu_0)}\int_{\partial(\tau\cu_0)}\frac{|f(x)-f(y)|^\xi}{|x-y|^{d-1+\alpha\xi}}\,d\mathcal H^{d-1}(y)\,d\mathcal H^{d-1}(x)\biggr)^{\nf1\xi}\,,
	\label{e.boundary.fractional.seminorm}
\end{equation}
The associated norm is
\begin{equation*}
	\norm{f}_{W^{\alpha,\xi}(\partial(\tau\cu_0))}\coloneqq\bigl(\norm{f}_{L^\xi(\partial(\tau\cu_0))}^\xi+[f]_{W^{\alpha,\xi}(\partial(\tau\cu_0))}^\xi\bigr)^{\nf1\xi}\,.
\end{equation*}
The double integral includes pairs on different faces. On open subsets of~$\R^{d-1}$ we use~\eqref{e.fractional.seminorm} with~$d$ replaced by~$d-1$.

\paragraph{Fractional traces.} On the~$(d-1)$-dimensional cube boundaries, the Sobolev exponent of~$W^{\alpha,\xi}$ is~$\xi_\partial^*$ of Section~\ref{ss.notation}; for~$\alpha=1-t$ and~$\xi=r$ it is~$r_\partial^*$ of~\eqref{e.sobolev.exponents}, which exceeds two by~\eqref{e.parameter.facts}.

\begin{lemma}[Sobolev inequality on cube boundaries]
\label{l.critical.trace.embedding}
Let~$0<\alpha<1$ and~$1\leq\xi<\infty$ with~$\alpha\xi<d-1$. There exists~$C(d,\alpha,\xi)<\infty$ such that, for~$\nf12\leq\tau\leq1$ and~$g\in W^{\alpha,\xi}(\partial(\tau\cu_0))$,
\begin{equation}
	\norm{g}_{L^{\xi_\partial^*}(\partial(\tau\cu_0))}\leq C\norm{g}_{W^{\alpha,\xi}(\partial(\tau\cu_0))}\,.
	\label{e.critical.trace.embedding}
\end{equation}
In particular, if~$\xi_\partial^*\geq2$, convergence in~$W^{\alpha,\xi}(\partial(\tau\cu_0))$ implies convergence in~$L^2(\partial(\tau\cu_0))$.
\end{lemma}

\begin{proof}
An open unit cube of dimension~$d-1$ is a bounded Lipschitz domain, hence an extension domain for~$W^{\alpha,\xi}$~\cite[Theorem~5.4]{DNPV.fractional}, and on such a domain~$W^{\alpha,\xi}$ embeds continuously in~$L^{\xi_\partial^*}$~\cite[Theorem~6.7]{DNPV.fractional}. A face~$\Phi$ of~$\partial(\tau\cu_0)$ is a translate of the unit face dilated by~$\tau\in[\nf12,1]$. Under this dilation the~$L^{\xi_\partial^*}$ norm, the~$L^\xi$ norm and the seminorm change by the factors~$\tau^{(d-1)/\xi_\partial^*}$,~$\tau^{(d-1)/\xi}$ and~$\tau^{(d-1)/\xi-\alpha}$, which are bounded above and below. Hence~$\norm{g}_{L^{\xi_\partial^*}(\Phi)}\leq C\norm{g}_{W^{\alpha,\xi}(\Phi)}$, where the seminorm on~$\Phi$ is the double integral~\eqref{e.boundary.fractional.seminorm} restricted to~$\Phi\times\Phi$. The full double integral contains the~$2d$ double integrals over single faces, and~$\xi\leq\xi_\partial^*$, so
\begin{equation*}
	\norm{g}_{L^{\xi_\partial^*}(\partial(\tau\cu_0))}^{\xi_\partial^*}=\sum_\Phi\norm{g}_{L^{\xi_\partial^*}(\Phi)}^{\xi_\partial^*}\leq C\Bigl(\sum_\Phi\norm{g}_{W^{\alpha,\xi}(\Phi)}^\xi\Bigr)^{\nf{\xi_\partial^*}\xi}\leq C\norm{g}_{W^{\alpha,\xi}(\partial(\tau\cu_0))}^{\xi_\partial^*}\,.
\end{equation*}
The last assertion follows from H\"older's inequality, since~$\partial(\tau\cu_0)$ has area~$2d\tau^{d-1}\leq2d$.
\end{proof}

\begin{lemma}[Cutoff estimate for~$W^{\alpha,\xi}$]
\label{l.fractional.cutoff}
Let~$U\subset\Rd$ be open, let~$0<\varepsilon\leq1$, and let~$\varphi$ be Lipschitz with~$|\varphi|\leq1$, Lipschitz constant at most~$1/\varepsilon$, and support at distance at least~$\varepsilon$ from~$\Rd\setminus U$. For~$0<\alpha<1$,~$1<\xi<\infty$ and~$w\in W^{\alpha,\xi}(U)$, the zero extension of~$\varphi w$ satisfies
\begin{equation}
	\norm{\varphi w}_{W^{\alpha,\xi}(\Rd)}\leq C\bigl([w]_{W^{\alpha,\xi}(U)}+\varepsilon^{-\alpha}\norm{w}_{L^\xi(U)}\bigr)\,,\qquad C=C(d,\alpha,\xi)\,.
	\label{e.fractional.cutoff}
\end{equation}
\end{lemma}

\begin{proof}
For~$x,y\in U$, decompose the difference as
\begin{equation*}
	\varphi(x)w(x)-\varphi(y)w(y)=\varphi(x)\bigl(w(x)-w(y)\bigr)+\bigl(\varphi(x)-\varphi(y)\bigr)w(y)\,.
\end{equation*}
The first term contributes~$C[w]^\xi_{W^{\alpha,\xi}(U)}$. For the second,~$|\varphi(x)-\varphi(y)|\leq\min\{2,|x-y|/\varepsilon\}$, and
\begin{equation*}
	\int_{\Rd}\frac{\min\{2,|z|/\varepsilon\}^\xi}{|z|^{d+\alpha\xi}}\,dz=C\varepsilon^{-\alpha\xi}\,.
\end{equation*}
Pairs with~$x\in\supp\varphi$ and~$y\notin U$ satisfy~$|x-y|\geq\varepsilon$ and contribute at most~$C\varepsilon^{-\alpha\xi}\norm{w}_{L^\xi(U)}^\xi$.
\end{proof}

The next proposition covers a neighborhood of the cube boundaries with radii in~$J$ by~$N\leq C\delta^{-d}$ threefold enlargements of triadic cubes, with side lengths comparable to~$\delta$. The cutoff functions contribute the factor~$\delta^{-(1-t)}$. The localization near~$\rho_1\overline{\cu}_0$ is used with the Sobolev inequality on~$\rho_1\cu_0$ in Sections~\ref{s.local.boundedness} and~\ref{s.harnack}.

\begin{proposition}[Localization]
\label{p.fractional.localization}
There exists a collection~$\{Q_i\Subset\rho_2\cu_0\}_{i=1}^N$ of threefold enlargements of triadic cubes and a collection~$\{\varphi_i\in C^\infty_c(Q_i)\}_{i=1}^N$ of nonnegative functions, both depending only on~$\rho_1,\rho_2$, such that for every~$v\in H^1_\a(\cu_0)$ the function~$F\coloneqq\sum_i\varphi_iv$, extended by zero to~$\Rd$, equals~$v$ almost everywhere near every boundary~$\partial(\tau\cu_0)$ with~$\tau\in J$, and
\begin{equation}
	\norm{F}_{W^{1-t,r}(\Rd)}\leq C\delta^{-(1-t)}\bigl(\lambda_{t,1,q}(\cu_0)^{-\nf12}\Ener_{\rho_2\cu_0}(v)^{\nf12}+\norm{v}_{L^r(\rho_2\cu_0)}\bigr)\,.
	\label{e.localized.fractional.bound}
\end{equation}
If~$v_j\to v$ in~$H^1_\a(\cu_0)$, the corresponding localized functions~$F_j$ converge to~$F$ in~$W^{1-t,r}(\Rd)$. A second choice of cubes~$Q_i$ and functions~$\varphi_i$, depending only on~$\rho_1,\rho_2$, gives the same bound and convergence with~$F=v$ near~$\rho_1\overline{\cu}_0$ instead of near the cube boundaries.
\end{proposition}

\begin{proof}
Choose~$m$ with~$\delta/192<3^{-m}\leq\delta/64$ and let~$Q_1,\ldots,Q_N$ be the threefold enlargements of the triadic cubes~$z+\cu_{-m}$,~$z\in3^{-m}\Zd$, whose closures meet the closed annulus~$\{\rho_1+\delta/8\leq2|x|_\infty\leq\rho_1+5\delta/8\}$. On~$\overline{Q_i}$ one has~$2|x|_\infty\leq\rho_1+5\delta/8+4\cdot3^{-m}<\rho_2$, so~$Q_i\Subset\rho_2\cu_0$; counting the centers gives~$N\leq C3^{md}\leq C\delta^{-d}$, and the overlap is at most~$4^d$. The closed central thirds cover the annulus and lie at distance~$3^{-m}$ from~$\partial Q_i$. Take bumps~$b_i$ equal to one on the closed central thirds, supported at distance~$3^{-m}/2$ from~$\partial Q_i$, with~$|\nabla b_i|\leq C3^m$. Let~$\eta:[0,\infty)\to[0,1]$ be smooth, zero on~$[0,\nf12]$ and one on~$[\nf34,\infty)$. Define
\begin{equation*}
	\varphi_i(x)\coloneqq
	\begin{cases}
		\displaystyle\frac{b_i(x)\eta(\sum_l b_l(x))}{\sum_l b_l(x)},&\sum_l b_l(x)>0\,,\\[6pt]
		0,&\sum_l b_l(x)=0\,.
	\end{cases}
\end{equation*}
Then~$\sum_i\varphi_i=1$ near the annulus, which contains a neighborhood of every cube boundary with radius in~$J$, and~$|\nabla\varphi_i|\leq C3^m$. For the construction near~$\rho_1\overline{\cu}_0$, take the closed central thirds meeting that cube.

\smallskip

By Proposition~\ref{p.lower.fractional},~$v\in W^{1-t,r}(Q_i)$. Lemma~\ref{l.fractional.cutoff} with~$\alpha=1-t$,~$\xi=r$ and~$\varepsilon=c3^{-m}$, for a dimensional~$c>0$, and bounded overlap give
\begin{equation}
	\norm{F}_{W^{1-t,r}(\Rd)}^r\leq C\sum_i\norm{\varphi_iv}_{W^{1-t,r}(\Rd)}^r\leq C\sum_i[v]_{W^{1-t,r}(Q_i)}^r+C\delta^{-(1-t)r}\sum_i\norm{v}_{L^r(Q_i)}^r\,.
	\label{e.localization.sum}
\end{equation}
Each~$Q_i$ has the form~$z+\cu_{1-m}$ with~$z\in3^{-m}\Zd$, so the local bound~\eqref{e.lower.fractional.local} with~$n=m-1$ applies to it. With Minkowski's inequality, H\"older's inequality over~$i$ with~$\nf1r=\nf1{(2q)}+\nf12$, and bounded overlap, this gives
\begin{align*}
	\biggl(\sum_i[v]_{W^{1-t,r}(Q_i)}^r\biggr)^{\nf1r}&\leq C\sum_{k\geq m}3^{-kt}\biggl(\sum_i\sum_{\triangle\in\mathscr T_k(Q_i)}|\triangle|\,|\a_*^{-1}(\triangle)|^q\biggr)^{\nf1{2q}}\biggl(\sum_i\Ener_{Q_i}(v)\biggr)^{\nf12}\\
	&\leq C\Ener_{\rho_2\cu_0}(v)^{\nf12}\sum_{k\geq m}3^{-kt}\biggl(\avsum_{\triangle\in\mathscr T_k}|\a_*^{-1}(\triangle)|^q\biggr)^{\nf1{2q}}\,,
\end{align*}
and~\eqref{e.lower.half.series} bounds the last sum by~$(1-3^{-t})^{-1}\lambda_{t,1,q}(\cu_0)^{-\nf12}$. Inserting this into~\eqref{e.localization.sum} and using bounded overlap for the~$L^r$ terms yields
\begin{equation*}
	\norm{F}_{W^{1-t,r}(\Rd)}\leq C\bigl(\lambda_{t,1,q}(\cu_0)^{-\nf12}\Ener_{\rho_2\cu_0}(v)^{\nf12}+\delta^{-(1-t)}\norm{v}_{L^r(\rho_2\cu_0)}\bigr)\,,
\end{equation*}
which implies~\eqref{e.localized.fractional.bound}, since~$\delta\leq1$. For the convergence, apply~\eqref{e.localization.sum} to~$v_j-v$, whose restrictions converge in each~$W^{1-t,r}(Q_i)$ by Proposition~\ref{p.lower.fractional}.
\end{proof}

Slicing transfers a bound in~$W^{\alpha,\xi}(\Rd)$ to the cube boundaries~$\partial(\tau\cu_0)$, on average over~$\tau$.

\begin{lemma}[Integrated slicing]
\label{l.fractional.slicing}
Let~$F\in W^{\alpha,\xi}(\Rd)$ with~$0<\alpha<1<\xi<\infty$. For almost every~$\tau\in(\nf12,1)$ the restriction of~$F$ to~$\partial(\tau\cu_0)$ is well defined, and for every measurable~$I\subseteq(\nf12,1)$,
\begin{equation}
    \label{e.boundary.slicing}
    \int_I
    \norm{F|_{\partial(\tau\cu_0)}}_{W^{\alpha,\xi}(\partial(\tau\cu_0))}^\xi
    \,d\tau
    \leq C2^\xi \norm{F}_{W^{\alpha,\xi}(\Rd)}^\xi
    \,,
\end{equation}
where~$C$ depends only on~$d$.
\end{lemma}
\begin{proof}
Fix a Borel representative of~$F$.
Write~$S_\tau\coloneqq\partial(\tau\cu_0)$.
Throughout the proof,~$C$ depends only on~$d$.
For~$x\in S_\tau$ and~$\rho>0$, one has
~$\mathcal H^{d-1}(S_\tau\cap B(x,\rho))\leq C\rho^{d-1}$.
Decomposing into dyadic annuli therefore gives
\begin{equation*}
    \begin{aligned}
        \int_{S_\tau\setminus(S_\tau\cap B(x,\rho))}
        \frac{d\mathcal H^{d-1}(y)}{|x-y|^{2d-1+\alpha\xi}}
        &\leq \sum_{j=0}^{\infty}
        \frac{\mathcal H^{d-1}(S_\tau\cap B(x,2^{j+1}\rho))}
             {(2^j\rho)^{2d-1+\alpha\xi}} \\
        &\leq C\rho^{-d-\alpha\xi}
        \sum_{j=0}^{\infty}2^{-j(d+\alpha\xi)}
        \leq C\rho^{-d-\alpha\xi}
        \,.
    \end{aligned}
\end{equation*}

For distinct~$x,y\in S_\tau$, put~$\ell\coloneqq|x-y|$.
Inserting~$F(z)$ and averaging over
~$B((x+y)/2,\ell/2)$, which is contained in both
~$B(x,\ell)$ and~$B(y,\ell)$, yields
\begin{equation*}
	|F(x)-F(y)|^\xi
	\leq C2^\xi\ell^{-d}\biggl(
	\int_{B(x,\ell)}|F(x)-F(z)|^\xi\,dz
	+\int_{B(y,\ell)}|F(y)-F(z)|^\xi\,dz
	\biggr)
	\,.
\end{equation*}
Divide by~$\ell^{d-1+\alpha\xi}$ and integrate over
~$S_\tau\times S_\tau$.
The two contributions are equal by symmetry.
Since~$z\in B(x,|x-y|)$ implies~$|x-y|>|x-z|$,
Tonelli's theorem and the annular estimate give
\begin{equation*}
    \begin{aligned}
        [F|_{S_\tau}]_{W^{\alpha,\xi}(S_\tau)}^\xi
        &\leq C2^\xi
        \int_{S_\tau}\int_{\Rd}|F(x)-F(z)|^\xi \\
        &\qquad\times
        \left(
            \int_{S_\tau\setminus(S_\tau\cap B(x,|x-z|))}
            \frac{d\mathcal H^{d-1}(y)}{|x-y|^{2d-1+\alpha\xi}}
        \right)\,dz\,d\mathcal H^{d-1}(x) \\
        &\leq C2^\xi
        \int_{S_\tau}\int_{\Rd}
        \frac{|F(x)-F(z)|^\xi}{|x-z|^{d+\alpha\xi}}
        \,dz\,d\mathcal H^{d-1}(x)
        \,.
    \end{aligned}
\end{equation*}
Integrating in~$\tau$ and applying the coarea formula~\eqref{e.cubical.coarea}, we obtain
\begin{equation*}
    \int_I
    [F|_{S_\tau}]_{W^{\alpha,\xi}(S_\tau)}^\xi\,d\tau
    \leq C2^\xi
    \int_{\Rd}\int_{\Rd}
    \frac{|F(x)-F(z)|^\xi}{|x-z|^{d+\alpha\xi}}
    \,dz\,dx
    =C2^\xi[F]_{W^{\alpha,\xi}(\Rd)}^\xi
    \,.
\end{equation*}
The~$L^\xi$ bound follows from the same coarea formula.
Combining these estimates proves~\eqref{e.boundary.slicing}.
The coarea formula also shows that the restriction is independent
of the chosen representative for almost every~$\tau$.
\end{proof}

\paragraph{Maxima near a cube boundary.} For~$\nf12<\tau<1$ and~$k\geq0$, the largest norm of the matrices~$\a(\triangle)$,~$\triangle\in\mathscr T_k$, near~$\partial(\tau\cu_0)$ is
\begin{equation}
	A_k(\tau)\coloneqq\max\bigl\{|\a(\triangle)|:\triangle\in\mathscr T_k,\ \dist(\overline \triangle,\partial(\tau\cu_0))\leq100d\cdot3^{-k}\bigr\}\,,
	\label{e.boundary.maxima}
\end{equation}
with an empty maximum equal to zero. The Whitney simplices of~$\tau\overline{\cu}_0$ of size~$3^{-k}$ in~$\cu_0$ lie within Euclidean distance~$27\sqrt d\,3^{-k}<100d\cdot3^{-k}$ of~$\partial(\tau\cu_0)$, by~\eqref{e.whitney.distance}, so their matrices enter~$A_k(\tau)$. The estimate of Section~\ref{s.good.radius.energy} is controlled by the discounted series of these maxima,
\begin{equation}
	S(\tau)\coloneqq\sum_{k\geq0}3^{-k\sigma}A_k(\tau)^{\nf12}\,.
	\label{e.boundary.maxima.series}
\end{equation}
Its weight~$3^{-k\sigma}$ is the discount~$3^{-ks}$ of~$\Lambda_{s,1,p}(\cu_0)$ times~$3^{-k(d-1)/(2p)}$, which compensates the factor~$3^{k(d-1)/(2p)}$ lost by taking the maximum within distance comparable to~$3^{-k}$ of a cube boundary, of dimension~$d-1$: Lemma~\ref{l.radius.averages} shows that~$S\in L^{2p}$ in~$\tau$, with norm controlled by~$\Lambda_{s,1,p}(\cu_0)^{\nf12}$. For~$v\in H^1_\a(\cu_0)$ and~$\tau\in\R$ we also use the centered maximal function of the energy on cube boundaries,
\begin{equation}
	D_v(\tau)\coloneqq\sup_{\varepsilon>0}\frac1{2\varepsilon}\int_{(\tau-\varepsilon,\tau+\varepsilon)\cap(\rho_1,\rho_2)}\int_{\partial(l\cu_0)}\nabla v\cdot\a\nabla v\,d\mathcal H^{d-1}\,dl\,.
	\label{e.boundary.energy.maximal}
\end{equation}
It bounds the energy of~$v$ on the Whitney simplices of each size around~$\partial(\tau\cu_0)$; see~\eqref{e.energy.scale}.

\begin{lemma}[Integrability in the radius]
\label{l.radius.averages}
Let~$v\in H^1_\a(\cu_0)$. The functions~$S$ and~$D_v$ of~\eqref{e.boundary.maxima.series} and~\eqref{e.boundary.energy.maximal} are measurable. Moreover,
\begin{equation}
	\norm{S}_{L^{2p}(\frac12,1)}\leq\frac{C(d,p)}{1-3^{-s}}\Lambda_{s,1,p}(\cu_0)^{\nf12}\,,
	\label{e.boundary.maxima.bound}
\end{equation}
and for every~$M>0$,
\begin{equation}
	\bigl|\{\tau\in\R:D_v(\tau)>M\}\bigr|\leq\frac{6}{M}\Ener_{\rho_2\cu_0}(v)\,.
	\label{e.energy.weak.bound}
\end{equation}
\end{lemma}

\begin{proof}
For each simplex~$\triangle\in\mathscr T_k$ the function~$\tau\mapsto\dist(\overline\triangle,\partial(\tau\cu_0))$ is continuous, so~$A_k$ in~\eqref{e.boundary.maxima} is a finite maximum of measurable functions and~$S$ is measurable; moreover~$\triangle$ enters~$A_k(\tau)$ only for radii in an interval of length~$C3^{-k}$. Bounding the maximum by a sum,
\begin{equation*}
	\int_{1/2}^1 A_k(\tau)^p\,d\tau\leq C3^{-k}\sum_{\triangle\in\mathscr T_k}|\a(\triangle)|^p=C\,3^{k(d-1)}\avsum_{\triangle\in\mathscr T_k}|\a(\triangle)|^p\,.
\end{equation*}
Minkowski's inequality in~$L^{2p}$ gives
\begin{equation*}
	\norm{S}_{L^{2p}(\nf12,1)}\leq C\sum_{k\geq0}3^{-ks}\biggl(\avsum_{\triangle\in\mathscr T_k}|\a(\triangle)|^p\biggr)^{\nf1{2p}}=\frac{C}{1-3^{-s}}\Lambda_{s,1,p}(\cu_0)^{\nf12}\,,
\end{equation*}
which proves~\eqref{e.boundary.maxima.bound}.

\smallskip

Splitting~$\rho_2\cu_0\setminus\rho_1\overline{\cu}_0$ according to the largest signed coordinate gives the cubical coarea formula
\begin{equation}
	\int_{\rho_1}^{\rho_2}\int_{\partial(l\cu_0)}g\,d\mathcal H^{d-1}\,dl=2\int_{\rho_2\cu_0\setminus\rho_1\overline{\cu}_0}g\,dx\qquad(g\geq0)\,.
	\label{e.cubical.coarea}
\end{equation}
Hence the energy on cube boundaries, extended by zero outside~$(\rho_1,\rho_2)$, has integral at most~$2\Ener_{\rho_2\cu_0}(v)$, and~$D_v$ is its centered maximal function, which is lower semicontinuous. Cover a compact subset of~$\{D_v>M\}$ by finitely many centered intervals on which the average of this energy exceeds~$M$. Retain a longest interval, discard those meeting it, and repeat. The retained intervals are disjoint, and each discarded one lies in the threefold enlargement of a retained one, so the compact set has measure at most~$3/M$ times its integral, hence at most~$6\Ener_{\rho_2\cu_0}(v)/M$. Inner regularity gives~\eqref{e.energy.weak.bound}.
\end{proof}

The good radius is chosen once, for~$v$ and one sequence of smooth approximations, and it serves all truncations~$\min\{(v-k)_+,N\}$ of~$v$, with bounds in terms of~$v$ alone: Section~\ref{s.local.boundedness} takes~$k=b-a$ to compare two truncation levels of a subsolution, and Section~\ref{s.harnack} takes finite~$N$ to test with bounded powers of a supersolution. The case~$k=0$,~$N=\infty$ is the convergence of the approximations themselves.

\begin{proposition}[Existence of a good radius]
\label{p.good.radius}
Let~$v\in H^1_\a(\cu_0)$ be nonnegative with~$\norm{v}_{L^r(\rho_2\cu_0)}<\infty$, and let~$v_i\to v$ be nonnegative smooth approximations in~$H^1_\a(\cu_0)$. There exist~$\tau\in J$ and a subsequence such that~$v$ and all its truncations have traces on~$\partial(\tau\cu_0)$, which commute with truncations and are denoted by the same symbols, and, for every~$k\in\R$ and~$0<N\leq\infty$,
\begin{equation}
	\min\{(v_i-k)_+,N\}\to\min\{(v-k)_+,N\}\qquad\mbox{in }W^{1-t,r}(\partial(\tau\cu_0))\,.
	\label{e.good.radius.convergence}
\end{equation}
At this radius,
\begin{align}
	S(\tau)&\leq C\delta^{-\nf1{2p}}\Lambda_{s,1,p}(\cu_0)^{\nf12}\,,
	\label{e.good.radius.matrices}\\
	D_v(\tau)&\leq C\delta^{-1}\Ener_{\rho_2\cu_0}(v)\,,
	\label{e.good.radius.boundary.energy}\\
	\norm{v}_{W^{1-t,r}(\partial(\tau\cu_0))}&\leq C\delta^{-(1-t)-\nf1r}\bigl(\lambda_{t,1,q}(\cu_0)^{-\nf12}\Ener_{\rho_2\cu_0}(v)^{\nf12}+\norm{v}_{L^r(\rho_2\cu_0)}\bigr)\,,
	\label{e.good.radius.fractional}\\
	\norm{v}_{L^r(\partial(\tau\cu_0))}&\leq C\delta^{-\nf1r}\norm{v}_{L^r(\rho_2\cu_0)}\,,
	\label{e.good.radius.trace.lr}
\end{align}
and~$D_w\leq D_v$ everywhere for every truncation~$w=\min\{(v-k)_+,N\}$. If moreover~$\norm{v}_{L^2(\rho_2\cu_0)}<\infty$, the radius can be chosen so that also
\begin{equation}
	\norm{v}_{L^2(\partial(\tau\cu_0))}\leq C\delta^{-\nf12}\norm{v}_{L^2(\rho_2\cu_0)}\,.
	\label{e.good.radius.trace.l2}
\end{equation}
\end{proposition}

\begin{proof}
We first select a full-measure set of radii~$\tau\in J$ on which the traces of a subsequence satisfy~$v_i\to v$ in~$W^{1-t,r}(\partial(\tau\cu_0))$. We then exclude the radii where a quantitative bound fails and treat all truncations at one remaining radius.

\smallskip

\emph{Step 1.} Traces for almost every radius. Fix a representative of~$v$, to be used throughout the proof; for almost every~$\tau\in J$, its restriction to~$\partial(\tau\cu_0)$ is the trace of~$v$. Apply Proposition~\ref{p.fractional.localization} to~$v$ and all~$v_i$ with the same cover, so~$F_i\to F$ in~$W^{1-t,r}(\Rd)$, and pass to a subsequence with~$\sum_i\norm{F_i-F}_{W^{1-t,r}(\Rd)}^r<\infty$. By~\eqref{e.boundary.slicing}, for~$\alpha=1-t$ and~$\xi=r$, applied to each~$F_i-F$ and summed over~$i$, the~$r$-th powers of the norms of~$F_i-F$ on~$\partial(\tau\cu_0)$ have a finite sum for almost every~$\tau\in J$, and~$F$ and~$F_i$ agree with~$v$ and~$v_i$ on these cube boundaries for almost every radius, by~\eqref{e.cubical.coarea}. Hence, for almost every~$\tau\in J$, the series~$\sum_i\norm{v_i-v}_{W^{1-t,r}(\partial(\tau\cu_0))}^r$ converges, so~$v_i\to v$ in~$W^{1-t,r}(\partial(\tau\cu_0))$.

\smallskip

\emph{Step 2.} Excluding the bad radii. The interval~$J$ has length~$\delta/4$. Fix~$K\geq1$. By~\eqref{e.boundary.maxima.bound} and Chebyshev's inequality,
\begin{equation*}
	\bigl|\{\tau\in J:S(\tau)>K\delta^{-\nf1{2p}}\Lambda_{s,1,p}(\cu_0)^{\nf12}\}\bigr|\leq C\delta K^{-2p}\,,
\end{equation*}
and by~\eqref{e.energy.weak.bound},
\begin{equation*}
	\bigl|\{\tau\in J:D_v(\tau)>K\delta^{-1}\Ener_{\rho_2\cu_0}(v)\}\bigr|\leq6\delta K^{-1}\,.
\end{equation*}
Lemma~\ref{l.fractional.slicing} and~\eqref{e.localized.fractional.bound} for~$v$ show that~\eqref{e.good.radius.fractional} fails, with~$K$ in place of~$C$, on a set of measure at most~$C\delta K^{-r}$. For~$b\in\{r,2\}$, the coarea formula gives
\begin{equation*}
	\int_J\norm{v}_{L^b(\partial(\tau\cu_0))}^b\,d\tau\leq2\norm{v}_{L^b(\rho_2\cu_0)}^b\,.
\end{equation*}
Thus~\eqref{e.good.radius.trace.lr} fails on a set of measure at most~$2\delta K^{-r}$ and, if~$\norm{v}_{L^2(\rho_2\cu_0)}<\infty$,~\eqref{e.good.radius.trace.l2} fails on a set of measure at most~$2\delta K^{-2}$. If a controlling quantity vanishes, the corresponding bad set is null. Choosing~$K$ large, depending only on the parameters, makes the union of these sets shorter than~$|J|/2$, and any~$\tau$ outside it in the full-measure set of Step~1 will do.

\smallskip

\emph{Step 3.} Truncations. Fix~$\tau$ as in Step~2, and then~$k\in\R$ and~$0<N\leq\infty$; no further radii are excluded. Let~$w\coloneqq\min\{(v-k)_+,N\}$ and~$w_i\coloneqq\min\{(v_i-k)_+,N\}$. Since~$w=(v-k)_+$ if~$N=\infty$ and~$w=(v-k)_+-(v-k-N)_+$ if~$N<\infty$, Proposition~\ref{p.weighted.calculus}\ref{i.weighted.chain} gives~$w\in H^1_\a(\cu_0)$ and
\begin{equation*}
	\nabla w\cdot\a\nabla w=\indc_{\{k<v<k+N\}}\nabla v\cdot\a\nabla v\leq\nabla v\cdot\a\nabla v\,,
\end{equation*}
where~$\{k<v<k+\infty\}=\{v>k\}$. Hence~$D_w\leq D_v$ everywhere. Positive parts are continuous on~$H^1_\a(\cu_0)$ by the same part, so~$w_i\to w$ in~$H^1_\a(\cu_0)$.

\smallskip

The trace of~$v$ on~$\partial(\tau\cu_0)$ is the restriction of the representative of~$v$ fixed in Step~1. Composing this representative with the one-Lipschitz map~$y\mapsto\min\{(y-k)_+,N\}$ gives a representative of~$w$, whose restriction is the same truncation of the trace of~$v$. This restriction is the trace of~$w$, so traces commute with positive parts and truncations, and we now show that it is the limit of the traces of the~$w_i$. The~$w_i$ are Lipschitz on~$\tau\overline{\cu}_0$, since the~$v_i$ are smooth, and their traces are their restrictions. On~$\partial(\tau\cu_0)$ one has~$|w_i-w|\leq|v_i-v|$ pointwise, so~$w_i\to w$ in~$L^r(\partial(\tau\cu_0))$, by Step~1. For~$\alpha=1-t$ and~$\xi=r$, the seminorm~\eqref{e.boundary.fractional.seminorm} is the~$L^r$ norm of the differences~$(x,y)\mapsto f(x)-f(y)$ on~$\partial(\tau\cu_0)\times\partial(\tau\cu_0)$, with respect to the measure
\begin{equation*}
	|x-y|^{-(d-1)-(1-t)r}\,d\mathcal H^{d-1}(y)\,d\mathcal H^{d-1}(x)\,.
\end{equation*}
Since the series of Step~1 converges at~$\tau$, the~$v_i$ converge to~$v$ almost everywhere on~$\partial(\tau\cu_0)$, and this~$L^r$ space contains the function
\begin{equation*}
	(x,y)\mapsto\Bigl(|v(x)-v(y)|^r+\sum_i\bigl|(v_i-v)(x)-(v_i-v)(y)\bigr|^r\Bigr)^{\nf1r}\,.
\end{equation*}
It bounds the differences of~$v$, and twice it bounds those of each~$v_i$. Truncation is one-Lipschitz, so the differences of~$w$ and~$w_i$ obey the same bounds, and those of~$w_i-w$ are bounded by three times this function. They tend to zero almost everywhere, since truncation is continuous, and dominated convergence gives~$[w_i-w]_{W^{1-t,r}(\partial(\tau\cu_0))}\to0$. This completes the proof of~\eqref{e.good.radius.convergence}.
\end{proof}

\paragraph{The order of the choices.} We call a radius~$\tau\in J$ supplied by Proposition~\ref{p.good.radius} a \emph{good radius}. For~$v$ and its smooth approximations~$(v_i)$, the proposition fixes this radius and a subsequence before~$k$,~$N$ or the extension width~$h$ is chosen. At this radius,~\eqref{e.good.radius.convergence} holds for each fixed~$k$ and~$N$; convergence is not asserted uniformly in these parameters. The bounds~\eqref{e.good.radius.matrices}--\eqref{e.good.radius.trace.l2} are independent of~$k$,~$N$ and~$h$. The series~$S(\tau)$ already includes the matrix maxima at every scale, so the estimates of Section~\ref{s.good.radius.energy} hold for every admissible~$h$. Sections~\ref{s.local.boundedness} and~\ref{s.harnack} choose~$h$ afterward to make the fractional-seminorm term small.

\section{The piecewise harmonic extension}
\label{s.whitney.extension}

For Lipschitz~$f$ on~$\partial(\tau\cu_0)$, with~$\tau\in J$, we first construct the piecewise affine extension~$L_hf$ outside~$\tau\cu_0$. At free vertices its values are boundary averages of~$f$ multiplied by the cutoff~$\omega_z$. Values at hanging vertices are interpolated from larger simplices, and~$L_hf$ is affine on each Whitney simplex. Replacing~$L_hf$ on each simplex by its~$\a$-harmonic representative defines~$H_hf$. The estimates below bound the energy of~$H_hf$ using the matrices~$\a(\triangle)$ and boundary norms of~$f$. At a good radius, Section~\ref{s.good.radii} bounds~$S(\tau)$ and the stated trace norms of~$v$, and supplies convergence of the approximation traces used later.

\smallskip

Fix~$\nf12\leq\rho_1<\rho_2\leq1$ and~$\tau\in J$, and choose a triadic width
\begin{equation}
	h\in\{3^{-n}:n\in\N\}\,,\qquad 0<h\leq\frac{\rho_2-\tau}{100d}\,.
	\label{e.extension.width}
\end{equation}
Constants in this section depend only on~$d,\alpha,\xi$, and~$\operatorname{Lip}$ refers to Euclidean distance.

\paragraph{The piecewise affine extension.} A point~$x\notin\tau\cu_0$ has maximum-norm distance~$|x|_\infty-\tau/2$ to~$\tau\overline{\cu}_0$. Its coordinate projection~$y_x\in\partial(\tau\cu_0)$ is given by
\begin{equation*}
	(y_x)_i\coloneqq\max\{-\tau/2,\min\{x_i,\tau/2\}\}\,.
\end{equation*}
The projection is one-Lipschitz in~$x$. Let~$\omega(\varrho)\coloneqq\min\{1,(2-2\varrho)_+\}$, which equals one on~$[0,\nf12]$ and vanishes on~$[1,\infty)$. For Lipschitz~$f$ on~$\partial(\tau\cu_0)$ and a free vertex~$z$, write~$\omega_z\coloneqq\omega((|z|_\infty-\tau/2)/h)$ and let
\begin{equation}
	L_hf(z)\coloneqq\omega_z\fint_{\Sigma_z}f\,d\mathcal H^{d-1}\,,\qquad \Sigma_z\coloneqq\partial(\tau\cu_0)\cap B\bigl(y_z,(|z|_\infty-\tau/2)/(100d)\bigr)\,.
	\label{e.extension.definition}
\end{equation}
By Lemma~\ref{l.whitney.interpolation}, these values extend to a unique continuous function~$L_hf$ on~$\Rd\setminus\tau\overline{\cu}_0$ that is affine on every Whitney simplex, the \emph{piecewise affine extension} of~$f$. See Figure~\ref{f.whitney.vertices}. Let~$\mathscr W_h$ be the set of Whitney simplices whose selected cube~$D$ satisfies~$\dist_\infty(\overline D,\tau\overline{\cu}_0)<h$, and let~$\mathscr W_h^j$ consist of those of size~$3^{-j}$, in the sense of~\eqref{e.whitney.simplex.size}.

Neither set depends on~$f$. If a selected cube~$D$ satisfies~$\dist_\infty(\overline D,\tau\overline{\cu}_0)\geq h$, every free vertex~$z\in\overline D$ has~$|z|_\infty-\tau/2\geq h$, so its value~\eqref{e.extension.definition} vanishes. All prescribed values at free vertices of~$\overline D$ are therefore zero. Lemma~\ref{l.whitney.interpolation} gives~$L_hf=0$ on~$\overline D$.

\begin{proposition}[The piecewise affine extension]
\label{p.affine.extension}
For Lipschitz~$f$, the function~$L_hf$ is linear in~$f$, nonnegative when~$f\geq0$, and a Lipschitz extension of~$f$ to~$\Rd\setminus\tau\cu_0$. Its values lie between~$\min\{0,\min f\}$ and~$\max\{0,\max f\}$. The function~$L_hf$ vanishes on the Whitney simplices not in~$\mathscr W_h$, and the simplices of~$\mathscr W_h$ and the support of~$L_hf$ have closures in~$(\tau+3h)\overline{\cu}_0\Subset\rho_2\cu_0$. For~$0<\alpha<1$ and~$1\leq\xi\leq2$ there exists~$C(d,\alpha,\xi)<\infty$ such that, for every~$j\geq0$ and~$b\in\{2,\xi\}$,
\begin{align}
	\sum_{\triangle\in\mathscr W_h^j}|\triangle|\,|\nabla(L_hf|_\triangle)|^2&\leq C3^{-j}\bigl(3^j\,3^{-j(\alpha-(d-1)(\frac1\xi-\frac12))}[f]_{W^{\alpha,\xi}(\partial(\tau\cu_0))}\bigr)^2\notag\\
	&\qquad+C3^{-j}\bigl(h^{-1}3^{j(d-1)(\frac1b-\frac12)}\norm{f}_{L^b(\partial(\tau\cu_0))}\bigr)^2\,.
	\label{e.extension.scale}
\end{align}
\end{proposition}

\begin{proof}
\emph{Step 1.} Linearity, pointwise bounds and support. By the uniqueness in Lemma~\ref{l.whitney.interpolation},~$L_hf$ is linear in~$f$. The same lemma places the values of~$L_hf$ on the closure of a selected cube~$D$ between the smallest and the largest value~\eqref{e.extension.definition} at free vertices in~$\overline D$, hence between~$\min\{0,\min f\}$ and~$\max\{0,\max f\}$, since~$0\leq\omega_z\leq1$ at every free vertex. In particular~$L_hf\geq0$ if~$f\geq0$. As noted after~\eqref{e.extension.definition},~$L_hf$ vanishes on the Whitney simplices not in~$\mathscr W_h$. If the simplices in~$D$ belong to~$\mathscr W_h$, then~\eqref{e.whitney.distance} gives~$2\ell(D)<h$, and the points of~$\overline D$ have maximum-norm distance less than~$h+\ell(D)<\nf{3h}2$ from~$\tau\overline{\cu}_0$. With~$\ell(D)=3^{1-j}$, this gives
\begin{equation}
	3^{-j}<\nf h6\,,\qquad \overline D\subset(\tau+3h)\cu_0\,,\qquad (\tau+3h)\overline{\cu}_0\Subset\rho_2\cu_0\,,
	\label{e.extension.support}
\end{equation}
the last because~$3h<\rho_2-\tau$ by~\eqref{e.extension.width}. In particular, the simplices of~$\mathscr W_h$ lie in~$\cu_0$, so~$\mathscr W_h^j\subset\mathscr T_j$.

\smallskip

\emph{Step 2.} Differences of the values at free vertices. We bound the difference between the values at two free vertices by estimating the boundary averages and the cutoff values~$\omega_z,\omega_{z'}$ separately. Let~$D$ be a selected cube of side length~$3^{1-j}$ whose simplices belong to~$\mathscr W_h$, and let~$x_D$ be its center. Set~$\Sigma_D\coloneqq\partial(\tau\cu_0)\cap B(y_{x_D},\sqrt d\,\ell(D))$.

By~\eqref{e.whitney.distance}, the free vertices~$z\in\overline D$ satisfy~$6\cdot3^{-j}<|z|_\infty-\tau/2\leq27\cdot3^{-j}$. Hence the radius of~$\Sigma_z$ is comparable to~$3^{-j}$ and less than~$\ell(D)/10<\tau/2$.

Since~$|y_z-y_{x_D}|\leq|z-x_D|\leq\sqrt d\,\ell(D)/2$, we have~$\Sigma_z\subset\Sigma_D$. Boundary patches~$\partial(\tau\cu_0)\cap B(y,\varrho)$ with~$\varrho<\tau/2$ have area comparable to~$\varrho^{d-1}$, so~$\mathcal H^{d-1}(\Sigma_z)$ is comparable to~$3^{-j(d-1)}$.

For free~$z,z'\in\overline D$, Jensen's inequality and~$|x-y|\leq C3^{-j}$ for~$x,y\in\Sigma_D$ give
\begin{equation*}
	\biggl|\fint_{\Sigma_z}f-\fint_{\Sigma_{z'}}f\biggr|^\xi\leq C3^{2j(d-1)}\int_{\Sigma_D}\int_{\Sigma_D}|f(x)-f(y)|^\xi\leq C3^{j(d-1)}\bigl(3^{-j\alpha}[f]_{W^{\alpha,\xi}(\Sigma_D)}\bigr)^\xi\,,
\end{equation*}
where~$[f]_{W^{\alpha,\xi}(\Sigma_D)}$ restricts the double integral in~\eqref{e.boundary.fractional.seminorm} to~$\Sigma_D\times\Sigma_D$. For~$b\in\{2,\xi\}$, H\"older's inequality bounds each average by~$C3^{j(d-1)/b}\norm{f}_{L^b(\Sigma_D)}$. Since~$\omega$ is two-Lipschitz and~$|z-z'|_\infty\leq\ell(D)$, we have~$|\omega_z-\omega_{z'}|\leq C3^{-j}/h$. Separating the term with~$\omega_z$ from the term with~$\omega_z-\omega_{z'}$ in~$L_hf(z)-L_hf(z')$, we obtain
\begin{equation}
	|L_hf(z)-L_hf(z')|\leq C3^{-j\alpha}3^{\frac{d-1}\xi j}[f]_{W^{\alpha,\xi}(\Sigma_D)}+Ch^{-1}3^{-j}3^{\frac{d-1}bj}\norm{f}_{L^b(\Sigma_D)}\,.
	\label{e.extension.vertex.difference}
\end{equation}

\smallskip

\emph{Step 3.} Summation over the simplices of one size. The interpolation estimate converts the vertex differences into gradient bounds on the simplices. Bounded overlap of the patches~$\Sigma_D$ then allows us to sum these bounds. By~\eqref{e.interpolation.gradient} and~$\ell(D)=3^{1-j}$, the gradient of~$L_hf$ on~$D$ is at most~$C3^j$ times the right side of~\eqref{e.extension.vertex.difference}. Squaring and multiplying by~$|\triangle|=3^{-j}\,3^{-j(d-1)}/d!$ gives~\eqref{e.extension.scale} for a single Whitney simplex~$\triangle\subset D$, with~$\Sigma_D$ in place of~$\partial(\tau\cu_0)$.

A point of~$\Sigma_D$ is within~$C3^{-j}$ of~$x_D$, by~\eqref{e.whitney.distance}, and distinct cubes~$D$ contain distinct simplices of~$\mathscr T_j$, so, by~\eqref{e.simplex.count}, every point lies in at most~$C$ of the patches~$\Sigma_D$ of these cubes. Hence the sums over these cubes of~$\norm{f}^b_{L^b(\Sigma_D)}$ and of~$[f]^\xi_{W^{\alpha,\xi}(\Sigma_D)}$ are at most~$C\norm{f}^b_{L^b(\partial(\tau\cu_0))}$ and~$C[f]^\xi_{W^{\alpha,\xi}(\partial(\tau\cu_0))}$. For the fractional seminorms, each pair~$(x,y)$ belongs to at most~$C$ of the sets~$\Sigma_D\times\Sigma_D$, since each such patch contains~$x$.

Each selected cube~$D$ contains~$3^dd!$ Whitney simplices. For nonnegative numbers~$c_D$, the bound~$b\leq2$ gives~$\sum_Dc_D^2\leq(\sum_Dc_D^b)^{\nf2b}$, and the same inequality holds with~$\xi$ in place of~$b$. Summing over~$\mathscr W_h^j$ therefore proves~\eqref{e.extension.scale}.

\smallskip

\emph{Step 4.} The Lipschitz extension. For Lipschitz~$f$, two averages over patches in~$\Sigma_D$ differ by at most~$C3^{-j}\operatorname{Lip}(f)$, and each is at most~$\norm{f}_{L^\infty}$, so the vertex-difference argument of Step~2 and~\eqref{e.interpolation.gradient} give
\begin{equation*}
	|\nabla(L_hf|_\triangle)|\leq C\bigl(\operatorname{Lip}(f)+h^{-1}\norm{f}_{L^\infty}\bigr)\qquad(\triangle\in\mathscr W_h)\,.
\end{equation*}
Since~$L_hf=0$ on the other Whitney simplices, and a segment in~$\Rd\setminus\tau\overline{\cu}_0$ meets finitely many of them, this bounds the difference quotients of~$L_hf$ along such segments.

If~$x$ lies in the closure of a selected cube~$D$ and~$|x|_\infty-\tau/2<\nf h3$, then~\eqref{e.whitney.distance} gives~$|z-x|_\infty\leq\ell(D)<(|x|_\infty-\tau/2)/2$ for the free vertices~$z\in\overline D$, so~$\omega_z=1$ for all these vertices and their patches lie within~$C(|x|_\infty-\tau/2)$ of~$y_x$. Each prescribed value at a free vertex of~$D$ differs from~$f(y_x)$ by at most~$C\operatorname{Lip}(f)(|x|_\infty-\tau/2)$. Since~$L_hf(x)$ lies between these values by Lemma~\ref{l.whitney.interpolation},
\begin{equation*}
	|L_hf(x)-f(y_x)|\leq C\operatorname{Lip}(f)(|x|_\infty-\tau/2)\,.
\end{equation*}
Thus~$L_hf$ extends continuously to~$\partial(\tau\cu_0)$ with values~$f$. If a segment crosses the cube~$\tau\overline{\cu}_0$, the bound applies to its two exterior pieces by continuity. At the endpoints of the part of the segment inside the cube,~$L_hf=f$, so the Lipschitz bound for~$f$ controls the difference between these points. Hence~$L_hf$ is Lipschitz.
\end{proof}

Assume~\eqref{e.weighted.hypotheses} on~$\cu_0$. The \emph{piecewise harmonic extension}~$H_hf$ equals~$H_\triangle(L_hf|_\triangle)$ on each~$\triangle\in\mathscr W_h$ and zero on the other Whitney simplices. The coefficients enter only through these simplexwise harmonic replacements, and the simplex boundaries have measure zero.

\begin{proposition}[Energy and admissibility of the piecewise harmonic extension]
\label{p.whitney.extension}
For Lipschitz~$f$, the function~$H_hf$ is linear in~$f$, nonnegative when~$f\geq0$, and satisfies~$\min\{0,\min f\}\leq H_hf\leq\max\{0,\max f\}$ almost everywhere. It has Sobolev trace~$f$ on~$\partial(\tau\cu_0)$, belongs to~$W^{1,1}(\cu_0\setminus\tau\overline{\cu}_0)$ with finite energy, and vanishes outside~$(\tau+3h)\overline{\cu}_0$. For every~$j\geq0$ and~$\triangle\in\mathscr W_h^j$,
\begin{equation}
	\Ener_\triangle(H_hf)=|\triangle|\,\nabla(L_hf|_\triangle)\cdot\a(\triangle)\nabla(L_hf|_\triangle)\leq A_j(\tau)|\triangle|\,|\nabla(L_hf|_\triangle)|^2\,,
	\label{e.harmonic.energy}
\end{equation}
so the energy of~$H_hf$ on the simplices of~$\mathscr W_h^j$ obeys the estimate~\eqref{e.extension.scale} with the extra factor~$A_j(\tau)$ on its right side. If~$w$ is Lipschitz on~$\tau\overline{\cu}_0$ with~$w=f$ on~$\partial(\tau\cu_0)$, the function equal to~$w$ on~$\tau\overline{\cu}_0$ and to~$H_hf$ outside belongs to~$\Hzero{\a}(\cu_0)$, has compact support in~$\rho_2\cu_0$, and is nonnegative if~$w\geq0$.
\end{proposition}

\begin{proof}
On~$\triangle\in\mathscr W_h^j$ the function~$L_hf$ is affine, so the identity in~\eqref{e.harmonic.energy} is~\eqref{e.upper.dirichlet}. The inequality holds because the simplices of~$\mathscr W_h^j$ lie in~$\cu_0$, by Proposition~\ref{p.affine.extension}, and therefore enter the maximum~$A_j(\tau)$ of~\eqref{e.boundary.maxima}. Linearity follows from Proposition~\ref{p.affine.extension} and the linearity of harmonic replacement in Proposition~\ref{p.harmonic.replacement}. The bounds for~$L_hf$ and the maximum principle in Proposition~\ref{p.harmonic.replacement} give~$\min\{0,\min f\}\leq H_hf\leq\max\{0,\max f\}$ almost everywhere, including nonnegativity when~$f\geq0$.

\smallskip

To show that the function equal to~$w$ inside~$\tau\cu_0$ and to~$H_hf$ outside belongs to~$\Hzero{\a}(\cu_0)$, we join~$w$ to~$L_hf$ and then add, after zero extension, the functions
\begin{equation*}
	\eta_\triangle\coloneqq H_\triangle(L_hf|_\triangle)-L_hf|_\triangle\qquad(\triangle\in\mathscr W_h)\,.
\end{equation*}
We first show that these functions sum to an element of~$\Hzero{\a}(\cu_0)$. The function~$L_hf$ is Lipschitz and vanishes outside~$(\tau+3h)\cu_0$, so its exterior energy is finite, since~$\tr\a\in L^1(\cu_0)$. By~\eqref{e.harmonic.minimization}, each~$\eta_\triangle\in\Hzero{\a}(\triangle)$ has energy~$\Ener_\triangle(L_hf)-\Ener_\triangle(H_hf)$, so their energies are summable. Extending each~$\eta_\triangle$ by zero~\eqref{e.weighted.zero.extension} and summing over an enumeration of~$\mathscr W_h$, the partial sums form a Cauchy sequence in the energy seminorm, hence in~$H^1_\a(\cu_0)$ by~\eqref{e.weighted.coercivity}, and converge to some~$\eta\in\Hzero{\a}(\cu_0)$.

Their~$L^1$ limit identifies~$\eta$ with~$H_hf-L_hf$ outside~$\tau\overline{\cu}_0$ and with zero on~$\tau\cu_0$. Since~$\eta\in W^{1,1}(\cu_0)$ vanishes on~$\tau\cu_0$, its traces on~$\partial(\tau\cu_0)$ from both sides are zero. The function~$L_hf$ is continuous up to~$\partial(\tau\cu_0)$ with values~$f$, so~$H_hf=L_hf+\eta$ has trace~$f$, finite energy, and belongs to~$W^{1,1}(\cu_0\setminus\tau\overline{\cu}_0)$.

\smallskip

Finally, let~$u_0$ equal~$w$ on~$\tau\overline{\cu}_0$ and~$L_hf$ outside. A line segment meets the convex cube~$\tau\overline{\cu}_0$ in at most one subsegment. On that subsegment~$u_0=w$, while on the remaining parts~$u_0=L_hf$. The values match at the crossing points because both functions equal~$f$ on the boundary. The Lipschitz bounds for~$w$ and~$L_hf$ therefore give a Lipschitz bound for~$u_0$ along the whole segment. It is supported in~$(\tau+3h)\overline{\cu}_0\Subset\rho_2\cu_0$ and lies in~$\Hzero{\a}(\cu_0)$ by Proposition~\ref{p.weighted.calculus}\ref{i.weighted.products}. Thus~$u_0+\eta$ equals~$w$ on~$\tau\overline{\cu}_0$ and~$H_hf$ outside. It is nonnegative when~$w\geq0$.
\end{proof}

\section{Energy in cubes with good radii}
\label{s.good.radius.energy}

This section bounds the energy of a nonnegative subsolution~$w=(v-k)_+$ on~$\tau\cu_0$ at a good radius for~$v$ and its smooth approximations~$(v_i)$. We test with a function equal to~$w_i=(v_i-k)_+$ inside~$\tau\cu_0$ and to~$H_h(w_i|_{\partial(\tau\cu_0)})$ outside. This bounds the interior pairing by an exterior pairing. Lemma~\ref{l.exterior.integral} estimates that pairing over the Whitney simplices, and Proposition~\ref{p.good.radius.energy} passes the bound to the limit.

\smallskip

Throughout this section~\eqref{e.weighted.hypotheses} and the coarse ellipticity condition~\eqref{e.spatial.moment.range} hold on~$\cu_0$, the parameters are those of Section~\ref{s.triadic.simplices}, and~$\nf12\leq\rho_1<\rho_2\leq1$, with~$\delta$ and~$J$ as in~\eqref{e.selection.parameters}.

\begin{lemma}[Estimating the exterior integral]
\label{l.exterior.integral}
Let~$\tau\in J$, let~$h$ be a triadic width with~$0<h\leq\delta/(200d)$, let~$g$ be Lipschitz on~$\partial(\tau\cu_0)$, and let~$v\in H^1_\a(\cu_0)$. Then
\begin{equation}
	\biggl|\int_{\cu_0\setminus\tau\overline{\cu}_0}\nabla H_hg\cdot\a\nabla v\biggr|\leq C S(\tau)D_v(\tau)^{\nf12}\bigl(h^\theta[g]_{W^{1-t,r}(\partial(\tau\cu_0))}+h^{-\sigma}\norm{g}_{L^2(\partial(\tau\cu_0))}\bigr)\,,
	\label{e.summed.pairing}
\end{equation}
and
\begin{equation}
	\biggl|\int_{\cu_0\setminus\tau\overline{\cu}_0}\nabla H_hg\cdot\a\nabla v\biggr|\leq C S(\tau)D_v(\tau)^{\nf12}\bigl(h^\theta[g]_{W^{1-t,r}(\partial(\tau\cu_0))}+h^{-(\sigma+\sigma_*-t)}\norm{g}_{L^r(\partial(\tau\cu_0))}\bigr)\,.
	\label{e.summed.pairing.lr}
\end{equation}
\end{lemma}

\begin{proof}
\emph{Step 1.} The energy of~$v$ on~$\mathscr W_h^j$. Since~$\rho_2-\tau>\delta/2$, the width~$h$ satisfies~\eqref{e.extension.width}. By~\eqref{e.whitney.distance}, the Whitney simplices of size~$3^{-j}$ lie in~$(\tau+54\cdot3^{-j})\overline{\cu}_0\setminus\tau\overline{\cu}_0$. Whenever~$\mathscr W_h^j\neq\varnothing$,~\eqref{e.extension.support} gives
\begin{equation*}
	54\cdot3^{-j}<9h\leq\frac{9\delta}{200d}<\frac\delta4\,.
\end{equation*}
Hence the interval~$(\tau-54\cdot3^{-j},\tau+54\cdot3^{-j})$ lies in~$(\rho_1,\rho_2)$. The coarea formula~\eqref{e.cubical.coarea} and the definition of~$D_v$ give
\begin{equation}
	\sum_{\triangle\in\mathscr W_h^j}\Ener_\triangle(v)\leq\frac12\int_\tau^{\tau+54\cdot3^{-j}}\int_{\partial(l\cu_0)}\nabla v\cdot\a\nabla v\,d\mathcal H^{d-1}\,dl\leq54\cdot3^{-j}D_v(\tau)\,.
	\label{e.energy.scale}
\end{equation}

\smallskip

\emph{Step 2.} Cauchy--Schwarz on each~$\mathscr W_h^j$. The replacement~$H_hg$ has finite energy and vanishes on the Whitney simplices outside~$\mathscr W_h$. Both functions have finite energy, so by Cauchy--Schwarz the exterior integral is the absolutely convergent sum of the integrals over the simplices. On each~$\mathscr W_h^j$, Cauchy--Schwarz combines~\eqref{e.energy.scale} with~\eqref{e.harmonic.energy} and~\eqref{e.extension.scale} for~$\alpha=1-t$,~$\xi=r$ and~$b\in\{2,r\}$. After the square roots, the factors~$3^{-j}$ of the two energies give~$3^{-j}$, which cancels the gradient scale~$3^j$, and~$(1-t)-(d-1)(\nf1r-\nf12)=1-\sigma_*$ by~\eqref{e.reconstruction.parameters}, so that
\begin{align}
	\sum_{\triangle\in\mathscr W_h^j}\biggl|\int_\triangle\nabla H_hg\cdot\a\nabla v\biggr|&\leq C A_j(\tau)^{\nf12}D_v(\tau)^{\nf12}3^{-j(1-\sigma_*)}[g]_{W^{1-t,r}(\partial(\tau\cu_0))}\notag \\
	& \qquad +C A_j(\tau)^{\nf12}D_v(\tau)^{\nf12}h^{-1}3^{-j}3^{j(d-1)(\frac1b-\frac12)}\norm{g}_{L^b(\partial(\tau\cu_0))}\,.
	\label{e.pairing.scale}
\end{align}
The factor~$3^{j(d-1)(\frac1b-\frac12)}$ equals one for~$b=2$, and for~$b=r$ it is the loss~$3^{j(\sigma_*-t)}$ of the~$L^r$ norm, by~\eqref{e.reconstruction.parameters}. The weight of~$S$ in~\eqref{e.boundary.maxima.series} is~$3^{-j\sigma}$, and~$3^{j\sigma}\cdot3^{-j(1-\sigma_*)}=3^{-j\theta}$. Since~$1-\sigma>0$ and~$1-(\sigma+\sigma_*-t)=\theta+t>0$, the bound~$3^{-j}<\nf h6$ of~\eqref{e.extension.support}, valid whenever~$\mathscr W_h^j\neq\varnothing$, gives
\begin{equation*}
	3^{-j(1-\sigma_*)}\leq 3^{-j\sigma}h^\theta\,,\qquad h^{-1}3^{-j}\leq 3^{-j\sigma}h^{-\sigma}\,,\qquad h^{-1}3^{-j}3^{j(\sigma_*-t)}\leq 3^{-j\sigma}h^{-(\sigma+\sigma_*-t)}\,.
\end{equation*}
Thus each term of~\eqref{e.pairing.scale}, for~$b=2$ and for~$b=r$, is at most~$3^{-j\sigma}A_j(\tau)^{\nf12}$, the~$j$-th term of~$S(\tau)$, times~$CD_v(\tau)^{\nf12}$ and a factor independent of~$j$. Summing over~$j$ proves~\eqref{e.summed.pairing} and~\eqref{e.summed.pairing.lr}.
\end{proof}

\begin{proposition}[Energy bound in a cube with a good radius]
\label{p.good.radius.energy}
Let~$v$ and~$(v_i)$ satisfy the hypotheses of Proposition~\ref{p.good.radius}. Choose a good radius~$\tau\in J$ and a subsequence of~$(v_i)$ supplied by that proposition, and denote the subsequence again by~$(v_i)$. For~$k\in\R$, set~$w\coloneqq(v-k)_+$ and assume~$w\in\Csub(\cu_0)$. For every triadic width~$h$ with~$0<h\leq\delta/(200d)$,
\begin{equation}
	\Ener_{\tau\cu_0}(w)\leq CS(\tau)D_v(\tau)^{\nf12}\bigl(h^\theta[w]_{W^{1-t,r}(\partial(\tau\cu_0))}+h^{-\sigma}\norm{w}_{L^2(\partial(\tau\cu_0))}\bigr)\,.
	\label{e.good.radius.energy.estimate}
\end{equation}
\end{proposition}

\begin{proof}
Let~$w_i\coloneqq(v_i-k)_+$. These functions are nonnegative and, since the~$v_i$ are smooth, Lipschitz on~$\tau\overline{\cu}_0$. By Proposition~\ref{p.whitney.extension}, the function equal to~$w_i$ on~$\tau\overline{\cu}_0$ and to~$H_h(w_i|_{\partial(\tau\cu_0)})$ outside is a nonnegative element of~$\Hzero{\a}(\cu_0)$. Testing~\eqref{e.weighted.testing} for~$w$ with this function, and using~\eqref{e.summed.pairing} with~$w$ in place of~$v$ and~$w_i|_{\partial(\tau\cu_0)}$ in place of~$g$, together with~$D_w\leq D_v$,
\begin{align*}
	\int_{\tau\cu_0}\nabla w_i\cdot\a\nabla w&\leq-\int_{\cu_0\setminus\tau\overline{\cu}_0}\nabla H_h(w_i|_{\partial(\tau\cu_0)})\cdot\a\nabla w\\
	&\leq C S(\tau)D_v(\tau)^{\nf12}\bigl(h^\theta[w_i]_{W^{1-t,r}(\partial(\tau\cu_0))}+h^{-\sigma}\norm{w_i}_{L^2(\partial(\tau\cu_0))}\bigr)\,.
\end{align*}
As~$i\to\infty$,~$w_i\to w$ in~$H^1_\a(\cu_0)$ by Proposition~\ref{p.weighted.calculus}\ref{i.weighted.chain}. Thus the interior pairing in the preceding display tends to~$\Ener_{\tau\cu_0}(w)$ by~\eqref{e.weighted.pairing}. The boundary seminorm and~$L^2$ norm on the right side tend to those of~$w$, by~\eqref{e.good.radius.convergence} with~$N=\infty$ and Lemma~\ref{l.critical.trace.embedding}. Taking limits in this bound gives~\eqref{e.good.radius.energy.estimate}; no limit of the harmonic extensions is needed.
\end{proof}

\section{Caccioppoli inequality and local boundedness}
\label{s.local.boundedness}

This section proves Theorem~\ref{t.cg.local.boundedness} and Corollary~\ref{c.lr.bound}. It combines the energy bound in a cube with a good radius, Proposition~\ref{p.good.radius.energy}, with the bounds at a good radius in Proposition~\ref{p.good.radius}. With one level this gives a Caccioppoli inequality (Proposition~\ref{p.cg.caccioppoli}). With two truncation levels, the fractional Sobolev inequalities on the cube boundary and inside the cube give a recurrence with a superlinear gain (Proposition~\ref{p.two.level.recurrence}), which drives a De Giorgi iteration (Proposition~\ref{p.energy.to.sup}).

\smallskip

Throughout this section~\eqref{e.weighted.hypotheses} and the coarse ellipticity condition~\eqref{e.spatial.moment.range} hold on~$\cu_0$, the parameters are those of Section~\ref{s.triadic.simplices}, and, unless stated otherwise, constants depend only on~$d,p,q,s,t$. For radii~$\nf12\leq\rho_1<\rho_2\leq1$, the numbers~$\delta$ and~$J$ are those of~\eqref{e.selection.parameters}. By~\eqref{e.moment.comparison},~$\Theta\geq1$.

\subsection{The Caccioppoli inequality}

With~$k=0$, Proposition~\ref{p.good.radius.energy} and the bounds at a good radius~$\tau\in J$ control the energy of~$v$ in~$\tau\cu_0$ by~$\Theta^{\nf12}h^\theta$ times its energy on~$\rho_2\cu_0$, plus a term linear in the square root of that energy, with the factor~$h^{-\sigma}$ and the~$L^2$ norm of~$v$. We choose~$h$ so that the first term can be absorbed. Since~$h^{-\sigma}$ grows as~$h$ decreases, this choice produces the power~$\Theta^{\nf\sigma\theta}$ in the next proposition. The powers of~$(\rho_2-\rho_1)^{-1}$ in the energy bound in a cube with a good radius and in the Caccioppoli inequality are
\begin{equation}
	\gamma_1\coloneqq2-t+\frac1{2p}+\frac1{2q}\qquad\mbox{and}\qquad \gamma_2\coloneqq\frac{2\gamma_1(1-\sigma_*)}{\theta}\,,
	\label{e.localization.exponents}
\end{equation}
respectively; the choice of~$h$ multiplies~$\gamma_1$ by a factor involving~$\nf1\theta$.

\begin{proposition}[Caccioppoli inequality]
\label{p.cg.caccioppoli}
For a nonnegative~$v\in\Csub(\cu_0)$ and~$\nf12\leq\rho_1<\rho_2\leq1$, with~$\gamma_2$ as in~\eqref{e.localization.exponents},
\begin{equation}
	\Ener_{\rho_1\cu_0}(v)\leq C(\rho_2-\rho_1)^{-\gamma_2}\Lambda_{s,1,p}(\cu_0)\Theta^{\nf\sigma\theta}\norm{v}_{L^2(\rho_2\cu_0)}^2\,.
	\label{e.cg.caccioppoli}
\end{equation}
\end{proposition}

\begin{proof}
Let~$\delta=\rho_2-\rho_1$ and assume~$\norm{v}_{L^2(\rho_2\cu_0)}<\infty$. We first bound the energy on~$\rho_1\cu_0$ by a multiple of the energy on~$\rho_2\cu_0$ plus a term involving~$\norm{v}_{L^2(\rho_2\cu_0)}$, then choose~$h$ to make the multiple small, and finally remove the energy on~$\rho_2\cu_0$ by iterating over radii.

\smallskip

\emph{Step 1.} A cube with a good radius. Choose nonnegative smooth~$v_i\to v$ in~$H^1_\a(\cu_0)$ by Proposition~\ref{p.weighted.calculus}\ref{i.weighted.approximation}, and let~$\tau$ and the subsequence be given by Proposition~\ref{p.good.radius}, including the bound~\eqref{e.good.radius.trace.l2}. Proposition~\ref{p.good.radius.energy} with~$k=0$ and the bounds~\eqref{e.good.radius.matrices}--\eqref{e.good.radius.fractional} give, with~$E_{\rho_1}\coloneqq\Ener_{\rho_1\cu_0}(v)$ and~$E_{\rho_2}\coloneqq\Ener_{\rho_2\cu_0}(v)$,
\begin{align*}
	E_{\rho_1}\leq E_\tau&\leq C\delta^{-\nf1{2p}-\nf12-(1-t)-\nf1r}\Lambda_{s,1,p}(\cu_0)^{\nf12}h^\theta E_{\rho_2}^{\nf12}\bigl(\lambda_{t,1,q}(\cu_0)^{-\nf12}E_{\rho_2}^{\nf12}+\norm{v}_{L^r(\rho_2\cu_0)}\bigr)\notag\\
	&\qquad+C\delta^{-\nf1{2p}-1}\Lambda_{s,1,p}(\cu_0)^{\nf12}h^{-\sigma}\norm{v}_{L^2(\rho_2\cu_0)}E_{\rho_2}^{\nf12}\,.
\end{align*}
Since~$|\rho_2\cu_0|\leq1$ and~$r<2$, one has~$\norm{v}_{L^r(\rho_2\cu_0)}\leq\norm{v}_{L^2(\rho_2\cu_0)}$, and~$h^\theta\leq h^{-\sigma}$ because~$h<1$. The powers of~$\delta^{-1}$ add up to~$\nf1{(2p)}+\nf12+(1-t)+\nf1r=\gamma_1$ in the first term and to~$1+\nf1{(2p)}\leq\gamma_1$ in the second. Hence
\begin{equation}
	E_{\rho_1}\leq C\delta^{-\gamma_1}\Bigl(\Theta^{\nf12}h^\theta E_{\rho_2}+\Lambda_{s,1,p}(\cu_0)^{\nf12}h^{-\sigma}\norm{v}_{L^2(\rho_2\cu_0)}E_{\rho_2}^{\nf12}\Bigr)\,.
	\label{e.inner.energy.bound}
\end{equation}
\smallskip

\emph{Step 2.} The choice of~$h$. Let~$\eps_0\coloneqq2^{-\gamma_2-2}$, and choose~$0<c_0\leq1$, depending only on the parameters, such that~$Cc_0^\theta\leq1$ for the constant in~\eqref{e.inner.energy.bound}. Let~$h$ be the largest triadic width with
\begin{equation*}
	h\leq\min\biggl\{\frac{\delta}{200d},\,c_0\biggl(\frac{\eps_0\delta^{\gamma_1}}{\Theta^{\nf12}}\biggr)^{\nf1\theta}\biggr\}<3h\,.
\end{equation*}
Then~$C\delta^{-\gamma_1}\Theta^{\nf12}h^\theta\leq\eps_0$, and Young's inequality applied to the second term of~\eqref{e.inner.energy.bound} gives
\begin{equation*}
	E_{\rho_1}\leq2\eps_0 E_{\rho_2}+C\eps_0^{-1}\delta^{-2\gamma_1}\Lambda_{s,1,p}(\cu_0)h^{-2\sigma}\norm{v}_{L^2(\rho_2\cu_0)}^2\,.
\end{equation*}
The size of~$h^{-2\sigma}$ depends on which entry of the minimum in the choice of~$h$ is smaller. If the absorption condition is the smaller entry, then
\begin{equation*}
	h^{-2\sigma}\leq C\eps_0^{-\frac{2\sigma}{\theta}}\delta^{-\frac{2\gamma_1\sigma}{\theta}}\Theta^{\frac\sigma\theta}\qquad\mbox{and}\qquad2\gamma_1+\frac{2\gamma_1\sigma}{\theta}=\gamma_2\,,
\end{equation*}
where the identity uses~$\theta+\sigma=1-\sigma_*$. If~$\delta/(200d)$ is the smaller entry, then~$h^{-2\sigma}\leq C\delta^{-2\sigma}$ and~$2\gamma_1+2\sigma\leq\gamma_2$, since~$\gamma_1>1>\theta$. In both cases
\begin{equation}
	E_{\rho_1}\leq2\eps_0 E_{\rho_2}+C\eps_0^{-1-\nf{2\sigma}\theta}\delta^{-\gamma_2}\Lambda_{s,1,p}(\cu_0)\Theta^{\nf\sigma\theta}\norm{v}_{L^2(\rho_2\cu_0)}^2\,.
	\label{e.epsilon.caccioppoli}
\end{equation}

\smallskip

\emph{Step 3.} Iteration over radii. By the choice of~$\eps_0$, the factor~$2\eps_0=2^{-\gamma_2-1}$ multiplying~$E_{\rho_2}$ in~\eqref{e.epsilon.caccioppoli} outweighs the factor~$2^{\gamma_2}$ by which~$\delta^{-\gamma_2}$ grows when the gap between the radii is halved. Apply~\eqref{e.epsilon.caccioppoli} between the radii~$R_i\coloneqq\rho_2-2^{-i}\delta$ and~$R_{i+1}$, whose gap is~$2^{-i-1}\delta$. By induction,
\begin{equation*}
	\Ener_{\rho_1\cu_0}(v)\leq(2\eps_0)^N\Ener_{\rho_2\cu_0}(v)+C\delta^{-\gamma_2}\Lambda_{s,1,p}(\cu_0)\Theta^{\nf\sigma\theta}\norm{v}_{L^2(\rho_2\cu_0)}^2\sum_{i=0}^{N-1}(2\eps_0)^i2^{\gamma_2(i+1)}\,.
\end{equation*}
The series is at most~$2^{\gamma_2+1}$, and the first term tends to zero because the energy of~$v$ is finite.
\end{proof}

\subsection{A two-level recurrence}

We now compare two levels~$0\leq a<b$ of a subsolution~$u$ on nested cubes~$\rho_1\cu_0\subset\rho_2\cu_0$. The recurrence must bound at level~$b$ the same quantity it uses at level~$a$. The energy of~$(u-b)_+$ is estimated in a cube with a good radius through the fractional norm of the trace of~$(u-a)_+$, which by~\eqref{e.good.radius.fractional} requires both the energy and the~$L^r$ norm of~$(u-a)_+$. The~$L^r$ norm of~$(u-b)_+$ is estimated through the fractional Sobolev inequality on~$\rho_1\cu_0$, which requires the same two quantities. For a subsolution~$u\in\Csub(\cu_0)$, a level~$a\geq0$ and a radius~$\nf12\leq\rho_2\leq1$, we therefore measure the truncation~$(u-a)_+$ on~$\rho_2\cu_0$ by
\begin{equation}
	Y_a(\rho_2)\coloneqq\Bigl(\norm{(u-a)_+}_{L^r(\rho_2\cu_0)}^2+\lambda_{t,1,q}(\cu_0)^{-1}\Ener_{\rho_2\cu_0}\bigl((u-a)_+\bigr)\Bigr)^{\nf12}\,.
	\label{e.two.level.quantity}
\end{equation}
The weight~$\lambda_{t,1,q}(\cu_0)^{-1}$ normalizes the energy term as in the fractional estimate, since the energy controls the fractional seminorm with the factor~$\lambda_{t,1,q}(\cu_0)^{-1/2}$ by~\eqref{e.lower.fractional}. The quantity~$Y_a(\rho_2)$ is nonincreasing in~$a$ and nondecreasing in~$\rho_2$, and it is finite for~$\rho_2<1$ by Proposition~\ref{p.lower.fractional}. We also use the fractional Sobolev inequality~\cite[Theorem~6.5]{DNPV.fractional}: every measurable, compactly supported~$G$ on~$\Rd$ satisfies
\begin{equation}
	\norm{G}_{L^{r^*}(\Rd)}\leq C(d,q,t)[G]_{W^{1-t,r}(\Rd)}\,,
	\label{e.ordinary.fractional.sobolev}
\end{equation}
since~$(1-t)r<d$, as noted after~\eqref{e.parameter.facts}.

\smallskip

The gain comes from truncation. Where~$u>b$, the function~$(u-a)_+$ exceeds~$b-a$, so higher integrability of~$(u-a)_+$ makes~$(u-b)_+$ small when~$b-a$ is large. There are two sources of higher integrability. On the cube boundary, Lemma~\ref{l.critical.trace.embedding} places the trace of~$(u-a)_+$ in~$L^{r_\partial^*}$ with~$r_\partial^*>2$, while the energy bound in the cube involves the~$L^2$ norm of the trace of~$(u-b)_+$. On~$\rho_1\cu_0$,~\eqref{e.ordinary.fractional.sobolev} places~$(u-a)_+$ in~$L^{r^*}$ with~$r^*>r$, while~$Y_b$ involves the~$L^r$ norm of~$(u-b)_+$. In~\eqref{e.two.level.recurrence}, the boundary contribution contains~$\Theta^{\nf{(1-\sigma_*)}\theta}Y_a(\rho_2)^{r_\partial^*}(b-a)^{2-r_\partial^*}$, while the interior contribution contains~$Y_a(\rho_2)^{\nf{2r^*}r}(b-a)^{2-\nf{2r^*}r}$. Both powers of~$b-a$ are negative, so increasing the level gap reduces these contributions. The term~$\eps Y_a(\rho_2)^2$ collects the seminorm term of Proposition~\ref{p.good.radius.energy}, made small by the choice of~$h$, and the part of the boundary term split off by Young's inequality. The power of~$(\rho_2-\rho_1)^{-1}$ in the recurrence is
\begin{equation}
	\gamma_3\coloneqq\max\Bigl\{1+\frac1p+\Bigl((1-t)+\frac1r\Bigr)r_\partial^*+\frac{2\gamma_1\sigma}{\theta}\,,\ \frac{2(1-t)r^*}{r}\Bigr\}\,,
	\label{e.recurrence.exponent}
\end{equation}
the first entry from the boundary term and the second from the interior term.

\begin{proposition}[A two-level recurrence]
\label{p.two.level.recurrence}
For every~$0<\eps<1$ there is~$C_\eps<\infty$, depending only on~$\eps$ and~$d,p,q,s,t$, with the following property. Let~$u\in\Csub(\cu_0)$,~$0\leq a<b$ and~$\nf12\leq\rho_1<\rho_2\leq1$ with~$Y_a(\rho_2)<\infty$ in~\eqref{e.two.level.quantity}, and let~$\delta=\rho_2-\rho_1$. Then
\begin{equation}
	Y_b(\rho_1)^2\leq\eps Y_a(\rho_2)^2+C_\eps\delta^{-\gamma_3}\Bigl(\Theta^{\nf{(1-\sigma_*)}{\theta}}Y_a(\rho_2)^{r_\partial^*}(b-a)^{2-r_\partial^*}+Y_a(\rho_2)^{\nf{2r^*}r}(b-a)^{2-\nf{2r^*}r}\Bigr)\,.
	\label{e.two.level.recurrence}
\end{equation}
\end{proposition}

\begin{proof}
Write~$w_a=(u-a)_+$,~$w_b=(u-b)_+$,~$\Delta=b-a$ and~$Y=Y_a(\rho_2)$. If~$Y=0$, then~$w_a=0$ on~$\rho_2\cu_0$, hence~$w_b=0$ there, and there is nothing to prove; let~$Y>0$. We treat the energy and the~$L^r$ norm of~$w_b$ separately. Steps~1 and~2 estimate the energy in a cube with a good radius, as in the proof of Proposition~\ref{p.cg.caccioppoli}, with the gain in the boundary~$L^2$-trace term. Step~3 estimates the~$L^r$ norm with higher integrability on~$\rho_1\cu_0$.

\smallskip

\emph{Step 1.} The energy in a cube with a good radius. Choose nonnegative smooth~$v_i\to w_a$ in~$H^1_\a(\cu_0)$ by Proposition~\ref{p.weighted.calculus}\ref{i.weighted.approximation}, and let~$\tau$ and the subsequence be given by Proposition~\ref{p.good.radius} with~$v=w_a$. Since~$w_b=(w_a-\Delta)_+$ is a subsolution by Lemma~\ref{l.weighted.testing}, Proposition~\ref{p.good.radius.energy} applies with~$v=w_a$ and~$k=\Delta$; its proof approximates~$w_b$ by the Lipschitz functions~$(v_i-\Delta)_+$. On~$\partial(\tau\cu_0)$, the trace of~$w_b$ is~$(w_a-\Delta)_+$, formed from the trace of~$w_a$, by Proposition~\ref{p.good.radius}. Truncation at the level~$\Delta$ is one-Lipschitz, so~$[w_b]_{W^{1-t,r}}\leq[w_a]_{W^{1-t,r}}$ on~$\partial(\tau\cu_0)$. On the set where~$w_a>\Delta$,
\begin{equation*}
	w_b^2=(w_a-\Delta)^2\leq w_a^2\leq\Delta^{2-r_\partial^*}w_a^{r_\partial^*}\,.
\end{equation*}
The Sobolev inequality on cube boundaries~\eqref{e.critical.trace.embedding}, with~$\alpha=1-t$ and~$\xi=r$, then gives
\begin{equation*}
	\norm{w_b}_{L^2(\partial(\tau\cu_0))}^2\leq\Delta^{2-r_\partial^*}\norm{w_a}_{L^{r_\partial^*}(\partial(\tau\cu_0))}^{r_\partial^*}\leq C\Delta^{2-r_\partial^*}\norm{w_a}_{W^{1-t,r}(\partial(\tau\cu_0))}^{r_\partial^*}\,.
\end{equation*}
Each of~$\lambda_{t,1,q}(\cu_0)^{-\nf12}\Ener_{\rho_2\cu_0}(w_a)^{\nf12}$ and~$\norm{w_a}_{L^r(\rho_2\cu_0)}$ is at most~$Y$, so the bound~\eqref{e.good.radius.fractional} for~$w_a$ gives~$\norm{w_a}_{W^{1-t,r}(\partial(\tau\cu_0))}\leq C\delta^{-(1-t)-\nf1r}Y$. We insert these bounds in~\eqref{e.good.radius.energy.estimate}, together with the consequences of~\eqref{e.good.radius.matrices} and~\eqref{e.good.radius.boundary.energy}
\begin{equation*}
	S(\tau)\leq C\delta^{-\nf1{2p}}\Lambda_{s,1,p}(\cu_0)^{\nf12}\qquad\mbox{and}\qquad D_{w_a}(\tau)^{\nf12}\leq C\delta^{-\nf12}\lambda_{t,1,q}(\cu_0)^{\nf12}Y\,,
\end{equation*}
and divide by~$\lambda_{t,1,q}(\cu_0)$. Since~$\Lambda_{s,1,p}(\cu_0)^{\nf12}\lambda_{t,1,q}(\cu_0)^{-\nf12}=\Theta^{\nf12}$ and~$\Ener_{\rho_1\cu_0}(w_b)\leq\Ener_{\tau\cu_0}(w_b)$, this gives
\begin{equation}
	\frac{\Ener_{\rho_1\cu_0}(w_b)}{\lambda_{t,1,q}(\cu_0)}\leq C\delta^{-\gamma_1}\Theta^{\nf12}h^\theta Y^2+C\delta^{-\nf12-\nf1{2p}-((1-t)+\frac1r)\frac{r_\partial^*}2}\Theta^{\nf12}h^{-\sigma}Y^{1+\nf{r_\partial^*}2}\Delta^{1-\nf{r_\partial^*}2}\,,
	\label{e.recurrence.energy}
\end{equation}
for every triadic width~$0<h\leq\delta/(200d)$, where~$C$ depends only on the parameters. The first term has the form of the first term of~\eqref{e.inner.energy.bound}; in the second, the power~$1+\nf{r_\partial^*}2$ of~$Y$ exceeds two.

\smallskip

\emph{Step 2.} The choice of~$h$. As in the proof of Proposition~\ref{p.cg.caccioppoli}, we choose~$h$ so that the first term of~\eqref{e.recurrence.energy} is small. Choose~$0<c_0\leq1/(200d)$, depending only on the parameters, with~$Cc_0^\theta\leq\nf12$ for the constant in~\eqref{e.recurrence.energy}, and let~$h$ be the largest triadic width with
\begin{equation*}
	h\leq c_0\,\eps^{\nf1\theta}\Theta^{-\nf1{2\theta}}\delta^{\nf{\gamma_1}\theta}<3h\,.
\end{equation*}
Here the absorption condition alone determines~$h$: since~$\eps<1$,~$\Theta\geq1$,~$\gamma_1>1>\theta$ and~$\delta\leq\nf12$, this width satisfies~$h\leq\delta/(200d)$. The first term of~\eqref{e.recurrence.energy} is then at most~$\frac\eps2\,Y^2$. For the second term, Young's inequality with weight~$\eps$ and the lower bound~$h>c_0\eps^{1/\theta}\Theta^{-1/(2\theta)}\delta^{\gamma_1/\theta}/3$ give
\begin{equation*}
	\Theta h^{-2\sigma}\leq C\eps^{-\nf{2\sigma}\theta}\delta^{-\nf{2\gamma_1\sigma}\theta}\Theta^{1+\nf\sigma\theta}=C\eps^{-\nf{2\sigma}\theta}\delta^{-\nf{2\gamma_1\sigma}\theta}\Theta^{\nf{(1-\sigma_*)}\theta}\,,
\end{equation*}
since~$1+\nf\sigma\theta=\nf{(1-\sigma_*)}\theta$. Hence
\begin{equation*}
	\frac{\Ener_{\rho_1\cu_0}(w_b)}{\lambda_{t,1,q}(\cu_0)}\leq\eps Y^2+C\eps^{-1-\nf{2\sigma}\theta}\delta^{-1-\nf1p-((1-t)+\frac1r)r_\partial^*-\nf{2\gamma_1\sigma}\theta}\Theta^{\nf{(1-\sigma_*)}\theta}Y^{r_\partial^*}\Delta^{2-r_\partial^*}\,,
\end{equation*}
and the power of~$\delta^{-1}$ is at most~$\gamma_3$ by~\eqref{e.recurrence.exponent}. The identity above gives the power~$\Theta^{\nf{(1-\sigma_*)}\theta}$ in the boundary term of~\eqref{e.two.level.recurrence}.

\smallskip

\emph{Step 3.} The~$L^r$ norm. The localization near~$\rho_1\overline{\cu}_0$ from Proposition~\ref{p.fractional.localization}, applied to~$w_a$ between~$\rho_1$ and~$\rho_2$, gives a compactly supported~$F$ equal to~$w_a$ there with~$[F]_{W^{1-t,r}(\Rd)}\leq C\delta^{-(1-t)}Y$, so~\eqref{e.ordinary.fractional.sobolev} gives~$\norm{w_a}_{L^{r^*}(\rho_1\cu_0)}\leq C\delta^{-(1-t)}Y$. On the set where~$u>b$ one has~$w_a>\Delta$ and~$w_b\leq w_a$, hence~$w_b^r\leq\Delta^{r-r^*}w_a^{r^*}$ because~$r^*>r$. Integrating and raising to the power~$\nf2r$,
\begin{equation*}
	\norm{w_b}_{L^r(\rho_1\cu_0)}^2\leq\Delta^{2-\nf{2r^*}r}\norm{w_a}_{L^{r^*}(\rho_1\cu_0)}^{\nf{2r^*}r}\leq C\delta^{-\nf{2(1-t)r^*}r}Y^{\nf{2r^*}r}\Delta^{2-\nf{2r^*}r}\,.
\end{equation*}
Adding the bounds of Steps~2 and~3 and using~$\delta<1$ proves~\eqref{e.two.level.recurrence}.
\end{proof}

\subsection{Iteration over levels and proof of Theorem~\ref{t.cg.local.boundedness}}

To iterate Proposition~\ref{p.two.level.recurrence}, choose level gaps proportional to~$Y_a(\rho_2)$. The boundary term of~\eqref{e.two.level.recurrence} has the factor~$\Theta^{\nf{(1-\sigma_*)}\theta}$. A gap containing the factor~$\Theta^{\nf{(d-3+2\sigma_*)}{4\theta}}$ cancels it, since
\begin{equation*}
	\Theta^{\frac{1-\sigma_*}{\theta}}\bigl(\Theta^{\frac{d-3+2\sigma_*}{4\theta}}\bigr)^{2-r_\partial^*}=1\,.
\end{equation*}
The interior term has no positive power of~$\Theta$. To control the factor~$\delta^{-\gamma_3}$ as the radii approach~$\rho_1$, set
\begin{equation*}
	\gamma_4\coloneqq2\max\biggl\{\frac{\gamma_3}{r_\partial^*-2},\frac{\gamma_3}{2r^*/r-2}\biggr\}\,.
\end{equation*}
The resulting recurrence contracts~$Y$, and summing the level gaps bounds the supremum of~$u$.

\begin{proposition}[From energy to supremum]
\label{p.energy.to.sup}
There is~$C<\infty$, depending only on~$d,p,q,s,t$, such that every~$u\in\Csub(\cu_0)$ and~$\nf12\leq\rho_1<\rho_2\leq1$ with~$Y_0(\rho_2)<\infty$ satisfy
\begin{equation}
	\norm{u_+}_{L^\infty(\rho_1\cu_0)}\leq C(\rho_2-\rho_1)^{-\gamma_4}\Theta^{\nf{(d-3+2\sigma_*)}{4\theta}}Y_0(\rho_2)\,.
	\label{e.energy.to.sup}
\end{equation}
In particular, every~$u\in\Csub(\cu_0)$ is locally bounded above in~$\cu_0$.
\end{proposition}

\begin{proof}
Let~$\mu\coloneqq\nf{(d-3+2\sigma_*)}{(4\theta)}$, which is positive by~\eqref{e.parameter.facts}, since~$d\geq3$. Since~$r_\partial^*-2=\nf{4(1-\sigma_*)}{(d-3+2\sigma_*)}$ by~\eqref{e.sobolev.exponents}, one has
\begin{equation}
	\mu(r_\partial^*-2)=\frac{1-\sigma_*}{\theta}\,.
	\label{e.level.compensation}
\end{equation}
\smallskip

\emph{Step 1.} The choice of constants. By its definition,~$\gamma_4>\max\{\nf{\gamma_3}{(r_\partial^*-2)},\nf{\gamma_3}{(\nf{2r^*}r-2)}\}$. Choose~$\eps_1$ and~$\eps_0$ as below, let~$C_{\eps_0}$ be the constant of Proposition~\ref{p.two.level.recurrence} for~$\eps=\eps_0$, and then choose~$K\geq1$ so that
\begin{equation*}
	\eps_1\coloneqq2^{-\gamma_4-1}\,,\qquad \eps_0\coloneqq\frac{\eps_1^2}4\,,\qquad C_{\eps_0}K^{2-r_\partial^*}\leq\frac{\eps_1^2}4\,,\qquad C_{\eps_0}K^{2-2r^*/r}\leq\frac{\eps_1^2}4\,.
\end{equation*}
All these choices depend only on~$d,p,q,s,t$. The exponent~$\gamma_4$ makes the powers of the gap between radii positive in the two gain terms below. The contraction factor~$\eps_1$ is small enough that the steps, which contain the growing factor~$2^{n\gamma_4}$, still decrease geometrically. The choices of~$\eps_0$ and~$K$ make each of the three terms of~\eqref{e.two.level.recurrence} at most~$\frac{\eps_1^2}4$ times the square of the quantity at the previous level.

\smallskip

\emph{Step 2.} Choice of levels. Let~$R_n\coloneqq\rho_1+2^{-n}(\rho_2-\rho_1)$ and~$\delta_n\coloneqq R_n-R_{n+1}=2^{-n-1}(\rho_2-\rho_1)$, and set~$b_0\coloneqq0$. We choose~$b_{n+1}-b_n$ from the current value~$Y_{b_n}(R_n)$, with the factors~$K\Theta^\mu$ and~$\delta_n^{-\gamma_4}$:
\begin{equation*}
	b_{n+1}\coloneqq b_n+K\Theta^\mu\delta_n^{-\gamma_4}Y_{b_n}(R_n)\qquad(n\geq0)\,.
\end{equation*}
Write~$X_n\coloneqq Y_{b_n}(R_n)$, which is finite because~$Y$ is nonincreasing in the level and nondecreasing in the radius. If~$X_n=0$ for some~$n$, then~$u_+\leq b_n$ on~$R_n\cu_0\supseteq\rho_1\cu_0$, and the bound of Step~3 holds with~$b_n$ in place of~$b_\infty$. Otherwise~$b_{n+1}>b_n$, and~\eqref{e.two.level.recurrence} applies with~$(a,b,\rho_1,\rho_2)=(b_n,b_{n+1},R_{n+1},R_n)$ and~$b_{n+1}-b_n=K\Theta^\mu\delta_n^{-\gamma_4}X_n$. By~\eqref{e.level.compensation}, its boundary term equals
\begin{equation*}
	C_{\eps_0}K^{2-r_\partial^*}\delta_n^{\gamma_4(r_\partial^*-2)-\gamma_3}X_n^2\leq\frac{\eps_1^2}4X_n^2\,,
\end{equation*}
because the power of~$\delta_n<1$ is positive. The interior term satisfies
\begin{equation*}
	C_{\eps_0}K^{2-2r^*/r}\delta_n^{\gamma_4(2r^*/r-2)-\gamma_3}\Theta^{-\mu(2r^*/r-2)}X_n^2\leq\frac{\eps_1^2}4X_n^2\,,
\end{equation*}
since~$\Theta\geq1$ and the power of~$\delta_n$ is positive. Therefore~$X_{n+1}^2\leq\frac{3\eps_1^2}4\,X_n^2$, and~$X_n\leq\eps_1^nX_0$.

\smallskip

\emph{Step 3.} The limit level. Since~$\delta_n^{-\gamma_4}=2^{(n+1)\gamma_4}(\rho_2-\rho_1)^{-\gamma_4}$ and~$2^{\gamma_4}\eps_1=\nf12$, the steps decrease geometrically, and
\begin{equation*}
	\sum_{n\geq0}(b_{n+1}-b_n)\leq K\Theta^\mu\Bigl(\frac{\rho_2-\rho_1}2\Bigr)^{-\gamma_4}X_0\sum_{n\geq0}(2^{\gamma_4}\eps_1)^n=2K\Bigl(\frac{\rho_2-\rho_1}2\Bigr)^{-\gamma_4}\Theta^\mu Y_0(\rho_2)\,.
\end{equation*}
Hence~$b_n$ increases to a finite limit~$b_\infty$, and~$\norm{(u-b_\infty)_+}_{L^r(\rho_1\cu_0)}\leq X_n\to0$. Thus~$u_+\leq b_\infty$ almost everywhere on~$\rho_1\cu_0$, which is~\eqref{e.energy.to.sup}. For~$\rho_2<1$ the quantity~$Y_0(\rho_2)$ is finite, and every compact subset of~$\cu_0$ lies in some~$\rho_1\cu_0$ with~$\rho_1<\rho_2<1$; this proves local boundedness.
\end{proof}

\begin{proof}[Proof of Theorem~\ref{t.cg.local.boundedness}]
Local boundedness is part of Proposition~\ref{p.energy.to.sup}. For the estimate, we bound~$Y_0$ on an intermediate cube by the Caccioppoli inequality and then apply Proposition~\ref{p.energy.to.sup}. Let~$u\in\Csub(\cu_0)$ and~$\nf12\leq\rho_1<\rho_2\leq1$, assume~$\norm{u_+}_{L^2(\rho_2\cu_0)}<\infty$, and let~$\rho\coloneqq(\rho_1+\rho_2)/2$. The positive part~$u_+$ is a nonnegative subsolution by Lemma~\ref{l.weighted.testing}, so the Caccioppoli inequality~\eqref{e.cg.caccioppoli} between~$\rho$ and~$\rho_2$ gives
\begin{equation*}
	\lambda_{t,1,q}(\cu_0)^{-1}\Ener_{\rho\cu_0}(u_+)\leq C(\rho_2-\rho_1)^{-\gamma_2}\Theta^{1+\nf\sigma\theta}\norm{u_+}_{L^2(\rho_2\cu_0)}^2=C(\rho_2-\rho_1)^{-\gamma_2}\Theta^{\nf{(1-\sigma_*)}\theta}\norm{u_+}_{L^2(\rho_2\cu_0)}^2\,.
\end{equation*}
Since~$|\rho\cu_0|\leq1$, also~$\norm{u_+}_{L^r(\rho\cu_0)}\leq\norm{u_+}_{L^2(\rho_2\cu_0)}$, and therefore
\begin{equation*}
	Y_0(\rho)\leq C(\rho_2-\rho_1)^{-\frac{\gamma_2}2}\Theta^{\frac{1-\sigma_*}{2\theta}}\norm{u_+}_{L^2(\rho_2\cu_0)}\,.
\end{equation*}
Proposition~\ref{p.energy.to.sup} between~$\rho_1$ and~$\rho$ then gives~\eqref{e.cg.local.boundedness} with~$\gamma=\gamma_4+\nf{\gamma_2}2$: the exponent~$\nf{(d-3+2\sigma_*)}{4\theta}$ from Proposition~\ref{p.energy.to.sup} and the exponent~$\nf{(1-\sigma_*)}{2\theta}$ from the Caccioppoli bound on~$Y_0(\rho)$ add up to
\begin{equation*}
	\frac{d-3+2\sigma_*}{4\theta}+\frac{1-\sigma_*}{2\theta}=\frac{d-1}{4\theta}\,.
\end{equation*}
For~$\rho_2<1$, the right side of~\eqref{e.cg.local.boundedness} is finite because~$u_+$ is bounded on~$\rho_2\cu_0$.
\end{proof}

\begin{proof}[Proof of Corollary~\ref{c.lr.bound}]
We interpolate the~$L^2$ norm between the~$L^\infty$ and~$L^\eta$ norms and absorb the resulting power of the supremum by an iteration over radii. Fix~$0<\eta<2$ and~$\nf12\leq\rho_1<\rho_2<1$. For~$\rho\leq\rho_2$, set~$M(\rho)\coloneqq\norm{u_+}_{L^\infty(\rho\cu_0)}$, which is finite by Theorem~\ref{t.cg.local.boundedness}, and let~$\rho_1\leq\rho<\rho'\leq\rho_2$. Since~$u_+\leq M(\rho')$ on~$\rho'\cu_0$,
\begin{equation*}
	\norm{u_+}_{L^2(\rho'\cu_0)}^2\leq M(\rho')^{2-\eta}\norm{u_+}_{L^\eta(\rho_2\cu_0)}^\eta\,.
\end{equation*}
Theorem~\ref{t.cg.local.boundedness} between~$\rho$ and~$\rho'$ and Young's inequality with the exponents~$\nf2{(2-\eta)}$ and~$\nf2\eta$ now give
\begin{align*}
	M(\rho)&\leq C(\rho'-\rho)^{-\gamma}\Theta^{\nf{(d-1)}{4\theta}}M(\rho')^{1-\nf\eta2}\norm{u_+}_{L^\eta(\rho_2\cu_0)}^{\nf\eta2}\\
	&\leq\frac12M(\rho')+C(\rho'-\rho)^{-\nf{2\gamma}\eta}\Theta^{\nf{(d-1)}{2\eta\theta}}\norm{u_+}_{L^\eta(\rho_2\cu_0)}\,.
\end{align*}
Let
\begin{equation*}
	R_i\coloneqq\rho_1+\bigl(1-2^{-i\eta/(4\gamma)}\bigr)(\rho_2-\rho_1)\qquad(i\geq0)\,,
\end{equation*}
so that
\begin{equation*}
	R_{i+1}-R_i=\bigl(1-2^{-\eta/(4\gamma)}\bigr)2^{-i\eta/(4\gamma)}(\rho_2-\rho_1)\,.
\end{equation*}
Apply the bound for~$M(\rho)$ with~$(\rho,\rho')=(R_i,R_{i+1})$ for each~$i<n$. Iteration gives
\begin{equation*}
	M(\rho_1)\leq2^{-n}M(R_n)+C\bigl(1-2^{-\nf\eta{4\gamma}}\bigr)^{-\nf{2\gamma}\eta}(\rho_2-\rho_1)^{-\nf{2\gamma}\eta}\Theta^{\nf{(d-1)}{2\eta\theta}}\norm{u_+}_{L^\eta(\rho_2\cu_0)}\sum_{i<n}2^{-\nf i2}\,,
\end{equation*}
and~$2^{-n}M(R_n)\leq2^{-n}M(\rho_2)\to0$. This proves~\eqref{e.lr.bound} for~$\rho_2<1$, and the case~$\rho_2=1$ follows by letting~$\rho_2$ increase to one.
\end{proof}

\section{Weak Harnack and Harnack inequalities}
\label{s.harnack}

This section proves Theorem~\ref{t.harnack}. We first follow Moser's method~\cite{Moser2} to prove~\eqref{e.weak.harnack} at the fixed exponent~$\eta=\nf r4$. Iterating negative powers controls~$w^{-1}$ by its~$L^{p_*}$ norm, while iterating positive powers raises the integrability of~$w$ from~$p_*$ to~$\nf r4$. Bombieri's lemma~\cite{Bombieri.Giusti}, in the form proved in Appendix~\ref{a.bombieri}, uses the bound on~$\log w$ to control the product of the averages of~$w^{p_*}$ and~$w^{-p_*}$. The coarse ellipticity ratio~$\Theta$ determines the exponent~$p_*$ and gives the constant~$\exp(C\sqrt\Theta)$. Combining the resulting weak Harnack estimate with Corollary~\ref{c.lr.bound} proves~\eqref{e.harnack} for solutions. Section~\ref{ss.harnack.endpoint} then proves~\eqref{e.weak.harnack} at~$\eta=\nf{r^*}2$, using the summability across scales. All smaller exponents follow by H\"older's inequality.

\smallskip

We test with powers of~$w=u+\eps$, for~$\eps>0$. Every iteration starts from an exponent of size at most~$p_*$, a small multiple of~$\Theta^{-\nf12}$ fixed in~\eqref{e.small.exponent}. At such exponents the factor involving~$\Theta$ in the reverse inequalities is bounded, and at the~$j$-th step, where the exponent is~$(\nf{r^*}r)^j$ times larger, it grows only like a power of~$(\nf{r^*}r)^j$. The product of the constants is therefore~$e^{C/p_*}$, that is~$\exp(C\sqrt\Theta)$ for a larger~$C$. The positive iteration stops as soon as the exponent reaches~$\nf r4$, so every power~$w^m$ used has~$m<\nf14$. For solutions, Corollary~\ref{c.lr.bound} with~$\eta=\nf r4$ bounds the supremum by the~$L^{\nf r4}$ norm with a power of~$\Theta$, which is absorbed into the exponential. All constants are uniform in~$\eps$, which is sent to zero at the end.

\smallskip

Throughout this section~\eqref{e.weighted.hypotheses} and the coarse ellipticity condition~\eqref{e.spatial.moment.range} hold on~$\cu_0$, the parameters are those of Section~\ref{s.triadic.simplices}, and constants~$C$ depend only on~$d,p,q,s,t$. By~\eqref{e.moment.comparison},~$\Theta\geq1$. We write
\begin{equation*}
	\chi\coloneqq\frac{r^*}r>1\,,\qquad \gamma_5\coloneqq(1-t)+\frac{\gamma_1(1-t)}\theta\,,\qquad \gamma_6\coloneqq\frac{\gamma_5r\chi}{\chi-1}\,.
\end{equation*}
The normalized expressions~$\norm{f}_{\underline L^b(U)}$ of Section~\ref{ss.notation} are nondecreasing in~$b$ by H\"older's inequality. The factor~$\chi$ increases the integrability exponent at each step, and~$\gamma_5$ and~$\gamma_6$ govern the powers of the inverse gap between radii in one reverse inequality and after iteration.

\subsection{Energy of powers}

Moser's method starts from the inequalities satisfied by positive and negative powers of a supersolution.

\smallskip

Let~$u\geq0$ be a supersolution in~$\cu_0$, let~$\eps>0$ and~$w\coloneqq u+\eps$, and let~$m<\nf12$ with~$m\neq0$. By Lemma~\ref{l.weighted.testing}\ref{i.testing.powers} with~$U=\cu_0$, the power~$v\coloneqq w^m$ belongs to~$H^1_\a(\cu_0)\cap L^r(\cu_0)$ and satisfies~\eqref{e.power.testing} for every bounded nonnegative~$\phi\in H^1_\a(\cu_0)$ vanishing outside a compact subset of~$\cu_0$.

\smallskip

Dividing~\eqref{e.power.testing} by~$m$ shows that~$v$ is a supersolution if~$0<m<\nf12$ and a subsolution if~$m<0$. The additional term on the right is nonnegative when~$m>0$ and nonpositive when~$m<0$. If~$\phi$ equals~$v$ inside a cube, the interior part of the left side is~$m$ times the energy of~$v$ there and the right side is at least~$1-m$ times it. The energy therefore appears with the positive factor~$1-2m$, and the exterior integral enters with
\begin{equation*}
	c_m\coloneqq\frac{|m|}{1-2m}\,.
\end{equation*}
The factor~$c_m$ vanishes like~$|m|$ as~$m\to0$ and blows up as~$m\uparrow\nf12$.

We now test~\eqref{e.power.testing} with the function equal to an approximation of~$w^m$, truncated at~$N$ if~$m>0$, in a cube with a good radius and to the piecewise harmonic extension of its trace outside, and argue as in the proof of Proposition~\ref{p.cg.caccioppoli}. There are two differences. The right side of the resulting inequality contains the~$L^r$ norm of~$w^m$ rather than its~$L^2$ norm, because Lemma~\ref{l.positive.negative.moments} raises the integrability of~$w^m$ from~$L^r$ to~$L^{r^*}$, and~$r^*$ may be smaller than two. Accordingly we use the~$L^r$ form~\eqref{e.summed.pairing.lr} of the exterior integral, whose term involving the~$L^r$ norm of the boundary data carries~$h^{-(\sigma+\sigma_*-t)}$ instead of~$h^{-\sigma}$. And the exterior integral enters with the factor~$c_m$, so the absorption condition on~$h$ involves~$c_m^2\Theta$ instead of~$\Theta$.

\begin{proposition}[Caccioppoli inequality for powers]
\label{p.power.caccioppoli}
Let~$u\geq0$ be a supersolution in~$\cu_0$, let~$\eps>0$ and~$w\coloneqq u+\eps$. For~$m<\nf12$ with~$m\neq0$ and~$\nf12\leq\rho_1<\rho_2\leq1$,
\begin{equation}
	\Ener_{\rho_1\cu_0}(w^m)\leq C(\rho_2-\rho_1)^{-\nf{2\gamma_1(1-t)}\theta}\Lambda_{s,1,p}(\cu_0)\,c_m^2\bigl(1+c_m^2\Theta\bigr)^{\nf{(\sigma+\sigma_*-t)}\theta}\norm{w^m}_{L^r(\rho_2\cu_0)}^2\,.
	\label{e.power.caccioppoli}
\end{equation}
\end{proposition}

\begin{proof}
Set
\begin{equation*}
	v\coloneqq w^m\,,\qquad\delta\coloneqq\rho_2-\rho_1\,,\qquad E_\rho\coloneqq\Ener_{\rho\cu_0}(v)\quad(\rho\in\{\rho_1,\rho_2\})\,.
\end{equation*}
These energies are finite by Lemma~\ref{l.weighted.testing}\ref{i.testing.powers}. We test~\eqref{e.power.testing} in a cube with a good radius, choose~$h$, and iterate over radii.

\smallskip

\emph{Step 1.} Testing in a cube with a good radius. By Proposition~\ref{p.weighted.calculus}\ref{i.weighted.approximation} and a subsequence, choose nonnegative smooth~$v_i\to v$ in~$H^1_\a(\cu_0)$ and almost everywhere, with~$v_i\leq\eps^m$ if~$m<0$, since then~$v\leq\eps^m$. Let~$\tau\in J$ be a good radius and the subsequence given by Proposition~\ref{p.good.radius}. The approximations, the radius and the subsequence do not depend on~$N$. Set~$N=\infty$ if~$m<0$, and choose~$0<N<\infty$ if~$0<m<\nf12$. Put
\begin{equation*}
	\tilde v\coloneqq v\wedge N\,,\qquad\tilde v_i\coloneqq v_i\wedge N\,,
\end{equation*}
where~$v\wedge\infty=v$. Then~$0\leq\tilde v_i\leq N$ for finite~$N$, and~$0\leq\tilde v_i\leq\eps^m$ if~$m<0$. The same bounds hold for~$\tilde v$. The~$\tilde v_i$ need not be smooth, but they are Lipschitz on~$\tau\overline{\cu}_0$, as the~$v_i$ are smooth, and~$\tilde v_i\to\tilde v$ in~$H^1_\a(\cu_0)$ by Proposition~\ref{p.weighted.calculus}\ref{i.weighted.chain}.

Fix a triadic width~$0<h\leq\delta/(200d)$, which satisfies~\eqref{e.extension.width} since~$\rho_2-\tau>\delta/2$. By Proposition~\ref{p.whitney.extension}, define
\begin{equation*}
	\phi_i\coloneqq\begin{cases}
		\tilde v_i&\text{on }\tau\overline{\cu}_0\,,\\
		H_h(\tilde v_i|_{\partial(\tau\cu_0)})&\text{on }\cu_0\setminus\tau\overline{\cu}_0\,.
	\end{cases}
\end{equation*}
Proposition~\ref{p.whitney.extension} shows that~$\phi_i$ is a nonnegative element of~$\Hzero{\a}(\cu_0)$, bounded by~$N$, or by~$\eps^m$ if~$m<0$, and vanishing outside a compact subset of~$\cu_0$. We apply~\eqref{e.power.testing} with~$\phi=\phi_i$. The exterior part of its right side is nonnegative,
\begin{equation*}
	(1-m)\int_{\cu_0\setminus\tau\overline{\cu}_0}\frac{\phi_i}{v}\nabla v\cdot\a\nabla v\geq0\,.
\end{equation*}
Discarding this term gives
\begin{equation*}
	(1-m)\int_{\tau\cu_0}\frac{\tilde v_i}v\,\nabla v\cdot\a\nabla v-m\int_{\tau\cu_0}\nabla\tilde v_i\cdot\a\nabla v\leq|m|\,\biggl|\int_{\cu_0\setminus\tau\overline{\cu}_0}\nabla H_h(\tilde v_i|_{\partial(\tau\cu_0)})\cdot\a\nabla v\biggr|\,.
\end{equation*}
As~$i\to\infty$,~\eqref{e.weighted.pairing} and~\eqref{e.weighted.truncation.gradients} give
\begin{equation*}
	\int_{\tau\cu_0}\nabla\tilde v_i\cdot\a\nabla v\longrightarrow\int_{\tau\cu_0}\nabla\tilde v\cdot\a\nabla v=\Ener_{\tau\cu_0}(\tilde v)\,.
\end{equation*}
For the other integral, dominated convergence and~\eqref{e.power.source} give
\begin{equation*}
	\int_{\tau\cu_0}\frac{\tilde v_i}{v}\nabla v\cdot\a\nabla v\longrightarrow\int_{\tau\cu_0}\frac{\tilde v}{v}\nabla v\cdot\a\nabla v\geq\Ener_{\tau\cu_0}(\tilde v)\,.
\end{equation*}
The integrand is bounded by~$Nv^{-1}\nabla v\cdot\a\nabla v$, or by~$\eps^mv^{-1}\nabla v\cdot\a\nabla v$ if~$m<0$, and~$\tilde v/v=1$ where~$v<N$. As~$1-m>0$, it follows that~$(1-2m)\Ener_{\tau\cu_0}(\tilde v)$ is at most the limit superior of the right side. We use~\eqref{e.summed.pairing.lr} for~$v$ and the restrictions of the~$\tilde v_i$, the convergence~$\tilde v_i\to\tilde v$ on~$\partial(\tau\cu_0)$ in~\eqref{e.good.radius.convergence} with~$k=0$, and the truncation bounds on~$\partial(\tau\cu_0)$,
\begin{equation*}
	[\tilde v]_{W^{1-t,r}}\leq[v]_{W^{1-t,r}}\,,\qquad\norm{\tilde v}_{L^r}\leq\norm{v}_{L^r}\,.
\end{equation*}
These give
\begin{equation*}
	\Ener_{\rho_1\cu_0}(\tilde v)\leq\Ener_{\tau\cu_0}(\tilde v)\leq Cc_mS(\tau)D_v(\tau)^{\nf12}\bigl(h^\theta[v]_{W^{1-t,r}(\partial(\tau\cu_0))}+h^{-(\sigma+\sigma_*-t)}\norm{v}_{L^r(\partial(\tau\cu_0))}\bigr)\,.
\end{equation*}
We insert~\eqref{e.good.radius.matrices}--\eqref{e.good.radius.trace.lr} and use
\begin{align*}
	\Lambda_{s,1,p}(\cu_0)^{\nf12}\lambda_{t,1,q}(\cu_0)^{-\nf12}&=\Theta^{\nf12}\,,\qquad h^\theta\leq h^{-(\sigma+\sigma_*-t)}\,,\\
	\frac1{2p}+\frac12+(1-t)+\frac1r&=\gamma_1\,.
\end{align*}
These relations bound all powers of~$\delta^{-1}$ by~$\delta^{-\gamma_1}$. This gives
\begin{equation}
	\Ener_{\rho_1\cu_0}(\tilde v)\leq Cc_m\delta^{-\gamma_1}\Bigl(\Theta^{\nf12}h^\theta E_{\rho_2}+\Lambda_{s,1,p}(\cu_0)^{\nf12}h^{-(\sigma+\sigma_*-t)}\norm{v}_{L^r(\rho_2\cu_0)}E_{\rho_2}^{\nf12}\Bigr)\,.
	\label{e.power.good.radius}
\end{equation}
The right side does not depend on~$N$. For~$0<m<\nf12$, the energy densities~$\indc_{\{v<N\}}\nabla v\cdot\a\nabla v$ of~$v\wedge N$ increase to that of~$v$ as~$N\to\infty$, so~\eqref{e.power.good.radius} holds with~$E_{\rho_1}$ on the left, for every triadic width~$0<h\leq\delta/(200d)$ and every~$m$.

\smallskip

\emph{Step 2.} The choice of~$h$. As in the proof of Proposition~\ref{p.cg.caccioppoli}, the width must satisfy~$h\leq\delta/(200d)$ and make the first term of~\eqref{e.power.good.radius} a small multiple of~$E_{\rho_2}$; that term now carries the factor~$c_m\Theta^{\nf12}$. Let~$\eps_0\coloneqq2^{-\nf{2\gamma_1(1-t)}\theta-2}$, choose~$0<c_0\leq1$, depending only on the parameters, with~$Cc_0^\theta\leq1$ for the constant in~\eqref{e.power.good.radius}, and let~$h$ be the largest triadic width with
\begin{equation*}
	h\leq\min\biggl\{\frac{\delta}{200d},\,c_0\biggl(\frac{\eps_0\delta^{\gamma_1}}{(1+c_m^2\Theta)^{\nf12}}\biggr)^{\nf1\theta}\biggr\}<3h\,.
\end{equation*}
The bound
\begin{equation*}
	c_m\Theta^{\nf12}\leq(1+c_m^2\Theta)^{\nf12}
\end{equation*}
makes the first term of~\eqref{e.power.good.radius} at most~$\eps_0 E_{\rho_2}$. Young's inequality bounds the second term by
\begin{equation*}
	\eps_0 E_{\rho_2}+C\eps_0^{-1}c_m^2\delta^{-2\gamma_1}\Lambda_{s,1,p}(\cu_0)h^{-2(\sigma+\sigma_*-t)}\norm{v}_{L^r(\rho_2\cu_0)}^2\,.
\end{equation*} If the second entry of the minimum is smaller, the absorption condition determines~$h$ and gives
\begin{equation*}
	h^{-2(\sigma+\sigma_*-t)}\leq C\eps_0^{-\nf{2(\sigma+\sigma_*-t)}\theta}\delta^{-\nf{2\gamma_1(\sigma+\sigma_*-t)}\theta}(1+c_m^2\Theta)^{\nf{(\sigma+\sigma_*-t)}\theta}\,.
\end{equation*}
The exponent of~$\delta^{-1}$ combines with the preceding~$\delta^{-2\gamma_1}$ because
\begin{equation*}
	2\gamma_1+\frac{2\gamma_1(\sigma+\sigma_*-t)}\theta=\frac{2\gamma_1(1-t)}\theta\,,
\end{equation*}
using~$1-t=(\sigma+\sigma_*-t)+\theta$. Otherwise~$h$ is determined by~$\delta$ alone and
\begin{equation*}
	h^{-2(\sigma+\sigma_*-t)}\leq C\delta^{-2(\sigma+\sigma_*-t)}\,.
\end{equation*}
Since~$\gamma_1>1>\theta$,
\begin{equation*}
	2\gamma_1+2(\sigma+\sigma_*-t)\leq\frac{2\gamma_1(1-t)}\theta\,.
\end{equation*}
In both cases
\begin{equation}
	E_{\rho_1}\leq2\eps_0 E_{\rho_2}+C\delta^{-\nf{2\gamma_1(1-t)}\theta}\Lambda_{s,1,p}(\cu_0)\,c_m^2\bigl(1+c_m^2\Theta\bigr)^{\nf{(\sigma+\sigma_*-t)}\theta}\norm{v}_{L^r(\rho_2\cu_0)}^2\,.
	\label{e.power.eps}
\end{equation}

\smallskip

\emph{Step 3.} Iteration over radii. The choice of~$\eps_0$ makes~$2\eps_0$ outweigh the growth of~$\delta^{-\nf{2\gamma_1(1-t)}\theta}$ when the gap is halved. Set
\begin{equation*}
	R_i\coloneqq\rho_2-2^{-i}\delta\,,\qquad R_{i+1}-R_i=2^{-i-1}\delta\,.
\end{equation*}
Apply~\eqref{e.power.eps} between~$R_i$ and~$R_{i+1}$ and bound the~$L^r$ norms on~$R_{i+1}\cu_0$ by the one on~$\rho_2\cu_0$. By induction,
\begin{equation*}
	E_{\rho_1}\leq(2\eps_0)^nE_{\rho_2}+C\delta^{-\nf{2\gamma_1(1-t)}\theta}\Lambda_{s,1,p}(\cu_0)\,c_m^2\bigl(1+c_m^2\Theta\bigr)^{\nf{(\sigma+\sigma_*-t)}\theta}\norm{v}_{L^r(\rho_2\cu_0)}^2\sum_{i=0}^{n-1}(2\eps_0)^i2^{\frac{2\gamma_1(1-t)}\theta(i+1)}\,.
\end{equation*}
The ratio of the series is
\begin{equation*}
	2\eps_0\,2^{\nf{2\gamma_1(1-t)}\theta}=\tfrac12\,.
\end{equation*}
The first term tends to zero because~$E_{\rho_2}<\infty$.
\end{proof}

\subsection{Logarithmic estimate}

As~$m\to0$,
\begin{align*}
	m^{-1}\nabla w^m&=w^{m-1}\nabla u\longrightarrow w^{-1}\nabla u=\nabla\log w\,,\\
	\frac{c_m^2}{m^2}&=(1-2m)^{-2}\longrightarrow1\,,\qquad(1+c_m^2\Theta)^{\nf{(\sigma+\sigma_*-t)}\theta}\longrightarrow1\,.
\end{align*}
Dividing~\eqref{e.power.caccioppoli} by~$m^2$ gives, in the limit, an energy bound for~$\log w$ with no power of~$\Theta$. The matrices~$\a_*^{-1}(\triangle)$ convert it into a bound on the deviation of~$\log w$ from its mean, through the fractional regularity of Section~\ref{s.fractional.reconstruction}, and the upper and lower coefficient bounds combine into~$\sqrt\Theta$.

\begin{lemma}[Logarithmic estimate]
\label{l.log.estimate}
Let~$u\geq0$ be a supersolution in~$\cu_0$, let~$\eps>0$ and~$w\coloneqq u+\eps$. Then
\begin{equation}
	\Ener_{\frac{15}{16}\cu_0}(\log w)\leq C\Lambda_{s,1,p}(\cu_0)\,,
	\label{e.log.energy}
\end{equation}
and there is~$C_{\eqref{e.log.mean.zero}}$, depending only on~$d,p,q,s,t$, such that
\begin{equation}
	\norm[\big]{\log w-(\log w)_{\frac78\cu_0}}_{L^r(\frac78\cu_0)}\leq C_{\eqref{e.log.mean.zero}}\sqrt\Theta\,.
	\label{e.log.mean.zero}
\end{equation}
\end{lemma}

\begin{proof}
\emph{Step 1.} Energy. By Lemma~\ref{l.weighted.testing}\ref{i.testing.powers},~$\log w\in H^1_\a(\cu_0)$ and~$\nabla\log w=w^{-1}\nabla u$. For~$-1\leq m<0$,
\begin{equation*}
	m^{-2}\Ener_{\frac{15}{16}\cu_0}(w^m)=\int_{\frac{15}{16}\cu_0}w^{2m-2}\nabla u\cdot\a\nabla u\,,
\end{equation*}
and the integrand is at most~$\max\{1,\eps^{-4}\}\nabla u\cdot\a\nabla u$ and converges to~$w^{-2}\nabla u\cdot\a\nabla u$ as~$m\uparrow0$. Also~$w^{mr}\leq\max\{1,\eps^{-r}\}$ and~$w^{mr}\to1$, so~$\norm{w^m}_{L^r(\cu_0)}^2\to1$. Divide~\eqref{e.power.caccioppoli}, with~$\rho_1=\nf{15}{16}$ and~$\rho_2=1$, by~$m^2$ and let~$m\uparrow0$. The limits above and dominated convergence give~\eqref{e.log.energy}. The dominators depend on~$\eps$, but only the limits enter the estimate.

\smallskip

\emph{Step 2.} A finite cover. We bound the oscillation of~$\log w$ on overlapping cubes and connect their means. Let~$G\coloneqq\log w$, and consider the cubes~$Q_y\coloneqq y+3\cu_{-4}$,~$y\in3^{-4}\Zd$, whose central cubes~$y+\overline{\cu_{-4}}$ meet~$\frac78\overline{\cu}_0$. They cover~$\frac78\overline{\cu}_0$, and their points satisfy
\begin{equation*}
	|x|_\infty\leq\tfrac7{16}+2\cdot3^{-4}<\tfrac{15}{32}\,.
\end{equation*}
Thus~$\overline{Q_y}\subset\frac{15}{16}\cu_0$. By~\eqref{e.fractional.mean.norm} applied to~$G-(G)_{Q_y}$,~\eqref{e.lower.fractional} and~\eqref{e.log.energy},
\begin{equation*}
	\norm{G-(G)_{Q_y}}_{L^r(Q_y)}\leq C[G]_{W^{1-t,r}(Q_y)}\leq C\lambda_{t,1,q}(\cu_0)^{-\nf12}\Ener_{\frac{15}{16}\cu_0}(G)^{\nf12}\leq C\sqrt\Theta\,.
\end{equation*}
The centers form a box in~$3^{-4}\Zd$, and cubes whose centers differ by one grid step overlap in a set of volume~$2\cdot3^{-4}\cdot3^{-3(d-1)}$. For such neighbors, averaging the difference of the two means over the overlap and applying H\"older's inequality give
\begin{equation*}
	|(G)_{Q_y}-(G)_{Q_{y'}}|\leq C\sqrt\Theta\,.
\end{equation*}
The number of cubes and the lengths of grid paths between them depend only on~$d$, so, for one fixed~$y_0$, summing over the cover gives
\begin{equation*}
	\norm{G-(G)_{Q_{y_0}}}_{L^r(\frac78\cu_0)}\leq C\sqrt\Theta\,.
\end{equation*} Replacing~$(G)_{Q_{y_0}}$ by~$(G)_{\frac78\cu_0}$ at most doubles the left side, by H\"older's inequality, which proves~\eqref{e.log.mean.zero}.
\end{proof}

\subsection{Power iteration}

Proposition~\ref{p.power.caccioppoli} and the fractional Sobolev inequality~\eqref{e.ordinary.fractional.sobolev} bound the~$L^{r^*}(\rho_1\cu_0)$ norm of~$w^m$ by its~$L^r(\rho_2\cu_0)$ norm. This raises the integrability exponent of~$w$ when~$m>0$, or of~$w^{-1}$ when~$m<0$, from~$|m|r$ to~$\chi|m|r$.

\begin{lemma}[Reverse H\"older inequalities for powers]
\label{l.positive.negative.moments}
Let~$u\geq0$ be a supersolution in~$\cu_0$, let~$\eps>0$ and~$w\coloneqq u+\eps$. There is~$C_{\eqref{e.positive.negative.moments}}\geq1$, depending only on~$d,p,q,s,t$, such that for~$m<\nf12$ with~$m\neq0$ and~$\nf12\leq\rho_1<\rho_2\leq1$,
\begin{equation}
	\norm{w^{\operatorname{sgn}m}}_{\underline L^{\chi|m|r}(\rho_1\cu_0)}^{|m|r}\leq C_{\eqref{e.positive.negative.moments}}(\rho_2-\rho_1)^{-\gamma_5r}\bigl(1+c_m^2\Theta\bigr)^{\nf{(1-t)r}{2\theta}}\norm{w^{\operatorname{sgn}m}}_{\underline L^{|m|r}(\rho_2\cu_0)}^{|m|r}\,.
	\label{e.positive.negative.moments}
\end{equation}
\end{lemma}

\begin{proof}
Let~$\delta\coloneqq\rho_2-\rho_1$ and~$\rho\coloneqq(\rho_1+\rho_2)/2$. The localization near~$\rho_1\overline{\cu}_0$ from Proposition~\ref{p.fractional.localization}, applied to~$w^m$ between~$\rho_1$ and~$\rho$, gives a compactly supported~$F$ equal to~$w^m$ there. By~\eqref{e.ordinary.fractional.sobolev} and Proposition~\ref{p.power.caccioppoli} between~$\rho$ and~$\rho_2$,
\begin{align*}
	\norm{w^m}_{L^{r^*}(\rho_1\cu_0)}&\leq C\delta^{-(1-t)}\bigl(\lambda_{t,1,q}(\cu_0)^{-\nf12}\Ener_{\rho\cu_0}(w^m)^{\nf12}+\norm{w^m}_{L^r(\rho\cu_0)}\bigr)\\
	&\leq C\delta^{-\gamma_5}\Bigl(1+c_m\Theta^{\nf12}\bigl(1+c_m^2\Theta\bigr)^{\nf{(\sigma+\sigma_*-t)}{2\theta}}\Bigr)\norm{w^m}_{L^r(\rho_2\cu_0)}\\
	&\leq C\delta^{-\gamma_5}\bigl(1+c_m^2\Theta\bigr)^{\nf{(1-t)}{2\theta}}\norm{w^m}_{L^r(\rho_2\cu_0)}\,,
\end{align*}
where the last step uses~$1-t=(\sigma+\sigma_*-t)+\theta$. Raising to the power~$r$ and using
\begin{equation*}
	w^{mr^*}=(w^{\operatorname{sgn}m})^{\chi|m|r}\,,\qquad w^{mr}=(w^{\operatorname{sgn}m})^{|m|r}\,,
\end{equation*}
gives~\eqref{e.positive.negative.moments} for the unnormalized integrals. Passing to the normalized ones introduces the factor
\begin{equation*}
	|\rho_1\cu_0|^{-1/\chi}|\rho_2\cu_0|\leq2^d\,.
\end{equation*}
\end{proof}

The factor~$1+c_m^2\Theta$ in~\eqref{e.positive.negative.moments} is bounded as long as~$|m|$ is at most of order~$\Theta^{-\nf12}$ and~$m$ stays away from~$\nf12$. We therefore start all iterations from a small exponent. With~$C_{\eqref{e.log.mean.zero}}$ as in~\eqref{e.log.mean.zero}, let
\begin{equation}
	\begin{aligned}
		c&\coloneqq\min\biggl\{\frac r4,\frac1{1+C_{\eqref{e.log.mean.zero}}}\biggr\}\,,\\
		p_*&\coloneqq\frac c{\sqrt\Theta}\,.
	\end{aligned}
	\label{e.small.exponent}
\end{equation}
Since~$\Theta\geq1$ and~$r<2$,
\begin{equation*}
	p_*\leq c\leq\tfrac r4<\tfrac12\,.
\end{equation*}
The bound~$c\leq\nf1{(1+C_{\eqref{e.log.mean.zero}})}$ makes~$p_*$ times the deviation~$C_{\eqref{e.log.mean.zero}}\sqrt\Theta$ of~$\log w$ in~\eqref{e.log.mean.zero} at most one. This is where the order~$\Theta^{-\nf12}$ comes from.

\smallskip

We iterate~\eqref{e.positive.negative.moments} in three ways, halving the gap between the radii at each step. Between small exponents~$0<a<b\leq1$ of~$w^{\pm p_*}$, the iteration gives the hypothesis of Lemma~\ref{l.bombieri}. Along the negative exponents~$-p_*\chi^j$, it reaches the infimum of~$w$. Along the positive exponents~$p_*\chi^j$, it stops when the exponent reaches~$\nf r4$. In each case the factor involving~$\Theta$ at the~$j$-th step is at most~$(2\chi^{2j})^{\nf{(1-t)r}{(2\theta)}}$. Its logarithm is summable against the weights~$\chi^{-j}$ with which the steps enter.

\begin{lemma}[Moser iteration]
\label{l.moment.iterations}
Let~$u\geq0$ be a supersolution in~$\cu_0$, let~$\eps>0$,~$w\coloneqq u+\eps$ and~$\nf12\leq\rho_1<\rho_2\leq1$. There is~$C_{\eqref{e.small.moments}}\geq1$, depending only on~$d,p,q,s,t$, such that for~$Z\in\{w^{p_*},w^{-p_*}\}$ and~$0<a<b\leq1$,
\begin{equation}
	\norm{Z}_{\underline L^b(\rho_1\cu_0)}\leq\bigl(C_{\eqref{e.small.moments}}(\rho_2-\rho_1)^{-\gamma_6}\bigr)^{\nf1a}\norm{Z}_{\underline L^a(\rho_2\cu_0)}\,.
	\label{e.small.moments}
\end{equation}
Moreover,
\begin{equation}
	\biggl(\fint_{\rho_2\cu_0}w^{-p_*}\biggr)^{-\nf1{p_*}}\leq e^{C/p_*}(\rho_2-\rho_1)^{-\nf{\gamma_6}{p_*}}\essinf_{\rho_1\cu_0}w\,,
	\label{e.negative.moments}
\end{equation}
and
\begin{equation}
	\norm{w}_{\underline L^{\nf r4}(\rho_1\cu_0)}\leq e^{C/p_*}(\rho_2-\rho_1)^{-\nf{\gamma_6}{p_*}}\norm{w}_{\underline L^{p_*}(\rho_2\cu_0)}\,.
	\label{e.positive.moments}
\end{equation}
\end{lemma}

\begin{proof}
Write~$\delta\coloneqq\rho_2-\rho_1$ and~$R_j\coloneqq\rho_1+2^{-j}\delta$ for~$j\geq0$, so that~$R_0=\rho_2$ and~$R_j-R_{j+1}=2^{-j-1}\delta$. All three bounds follow by applying~\eqref{e.positive.negative.moments} between~$R_{j+1}$ and~$R_j$ along powers~$m_j$ with~$|m_{j+1}|=\chi|m_j|$, and multiplying. We use
\begin{equation}
	\sum_{j\geq0}\chi^{-j}=\frac\chi{\chi-1}\,,\qquad \sum_{j\geq0}(j+1)\chi^{-j}=\frac{\chi^2}{(\chi-1)^2}\,,
	\label{e.iteration.sums}
\end{equation}
so that~$\gamma_5r\sum_j\chi^{-j}=\gamma_6$, and the final comparison~$|R_n\cu_0|\leq2^d|\rho_1\cu_0|$ for the normalized expressions.

\smallskip

\emph{Step 1.} Small powers. Let~$Z=w^{\pm p_*}$, and let~$N$ be the least integer~$n\geq1$ with~$a\chi^n\geq b$. Set
\begin{equation*}
	a_j\coloneqq a\chi^j\quad(0\leq j\leq N)\,,\qquad m_j\coloneqq\pm\nf{p_*a_j}r\quad(0\leq j<N)\,.
\end{equation*}
Then~$a_j<b\leq1$ for~$j<N$. On the positive branch~$1-2m_j\geq1-\nf{2p_*}r\geq\nf12$, and on the negative branch~$1-2m_j\geq1$. Thus, in both cases,
\begin{equation*}
	c_{m_j}^2\Theta\leq(2c/r)^2\leq1\,.
\end{equation*}
For~$b'>0$,
\begin{equation*}
	\norm{Z}_{\underline L^{b'}(U)}^{a_j}=\norm{w^{\pm1}}_{\underline L^{p_*b'}(U)}^{p_*a_j}\,.
\end{equation*}
The estimate~\eqref{e.positive.negative.moments} with~$m=m_j$ therefore reads
\begin{equation*}
	\norm{Z}_{\underline L^{a_{j+1}}(R_{j+1}\cu_0)}\leq\Bigl(C_{\eqref{e.positive.negative.moments}}\delta^{-\gamma_5r}2^{\gamma_5r(j+1)}2^{\nf{(1-t)r}{2\theta}}\Bigr)^{\nf1{a_j}}\norm{Z}_{\underline L^{a_j}(R_j\cu_0)}\,.
\end{equation*}
We multiply these bounds for~$j<N$ and pass from~$\underline L^{a_N}(R_N\cu_0)$ to~$\underline L^b(\rho_1\cu_0)$ by H\"older's inequality, with the factor~$2^{d/a_N}\leq2^{d/a}$. By~\eqref{e.iteration.sums}, the logarithm of the total factor is at most
\begin{multline*}
	\frac1a\Bigl(d\log2+\Bigl(\log C_{\eqref{e.positive.negative.moments}}+\frac{(1-t)r}{2\theta}\log2+\gamma_5r\log\frac1\delta\Bigr)\sum_{j\geq0}\chi^{-j}+\gamma_5r\log2\sum_{j\geq0}(j+1)\chi^{-j}\Bigr)\\
	\leq\frac{C+\gamma_6\log(1/\delta)}a\,,
\end{multline*}
which proves~\eqref{e.small.moments}.

\smallskip

\emph{Step 2.} Negative powers. Let~$m_j\coloneqq-\nf{p_*\chi^j}r$ for all~$j\geq0$. Then
\begin{equation*}
	c_{m_j}^2\Theta\leq(c/r)^2\chi^{2j}\leq\chi^{2j}\,.
\end{equation*}
The power of~$1+c_{m_j}^2\Theta$ in~\eqref{e.positive.negative.moments} is therefore at most~$(2\chi^{2j})^{\nf{(1-t)r}{(2\theta)}}$. Applying~\eqref{e.positive.negative.moments} with~$m=m_j$, taking the power~$1/(p_*\chi^j)$ and multiplying for~$j<n$,
\begin{equation*}
	\norm{w^{-1}}_{\underline L^{p_*\chi^n}(R_n\cu_0)}\leq\exp\biggl(\frac{C+\gamma_6\log(1/\delta)}{p_*}\biggr)\norm{w^{-1}}_{\underline L^{p_*}(\rho_2\cu_0)}\,,
\end{equation*}
uniformly in~$n$. The factor involving~$1+c_{m_j}^2\Theta$ contributes a logarithmic sum bounded by~\eqref{e.iteration.sums}. If~$L<\esssup_{\rho_1\cu_0}w^{-1}$, the set~$E\coloneqq\{x\in\rho_1\cu_0:w^{-1}(x)>L\}$ has positive measure and lies in every~$R_n\cu_0$. Consequently,
\begin{equation*}
	\norm{w^{-1}}_{\underline L^{p_*\chi^n}(R_n\cu_0)}\geq L\biggl(\frac{|E|}{|R_n\cu_0|}\biggr)^{1/(p_*\chi^n)}\longrightarrow L\,.
\end{equation*} Letting~$n\to\infty$ in this lower bound and the preceding upper bound, then letting~$L\uparrow\esssup_{\rho_1\cu_0}w^{-1}$, proves~\eqref{e.negative.moments}.

\smallskip

\emph{Step 3.} Positive powers. Let~$N$ be the least integer~$n\geq0$ with~$p_*\chi^n\geq\nf r4$, and set~$m_j\coloneqq\nf{p_*\chi^j}r$ for~$j<N$. Then~$m_j<\nf14$, so~$c_{m_j}\leq2m_j$ and
\begin{equation*}
	c_{m_j}^2\Theta\leq(2c/r)^2\chi^{2j}\leq\chi^{2j}\,.
\end{equation*}
Applying~\eqref{e.positive.negative.moments} for~$j<N$ and summing the logarithms with~\eqref{e.iteration.sums}, as in Step~2, gives
\begin{equation*}
	\norm{w}_{\underline L^{p_*\chi^N}(R_N\cu_0)}\leq\exp\biggl(\frac{C+\gamma_6\log(1/\delta)}{p_*}\biggr)\norm{w}_{\underline L^{p_*}(\rho_2\cu_0)}\,.
\end{equation*}
For~$N=0$, this bound holds because~$R_0=\rho_2$. Since~$p_*\chi^N\geq\nf r4$ and~$\rho_1\cu_0\subseteq R_N\cu_0$, H\"older's inequality gives
\begin{equation*}
	\norm{w}_{\underline L^{\nf r4}(\rho_1\cu_0)}\leq2^{d/(p_*\chi^N)}\norm{w}_{\underline L^{p_*\chi^N}(R_N\cu_0)}\,.
\end{equation*} Absorbing~$2^{d/(p_*\chi^N)}\leq2^{d/p_*}$ into the exponential proves~\eqref{e.positive.moments}.
\end{proof}

\subsection{From powers to Harnack}

It remains to connect the integrals of~$w^{p_*}$ and~$w^{-p_*}$. After normalization by the exponential of~$\pm p_*$ times the mean of~$\log w$, these powers have logarithms bounded by one in~$L^1$, by~\eqref{e.log.mean.zero} and the choice of~$p_*$, and they satisfy the reverse inequalities~\eqref{e.small.moments} between small exponents. Lemma~\ref{l.bombieri} turns these two facts into bounds on their integrals, and the normalizations cancel in the product.

\begin{lemma}[Crossover estimate]
\label{l.crossover}
Let~$u\geq0$ be a supersolution in~$\cu_0$, let~$\eps>0$ and~$w\coloneqq u+\eps$. Then
\begin{equation}
	\fint_{\frac34\cu_0}w^{p_*}\fint_{\frac34\cu_0}w^{-p_*}\leq C\,.
	\label{e.crossover}
\end{equation}
\end{lemma}

\begin{proof}
Set
\begin{align*}
	Y&\coloneqq\exp\bigl(-p_*(\log w)_{\frac78\cu_0}\bigr)w^{p_*}\,,\\
	\log Y&=p_*\bigl(\log w-(\log w)_{\frac78\cu_0}\bigr)\,,\qquad\log Y^{-1}=-\log Y\,.
\end{align*}
By H\"older's inequality,~\eqref{e.log.mean.zero} and~$cC_{\eqref{e.log.mean.zero}}\leq1$,
\begin{equation*}
	\norm{\log Y}_{L^1(\frac78\cu_0)}\leq p_*C_{\eqref{e.log.mean.zero}}\sqrt\Theta\leq1\,.
\end{equation*}
Because~$w^{p_*}\leq1+w$ and~$w^{-p_*}\leq\eps^{-p_*}$, both~$Y$ and~$Y^{-1}$ are positive, finite almost everywhere, and integrable on~$\frac78\cu_0$. For~$\nf34\leq\rho_1<\rho_2\leq\nf78$ and~$0<a<1$, the bound~\eqref{e.small.moments} with~$b=1$,~$|\rho_1\cu_0|\leq1$ and~$|\rho_2\cu_0|\geq2^{-d}$ gives
\begin{equation*}
	\norm{Z}_{L^1(\rho_1\cu_0)}\leq\bigl(2^dC_{\eqref{e.small.moments}}(\rho_2-\rho_1)^{-\gamma_6}\bigr)^{\nf1a}\norm{Z}_{L^a(\rho_2\cu_0)}\qquad(Z\in\{w^{p_*},w^{-p_*}\})\,.
\end{equation*}
The reverse inequality for~$Z$ is invariant under multiplication by a positive constant, so it holds for~$Y$ and~$Y^{-1}$. Lemma~\ref{l.bombieri} with~$A_1=1$,~$A_2=2^dC_{\eqref{e.small.moments}}$ and~$\kappa=\gamma_6$ gives
\begin{equation*}
	\int_{\frac34\cu_0}Y\leq C\,,\qquad\int_{\frac34\cu_0}Y^{-1}\leq C\,.
\end{equation*}
The normalizing factors cancel in the product, and dividing by~$|\frac34\cu_0|^2$ gives~\eqref{e.crossover}.
\end{proof}

\begin{proof}[Proof of Theorem~\ref{t.harnack} for~$\eta=\nf r4$ and of~\eqref{e.harnack}]
We combine the crossover with the negative and positive iterations of Lemma~\ref{l.moment.iterations} on nested cubes, and then let~$\eps\downarrow0$. Let~$u\geq0$ be a supersolution, let~$\eps>0$ and~$w=u+\eps$. Raising~\eqref{e.crossover} to the power~$\nf1{p_*}$ gives
\begin{equation*}
	\norm{w}_{\underline L^{p_*}(\frac34\cu_0)}\norm{w^{-1}}_{\underline L^{p_*}(\frac34\cu_0)}\leq e^{C/p_*}\,.
\end{equation*} The bounds~\eqref{e.negative.moments} between~$\nf12$ and~$\nf34$ and~\eqref{e.positive.moments} between~$\nf58$ and~$\nf34$ give
\begin{equation*}
	\norm{w^{-1}}_{\underline L^{p_*}(\frac34\cu_0)}^{-1}\leq e^{C/p_*}\essinf_{\frac12\cu_0}w\qquad\mbox{and}\qquad\norm{w}_{\underline L^{\nf r4}(\frac58\cu_0)}\leq e^{C/p_*}\norm{w}_{\underline L^{p_*}(\frac34\cu_0)}\,.
\end{equation*}
Combining the three bounds and using~$\nf1{p_*}=\nf{\sqrt\Theta}c$,
\begin{equation}
	\norm{u+\eps}_{\underline L^{\nf r4}(\frac58\cu_0)}\leq e^{C\sqrt\Theta}\essinf_{\frac12\cu_0}(u+\eps)\,.
	\label{e.weak.harnack.eps}
\end{equation}
Since~$u\leq u+\eps$ and~$\essinf(u+\eps)=\essinf u+\eps$, letting~$\eps\downarrow0$ in~\eqref{e.weak.harnack.eps} proves~\eqref{e.weak.harnack} at~$\eta=\nf r4$.

\smallskip

Let~$u\geq0$ be a solution. Then~$w=u+\eps$ is a nonnegative solution, hence a subsolution. Corollary~\ref{c.lr.bound} with~$\eta=\nf r4$ between~$\nf12$ and~$\nf58$ gives
\begin{equation*}
	\esssup_{\frac12\cu_0}w\leq C\Theta^{\nf{2(d-1)}{r\theta}}\norm{w}_{L^{\nf r4}(\frac58\cu_0)}\leq C\Theta^{\nf{2(d-1)}{r\theta}}\norm{w}_{\underline L^{\nf r4}(\frac58\cu_0)}\,,
\end{equation*}
since~$|\frac58\cu_0|\leq1$. Since~$\log\Theta\leq\sqrt\Theta$ and~$\sqrt\Theta\geq1$, the power of~$\Theta$ and the prefactor are absorbed into the exponential in~\eqref{e.weak.harnack.eps}, giving
\begin{equation*}
	\esssup_{\frac12\cu_0}(u+\eps)\leq e^{C\sqrt\Theta}\essinf_{\frac12\cu_0}(u+\eps)\,.
\end{equation*}
Letting~$\eps\downarrow0$ proves~\eqref{e.harnack}. Section~\ref{ss.harnack.endpoint} proves~\eqref{e.weak.harnack} for the full range~$0<\eta\leq\nf{r^*}2$.
\end{proof}

\subsection{Endpoint weak Harnack inequality}
\label{ss.harnack.endpoint}

% Added 2026-10-01. Mathematical source: weak-harnack-nonradial-endpoint.tex (SHA-256
% f159f077960ad92124ee8175f784f45c4237e30b15a9ee90e46393c5749776f4): Lemma 2.1 (lem:reconstruction),
% Lemma 2.2 (lem:level-set, moved into the proof of p.endpoint.potential), Proposition 3.1 (prop:potential),
% Section 4 (eq:interior-weak-harnack, lem:capacitary, lem:source-mass) and Section 5 (thm:endpoint).
% Notation of the note -> this manuscript: Q -> \cu_0, h_k -> 3^{-k}, \E -> \Ener_{\cu_0},
% \A(v,z) -> \int\nabla v\cdot\a\nabla z, \eta_c -> r^*/2, D/2 -> d/r^*, \eta_0 -> r/4, H -> \exp(C\sqrt\Theta),
% U -> \frac34\cu_0, W -> V, h -> \psi, B_0,B_1 -> Q_0,Q_1, \zeta -> \vartheta, M -> \nu(\cu_0) or \mu(\frac34\cu_0).

This subsection proves~\eqref{e.weak.harnack} at~$\eta=\nf{r^*}2$ and obtains all smaller exponents by H\"older's inequality. Reaching the endpoint through~\eqref{e.positive.negative.moments} would require~$m=\nf12$, where~$c_m=\nf{|m|}{(1-2m)}$ diverges. Instead, we use the summability of the scale terms in~\eqref{e.lower.half.series}. Lemma~\ref{l.dirichlet.reconstruction} keeps these terms separate for functions with zero boundary values, and Proposition~\ref{p.endpoint.potential} bounds the~$L^{\nf{r^*}2}$ norm of a potential by the mass of its source. For a supersolution~$u$, Lemma~\ref{l.source.mass} bounds the mass of~$-\nabla\cdot\a\nabla u$ inside~$\frac34\cu_0$ by the infimum of~$u$. The difference between~$u$ and the potential of this restricted measure is a solution in~$\frac34\cu_0$. The other inputs are Corollary~\ref{c.lr.bound}, the weak Harnack inequality~\eqref{e.weak.harnack} at~$\eta=\nf r4$ and the Harnack inequality~\eqref{e.harnack}, both proved above, together with the results of Sections~\ref{s.harmonic.matrices}--\ref{s.fractional.reconstruction} and Appendix~\ref{a.weighted.proofs}.

\paragraph{Interior estimates.} We apply Corollary~\ref{c.lr.bound}, the weak Harnack inequality~\eqref{e.weak.harnack} at~$\eta=\nf r4$ and the Harnack inequality~\eqref{e.harnack} on cubes~$Q'\subseteq\cu_0$ defined below. Let~$n\geq1$ and~$y\in3^{-n}\Zd$ be such that~$y+\cu_{1-n}\subseteq\cu_0$. Set
\begin{equation*}
	Q'\coloneqq y+\cu_{1-n}\,,\qquad\Phi(x)\coloneqq y+3^{1-n}x\,.
\end{equation*}
The cube~$Q'$ is the threefold enlargement of~$y+\cu_{-n}$. Composition with~$\Phi$ preserves smoothness, means and compact supports and multiplies energies by~$3^{(1-n)(2-d)}$, so it maps~$H^1_\a(Q')$,~$\Hzero{\a}(Q')$, the solutions and the supersolutions in~$Q'$ onto the corresponding spaces and classes for the field~$\a\circ\Phi$ in~$\cu_0$, which satisfies~\eqref{e.weighted.hypotheses}. The map~$w\mapsto3^{n-1}w\circ\Phi$ preserves affine boundary values up to constants and gives
\begin{equation*}
	\nabla(3^{n-1}w\circ\Phi)=(\nabla w)\circ\Phi\,.
\end{equation*}
Thus energies and volumes change by the same factor. By~\eqref{e.upper.dirichlet} and~\eqref{e.lower.all.functions}, the coarse-grained matrices of~$\a\circ\Phi$ on a simplex~$\triangle\subseteq\cu_0$ are~$\a(\Phi(\triangle))$ and~$\a_*(\Phi(\triangle))$.

For~$k\geq1$ we have~$y\in3^{1-n-k}\Zd$, so~$\Phi$ maps~$\mathscr T_k$ one-to-one into~$\mathscr T_{k+n-1}(Q')$, and~\eqref{e.partition.restricted} bounds the mean of~$|\a(\Phi(\triangle))|^p$ over~$\triangle\in\mathscr T_k$ by~$3^{(n-1)d}$ times the mean of~$|\a(\triangle)|^p$ over~$\mathscr T_{k+n-1}$, and likewise for~$|\a_*^{-1}|^q$. Because~$y$ need not lie on the coarser grid~$3^{1-n}\Zd$, the simplices~$\Phi(\triangle)$ for~$\triangle\in\mathscr T_0$ need not belong to any~$\mathscr T_j$. Each simplex of~$\mathscr T_0$ is, up to faces, a union of simplices of~$\mathscr T_1$. The subadditivity bounds~\eqref{e.upper.subadditivity}--\eqref{e.lower.cube.aggregation} and Jensen's inequality therefore bound the two matrix moments of~$\a\circ\Phi$ over~$\mathscr T_0$ by those over~$\mathscr T_1$.

Reindexing the scales in~\eqref{e.upper.half.series} and~\eqref{e.lower.half.series}, which introduces the factors~$3^{(n-1)s}$ and~$3^{(n-1)t}$, shows that~$\Lambda_{s,1,p}(\cu_0)$ and~$\lambda_{t,1,q}(\cu_0)^{-1}$, computed for~$\a\circ\Phi$, are at most~$C(n)$ times their values for~$\a$, where~$C(n)$ depends only on~$n$ and~$d,p,q,s,t$. For the specific map~$\Phi$ above, these bounds give an ellipticity ratio at most~$C(n)\Theta$ for~$\a\circ\Phi$. Hence Corollary~\ref{c.lr.bound}, the inequality~\eqref{e.weak.harnack} for~$\eta=\nf r4$ and~\eqref{e.harnack} hold for~$\a\circ\Phi$ and, by composition with~$\Phi^{-1}$, on~$Q'$ and its dilations about its center, with additional factors depending only on~$n$; since~$\Theta\geq1$, these factors and the powers of~$\Theta$ are absorbed into~$\exp(C\sqrt\Theta)$.

\smallskip

We use~$n=4$ and the cubes~$Q_y\coloneqq y+\cu_{-3}$,~$y\in3^{-4}\Zd$. If the centers of~$Q_y$ and~$Q_{y'}$ differ by~$3^{-4}$ in one coordinate, we call them neighbors. Their central halves~$\frac12Q_y=y+\frac12\cu_{-3}$ overlap in a set occupying a fixed positive fraction of each half. Let~$u$ be a nonnegative supersolution in~$\cu_0$, and take the~$y$ for which~$y+\overline{\cu}_{-4}$ meets~$\frac{15}{16}\overline{\cu}_0$. The points of these~$Q_y$ satisfy
\begin{equation*}
	|x|_\infty\leq\tfrac{15}{32}+2\cdot3^{-4}<\tfrac12\,.
\end{equation*}
Their centers form a box in~$3^{-4}\Zd$ whose number of points depends only on~$d$. For neighboring centers, the intersection~$\frac12Q_y\cap\frac12Q_{y'}$ lies in~$\frac58Q_y$ and occupies a fixed positive fraction of~$\frac58Q_y$. Comparing the~$L^{\nf r4}$ averages and applying the rescaled weak Harnack estimate~\eqref{e.weak.harnack} on~$Q_y$ give
\begin{equation*}
	\essinf_{\frac12Q_{y'}}u\leq\exp(C\sqrt\Theta)\essinf_{\frac12Q_y}u\,.
\end{equation*} The cubes~$y+\cu_{-4}\subseteq\frac12Q_y$ cover~$\frac{15}{16}\cu_0$ up to a null set, so some~$y_1$ satisfies~$\essinf_{\frac12Q_{y_1}}u\leq\essinf_{\frac{15}{16}\cu_0}u$. Chaining this comparison from~$y_1$ and summing the rescaled weak Harnack estimates~\eqref{e.weak.harnack} over the finitely many cubes~$Q_y$ give, for every nonnegative supersolution~$u$ in~$\cu_0$,
\begin{equation}
	\norm{u}_{L^{\nf r4}(\frac{15}{16}\cu_0)}\leq\exp\bigl(C\sqrt\Theta\bigr)\essinf_{\frac{15}{16}\cu_0}u\,.
	\label{e.interior.weak.harnack}
\end{equation}
For the Harnack estimate, take the cubes~$Q_y$ for which~$y+\overline{\cu}_{-4}$ meets~$\frac58\overline{\cu}_0$. Their points satisfy
\begin{equation*}
	|x|_\infty\leq\tfrac5{16}+2\cdot3^{-4}<\tfrac38\,,
\end{equation*}
so they lie in~$\frac34\cu_0$. For neighboring centers, the overlap gives~$\essinf_{\frac12Q_{y'}}u\leq\esssup_{\frac12Q_y}u$. Chaining the rescaled Harnack estimate~\eqref{e.harnack} over these cubes gives, for every nonnegative~$u\in\Csol(\frac34\cu_0)$,
\begin{equation}
	\esssup_{\frac58\cu_0}u\leq\exp\bigl(C\sqrt\Theta\bigr)\essinf_{\frac12\cu_0}u\,.
	\label{e.interior.harnack}
\end{equation}

\paragraph{Reconstruction with zero boundary values.} The fractional estimate~\eqref{e.lower.fractional} uses the matrices~$\a_*^{-1}(\triangle)$,~$\triangle\in\mathscr T_k$, only through the weighted sum over~$k$ in~\eqref{e.lower.half.series}. At the endpoint we keep the contribution of each scale apart and use that these contributions are summable. For~$k\geq1$ let
\begin{equation}
	b_k\coloneqq3^{-kt}\biggl(\avsum_{\triangle\in\mathscr T_k}|\a_*^{-1}(\triangle)|^q\biggr)^{\nf1{2q}}\,,\qquad\mbox{so that}\qquad\sum_{k\geq1}b_k\leq\frac{\lambda_{t,1,q}(\cu_0)^{-\nf12}}{1-3^{-t}}
	\label{e.endpoint.scales}
\end{equation}
by~\eqref{e.lower.half.series}. For~$v\in\Hzero{\a}(\cu_0)$, the mean gradient~$P_0\nabla v$ vanishes by~\eqref{e.weighted.coercivity}. We reconstruct~$v$ from the successive differences~$(P_k-P_{k-1})\nabla v$ using Newtonian potentials on~$\Rd$. The series starts at~$k=1$, and the estimates for each term retain the factor~$b_k$ in~$L^r$ and~$L^\infty$.

% Statement adapted from Lemma 2.1 (lem:reconstruction) of weak-harnack-nonradial-endpoint.tex.
% Proof revised 2026-10-01 to use whole-space Newtonian potentials of the zero-extended gradient increments.
% The terms need not have zero boundary values. The statement and the two scale estimates are unchanged.
\begin{lemma}[Scale decomposition for functions with zero boundary values]
\label{l.dirichlet.reconstruction}
There is~$C<\infty$, depending only on~$d$ and~$q$, such that every~$v\in\Hzero{\a}(\cu_0)$ is the sum~$v=\sum_{k\geq1}v_k$ of a series converging in~$L^r(\cu_0)$, whose terms~$v_k\in L^\infty(\cu_0)$ satisfy
\begin{equation}
	\norm{v_k}_{L^r(\cu_0)}\leq C3^{-k(1-t)}b_k\Ener_{\cu_0}(v)^{\nf12}\qquad\mbox{and}\qquad\norm{v_k}_{L^\infty(\cu_0)}\leq C3^{\frac d{r^*}k}b_k\Ener_{\cu_0}(v)^{\nf12}\,.
	\label{e.dirichlet.blocks}
\end{equation}
\end{lemma}

\begin{proof}
We construct the terms~$v_k$ using Newtonian potentials, prove the two bounds in~\eqref{e.dirichlet.blocks}, and sum the series. We write~$P_k=P_k(\cdot\,;\cu_0)$, defined in Section~\ref{s.fractional.reconstruction} with~$n=0$ and~$z=0$, and~$r'=\nf r{(r-1)}$.

\smallskip

\emph{Step 1.} Newtonian potentials. Let~$v\in C^\infty_c(\cu_0)$. For~$k\geq1$,
\begin{equation*}
	P_0\nabla v=(\nabla v)_{\cu_0}=0\,,\qquad G_k\coloneqq(P_k-P_{k-1})\nabla v\,.
\end{equation*}
Lemma~\ref{l.lower.averages} with~$Q=\cu_0$ and~$k\geq1$, together with~\eqref{e.partition.average}, gives
\begin{equation*}
	\norm{P_k\nabla v}_{L^r(\cu_0)}\leq3^{kt}b_k\Ener_{\cu_0}(v)^{\nf12}\,.
\end{equation*}
Since~$P_{k-1}\nabla v$ is the average of~$P_k\nabla v$ on each triadic cube of side length~$3^{1-k}$, Jensen's inequality gives
\begin{equation}
	\norm{G_k}_{L^r(\cu_0)}\leq2\cdot3^{kt}b_k\Ener_{\cu_0}(v)^{\nf12}\,.
	\label{e.endpoint.gradient.increments}
\end{equation}
Let~$\Gamma$ be the fundamental solution of~$-\Delta$ on~$\Rd$. Extend~$G_k$ by zero outside~$\cu_0$ and define
\begin{equation}
	v_k(x)\coloneqq\int_{\cu_0}\nabla\Gamma(y-x)\cdot G_k(y)\,dy\qquad(x\in\Rd)\,.
	\label{e.endpoint.newtonian}
\end{equation}
The integral is absolutely convergent because~$G_k$ is bounded and~$|\nabla\Gamma(x)|\leq C|x|^{1-d}$. For~$\phi\in C^\infty_c(\cu_0)$, extended by zero to~$\Rd$, let~$z_\phi=\Gamma*\phi$. Then~$z_\phi$ is smooth and solves~$-\Delta z_\phi=\phi$. The Calder\'on--Zygmund estimate on~$\Rd$ gives
\begin{equation}
	\norm{\nabla^2z_\phi}_{L^{r'}(\cu_0)}\leq C\norm{\phi}_{L^{r'}(\cu_0)}\,.
	\label{e.dirichlet.estimates}
\end{equation}
By Fubini's theorem and~\eqref{e.endpoint.newtonian},
\begin{equation}
	\int_{\cu_0}v_k\phi=\int_{\cu_0}G_k\cdot\nabla z_\phi\,.
	\label{e.dirichlet.duality}
\end{equation}

\smallskip

\emph{Step 2.} The bound in~$L^r$. The field~$G_k$ has zero average on every triadic cube~$F$ of side length~$3^{1-k}$. Subtracting~$(\nabla z_\phi)_F$ on each~$F$ in~\eqref{e.dirichlet.duality}, the Poincar\'e and H\"older inequalities on these cubes and the Hessian bound in~\eqref{e.dirichlet.estimates} give
\begin{equation*}
	\biggl|\int_{\cu_0}v_k\phi\biggr|=\biggl|\sum_F\int_FG_k\cdot\bigl(\nabla z_\phi-(\nabla z_\phi)_F\bigr)\biggr|\leq C3^{-k}\norm{G_k}_{L^r(\cu_0)}\norm{\phi}_{L^{r'}(\cu_0)}\,.
\end{equation*}
By duality,
\begin{equation*}
	\norm{v_k}_{L^r(\cu_0)}\leq C3^{-k}\norm{G_k}_{L^r(\cu_0)}\,.
\end{equation*}
The bound~\eqref{e.endpoint.gradient.increments} on~$\norm{G_k}_{L^r(\cu_0)}$ gives the~$L^r$ estimate in~\eqref{e.dirichlet.blocks}.

\smallskip

\emph{Step 3.} The bound in~$L^\infty$. The formula~\eqref{e.endpoint.newtonian} and the bound on~$\nabla\Gamma$ give
\begin{equation}
	|v_k(y)|\leq C\int_{\cu_0}\frac{|G_k(x)|}{|x-y|^{d-1}}\,dx\qquad(y\in\cu_0)\,.
	\label{e.dirichlet.potential.pairing}
\end{equation}
We split the integral in~\eqref{e.dirichlet.potential.pairing} at~$|x-y|=3^{-k}$. The field~$G_k$ is constant on triadic cubes of volume~$3^{-kd}$, so the part near~$y$ satisfies
\begin{align*}
	\norm{G_k}_{L^\infty(\cu_0)}&\leq3^{\frac drk}\norm{G_k}_{L^r(\cu_0)}\,,\\
	\int_{\{|x-y|<3^{-k}\}}\frac{|G_k(x)|}{|x-y|^{d-1}}\,dx&\leq C3^{-k}\norm{G_k}_{L^\infty(\cu_0)}\leq C3^{(\frac dr-1)k}\norm{G_k}_{L^r(\cu_0)}\,.
\end{align*}
Since~$r<2<d$, the exponent~$(1-d)r'+d$ is negative, and
\begin{equation*}
	\int_{\{|h|\geq3^{-k}\}}|h|^{(1-d)r'}\,dh=C3^{-k((1-d)r'+d)}\,.
\end{equation*}
H\"older's inequality therefore gives
\begin{equation*}
	\int_{\{|x-y|\geq3^{-k}\}}\frac{|G_k(x)|}{|x-y|^{d-1}}\,dx\leq C3^{(\frac dr-1)k}\norm{G_k}_{L^r(\cu_0)}\,.
\end{equation*}
Together, the near and far parts of~\eqref{e.dirichlet.potential.pairing} give
\begin{equation*}
	\norm{v_k}_{L^\infty(\cu_0)}\leq C3^{(\frac dr-1)k}\norm{G_k}_{L^r(\cu_0)}\,.
\end{equation*}
By~\eqref{e.sobolev.exponents},
\begin{equation}
	\frac dr-\frac d{r^*}=1-t\,,
	\label{e.endpoint.exponents}
\end{equation}
so~$(\frac dr-1)k+tk=\frac d{r^*}k$. The bound~\eqref{e.endpoint.gradient.increments} on~$\norm{G_k}_{L^r(\cu_0)}$ gives the~$L^\infty$ estimate in~\eqref{e.dirichlet.blocks}.

\smallskip

\emph{Step 4.} Summation and the completion. By~\eqref{e.dirichlet.blocks} and~\eqref{e.endpoint.scales}, the series~$\sum_kv_k$ converges in~$L^r(\cu_0)$. For~$\phi\in C^\infty_c(\cu_0)$, summing~\eqref{e.dirichlet.duality} over~$k\leq N$ and using~$P_0\nabla v=0$,
\begin{equation*}
	\sum_{k=1}^N\int_{\cu_0}v_k\phi=\int_{\cu_0}P_N\nabla v\cdot\nabla z_\phi\longrightarrow\int_{\cu_0}\nabla v\cdot\nabla z_\phi=\int_{\cu_0}v\phi\,,
\end{equation*}
The limit follows from the uniform convergence~$P_N\nabla v\to\nabla v$. The last equality uses~$-\Delta z_\phi=\phi$ and integration by parts, since~$v$ has compact support in~$\cu_0$. Hence~$\sum_kv_k=v$. For general~$v\in\Hzero{\a}(\cu_0)$, take~$v^{(j)}\in C^\infty_c(\cu_0)$ converging to~$v$ in~$H^1_\a(\cu_0)$. The construction is linear, so~\eqref{e.dirichlet.blocks} for~$v^{(j)}-v^{(i)}$ shows that, for each~$k$, the terms~$v_k^{(j)}$ converge in~$L^r(\cu_0)$ and in~$L^\infty(\cu_0)$ to some~$v_k$ satisfying~\eqref{e.dirichlet.blocks}, and, summed over~$k$ with~\eqref{e.endpoint.scales}, that~$\sum_kv_k^{(j)}\to\sum_kv_k$ in~$L^r(\cu_0)$. Since~$v^{(j)}\to v$ in~$L^1(\cu_0)$ by Proposition~\ref{p.weighted.calculus}\ref{i.weighted.identification}, the sum is~$v$.
\end{proof}

\paragraph{Potentials of measures.} Let~$\nu$ be a finite nonnegative Radon measure on~$\cu_0$ such that~$\phi\mapsto\int_{\cu_0}\phi\,d\nu$,~$\phi\in C^\infty_c(\cu_0)$, extends to a continuous linear functional on~$\Hzero{\a}(\cu_0)$, on which~$\Ener_{\cu_0}^{\nf12}$ is a norm equivalent to~\eqref{e.weighted.norm} by~\eqref{e.weighted.coercivity}. We write~$\nu(g)$ for the value of this extension at~$g\in\Hzero{\a}(\cu_0)$, and~$\nu(\cu_0)$ for the total mass of the measure. The value~$\nu(g)$ is defined by continuity, without assigning values to~$g$ on sets of measure zero. By Proposition~\ref{p.weighted.calculus}\ref{i.weighted.approximation}, a nonnegative~$g\in\Hzero{\a}(\cu_0)$ is the limit of nonnegative functions in~$C^\infty_c(\cu_0)$, which can be taken at most~$K$ if~$g\leq K$. Passing to the limit in~$0\leq\int\phi\,d\nu\leq K\nu(\cu_0)$ gives, for~$g\in\Hzero{\a}(\cu_0)$ and~$K\geq0$,
\begin{equation}
	\nu(g)\geq0\quad\mbox{if }g\geq0\,,\qquad\mbox{and}\qquad\nu(g)\leq K\nu(\cu_0)\quad\mbox{if }0\leq g\leq K\,.
	\label{e.measure.evaluation}
\end{equation}
The potential of~$\nu$ is the~$v\in\Hzero{\a}(\cu_0)$ satisfying
\begin{equation*}
	\int_{\cu_0}\nabla g\cdot\a\nabla v=\nu(g)\qquad(g\in\Hzero{\a}(\cu_0))\,.
\end{equation*}
It exists and is unique by the Lax--Milgram lemma, as in the proof of Proposition~\ref{p.harmonic.replacement}.

% Adapted from Proposition 3.1 (prop:potential) of weak-harnack-nonradial-endpoint.tex, with its Lemma 2.2
% (lem:level-set) as Step 2 of the proof; stated only for measures whose action extends continuously to the
% zero-boundary weighted space.
\begin{proposition}[Endpoint potential estimate]
\label{p.endpoint.potential}
There is~$C<\infty$, depending only on~$d,q,t$, such that for every finite nonnegative Radon measure~$\nu$ on~$\cu_0$ whose action on~$C^\infty_c(\cu_0)$ extends continuously to~$\Hzero{\a}(\cu_0)$, the potential~$v$ of~$\nu$ satisfies
\begin{equation}
	\norm{v}_{L^{\nf{r^*}2}(\cu_0)}\leq C\lambda_{t,1,q}(\cu_0)^{-1}\nu(\cu_0)\,.
	\label{e.endpoint.potential}
\end{equation}
\end{proposition}

\begin{proof}
We first show~$v\geq0$ and prove the truncation bound~\eqref{e.endpoint.truncation}. The estimates for~$v_k$ in~\eqref{e.dirichlet.blocks} then give~\eqref{e.endpoint.distribution}, which bounds the decreasing rearrangement of~$v$ on each interval~$(3^{-(N+1)d},3^{-Nd}]$. Integrating these bounds proves~\eqref{e.endpoint.potential}, even when~$\nf{r^*}2<1$.

\smallskip

\emph{Step 1.} Positivity and truncations. By Proposition~\ref{p.weighted.calculus}\ref{i.weighted.chain} and~\ref{i.weighted.products},~$(-v)_+$ is a nonnegative element of~$\Hzero{\a}(\cu_0)$ with gradient~$-\indc_{\{v<0\}}\nabla v$. Testing the equation of~$v$ with~$(-v)_+$ and using~\eqref{e.measure.evaluation} gives
\begin{equation*}
	-\Ener_{\cu_0}((-v)_+)=\nu((-v)_+)\geq0\,.
\end{equation*}
Thus~$(-v)_+$ has zero gradient and vanishes by Proposition~\ref{p.weighted.calculus}\ref{i.weighted.zero}. Hence~$v\geq0$. For~$K>0$, the truncation~$v\wedge K=T_K(v)$ of Appendix~\ref{a.weighted.proofs} belongs to~$\Hzero{\a}(\cu_0)$, has gradient~$\indc_{\{v<K\}}\nabla v$, and satisfies~$0\leq v\wedge K\leq K$. Testing the equation of~$v$ with~$v\wedge K$ and using~\eqref{e.measure.evaluation},
\begin{equation}
	\Ener_{\cu_0}(v\wedge K)=\int_{\cu_0}\nabla(v\wedge K)\cdot\a\nabla v=\nu(v\wedge K)\leq K\nu(\cu_0)\,.
	\label{e.endpoint.truncation}
\end{equation}

\smallskip

\emph{Step 2.} A distribution estimate. For integers~$N\geq0$ let
\begin{equation*}
	c_N\coloneqq\sum_{1\leq k\leq N}3^{-\frac d{r^*}(N-k)}b_k+\sum_{k>N}3^{-(1-t)(k-N)}b_k\,.
\end{equation*}
The terms with~$k\leq N$ have scale at least~$3^{-N}$ and will use the~$L^\infty$ bound in~\eqref{e.dirichlet.blocks}. The terms with~$k>N$ will use its~$L^r$ bound. Exchanging the order of summation, both geometric weights sum to constants depending only on~$d,q,t$, so
\begin{equation*}
	\sum_{N\geq0}c_N\leq C\sum_{k\geq1}b_k\,.
\end{equation*}
Since~$r^*>r>1$ by~\eqref{e.parameter.facts}, the~$\ell^{r^*}$ norm of a nonnegative sequence is at most its~$\ell^1$ norm, and therefore
\begin{equation}
	\sum_{N\geq0}c_N^{r^*}\leq\biggl(\sum_{N\geq0}c_N\biggr)^{r^*}\leq C\biggl(\sum_{k\geq1}b_k\biggr)^{r^*}\,.
	\label{e.endpoint.sums}
\end{equation}
Now let~$f\in\Hzero{\a}(\cu_0)$,~$K>0$ and~$N\geq0$, and suppose that~$|f|\geq K$ almost everywhere on a measurable set~$A\subseteq\cu_0$ with~$3^{-(N+1)d}<|A|\leq3^{-Nd}$. Decompose~$f=\sum_kf_k$ by Lemma~\ref{l.dirichlet.reconstruction}. Averaging~$|f|$ over~$A$, and using H\"older's inequality on~$A$ for the terms with~$k>N$,
\begin{equation*}
	K\leq\sum_{1\leq k\leq N}\norm{f_k}_{L^\infty(\cu_0)}+|A|^{-\nf1r}\sum_{k>N}\norm{f_k}_{L^r(\cu_0)}\,.
\end{equation*}
By~\eqref{e.sobolev.exponents},
\begin{equation*}
	|A|^{-\nf1r}3^{-(1-t)k}<3^{\frac dr(N+1)}3^{-(1-t)k}=3^{\nf dr}3^{\frac d{r^*}N}3^{-(1-t)(k-N)}\,.
\end{equation*}
Together with~\eqref{e.dirichlet.blocks}, this gives
\begin{equation}
	K\leq C3^{\frac d{r^*}N}c_N\Ener_{\cu_0}(f)^{\nf12}\,.
	\label{e.endpoint.distribution}
\end{equation}

\smallskip

\emph{Step 3.} The decreasing rearrangement. For~$0<\vartheta\leq1$ let
\begin{equation*}
	v^*(\vartheta)\coloneqq\inf\{K\geq0:|\{v>K\}|\leq\vartheta\}\,.
\end{equation*}
Since~$K\mapsto|\{v>K\}|$ is right-continuous,~$v^*(\vartheta)>K$ exactly when~$|\{v>K\}|>\vartheta$. Hence~$v^*$ on~$(0,1)$ and~$v$ on~$\cu_0$ have the same distribution function, and for~$\eta>0$,
\begin{equation*}
	\int_{\cu_0}v^\eta=\int_0^1v^*(\vartheta)^\eta\,d\vartheta\,.
\end{equation*}
Let~$N\geq0$,~$3^{-(N+1)d}<\vartheta\leq3^{-Nd}$ and~$0<K<v^*(\vartheta)$. Then~$|\{v>K\}|>\vartheta$, and since Lebesgue measure has no atoms, there is a measurable~$A\subseteq\{v>K\}$ with~$|A|=\vartheta$. The truncation~$v\wedge K$ equals~$K$ on~$A$, so~\eqref{e.endpoint.distribution} for~$f=v\wedge K$ and~\eqref{e.endpoint.truncation} give
\begin{equation*}
	K\leq C3^{\frac d{r^*}N}c_N\bigl(K\nu(\cu_0)\bigr)^{\nf12}\,.
\end{equation*}
Dividing by~$K^{\nf12}$, squaring and letting~$K\uparrow v^*(\vartheta)$,
\begin{equation}
	v^*(\vartheta)\leq C3^{\frac{2d}{r^*}N}c_N^2\,\nu(\cu_0)\qquad\bigl(3^{-(N+1)d}<\vartheta\leq3^{-Nd}\bigr)\,.
	\label{e.endpoint.rearrangement}
\end{equation}

\smallskip

\emph{Step 4.} Summation over the scales. The intervals~$(3^{-(N+1)d},3^{-Nd}]$,~$N\geq0$, cover~$(0,1)$, and each has length at most~$3^{-Nd}$. We raise~\eqref{e.endpoint.rearrangement} to the power~$\nf{r^*}2$; the powers of three cancel, and~\eqref{e.endpoint.sums} gives
\begin{align*}
	\int_{\cu_0}v^{\nf{r^*}2}=\int_0^1v^*(\vartheta)^{\nf{r^*}2}\,d\vartheta&\leq C\nu(\cu_0)^{\nf{r^*}2}\sum_{N\geq0}3^{-dN}\bigl(3^{\frac{2d}{r^*}N}c_N^2\bigr)^{\nf{r^*}2}\\
	&=C\nu(\cu_0)^{\nf{r^*}2}\sum_{N\geq0}c_N^{r^*}\leq C\nu(\cu_0)^{\nf{r^*}2}\biggl(\sum_{k\geq1}b_k\biggr)^{r^*}\,.
\end{align*}
Raising to the power~$\nf2{r^*}$ and using~\eqref{e.endpoint.scales} proves~\eqref{e.endpoint.potential}.
\end{proof}

The summability of~$(b_k)$ enters through~\eqref{e.endpoint.sums}. A bound on each~$b_k$ alone would give from~\eqref{e.endpoint.rearrangement} only the weak-type estimate
\begin{equation*}
	v^*(\vartheta)\leq C\vartheta^{-\nf2{r^*}}\nu(\cu_0)\sup_kb_k^2\,.
\end{equation*}

% Adapted from Section 4 of weak-harnack-nonradial-endpoint.tex (lem:capacitary and lem:source-mass, merged): the
% capacitary function and its properties are Steps 4--6 of the proof. The interior estimates on fixed cubes are
% e.interior.weak.harnack and Corollary c.lr.bound on Q_1, justified in the paragraph "Interior cubes".
\begin{lemma}[Interior source measure estimate]
\label{l.source.mass}
Let~$u\geq0$ be a supersolution in~$\cu_0$. Then~$\mu\coloneqq-\nabla\cdot\a\nabla u$ is a nonnegative Radon measure on~$\cu_0$, possibly of infinite total mass, with
\begin{equation*}
	\int_{\cu_0}\phi\,d\mu=\int_{\cu_0}\nabla\phi\cdot\a\nabla u\qquad(\phi\in C^\infty_c(\cu_0))\,.
\end{equation*}
Moreover,
\begin{equation}
	\mu\bigl(\tfrac34\cu_0\bigr)\leq C\Lambda_{s,1,p}(\cu_0)\exp\bigl(C\sqrt\Theta\bigr)\essinf_{\frac12\cu_0}u\,.
\label{e.source.mass}
\end{equation}
The restriction~$\nu$ of~$\mu$ to~$\frac34\cu_0$ has a continuous action on~$\Hzero{\a}(\cu_0)$. Its potential~$V\in\Hzero{\a}(\cu_0)$ satisfies~$0\leq V\leq u$ almost everywhere in~$\cu_0$, and~$u-V\in\Csol(\frac34\cu_0)$.
\end{lemma}

\begin{proof}
Let~$\nu$ be the restriction of~$\mu$ to~$\frac34\cu_0$ and let~$V$ be its potential. We show~$0\leq V\leq u$ and that~$V$ is a solution outside~$\frac34\cu_0$. For a cube~$Q_0$ in this exterior region, we construct a capacitary function~$\psi$ whose value~$\nu(\psi)$ is at least~$\exp(-C\sqrt\Theta)\mu(\frac34\cu_0)$. The equation of~$V$ and the energy of~$\psi$ bound~$\nu(\psi)$ above by~$\esssup_{Q_0}V\,\Ener_{\cu_0}(\psi)$.

\smallskip

\emph{Step 1.} The measures. For~$\phi\in C^\infty_c(\cu_0)$ let
\begin{equation*}
	\mu(\phi)\coloneqq\int_{\cu_0}\nabla\phi\cdot\a\nabla u\,.
\end{equation*}
This is nonnegative if~$\phi\geq0$, since~$-u\in\Csub(\cu_0)$. For a compact~$E\subset\cu_0$, fix a cutoff~$\varphi_E\in C^\infty_c(\cu_0)$ with~$0\leq\varphi_E\leq1$ and~$\varphi_E=1$ on~$E$. Positivity, applied to~$\norm{\phi}_{L^\infty(\cu_0)}\varphi_E\pm\phi$, gives
\begin{equation*}
	|\mu(\phi)|\leq\mu(\varphi_E)\norm{\phi}_{L^\infty(\cu_0)}\qquad(\supp\phi\subseteq E)\,.
\end{equation*}
Hence~$\mu$ extends, by uniform approximation, to a positive linear functional on~$C_c(\cu_0)$, which by the Riesz representation theorem is integration against a nonnegative Radon measure, also denoted~$\mu$. Let~$\nu$ be its restriction to~$\frac34\cu_0$, which is finite. For~$\phi\in C^\infty_c(\cu_0)$, set
\begin{equation*}
	\phi_\eps\coloneqq(\phi^2+\eps^2)^{\nf12}-\eps\,.
\end{equation*}
The functions~$\phi_\eps\in C^\infty_c(\cu_0)$ converge to~$|\phi|$ uniformly as~$\eps\downarrow0$, with gradients bounded by~$|\nabla\phi|$ and converging to~$\nabla|\phi|$ almost everywhere. Dominated convergence, with~$\a\nabla u\in L^1(\cu_0)$ and~$\mu(\supp\phi)<\infty$, gives~$\int|\phi|\,d\mu=\int\nabla|\phi|\cdot\a\nabla u$, and hence
\begin{equation*}
	\biggl|\int_{\cu_0}\phi\,d\nu\biggr|\leq\int_{\cu_0}|\phi|\,d\mu=\int_{\cu_0}\nabla|\phi|\cdot\a\nabla u\leq\Ener_{\cu_0}(u)^{\nf12}\Ener_{\cu_0}(\phi)^{\nf12}\,.
\end{equation*}
Thus~$\nu$ satisfies the hypothesis of Proposition~\ref{p.endpoint.potential}. Since~$\nu$ is the restriction of~$\mu$, we have~$\int_{\cu_0}\nabla g\cdot\a\nabla u=\mu(g)\geq\nu(g)$ for nonnegative~$g\in C^\infty_c(\cu_0)$. The pairing with~$u$ is continuous on~$\Hzero{\a}(\cu_0)$ by~\eqref{e.weighted.pairing}. Approximation by nonnegative smooth functions (Proposition~\ref{p.weighted.calculus}\ref{i.weighted.approximation}) gives
\begin{equation}
	\int_{\cu_0}\nabla g\cdot\a\nabla u\geq\nu(g)\qquad\mbox{for every nonnegative }g\in\Hzero{\a}(\cu_0)\,.
	\label{e.source.restriction}
\end{equation}

\smallskip

\emph{Step 2.} The potential of~$\nu$. Let~$V$ be the potential of~$\nu$; then~$V\geq0$, by Step~1 of the proof of Proposition~\ref{p.endpoint.potential}. To show~$V\leq u$, take~$\phi_j\in C^\infty_c(\cu_0)$ converging to~$V$, and nonnegative smooth functions~$u_j$ of finite energy converging to~$u$, in~$H^1_\a(\cu_0)$ (Proposition~\ref{p.weighted.calculus}\ref{i.weighted.approximation}). The function~$(\phi_j-u_j)_+$ is nonnegative and vanishes outside~$\supp\phi_j$, so it belongs to~$\Hzero{\a}(\cu_0)$ by Lemma~\ref{l.weighted.testing}\ref{i.testing.compact}, and it converges to~$g\coloneqq(V-u)_+$ by Proposition~\ref{p.weighted.calculus}\ref{i.weighted.chain}. Hence~$g$ is a nonnegative element of~$\Hzero{\a}(\cu_0)$ with gradient~$\indc_{\{V>u\}}\nabla(V-u)$, and the equation of~$V$ and~\eqref{e.source.restriction} give
\begin{equation*}
	\Ener_{\cu_0}(g)=\int_{\cu_0}\nabla g\cdot\a\nabla V-\int_{\cu_0}\nabla g\cdot\a\nabla u=\nu(g)-\int_{\cu_0}\nabla g\cdot\a\nabla u\leq0\,.
\end{equation*}
So~$g=0$, that is,~$V\leq u$. The definitions give
\begin{align*}
	\int_{\cu_0}\nabla\phi\cdot\a\nabla(u-V)&=\int_{\cu_0}\phi\,d\mu-\int_{\cu_0}\phi\,d\nu=0&&\bigl(\phi\in C^\infty_c(\tfrac34\cu_0)\bigr)\,,\\
	\int_{\cu_0}\nabla\phi\cdot\a\nabla V&=\int_{\cu_0}\phi\,d\nu=0&&\bigl(\phi\in C^\infty_c(\cu_0\setminus\tfrac34\overline{\cu}_0)\bigr)\,.
\end{align*}
With the restrictions of Proposition~\ref{p.weighted.calculus}\ref{i.weighted.domains},~$u-V$ is therefore a nonnegative solution in~$\frac34\cu_0$, and~$V$ is a nonnegative solution in every cube contained in~$\cu_0\setminus\frac34\overline{\cu}_0$.

\smallskip

\emph{Step 3.} A bound for~$V$ away from~$\frac34\cu_0$. Let~$e_1$ be the first coordinate vector and set
\begin{equation*}
	y_0\coloneqq\tfrac{34}{81}e_1,\qquad Q_1\coloneqq y_0+\cu_{-3},\qquad Q_0\coloneqq\tfrac12Q_1=y_0+\tfrac12\cu_{-3}\,.
\end{equation*}
For~$x\in\overline Q_1$,
\begin{equation*}
	\tfrac38<\tfrac{65}{162}\leq x_1\leq\tfrac{71}{162}<\tfrac{15}{32},\qquad |x_j|\leq\tfrac3{162}<\tfrac{15}{32}\quad(j\geq2)\,.
\end{equation*}
Thus~$\overline Q_1\subset\frac{15}{16}\cu_0\setminus\frac34\overline{\cu}_0$. Since~$y_0\in3^{-4}\Zd$, the rescaling argument before~\eqref{e.interior.weak.harnack} applies to~$Q_1$ with~$n=4$, and~$V$ is a nonnegative solution in~$Q_1$ by Step~2. Corollary~\ref{c.lr.bound} on~$Q_1$ with~$\eta=\nf r4$,~$\rho_1=\nf12$ and~$\rho_2=1$, then~$V\leq u$ and~\eqref{e.interior.weak.harnack} give
\begin{equation}
	\esssup_{Q_0}V\leq\exp\bigl(C\sqrt\Theta\bigr)\norm{V}_{L^{\nf r4}(Q_1)}\leq\exp\bigl(C\sqrt\Theta\bigr)\norm{u}_{L^{\nf r4}(\frac{15}{16}\cu_0)}\leq\exp\bigl(C\sqrt\Theta\bigr)\essinf_{\frac12\cu_0}u\,.
	\label{e.remote.bound}
\end{equation}

\smallskip

\emph{Step 4.} The capacitary function. We call an element of~$\Hzero{\a}(\cu_0)$ admissible if it is at least one almost everywhere on~$Q_0$. The admissible functions form a convex set, which is closed because convergence in~$H^1_\a(\cu_0)$ implies convergence in~$L^1(\cu_0)$, and nonempty by Step~5. Since~$\Ener_{\cu_0}^{\nf12}$ is a Hilbert norm on~$\Hzero{\a}(\cu_0)$, there is a unique admissible~$\psi$ of least energy, the capacitary function of~$Q_0$. The truncation~$\psi_+\wedge1$ is admissible and, by~\eqref{e.weighted.truncation.gradients}, has no larger energy, so uniqueness gives~$\psi=\psi_+\wedge1$. Thus~$0\leq\psi\leq1$, and~$\psi=1$ almost everywhere on~$Q_0$.

For~$g\in\Hzero{\a}(\cu_0)$, the function~$\psi+\eps g$ is admissible for every~$\eps>0$ if~$g\geq0$, and for every~$\eps\in\R$ if~$g=0$ almost everywhere on~$Q_0$. Comparing energies gives
\begin{equation*}
	\int_{\cu_0}\nabla g\cdot\a\nabla\psi\geq0\quad\mbox{if }g\geq0\,,\qquad\int_{\cu_0}\nabla g\cdot\a\nabla\psi=0\quad\mbox{if }g=0\mbox{ a.e. on }Q_0\,.
\end{equation*}
In particular,~$\psi$ is a nonnegative supersolution in~$\cu_0$.

Now let~$z\in\Hzero{\a}(\cu_0)$ be nonnegative, with~$z\leq L$ almost everywhere on~$Q_0$ for a constant~$L\geq0$. Then~$(z-L\psi)_+\in\Hzero{\a}(\cu_0)$ vanishes almost everywhere on~$Q_0$, and
\begin{equation*}
	L\psi-z\wedge(L\psi)=(L\psi-z)_+\geq0\,,\qquad z=z\wedge(L\psi)+(z-L\psi)_+\,.
\end{equation*}
We obtain
\begin{equation}
	0\leq\int_{\cu_0}\nabla z\cdot\a\nabla\psi=\int_{\cu_0}\nabla\bigl(z\wedge(L\psi)\bigr)\cdot\a\nabla\psi\leq L\Ener_{\cu_0}(\psi)\,.
	\label{e.capacitary.test}
\end{equation}

\smallskip

\emph{Step 5.} The energy of~$\psi$. We bound~$\Ener_{\cu_0}(\psi)$ using a piecewise affine function~$f$ and its~$\a$-harmonic replacement on each simplex. By~\eqref{e.upper.dirichlet}, the energy of the replacement is expressed through the matrices~$\a(\triangle)$. Let~$f$ be the continuous function on~$\overline{\cu}_0$ that is affine on each simplex of~$\mathscr T_4$, equal to one at the vertices of the simplices whose closures meet~$\overline Q_0$, and zero at the other vertices. Then~$f=1$ on each simplex whose closure meets~$\overline Q_0$, and~$f=0$ on each simplex without a vertex of value one. Hence~$f=1$ on~$Q_0$, and~$f$ vanishes outside the points within maximum-norm distance~$2\cdot3^{-4}$ of~$Q_0$. These points satisfy
\begin{equation*}
	|x|_\infty<\tfrac{34}{81}+\tfrac1{108}+\tfrac2{81}<\tfrac12\,.
\end{equation*}
So~$f$ is Lipschitz with compact support in~$\cu_0$, and~$f\in\Hzero{\a}(\cu_0)$ by Proposition~\ref{p.weighted.calculus}\ref{i.weighted.products}. Its gradient is bounded by a constant depending only on~$d$.

For~$\triangle\in\mathscr T_4$,
\begin{equation*}
	H_\triangle(f|_\triangle)-f|_\triangle\in\Hzero{\a}(\triangle)\,.
\end{equation*}
By Proposition~\ref{p.weighted.calculus}\ref{i.weighted.domains}, this difference extends by zero to an element of~$\Hzero{\a}(\cu_0)$. Adding these finitely many extensions to~$f$ gives~$f_\a\in\Hzero{\a}(\cu_0)$ with~$f_\a=H_\triangle(f|_\triangle)$ on each~$\triangle$, by~\eqref{e.weighted.zero.extension}. On the simplices where~$f=1$ the representative is the constant one, by Proposition~\ref{p.harmonic.replacement}, so~$f_\a$ is admissible.

Since~$f$ is affine on each simplex,~\eqref{e.upper.dirichlet},~\eqref{e.partition.average}, Jensen's inequality and~\eqref{e.upper.half.series} give
\begin{equation}
	\Ener_{\cu_0}(\psi)\leq\Ener_{\cu_0}(f_\a)=\sum_{\triangle\in\mathscr T_4}|\triangle|\,\nabla f|_\triangle\cdot\a(\triangle)\nabla f|_\triangle\leq C\avsum_{\triangle\in\mathscr T_4}|\a(\triangle)|\leq C\Lambda_{s,1,p}(\cu_0)\,.
	\label{e.capacitary.energy}
\end{equation}

\smallskip

\emph{Step 6.} Conclusion. Since~$\psi$ is a nonnegative supersolution in~$\cu_0$ and~$\psi=1$ on~$Q_0\subseteq\frac{15}{16}\cu_0$, the bound~\eqref{e.interior.weak.harnack} gives
\begin{equation*}
	|Q_0|^{\nf4r}\leq\exp(C\sqrt\Theta)\essinf_{\frac{15}{16}\cu_0}\psi\,.
\end{equation*}
Therefore~$\psi\geq\exp(-C\sqrt\Theta)$ almost everywhere on~$\frac34\cu_0$. To bound~$\nu(\psi)$ from below, choose~$\phi_j\in C^\infty_c(\frac34\cu_0)$ with~$0\leq\phi_j\leq1$ increasing to one on~$\frac34\cu_0$. Then~$\phi_j\psi\in\Hzero{\a}(\cu_0)$ by Proposition~\ref{p.weighted.calculus}\ref{i.weighted.products}, so~$(1-\phi_j)\psi$ and~$\phi_j(\psi-\exp(-C\sqrt\Theta))$ are nonnegative elements of~$\Hzero{\a}(\cu_0)$, and~\eqref{e.measure.evaluation} gives
\begin{equation*}
	\nu(\psi)\geq\nu(\phi_j\psi)\geq\exp\bigl(-C\sqrt\Theta\bigr)\nu(\phi_j)=\exp\bigl(-C\sqrt\Theta\bigr)\int_{\cu_0}\phi_j\,d\mu\,.
\end{equation*}
By monotone convergence,~$\nu(\psi)\geq\exp(-C\sqrt\Theta)\mu(\frac34\cu_0)$. On the other hand, the equation of~$V$, the symmetry of~$\a$,~\eqref{e.capacitary.test} with~$z=V$ and~$L=\esssup_{Q_0}V$,~\eqref{e.remote.bound} and~\eqref{e.capacitary.energy} give
\begin{equation*}
	\nu(\psi)=\int_{\cu_0}\nabla\psi\cdot\a\nabla V=\int_{\cu_0}\nabla V\cdot\a\nabla\psi\leq\Bigl(\esssup_{Q_0}V\Bigr)\Ener_{\cu_0}(\psi)\leq C\Lambda_{s,1,p}(\cu_0)\exp\bigl(C\sqrt\Theta\bigr)\essinf_{\frac12\cu_0}u\,.
\end{equation*}
The lower and upper estimates for~$\nu(\psi)$ give~\eqref{e.source.mass}.
\end{proof}

% Adapted from Section 5 of weak-harnack-nonradial-endpoint.tex. Uses only Proposition p.endpoint.potential,
% Lemma l.source.mass and e.interior.harnack (from e.harnack, proved in the preceding subsection).
\begin{proof}[Completion of the proof of Theorem~\ref{t.harnack}]
Let~$u\geq0$ be a supersolution in~$\cu_0$. Set~$\mu=-\nabla\cdot\a\nabla u$, let~$\nu$ be its restriction to~$\frac34\cu_0$, and let~$V$ be the potential of~$\nu$. By Lemma~\ref{l.source.mass},~$0\leq V\leq u$,~$u-V$ is a solution in~$\frac34\cu_0$, and~$\nu$ satisfies the hypothesis of Proposition~\ref{p.endpoint.potential}. Also,~$\nu(\cu_0)=\mu(\frac34\cu_0)$. By Proposition~\ref{p.endpoint.potential} and Lemma~\ref{l.source.mass}, since~$\Theta=\Lambda_{s,1,p}(\cu_0)/\lambda_{t,1,q}(\cu_0)$,
\begin{equation*}
	\norm{V}_{L^{\nf{r^*}2}(\cu_0)}\leq C\lambda_{t,1,q}(\cu_0)^{-1}\mu\bigl(\tfrac34\cu_0\bigr)\leq C\Theta\exp\bigl(C\sqrt\Theta\bigr)\essinf_{\frac12\cu_0}u\,.
\end{equation*}
The function~$u-V$ is a nonnegative solution in~$\frac34\cu_0$, so~\eqref{e.interior.harnack} and~$V\geq0$ give
\begin{equation*}
	\esssup_{\frac58\cu_0}(u-V)\leq\exp\bigl(C\sqrt\Theta\bigr)\essinf_{\frac12\cu_0}(u-V)\leq\exp\bigl(C\sqrt\Theta\bigr)\essinf_{\frac12\cu_0}u\,.
\end{equation*}
On~$\frac58\cu_0$ we write~$u=V+(u-V)$. If~$\nf{r^*}2\geq1$, we use the triangle inequality in~$L^{\nf{r^*}2}(\frac58\cu_0)$. If~$\nf{r^*}2<1$, we integrate~$(a+b)^{\nf{r^*}2}\leq a^{\nf{r^*}2}+b^{\nf{r^*}2}$, valid for~$a,b\geq0$, and then use the convexity of~$y\mapsto y^{\nf2{r^*}}$. In both cases, since~$|\frac58\cu_0|$ depends only on~$d$,
\begin{equation*}
	\biggl(\fint_{\frac58\cu_0}u^{\nf{r^*}2}\biggr)^{\nf2{r^*}}\leq C\norm{V}_{L^{\nf{r^*}2}(\cu_0)}+C\esssup_{\frac58\cu_0}(u-V)\,.
\end{equation*}
Inserting the bounds for~$\norm{V}_{L^{\nf{r^*}2}(\cu_0)}$ and~$\esssup_{\frac58\cu_0}(u-V)$ and absorbing~$C\Theta$ into the exponential, since~$\Theta\geq1$, proves~\eqref{e.weak.harnack} for~$\eta=\nf{r^*}2$, with~$C$ depending only on~$d,p,q,s,t$. For~$0<\eta<\nf{r^*}2$, H\"older's inequality gives
\begin{equation*}
	\biggl(\fint_{\frac58\cu_0}u^\eta\biggr)^{\nf1\eta}\leq\biggl(\fint_{\frac58\cu_0}u^{\nf{r^*}2}\biggr)^{\nf2{r^*}}\,,
\end{equation*}
which proves~\eqref{e.weak.harnack} for every~$0<\eta\leq\nf{r^*}2$.
\end{proof}

\section{Coefficient moments and negative regularity}
\label{s.classical.moments}

The simplex bounds~\eqref{e.upper.norm.bound} and~\eqref{e.lower.norm.bound} give parts~\ref{c.classical.moments} and~\ref{c.negative.regularity} of Theorem~\ref{t.coefficient.conditions}. We combine part~\ref{c.negative.regularity} with Lemma~\ref{l.negative.sobolev} to prove part~\ref{c.negative.sobolev}, then apply the regularity theorems to obtain Corollary~\ref{c.sobolev.coefficients}. We finish by comparing the classical moment condition with the hypotheses of Bella and Sch\"affner.

\begin{proof}[Proof of Theorem~\ref{t.coefficient.conditions}\ref{c.classical.moments}]
By~\eqref{e.upper.norm.bound} and~\eqref{e.lower.norm.bound}, and since the norm of an average is at most the average of the norm, Jensen's inequality gives for every simplex
\begin{equation*}
	|\a(\triangle)|^p\leq|(\a)_\triangle|^p\leq\fint_\triangle|\a|^p\qquad\mbox{and}\qquad |\a_*^{-1}(\triangle)|^q\leq|(\a^{-1})_\triangle|^q\leq\fint_\triangle|\a^{-1}|^q\,.
\end{equation*}
By~\eqref{e.partition.average}, the means over~$\mathscr T_k$ are therefore at most~$\int_{\cu_0}|\a|^p$ and~$\int_{\cu_0}|\a^{-1}|^q$; for~$p=\infty$ the maximum over~$\mathscr T_k$ of~$|\a(\triangle)|\leq|(\a)_\triangle|$ is at most~$\norm{\a}_{L^\infty(\cu_0)}$, and likewise for~$q=\infty$. Since~$(1-3^{-s})\sum_k3^{-ks}=1$, the definitions~\eqref{e.upper.half.series}--\eqref{e.lower.half.series} give~\eqref{e.classical.series}, and~\eqref{e.spatial.moment.ratio} gives the bound on~$\Theta$.
\end{proof}

For a constant matrix the comparison is an equality, by the computation after~\eqref{e.lower.half.series}, so the constant one in~\eqref{e.classical.series} is attained.

\smallskip

Part~\ref{c.negative.regularity} follows from the same comparison on each simplex, combined with a comparison of the averages over simplices and over cubes.

\begin{proof}[Proof of Theorem~\ref{t.coefficient.conditions}\ref{c.negative.regularity}]
When~$p$ or~$q$ is infinite, the means over~$\mathscr T_k$ and over the cubes below are replaced by maxima, and the argument is the same. By~\eqref{e.upper.half.series} and~\eqref{e.upper.norm.bound},
\begin{equation*}
	\Lambda_{s,1,p}(\cu_0)\leq(1-3^{-s})^2\biggl(\sum_{k\geq0}3^{-ks}\biggl(\avsum_{\triangle\in\mathscr T_k}|(\a)_\triangle|^p\biggr)^{\nf1{2p}}\biggr)^2\,.
\end{equation*}
Let~$\triangle\in\mathscr T_k$ and let~$Q=z+\cu_{-k}$, with~$z\in3^{-k}\Zd\cap\cu_0$, be the triadic cube that contains it, so that~$|Q|=d!\,|\triangle|$. Since~$\a$ is positive definite almost everywhere,~$(\a)_\triangle\leq|\triangle|^{-1}\int_Q\a=d!\,(\a)_Q$ as quadratic forms, and hence~$|(\a)_\triangle|\leq d!\,|(\a)_Q|$. Each such cube contains~$d!$ simplices of~$\mathscr T_k$ of equal volume, so the~$L^p$ mean over~$\mathscr T_k$ of~$|(\a)_\triangle|$ is at most~$d!$ times the~$L^p$ mean of~$|(\a)_{z+\cu_{-k}}|$ over~$z\in3^{-k}\Zd\cap\cu_0$, and the right side above is at most~$d!\,(1-3^{-s})^2\norm{\a}_{\mathring{\underline B}^{-2s}_{p,\nf12}(\cu_0)}$. The same argument with~\eqref{e.lower.half.series},~\eqref{e.lower.norm.bound},~$\a^{-1}$,~$t$ and~$q$ bounds~$\lambda_{t,1,q}(\cu_0)^{-1}$. Since~$(1-3^{-s})(1-3^{-t})\leq1$,~\eqref{e.spatial.moment.ratio} and~\eqref{e.negative.regularity.series} give the bound on~$\Theta$.
\end{proof}

\begin{proof}[Proof of Theorem~\ref{t.coefficient.conditions}\ref{c.negative.sobolev}]
Let~$C$ denote constants depending only on~$d$,~$p$,~$q$,~$s$,~$t$ and~$\eps$. Since~$\a$ is positive definite almost everywhere and~$0<\eps\leq2s$, Lemma~\ref{l.negative.sobolev} gives
\begin{equation*}
	\norm{\a}_{\mathring{\underline B}^{-2s}_{p,\nf12}(\cu_0)}\leq C\norm{\a}_{W^{-2s+\eps,p}(\cu_0)}\,,
\end{equation*}
and the lemma applied to~$\a^{-1}$, with~$t$ and~$q$ in place of~$s$ and~$p$, bounds the quasi-norm~\eqref{e.negative.regularity.norm} of~$\a^{-1}$ of order~$-2t$ and integrability~$q$ in the same way. Both quasi-norms are therefore finite, and part~\ref{c.negative.regularity}, with~$(1-3^{-s})^2\leq1$ and~$(1-3^{-t})^2\leq1$, gives~\eqref{e.negative.sobolev.series} with the constant~$d!\,C$, which depends only on~$d,p,q,s,t,\eps$. The bound on~$\Theta$ follows from~\eqref{e.spatial.moment.ratio}.
\end{proof}

\begin{proof}[Proof of Corollary~\ref{c.sobolev.coefficients}]
Let~$C$ and~$\gamma$ denote constants depending only on~$d,p,q,\alpha,\beta$. Apply Theorem~\ref{t.coefficient.conditions}\ref{c.negative.sobolev} with
\begin{equation*}
	s\coloneqq\frac\alpha2+\frac\theta4\,,\qquad t\coloneqq\frac\beta2+\frac\theta4\,,\qquad\eps\coloneqq\frac\theta2\,.
\end{equation*}
These choices satisfy~$0<\eps\leq2\min\{s,t\}$, including when~$\alpha=0$ or~$\beta=0$, and~$-2s+\eps=-\alpha$ and~$-2t+\eps=-\beta$. Hence
\begin{equation*}
	\Lambda_{s,1,p}(\cu_0)\leq C\norm{\a}_{W^{-\alpha,p}(\cu_0)}\qand\lambda_{t,1,q}(\cu_0)^{-1}\leq C\norm{\a^{-1}}_{W^{-\beta,q}(\cu_0)}\,.
\end{equation*}
For these indices, the parameter in~\eqref{e.negative.regularity.range} equals~$\nf\theta2>0$. The preceding bounds give
\begin{equation*}
	\frac{\Lambda_{s,1,p}(\cu_0)}{\lambda_{t,1,q}(\cu_0)}\leq C\norm{\a}_{W^{-\alpha,p}(\cu_0)}\norm{\a^{-1}}_{W^{-\beta,q}(\cu_0)}\,.
\end{equation*}
Theorems~\ref{t.cg.local.boundedness} and~\ref{t.harnack} therefore apply with this ratio in place of~$\Theta$. The power~$\nf{(d-1)}{(4\theta)}$ in~\eqref{e.cg.local.boundedness}, with~$\nf\theta2$ in place of~$\theta$, becomes~$\nf{(d-1)}{(2\theta)}$ and gives~\eqref{e.sobolev.local.boundedness}. The Harnack estimate~\eqref{e.harnack} gives~\eqref{e.sobolev.harnack}.
\end{proof}

\paragraph{Comparison with Bella and Sch\"affner.} Section~\ref{ss.intro.discussion} compares the two results in general terms; here we record the precise hypotheses. Bella and Sch\"affner~\cite[Theorems~1--2]{Bella.Schaffner.regularity} prove local boundedness of weak solutions, the Harnack inequality for nonnegative solutions, and local upper boundedness of weak subsolutions, for possibly nonsymmetric coefficients with~$d\geq2$ and~$p,q\in(1,\infty]$. Besides~\eqref{e.classical.moment.range}, their coefficient hypotheses require
\begin{equation*}
	\sup_{e\neq0}\frac{|\a e|^2}{e\cdot\a e}\in L^p(\cu_0)\qquad\mbox{and}\qquad\left(\inf_{e\neq0}\frac{e\cdot\a e}{|e|^2}\right)^{-1}\in L^q(\cu_0)\,.
\end{equation*}
For symmetric~$\a$, diagonalization identifies these functions with~$|\a|$ and~$|\a^{-1}|$. Thus their coefficient hypotheses agree with those in Theorem~\ref{t.coefficient.conditions}\ref{c.classical.moments}, which also allows infinite exponents, and in our symmetric setting with~$d\geq3$ this part reaches their whole moment range. Their solution spaces are completions with the squared norm~$\int_U|\a|u^2+\int_U\nabla u\cdot\a\nabla u$. Ours uses the squared mean and weighted energy (Section~\ref{s.weighted.spaces}). The comparison here concerns the coefficient hypotheses and exponent range.

\section{Sharpness}
\label{s.sharpness}

We prove Theorem~\ref{t.sharpness} and the optimality of the powers of~$\Theta$ in Theorem~\ref{t.cg.local.boundedness} and Corollary~\ref{c.lr.bound}. The first two examples use thin cylinders parallel to the~$x_1$-axis; their volumes are proportional to the~$(d-1)$-st power of their radii. We first bound their volume fractions in simplices and cubes. For Theorem~\ref{t.sharpness}, disjoint cylinders carry nonnegative subsolutions with summable norms, and harmonic replacement gives a positive solution above their sum. For the polynomial bound, we construct~$\a_\eps$ and~$u_\eps$ using one thin cylinder. The positive part of~$u_\eps$ reaches values of order one but is nonzero only in a region of volume of order~$\eps^{d-1}$. A third, radial example, with fundamental solutions made constant near the pole, shows that the range of exponents in the weak Harnack inequality~\eqref{e.weak.harnack} cannot be enlarged.

\smallskip
Throughout this section we write~$x=(x_1,y)$ with~$y\in\R^{d-1}$, and regard functions of~$y$ as functions on~$\cu_0$. Two positive quantities are comparable, written with~$\asymp$, if each is at most~$C$ times the other. Constants~$c,C$ may change from line to line and depend only on~$d$ and the exponents fixed in each statement; in the proof of Theorem~\ref{t.sharpness}, the power~$\kappa$ is determined by those exponents. The constants never depend on the scale~$k$, the cylinder index~$j$, the radii and centers, or the conductivities.

\subsection{Averages of cylinders over simplices}
\label{ss.sharpness.cylinders}

The following upper bounds for cylinder averages will be used for both constructions.

\begin{lemma}[Averages of cylinders over simplices]
\label{l.sharpness.cylinder.averages}
Let~$0<\eps<1$ and~$y_0\in\R^{d-1}$, and define
\begin{equation*}
	E_{\eps,y_0}\coloneqq\{(x_1,y)\in\cu_0:|y-y_0|<\eps\}\,,\qquad f_\triangle\coloneqq\frac{|\triangle\cap E_{\eps,y_0}|}{|\triangle|}\,.
\end{equation*}
For~$1\leq p<\infty$ and~$k\geq0$,
\begin{equation}
	\biggl(\avsum_{\triangle\in\mathscr T_k}f_\triangle^p\biggr)^{\nf1p}\leq C\eps^{d-1}\max\{\eps,3^{-k}\}^{-(d-1)(1-\nf1p)}\,,
	\label{e.sharpness.cylinder.scales}
\end{equation}
and, for~$s>0$ and~$0<\nu<d-1$ with~$\nu\leq\nf{(d-1)}p+2s$,
\begin{equation}
	\sum_{k\geq0}3^{-ks}\biggl(\avsum_{\triangle\in\mathscr T_k}f_\triangle^p\biggr)^{\nf1{2p}}\leq C\eps^{\nf\nu2}\,.
	\label{e.sharpness.cylinder.discount}
\end{equation}
The constants depend only on~$d$ and~$p$, and in~\eqref{e.sharpness.cylinder.discount} also on~$s$ and~$\nu$. The same bounds hold for the fractions of the triadic subcubes of~$\cu_0$ of side length~$3^{-k}$ occupied by~$E_{\eps,y_0}$, with the arithmetic mean over these subcubes.
\end{lemma}

\begin{proof}
Since~$E_{\eps,y_0}$ lies in the product of~$(-\nf12,\nf12)$ and the ball of radius~$\eps$ about~$y_0$, one has~$|E_{\eps,y_0}|\leq C\eps^{d-1}$. Each simplex has volume~$\nf{3^{-kd}}{d!}$ and axial extent at most~$3^{-k}$, and by~\eqref{e.partition.average} the arithmetic mean over~$\mathscr T_k$ is the volume-weighted sum. Thus
\begin{equation*}
	\avsum_{\triangle\in\mathscr T_k}f_\triangle=|E_{\eps,y_0}|\leq C\eps^{d-1}\,,\qquad f_\triangle\leq\min\{1,C(3^k\eps)^{d-1}\}\leq C\eps^{d-1}\max\{\eps,3^{-k}\}^{-(d-1)}\,.
\end{equation*}
Hence~$\avsum_{\triangle\in\mathscr T_k}f_\triangle^p\leq(\max_{\triangle\in\mathscr T_k}f_\triangle)^{p-1}\avsum_{\triangle\in\mathscr T_k}f_\triangle$, and taking the~$p$-th root proves~\eqref{e.sharpness.cylinder.scales}. This argument also applies to cubes, whose volumes are~$3^{-kd}$.

\smallskip
For~$3^{-k}\geq\eps$, the summand in~\eqref{e.sharpness.cylinder.discount} is at most~$C\eps^{\nf{(d-1)}2}3^{\nf{k(d-1-\nu)}2}$, since~$s\geq\frac12(\nu-\nf{(d-1)}p)$. The geometric sum over these scales is at most~$C\eps^{\nf\nu2}$ because~$d-1-\nu>0$. For~$3^{-k}<\eps$, the sum is at most~$C\eps^{\nf{(d-1)}{2p}}\eps^s\leq C\eps^{\nf\nu2}$, since~$\eps<1$ and~$s+\nf{(d-1)}{(2p)}\geq\nf\nu2$. Adding these bounds proves~\eqref{e.sharpness.cylinder.discount}.
\end{proof}

\subsection{The coefficient range}
\label{ss.sharpness.range}

\begin{proof}[Proof of Theorem~\ref{t.sharpness}]
We construct a nonnegative subsolution in one cylinder, estimate its weighted norm, and bound the averages of the coefficient field and of its inverse cylinder by cylinder. Its axial growth is independent of the radius, which can therefore be chosen to make a sequence of these subsolutions summable. Harmonic replacement gives a positive solution above their sum.

\smallskip
\emph{Step 1.} Powers, profiles and the field. Since~$2(1-\theta_0)=\nf{(d-1)}p+\alpha+\nf{(d-1)}q+\beta$, set
\begin{equation*}
	\kappa\coloneqq\frac{\nf{(d-1)}p+\alpha}{1-\theta_0}\,,\qquad 2-\kappa=\frac{\nf{(d-1)}q+\beta}{1-\theta_0}\,.
\end{equation*}
Both numerators are positive and~$1-\theta_0\geq1$. Thus~$0<\kappa<2$ and
\begin{equation}
	\kappa\leq\frac{d-1}p+\alpha\,,\qquad 2-\kappa\leq\frac{d-1}q+\beta\,.
	\label{e.sharpness.scalar.powers}
\end{equation}

\smallskip
For~$j\geq1$, put~$b_j\coloneqq2^{-j-2}$ and~$y_j\coloneqq(b_j,0,\ldots,0)\in\R^{d-1}$. We first allow any~$0<\varrho_j<\nf{b_j}{(16d)}$ and set~$a_j^+\coloneqq b_j\varrho_j^{-\kappa}$ and~$a_j^-\coloneqq b_j^{-1}\varrho_j^{2-\kappa}$. Define
\begin{equation}
	\psi_j(\varrho)\coloneqq\begin{cases}1+\dfrac{d-1}{2b_j^2}(\varrho_j^2-\varrho^2)&\mbox{if }0\leq \varrho\leq\varrho_j\,,\\(2-\varrho/\varrho_j)^{d-1}&\mbox{if }\varrho_j<\varrho<2\varrho_j\,,\\0&\mbox{if }\varrho\geq2\varrho_j\,,\end{cases}
	\label{e.sharpness.scalar.profile}
\end{equation}
and
\begin{equation}
	v_j(x_1,y)\coloneqq j\exp\biggl(\frac{d-1}{b_j}(x_1+\nf12)\biggr)\psi_j(|y-y_j|)\,.
	\label{e.sharpness.scalar.subsolution}
\end{equation}
The profile~\eqref{e.sharpness.scalar.profile} is continuous, satisfies~$0\leq\psi_j\leq2$, and is at least one on~$[0,\varrho_j]$. In expressions involving transverse derivatives or integrals, we also write~$\psi_j$ for the function~$y\mapsto\psi_j(|y-y_j|)$.

\smallskip
For~$i>j$, the centers have distance at least~$\nf{b_j}2$, whereas~$2\varrho_j+2\varrho_i<\nf{3b_j}{(16d)}<\nf{b_j}2$, so the closed transverse balls are disjoint. They lie in the transverse cube and accumulate only at zero. Define
\begin{equation}
	a(y)\coloneqq\begin{cases}a_j^+&\mbox{if }|y-y_j|<\varrho_j\mbox{ for some }j\,,\\a_j^-&\mbox{if }\varrho_j<|y-y_j|<2\varrho_j\mbox{ for some }j\,,\\1&\mbox{otherwise}\,,\end{cases}\qquad\a(x_1,y)\coloneqq a(y)I\,.
	\label{e.sharpness.scalar.field}
\end{equation}

\smallskip
\emph{Step 2.} The subsolution and its energy. We call~$E_{\varrho_j,y_j}$ the inner cylinder and~$\{x\in\cu_0:\varrho_j<|y-y_j|<2\varrho_j\}$ the annulus; by~\eqref{e.sharpness.scalar.field},~$a=a_j^+$ on the first and~$a=a_j^-$ on the second. In the inner cylinder,~$\Delta_y\psi_j=-\nf{(d-1)^2}{b_j^2}$, so the function~$v_j$ in~\eqref{e.sharpness.scalar.subsolution} satisfies
\begin{equation*}
	\nabla\cdot(\a\nabla v_j)=a_j^+j\exp\biggl(\frac{d-1}{b_j}(x_1+\nf12)\biggr)\frac{(d-1)^2}{b_j^2}(\psi_j-1)\geq0\,.
\end{equation*}
In the annulus, the transverse Laplacian at radius~$\varrho$ is
\begin{equation*}
	\psi_j''(\varrho)+\frac{d-2}{\varrho}\psi_j'(\varrho)=\frac{2(d-1)(d-2)}{\varrho_j^2}(2-\varrho/\varrho_j)^{d-3}\Bigl(1-\frac{\varrho_j}{\varrho}\Bigr)\geq0\,,
\end{equation*}
and the axial second derivative is nonnegative. At~$\varrho=\varrho_j$, the radial fluxes agree because
\begin{equation*}
	a_j^+\psi_j'(\varrho_j-)=-\frac{(d-1)a_j^+\varrho_j}{b_j^2}=-\frac{(d-1)a_j^-}{\varrho_j}=a_j^-\psi_j'(\varrho_j+)\,.
\end{equation*}
At~$\varrho=2\varrho_j$, both radial derivatives are zero because~$d\geq3$. The radial function is smooth at its center. Let~$0\leq\varphi\in C^\infty_c(\cu_0)$; it vanishes near~$\partial\cu_0$, which contains the axial ends of the cylinder. Integrating by parts in the inner cylinder and in the annulus, the volume terms~$-\int\varphi\nabla\cdot(\a\nabla v_j)$ are nonpositive by the two inequalities above, the terms on~$\{|y-y_j|=\varrho_j\}$ cancel because the radial fluxes agree, and those on~$\{|y-y_j|=2\varrho_j\}$ vanish. Hence~$\int_{\cu_0}\nabla\varphi\cdot\a\nabla v_j\leq0$.

\smallskip
On the inner ball~$|\psi_j'|\leq C\nf{\varrho_j}{b_j^2}$, and on the annulus~$|\psi_j'|\leq\nf C{\varrho_j}$. Both regions have volume at most~$C\varrho_j^{d-1}$. The annular contribution to the first integral below and the inner contribution to the second are smaller than the respective bounds by the factor~$\nf{\varrho_j^2}{b_j^2}<1$. Thus
\begin{equation*}
	\int a\psi_j^2\leq Cb_j\varrho_j^{d-1-\kappa}\,,\qquad\int a|\nabla_y\psi_j|^2\leq Cb_j^{-1}\varrho_j^{d-1-\kappa}\,.
\end{equation*}
Here the integrals are over~$(-\nf12,\nf12)^{d-1}$. The axial factor and its derivative are bounded by~$j\exp(\nf{(d-1)}{b_j})$ and~$j(d-1)b_j^{-1}\exp(\nf{(d-1)}{b_j})$. Consequently, since~$(\nf{(d-1)}{b_j})^2b_j=(d-1)^2b_j^{-1}$,
\begin{equation}
	|(v_j)_{\cu_0}|^2+\Ener_{\cu_0}(v_j)\leq Cj^2\exp\bigl(\nf{2(d-1)}{b_j}\bigr)b_j^{-1}\varrho_j^{d-1-\kappa}\,.
	\label{e.sharpness.scalar.summability}
\end{equation}
For the mean, use~$|(v_j)_{\cu_0}|\leq Cj\exp(\nf{(d-1)}{b_j})\varrho_j^{d-1}$,~$b_j<1$ and~$2(d-1)>d-1-\kappa$. Constants in these estimates depend only on~$d$ and are independent of~$j$ and~$\varrho_j$.

\smallskip
\emph{Step 3.} Radii and coefficient bounds. Now fix each~$\varrho_j$ small enough that
\begin{equation}
	\varrho_j<\frac{b_j}{16d}\,,\qquad a_j^+\geq1\geq a_j^-\,,\qquad j^2\exp\bigl(\nf{2(d-1)}{b_j}\bigr)b_j^{-1}\varrho_j^{d-1-\kappa}\leq2^{-j}\,.
	\label{e.sharpness.scalar.radii}
\end{equation}
This is possible because~$0<\kappa<2$ and~$d-1-\kappa>0$. The resulting field is measurable and positive, and
\begin{equation}
	a\leq1+\sum_{j\geq1}b_j\varrho_j^{-\kappa}\indc_{E_{\varrho_j,y_j}}\,,\qquad a^{-1}\leq1+\sum_{j\geq1}b_j\varrho_j^{-(2-\kappa)}\indc_{E_{2\varrho_j,y_j}}\,.
	\label{e.sharpness.scalar.majorants}
\end{equation}
Integrating the majorants~\eqref{e.sharpness.scalar.majorants} gives powers~$\varrho_j^{d-1-\kappa}$ and~$\varrho_j^{d-3+\kappa}$, whose exponents are positive because~$0<\kappa,2-\kappa<2\leq d-1$. Since~$\varrho_j<1$ and~$\sum_jb_j<\infty$, one has~$a,a^{-1}\in L^1(\cu_0)$ and hence~$\tr\a,\tr\a^{-1}\in L^1(\cu_0)$.

\smallskip
For cube averages,~$|(\a)_{z+\cu_{-k}}|=(a)_{z+\cu_{-k}}$ and~$|(\a^{-1})_{z+\cu_{-k}}|=(a^{-1})_{z+\cu_{-k}}$. Lemma~\ref{l.sharpness.cylinder.averages} applies because~$2\varrho_j<1$,~$0<\kappa,2-\kappa<d-1$, and~\eqref{e.sharpness.scalar.powers} holds. For~$\alpha>0$, apply~\eqref{e.sharpness.cylinder.discount} for cubes with exponent~$p$, scale discount~$s=\nf\alpha2$ and~$\nu=\kappa$. For a finite partial sum of the first majorant, Minkowski's inequality for the spatial~$p$-mean, square-root subadditivity and~\eqref{e.sharpness.cylinder.discount} bound the contribution of cylinder~$j$ by~$C(b_j\varrho_j^{-\kappa})^{\nf12}\varrho_j^{\nf\kappa2}=Cb_j^{\nf12}$. The partial majorants increase; monotone convergence passes their cube averages and nonnegative scale sums to the limit, and the bound passes to~$a$, which is at most the full majorant. By~\eqref{e.negative.regularity.norm},
\begin{equation}
	\norm{\a}_{\mathring{\underline B}^{-\alpha}_{p,\nf12}(\cu_0)}^{\nf12}\leq\frac1{1-3^{-\nf\alpha2}}+C\sum_{j\geq1}b_j^{\nf12}<\infty\,.
	\label{e.sharpness.scalar.upper.norm}
\end{equation}
For~$\beta>0$, apply~\eqref{e.sharpness.cylinder.discount} to the second majorant for cubes with exponent~$q$, scale discount~$s=\nf\beta2$,~$\nu=2-\kappa$ and radius~$2\varrho_j$. The contribution is at most~$C(b_j\varrho_j^{-(2-\kappa)})^{\nf12}(2\varrho_j)^{\nf{(2-\kappa)}2}\leq Cb_j^{\nf12}$. Applying spatial Minkowski, square-root subadditivity and monotone convergence as above gives
\begin{equation}
	\norm{\a^{-1}}_{\mathring{\underline B}^{-\beta}_{q,\nf12}(\cu_0)}^{\nf12}\leq\frac1{1-3^{-\nf\beta2}}+C\sum_{j\geq1}b_j^{\nf12}<\infty\,.
	\label{e.sharpness.scalar.lower.norm}
\end{equation}
The constants here may also depend on the fixed exponents of Theorem~\ref{t.sharpness}.

\smallskip
If~$\alpha=0$, Minkowski's inequality and~\eqref{e.sharpness.scalar.powers} give~$\norm{\a}_{L^p(\cu_0)}\leq1+C\sum_jb_j\varrho_j^{\nf{(d-1)}p-\kappa}\leq1+C\sum_jb_j<\infty$. If~$\beta=0$, they give~$\norm{\a^{-1}}_{L^q(\cu_0)}\leq1+C\sum_jb_j\varrho_j^{\nf{(d-1)}q-(2-\kappa)}\leq1+C\sum_jb_j<\infty$. These arguments apply independently, including when both orders vanish; equality in either inequality of~\eqref{e.sharpness.scalar.powers} is allowed. Together with~\eqref{e.sharpness.scalar.upper.norm} and~\eqref{e.sharpness.scalar.lower.norm}, this proves the assertions of Theorem~\ref{t.sharpness} on the coefficient field.

\smallskip
\emph{Step 4.} Convergence and harmonic replacement. Each~$v_j$ is Lipschitz on~$\overline{\cu}_0$. Mollifications of a Lipschitz extension converge uniformly, with uniformly bounded gradients converging almost everywhere. The energy densities of the approximation errors are bounded by four times the squared Lipschitz constant of the extension times~$a\in L^1(\cu_0)$. Dominated convergence and convergence of the means give convergence in the norm~\eqref{e.weighted.norm}. By Proposition~\ref{p.weighted.calculus}\ref{i.weighted.identification} the limit is~$v_j$, so~$v_j\in H^1_\a(\cu_0)$, and Step~2 makes it a subsolution. By~\eqref{e.sharpness.scalar.summability} and~\eqref{e.sharpness.scalar.radii},~$\sum_j\norm{v_j}_{H^1_\a(\cu_0)}\leq C\sum_j2^{-\nf j2}<\infty$. The series converges in~$H^1_\a(\cu_0)$ and, by Proposition~\ref{p.weighted.calculus}\ref{i.weighted.identification}, in~$L^1(\cu_0)$. Monotone convergence identifies the limit with~$v\coloneqq\sum_jv_j$. Continuity of the pairing~\eqref{e.weighted.pairing} passes the subsolution inequality to~$v$.

\smallskip
Let~$H_{\cu_0}v$ be given by Proposition~\ref{p.harmonic.replacement}, and put~$w\coloneqq H_{\cu_0}v-v\in\Hzero{\a}(\cu_0)$. Its negative part~$(-w)_+$ belongs to~$\Hzero{\a}(\cu_0)$ by Proposition~\ref{p.weighted.calculus}\ref{i.weighted.products}. By~\eqref{e.weighted.truncation.gradients},~$\nabla(-w)_+=-\indc_{\{w<0\}}\nabla w$, so the orthogonality~\eqref{e.harmonic.orthogonality} and the testing inequality~\eqref{e.weighted.testing} give
\begin{equation*}
	\Ener_{\cu_0}((-w)_+)=-\int_{\cu_0}\nabla(-w)_+\cdot\a\nabla w=\int_{\cu_0}\nabla(-w)_+\cdot\a\nabla v\leq0\,.
\end{equation*}
By~\eqref{e.weighted.coercivity},~$(-w)_+=0$ and~$w\geq0$. Hence~$u\coloneqq1+H_{\cu_0}v\in\Csol(\cu_0)$ satisfies~$u\geq1+v\geq1$ almost everywhere.

\smallskip
On~$E_{\varrho_j,y_j}$, one has~$v_j\geq j$. Every neighborhood of a point~$(x_1,0)$ with~$|x_1|<\nf12$ contains the product of a nonempty open subinterval of~$(-\nf12,\nf12)$ and a transverse ball about zero. Since~$|y_j|+\varrho_j\to0$, this product contains a positive-measure subset of~$E_{\varrho_j,y_j}$ for every sufficiently large~$j$. On it~$u\geq1+j$, proving the asserted unboundedness near every point of the segment and in~$\frac12\cu_0$. Finally,~$u\geq1$ and~$u\in L^1(\cu_0)$ by Proposition~\ref{p.weighted.calculus}\ref{i.weighted.identification}, so~$0<\essinf_{\frac12\cu_0}u<\infty$.
\end{proof}

\smallskip
Bella and Sch\"affner~\cite[Section~3]{Bella.Schaffner.plaplace} also alternate large and small scalar conductivities about a line, on coaxial annuli, and obtain unbounded subsolutions that are an exponential in~$x_1$ times a function of~$|y|$. For the linear equation their Theorem~2 includes the endpoint~$\nf1p+\nf1q=\nf2{(d-1)}$ only when~$d\geq4$, and they note that their construction fails at this endpoint when~$d=3$. Here the cylinders are disjoint, and harmonic replacement turns the sum of the subsolutions into a positive solution.

\subsection{The polynomial bound}
\label{ss.sharpness.polynomial}

A single anisotropic cylinder shows that the powers~$\nf{(d-1)}{(4\theta)}$ in Theorem~\ref{t.cg.local.boundedness} and~$\nf{(d-1)}{(2\eta\theta)}$ in Corollary~\ref{c.lr.bound} cannot be lowered. The solution changes sign, and its positive part has height of order one and transverse radius of order~$\eps$.

\begin{proposition}[Sharpness of the polynomial bound]
\label{p.sharpness.polynomial}
Let~$d\geq3$, and let~$1<p,q<\infty$ and~$s,t>0$ satisfy~$\theta>0$, with~$\theta$ as in~\eqref{e.intro.parameters}. There are symmetric coefficient fields~$\a_\eps$ on~$\cu_0$, for~$0<\eps<\nf18$, each uniformly elliptic with ellipticity constants depending on~$\eps$, and solutions~$u_\eps\in\Csol(\cu_0)$ for~$\a=\a_\eps$, such that
\begin{equation}
	\sup_{0<\eps<\nf18}\int_{\cu_0}\bigl(\tr\a_\eps+\tr\a_\eps^{-1}\bigr)<\infty\,.
	\label{e.sharpness.polynomial.traces}
\end{equation}
For these fields, write~$\Lambda_\eps\coloneqq\Lambda_{s,1,p}(\cu_0)$,~$\lambda_\eps\coloneqq\lambda_{t,1,q}(\cu_0)$, and~$\Theta_\eps\coloneqq\nf{\Lambda_\eps}{\lambda_\eps}$. Then
\begin{equation}
	\Lambda_\eps\asymp\eps^{-2\theta}\,,\qquad\lambda_\eps\asymp1\,,\qquad\Theta_\eps\asymp\eps^{-2\theta}\,.
	\label{e.sharpness.polynomial.constants}
\end{equation}
The positive parts~$(u_\eps)_+$ belong to~$\Csub(\cu_0)$, and for every~$\eta>0$ and~$\nf12\leq\rho_1<\rho_2\leq1$,
\begin{equation}
	\norm{(u_\eps)_+}_{L^\infty(\rho_1\cu_0)}=1+\frac{\rho_1^2}4\qquad\mbox{and}\qquad\norm{(u_\eps)_+}_{L^\eta(\rho_2\cu_0)}\asymp\eps^{\nf{(d-1)}\eta}\,.
	\label{e.sharpness.polynomial.norms}
\end{equation}
The constants in~\eqref{e.sharpness.polynomial.constants} depend only on~$d,p,q,s,t$, and those in~\eqref{e.sharpness.polynomial.norms} also on~$\eta$; none depends on~$\eps$,~$\rho_1$ or~$\rho_2$. Consequently, for fixed~$\nf12\leq\rho_1<\rho_2\leq1$ and~$0<\eta\leq2$, and every real~$\upsilon<\nf{(d-1)}{(2\eta\theta)}$,
\begin{equation}
	\frac{\norm{(u_\eps)_+}_{L^\infty(\rho_1\cu_0)}}{\Theta_\eps^\upsilon\norm{(u_\eps)_+}_{L^\eta(\rho_2\cu_0)}}\longrightarrow\infty\qquad\mbox{as }\eps\downarrow0\,.
	\label{e.sharpness.polynomial.failure}
\end{equation}
\end{proposition}

\smallskip
By~\eqref{e.sharpness.polynomial.constants}, each~$\a_\eps$ satisfies the coarse ellipticity condition~\eqref{e.spatial.moment.range}. For fixed~$\rho_1$ and~$\rho_2$, the factors~$(\rho_2-\rho_1)^{-\gamma}$ in~\eqref{e.cg.local.boundedness} and~$(\rho_2-\rho_1)^{-\nf{2\gamma}\eta}$ in~\eqref{e.lr.bound} do not depend on~$\eps$. Hence~\eqref{e.sharpness.polynomial.failure}, with~$\eta=2$ and with~$0<\eta<2$, shows that the powers~$\nf{(d-1)}{(4\theta)}$ in~\eqref{e.cg.local.boundedness} and~$\nf{(d-1)}{(2\eta\theta)}$ in~\eqref{e.lr.bound} are optimal for estimates of the positive parts of the sign-changing solutions~$u_\eps$, for each fixed~$d,p,q,s,t,\rho_1,\rho_2$.

\begin{proof}[Proof of Proposition~\ref{p.sharpness.polynomial}]
We construct the solution, estimate its positive part, and bound~$\Lambda_\eps$ and~$\lambda_\eps^{-1}$ from above by coefficient averages. Theorem~\ref{t.cg.local.boundedness} then gives the matching lower bounds.

\smallskip
\emph{Step 1.} The field and solution. The condition~$\theta>0$ gives~$0<s+\nf{(d-1)}{2p}<1$ and~$0<t+\nf{(d-1)}{2q}<1$. Define
\begin{equation*}
	a_{\perp,\eps}\coloneqq\eps^{2t+\nf{(d-1)}q}\,,\qquad a_{\parallel,\eps}\coloneqq(d-1)\eps^{2t+\nf{(d-1)}q-2}\,,
\end{equation*}
and
\begin{equation}
	\a_\eps(x_1,y)\coloneqq\begin{cases}\operatorname{diag}(a_{\parallel,\eps},a_{\perp,\eps}I_{d-1})&\mbox{if }|y|<2\eps\,,\\I&\mbox{if }|y|\geq2\eps\,.\end{cases}
	\label{e.sharpness.polynomial.field}
\end{equation}
Then~$0<a_{\perp,\eps}<1<a_{\parallel,\eps}<\infty$, so the field~\eqref{e.sharpness.polynomial.field} is uniformly elliptic for each fixed~$\eps$. The contributions of~$E_{2\eps,0}$ to~$\int_{\cu_0}\tr\a_\eps$ and~$\int_{\cu_0}\tr\a_\eps^{-1}$ are at most constants times~$\eps^{d-3+2t+\nf{(d-1)}q}$ and~$\eps^{d-1-2t-\nf{(d-1)}q}$, respectively. Both exponents are positive, which proves~\eqref{e.sharpness.polynomial.traces}.

\smallskip
Define~$F_\eps(\varrho)\coloneqq\nf{\varrho^2}{\eps^2}$ for~$0\leq \varrho\leq2\eps$, and continue~$F_\eps$ continuously from~$F_\eps(2\eps)=4$ by
\begin{equation}
	\varrho^{d-2}F_\eps'(\varrho)=\frac{2a_{\perp,\eps}(2\eps)^{d-1}}{\eps^2}+\frac2{d-1}\bigl(\varrho^{d-1}-(2\eps)^{d-1}\bigr)\qquad(\varrho>2\eps)\,.
	\label{e.sharpness.polynomial.profile}
\end{equation}
Set~$u_\eps(x_1,y)\coloneqq1+x_1^2-F_\eps(|y|)$. Inside the cylinder,
\begin{equation*}
	\nabla\cdot(\a_\eps\nabla u_\eps)=2a_{\parallel,\eps}-\frac{2(d-1)a_{\perp,\eps}}{\eps^2}=0\,.
\end{equation*}
Outside, differentiation of~\eqref{e.sharpness.polynomial.profile} gives~$\Delta_y(F_\eps(|y|))=2$, so~$\Delta u_\eps=0$. At~$\varrho=2\eps$,
\begin{equation*}
	F_\eps'(2\eps+)=a_{\perp,\eps}F_\eps'(2\eps-)\,,
\end{equation*}
so the normal fluxes agree. Integration by parts therefore proves the weak equation on~$\cu_0$. For fixed~$\eps$, the profile has bounded derivative on~$[0,\nf{\sqrt{d-1}}2]$, so~$u_\eps$ is Lipschitz. Mollifying a Lipschitz extension gives convergence in~$H^1_{\a_\eps}(\cu_0)$ since~$\a_\eps$ is bounded, with limit~$u_\eps$ by Proposition~\ref{p.weighted.calculus}\ref{i.weighted.identification}. Hence~$u_\eps\in\Csol(\cu_0)$, and Lemma~\ref{l.weighted.testing}\ref{i.testing.truncation} gives~$(u_\eps)_+\in\Csub(\cu_0)$.

\smallskip
\emph{Step 2.} The positive part and its norms. The right side of~\eqref{e.sharpness.polynomial.profile} is positive for~$\varrho>2\eps$, so~$F_\eps'(\varrho)>0$ there and~$F_\eps(\varrho)\geq F_\eps(2\eps)=4$, whereas~$1+x_1^2\leq\nf54$ on~$\cu_0$. Thus~$u_\eps$ changes sign and, on all of~$\cu_0$,
\begin{equation}
	(u_\eps)_+(x_1,y)=\Bigl(1+x_1^2-\frac{|y|^2}{\eps^2}\Bigr)_+\,.
	\label{e.sharpness.polynomial.positive}
\end{equation}
By~\eqref{e.sharpness.polynomial.positive}, the positive part is at most~$\nf54$, vanishes unless~$|y|<\sqrt{\nf54}\,\eps$, and is at least~$\nf34$ when~$|y|<\nf\eps2$. These transverse balls lie in~$(-\nf{\rho_2}2,\nf{\rho_2}2)^{d-1}$ for~$\eps<\nf18$ and~$\rho_2\geq\nf12$. Integrating these bounds over the axial interval of length~$\rho_2$ gives, for every~$\eta>0$,
\begin{equation*}
	c\eps^{d-1}\leq\int_{\rho_2\cu_0}(u_\eps)_+^\eta\leq C\eps^{d-1}\,.
\end{equation*}
The constants depend only on~$d,\eta$, uniformly in~$\rho_2\in[\nf12,1]$. Taking the~$\nf1\eta$ power gives the asserted~$L^\eta$ comparison. The essential supremum on~$\rho_1\cu_0$ is~$1+\nf{\rho_1^2}4$: there~$(u_\eps)_+\leq1+x_1^2<1+\nf{\rho_1^2}4$, and every smaller value is exceeded on a set of positive measure near~$(\pm\nf{\rho_1}2,0)$. This proves~\eqref{e.sharpness.polynomial.norms}, including~$\rho_2=1$.

\smallskip
\emph{Step 3.} The coarse ellipticity constants and ratio. Use the fractions~$f_\triangle$ of Lemma~\ref{l.sharpness.cylinder.averages} for~$E_{2\eps,0}$. The averages of~$\a_\eps$ and~$\a_\eps^{-1}$ over~$\triangle$ are diagonal, with largest entries~$1+(a_{\parallel,\eps}-1)f_\triangle$ and~$1+(a_{\perp,\eps}^{-1}-1)f_\triangle$, respectively. For the coarse-grained matrices~$\a(\triangle)$ and~$\a_*(\triangle)$ of~$\a=\a_\eps$, the bounds~\eqref{e.upper.norm.bound} and~\eqref{e.lower.norm.bound} give
\begin{equation*}
	|\a(\triangle)|\leq1+a_{\parallel,\eps}f_\triangle\,,\qquad|\a_*^{-1}(\triangle)|\leq1+a_{\perp,\eps}^{-1}f_\triangle\,.
\end{equation*}
By spatial Minkowski, square-root subadditivity and~$(1-3^{-s})\sum_{k\geq0}3^{-ks}=1$, the definition~\eqref{e.upper.half.series} gives
\begin{equation*}
	\Lambda_\eps^{\nf12}\leq1+a_{\parallel,\eps}^{\nf12}\sum_{k\geq0}3^{-ks}\biggl(\avsum_{\triangle\in\mathscr T_k}f_\triangle^p\biggr)^{\nf1{2p}}\leq1+Ca_{\parallel,\eps}^{\nf12}\eps^{s+\nf{(d-1)}{2p}}=1+C\sqrt{d-1}\,\eps^{-\theta}\,.
\end{equation*}
The second inequality is~\eqref{e.sharpness.cylinder.discount} with~$\nu=2s+\nf{(d-1)}p\in(0,2)\subseteq(0,d-1)$. The same estimate with~$q$,~$t$ and~$2t+\nf{(d-1)}q$ in place of~$p$,~$s$ and~$\nu$, together with~\eqref{e.lower.half.series}, gives~$\lambda_\eps^{-\nf12}\leq1+Ca_{\perp,\eps}^{-\nf12}\eps^{t+\nf{(d-1)}{2q}}=1+C$. Hence
\begin{equation*}
	\Lambda_\eps\leq C\eps^{-2\theta}\,,\qquad\lambda_\eps^{-1}\leq C\,,\qquad\Theta_\eps\leq C\eps^{-2\theta}\,.
\end{equation*}

\smallskip
These bounds verify the coefficient hypotheses of Theorem~\ref{t.cg.local.boundedness}, and~$u_\eps$ is a solution. By~\eqref{e.cg.local.boundedness} with~$\rho_1=\nf12$,~$\rho_2=\nf34$, and by~\eqref{e.sharpness.polynomial.norms} with~$\eta=2$,
\begin{equation*}
	1\leq\norm{(u_\eps)_+}_{L^\infty(\frac12\cu_0)}\leq C\Theta_\eps^{\nf{(d-1)}{4\theta}}\norm{(u_\eps)_+}_{L^2(\frac34\cu_0)}\leq C\Theta_\eps^{\nf{(d-1)}{4\theta}}\eps^{\nf{(d-1)}2}\,.
\end{equation*}
Thus~$\Theta_\eps\geq c\eps^{-2\theta}$. Combining this with the preceding bounds gives~$\Lambda_\eps=\Theta_\eps\lambda_\eps\geq c\eps^{-2\theta}$ and~$\lambda_\eps=\nf{\Lambda_\eps}{\Theta_\eps}\leq C$, proving all three comparisons in~\eqref{e.sharpness.polynomial.constants}. The constants depend only on~$d,p,q,s,t$, as those of Theorem~\ref{t.cg.local.boundedness} do.

\smallskip
For every fixed real~$\upsilon$, including negative~$\upsilon$, the two-sided comparison for~$\Theta_\eps$ gives~$\Theta_\eps^\upsilon\asymp\eps^{-2\theta\upsilon}$. By~\eqref{e.sharpness.polynomial.norms}, the ratio in~\eqref{e.sharpness.polynomial.failure} is therefore comparable to~$\eps^{2\theta\upsilon-\nf{(d-1)}\eta}$ for fixed~$\rho_1,\rho_2,\eta,\upsilon$, with constants depending also on~$\upsilon$. It tends to infinity exactly when~$\upsilon<\nf{(d-1)}{(2\eta\theta)}$.
\end{proof}

\subsection{The exponent of the weak Harnack inequality}
\label{ss.sharpness.weak.harnack}

Under the hypotheses of Theorem~\ref{t.cg.local.boundedness}, Theorem~\ref{t.harnack} gives the weak Harnack inequality~\eqref{e.weak.harnack} for~$0<\eta\leq\eta_c$, where
\begin{equation*}
	\eta_c\coloneqq\frac{r^*}2=\frac d{d-2+2t+\nf dq}\,;
\end{equation*}
the last expression follows from~\eqref{e.sobolev.exponents} and~$r=\nf{2q}{(q+1)}$. We show that this range cannot be enlarged, for any fixed admissible choice of~$d,p,q,s,t$. In this subsection we write~$\varrho=|x|$ for the radial variable, and set~$\nu\coloneqq2t+\nf dq$; thus~$\eta_c=\nf d{(d-2+\nu)}$.

\smallskip
The coefficient field vanishes at the origin like~$|x|^\nu$, up to a logarithmic factor. On the simplices of~$\mathscr T_k$ near the origin, of total volume of order~$3^{-kd}$, the average of its inverse is of order~$3^{k\nu}$. Apart from the logarithm, these simplices therefore contribute a term of order~$3^{-kt}3^{k(\frac\nu2-\frac d{2q})}=1$ to the series~\eqref{e.lower.half.series}, and the third power of the logarithm in the field makes the series converge. Each supersolution is the fundamental solution with pole at the origin, of size~$\varrho^{-(d-2+\nu)}$ up to a logarithm, made constant in a ball~$B(0,\eps)$. The value~$\eps^{-(d-2+\nu)}$ on a ball of volume of order~$\eps^d$ makes the moments of order~$\eta>\nf d{(d-2+\nu)}=\eta_c$ unbounded as~$\eps\downarrow0$, whereas the infimum on~$\frac12\cu_0$ does not depend on~$\eps$.

\begin{proposition}[Sharpness of the weak Harnack exponent]
\label{p.sharpness.weak.harnack}
Let~$d\geq3$, and let~$1<p,q<\infty$ and~$s,t>0$ satisfy~$\theta>0$, with~$\theta$ as in~\eqref{e.intro.parameters}. There are a coefficient field~$\a(x)=a(|x|)I$ on~$\cu_0$ with~$\tr\a,\tr\a^{-1}\in L^1(\cu_0)$,~$\Lambda_{s,1,p}(\cu_0)<\infty$ and~$\lambda_{t,1,q}(\cu_0)>0$, and nonnegative supersolutions~$u_\eps$ in~$\cu_0$, for~$0<\eps<\nf18$, such that~$\essinf_{\frac12\cu_0}u_\eps>0$ does not depend on~$\eps$ and, for every~$\eta>\eta_c$,
\begin{equation}
	\frac1{\essinf_{\frac12\cu_0}u_\eps}\biggl(\fint_{\frac58\cu_0}u_\eps^\eta\biggr)^{\nf1\eta}\longrightarrow\infty\qquad\mbox{as }\eps\downarrow0\,.
	\label{e.sharpness.weak.harnack.failure}
\end{equation}
\end{proposition}

\smallskip
The field and the supersolutions are the same for every~$\eta>\eta_c$. Since~$\Theta<\infty$ for this field,~\eqref{e.sharpness.weak.harnack.failure} shows that~\eqref{e.weak.harnack} fails for every~$\eta>\eta_c$ with any finite constant, even one depending on~$\a$. Hence the range~$0<\eta\leq\eta_c$ of Theorem~\ref{t.harnack} is optimal for each fixed~$d,p,q,s,t$. This concerns the exponent only; the coefficient range and the growth of the constant in~$\Theta$ are treated by Theorem~\ref{t.sharpness} and by the example after Theorem~\ref{t.harnack}.

\begin{proof}[Proof of Proposition~\ref{p.sharpness.weak.harnack}]
We bound the coarse ellipticity constants of the field through the averages of~$a$ and~$a^{-1}$, show that the functions~$u_\eps$ are supersolutions with an explicit source, and compare their moments with their infimum. The constants depend only on~$d,p,q,s,t$, and in Step~4 also on~$\eta$; they never depend on~$k$ or~$\eps$.

\smallskip
\emph{Step 1.} The field and its upper constant. Since~$\nu=2\sigma_*+\nf1q$ and~$\sigma_*<1$ by~\eqref{e.parameter.facts},
\begin{equation}
	0<\nu<2+\frac1q<3\leq d\qquad\mbox{and}\qquad\nu q-d=2tq>0\,.
	\label{e.sharpness.weak.harnack.exponents}
\end{equation}
Let~$R\coloneqq\sqrt d$, so that~$|x|<\nf R2$ for~$x\in\cu_0$, and define
\begin{equation*}
	a(\varrho)\coloneqq\varrho^\nu\log^3(eR/\varrho)\quad(0<\varrho\leq R)\,,\qquad\a(x)\coloneqq a(|x|)I\quad(x\in\cu_0\setminus\{0\})\,,
\end{equation*}
and~$\a(0)\coloneqq I$. The logarithm is at least one on~$(0,R]$, so~$a$ is positive and continuous there, and~$a(\varrho)\to0$ as~$\varrho\downarrow0$ because~$\nu>0$; thus~$a$ is bounded. Moreover~$a(\varrho)^{-1}\leq\varrho^{-\nu}$, and~$|x|^{-\nu}$ is integrable near the origin because~$\nu<d$. Hence~$\tr\a=da(|x|)$ and~$\tr\a^{-1}=da(|x|)^{-1}$ belong to~$L^1(\cu_0)$. The field is fixed from now on; it depends neither on~$\eps$ nor on~$\eta$.

\smallskip
For a simplex~$\triangle$ we write~$(a)_\triangle$ and~$(a^{-1})_\triangle$ for the averages of~$a(|x|)$ and~$a(|x|)^{-1}$ over~$\triangle$, so that~$(\a)_\triangle=(a)_\triangle I$ and~$(\a^{-1})_\triangle=(a^{-1})_\triangle I$. By~\eqref{e.upper.norm.bound} and~\eqref{e.lower.norm.bound},
\begin{equation}
	|\a(\triangle)|\leq(a)_\triangle\leq\sup_{(0,R]}a\,,\qquad|\a_*^{-1}(\triangle)|\leq(a^{-1})_\triangle\,.
	\label{e.sharpness.weak.harnack.matrices}
\end{equation}
Since the weights~$(1-3^{-s})3^{-ks}$ in~\eqref{e.upper.half.series} sum to one, the first bound gives~$\Lambda_{s,1,p}(\cu_0)\leq\sup_{(0,R]}a<\infty$.

\smallskip
\emph{Step 2.} The lower constant. We show that, for every~$k\geq0$,
\begin{equation}
	\avsum_{\triangle\in\mathscr T_k}\bigl((a^{-1})_\triangle\bigr)^q\leq C3^{k(\nu q-d)}(1+k)^{-3q}\,.
	\label{e.sharpness.weak.harnack.inverse}
\end{equation}
The simplices of~$\mathscr T_k$ lie in cubes of side length~$3^{-k}$, so they have diameter at most~$R3^{-k}$. By~\eqref{e.simplex.count} with~$x=0$ and~$N=R$, at most~$C$ of them have closures meeting~$B(0,R3^{-k})$, and these lie in~$B(0,2R3^{-k})$. There~$\log(eR/|x|)\geq1-\log2+k\log3\geq(1-\log2)(1+k)$, so polar coordinates and~$\nu<d$ give, for each of these simplices,
\begin{equation*}
	(a^{-1})_\triangle\leq\frac{d!\,3^{kd}}{(1-\log2)^3(1+k)^3}\int_{B(0,2R3^{-k})}|x|^{-\nu}\,dx\leq C3^{k\nu}(1+k)^{-3}\,.
\end{equation*}
By~\eqref{e.partition.average}, their contribution to the left side of~\eqref{e.sharpness.weak.harnack.inverse} is at most~$C3^{k(\nu q-d)}(1+k)^{-3q}$. Here the averages over simplices are essential: since~$\nu q>d$ by~\eqref{e.sharpness.weak.harnack.exponents}, the function~$a(|x|)^{-q}$ is not integrable near the origin. The other simplices of~$\mathscr T_k$ do not meet~$B(0,R3^{-k})$. For them, Jensen's inequality bounds~$((a^{-1})_\triangle)^q$ by the average of~$a(|x|)^{-q}$ over~$\triangle$, and~\eqref{e.partition.average}, polar coordinates and~$d-1-\nu q=-1-2tq$ bound their contribution by
\begin{equation*}
	\int_{\cu_0\setminus B(0,R3^{-k})}\frac{dx}{a(|x|)^q}\leq C\int_{R3^{-k}}^R\frac{d\varrho}{\varrho^{1+2tq}\log^{3q}(eR/\varrho)}\,.
\end{equation*}
The logarithm is at least~$1+\frac k2\log3\geq\frac12(1+k)$ for~$\varrho\leq R3^{-\nf k2}$, and at least one for~$\varrho\leq R$. Splitting the integral at~$R3^{-\nf k2}$, we bound it by
\begin{equation*}
	\frac{2^{3q}(R3^{-k})^{-2tq}}{2tq(1+k)^{3q}}+\frac{(R3^{-\nf k2})^{-2tq}}{2tq}\leq C3^{2ktq}(1+k)^{-3q}\,,
\end{equation*}
since~$3^{-ktq}(1+k)^{3q}$ is bounded in~$k$. As~$2tq=\nu q-d$ by~\eqref{e.sharpness.weak.harnack.exponents}, this proves~\eqref{e.sharpness.weak.harnack.inverse}.

\smallskip
By~\eqref{e.sharpness.weak.harnack.matrices} and~\eqref{e.sharpness.weak.harnack.inverse}, and since~$\nf\nu2-\nf d{2q}=t$,
\begin{equation*}
	3^{-kt}\biggl(\avsum_{\triangle\in\mathscr T_k}|\a_*^{-1}(\triangle)|^q\biggr)^{\nf1{2q}}\leq C3^{-kt}3^{k(\frac\nu2-\frac d{2q})}(1+k)^{-\nf32}=C(1+k)^{-\nf32}\,.
\end{equation*}
Summing over~$k$, the definition~\eqref{e.lower.half.series} gives~$\lambda_{t,1,q}(\cu_0)^{-\nf12}\leq C(1-3^{-t})\sum_{k\geq0}(1+k)^{-\nf32}<\infty$, so~$\lambda_{t,1,q}(\cu_0)>0$.

\smallskip
\emph{Step 3.} The supersolutions. For~$0<\varrho\leq R$ and~$0<\eps<\nf18$, let
\begin{equation*}
	G(\varrho)\coloneqq\int_\varrho^R\frac{d\tau}{\tau^{d-1+\nu}\log^3(eR/\tau)}\,,\qquad u_\eps(x)\coloneqq G(\max\{|x|,\eps\})\quad(x\in\cu_0)\,.
\end{equation*}
The function~$G$ is decreasing, vanishes at~$R$, and satisfies~$a(\varrho)G'(\varrho)=-\varrho^{1-d}$. Since~$\max\{|x|,\eps\}<R$ on~$\cu_0$, the function~$u_\eps$ is positive there. It equals~$G(\eps)$ on~$B(0,\eps)$, and~$\nabla u_\eps(x)=G'(|x|)\nf x{|x|}$ for~$|x|>\eps$. Hence
\begin{equation*}
	\a\nabla u_\eps(x)=\begin{cases}0&\mbox{if }|x|<\eps\,,\\-|x|^{1-d}\dfrac x{|x|}&\mbox{if }|x|>\eps\,.\end{cases}
\end{equation*}
The vector field~$x|x|^{-d}$ is divergence free in~$\Rd\setminus\{0\}$. For~$\varphi\in C^\infty_c(\cu_0)$, the divergence theorem in~$\{|x|>\eps\}$, whose outer unit normal on the sphere~$\partial B(0,\eps)\subseteq\cu_0$ is~$-\nf x{|x|}$, gives
\begin{equation}
	\int_{\cu_0}\nabla\varphi\cdot\a\nabla u_\eps=-\int_{\{|x|>\eps\}}\nabla\varphi\cdot\frac x{|x|^d}\,dx=\eps^{1-d}\int_{\partial B(0,\eps)}\varphi\,d\mathcal H^{d-1}\,.
	\label{e.sharpness.weak.harnack.source}
\end{equation}
Thus~$-\nabla\cdot\a\nabla u_\eps$ is the nonnegative measure~$\eps^{1-d}\mathcal H^{d-1}$ on~$\partial B(0,\eps)$, of total mass~$\mathcal H^{d-1}(\partial B(0,1))$.

\smallskip
For fixed~$\eps$, the function~$x\mapsto G(\min\{\max\{|x|,\eps\},R\})$ is a Lipschitz extension of~$u_\eps$ to~$\Rd$. Its mollifications converge to it uniformly, so their means over~$\cu_0$ converge, and their gradients are bounded by its Lipschitz constant and converge almost everywhere. The energy densities of the approximation errors are therefore bounded by four times the squared Lipschitz constant times~$a(|x|)$, which is integrable, and dominated convergence gives convergence in the norm~\eqref{e.weighted.norm}. By Proposition~\ref{p.weighted.calculus}\ref{i.weighted.identification} the limit is~$u_\eps$, so~$u_\eps\in H^1_\a(\cu_0)$, and~\eqref{e.sharpness.weak.harnack.source} shows that~$u_\eps$ is a nonnegative supersolution. For each fixed~$\eps>0$, the truncation makes~$u_\eps$ integrable; when~$\nu>2$, the untruncated profile~$G(|x|)$ is not integrable on~$\cu_0$.

\smallskip
\emph{Step 4.} Asymptotics and the ratio. We show that
\begin{equation}
	\lim_{\varrho\downarrow0}(d-2+\nu)\varrho^{d-2+\nu}\log^3(eR/\varrho)\,G(\varrho)=1\,.
	\label{e.sharpness.weak.harnack.asymptotics}
\end{equation}
We compare~$G$ with~$g(\varrho)\coloneqq(d-2+\nu)^{-1}\varrho^{2-d-\nu}\log^{-3}(eR/\varrho)$, which tends to infinity as~$\varrho\downarrow0$. Differentiation gives
\begin{equation*}
	g'(\varrho)=G'(\varrho)\biggl(1-\frac3{(d-2+\nu)\log(eR/\varrho)}\biggr)\,.
\end{equation*}
The factor in parentheses is less than one, and for each~$0<\delta<1$ it exceeds~$1-\delta$ on an interval~$(0,\varrho_\delta)$. Since~$G'<0$ and~$G(R)=0$, integrating~$g'$ over~$(\varrho,R)$ and over~$(\varrho,\varrho_\delta)$ gives
\begin{equation*}
	g(\varrho)-g(R)\leq G(\varrho)\leq\frac{g(\varrho)}{1-\delta}+G(\varrho_\delta)\qquad(0<\varrho<\varrho_\delta)\,.
\end{equation*}
Dividing by~$g(\varrho)$, letting~$\varrho\downarrow0$ and then~$\delta\downarrow0$ proves~\eqref{e.sharpness.weak.harnack.asymptotics}.

\smallskip
Now let~$\eta>\eta_c$, so that~$d-\eta(d-2+\nu)<0$. By~\eqref{e.sharpness.weak.harnack.asymptotics},~$G(\eps)\geq\frac12g(\eps)$ for every sufficiently small~$\eps$. Since~$u_\eps=G(\eps)$ on~$B(0,\eps)\subseteq\frac58\cu_0$ and~$|\frac58\cu_0|\leq1$,
\begin{equation*}
	\fint_{\frac58\cu_0}u_\eps^\eta\geq|B(0,\eps)|\,G(\eps)^\eta\geq\frac{c\,\eps^{d-\eta(d-2+\nu)}}{\log^{3\eta}(eR/\eps)}\longrightarrow\infty\qquad\mbox{as }\eps\downarrow0\,,
\end{equation*}
because the power of~$\eps$ is negative and the logarithm grows more slowly than every power of~$\eps^{-1}$. On the other hand,~$|x|<\nf R4$ on~$\frac12\cu_0$ and~$\eps<\nf18<\nf R4$, so~$u_\eps\geq G(\nf R4)$ there, and~$u_\eps(x)=G(|x|)$ tends to~$G(\nf R4)$ as~$x$ tends to a vertex of~$\frac12\cu_0$. As~$u_\eps$ is continuous,~$\essinf_{\frac12\cu_0}u_\eps=G(\nf R4)>0$, independently of~$\eps$. This proves~\eqref{e.sharpness.weak.harnack.failure}.
\end{proof}

\smallskip
At~$\eta=\eta_c$ the construction is consistent with Theorem~\ref{t.harnack}. Since~$\varrho^{d-1}\varrho^{-(d-2+\nu)\eta_c}=\varrho^{-1}$, the limit~\eqref{e.sharpness.weak.harnack.asymptotics} shows that~$\varrho^{d-1}G(\varrho)^{\eta_c}$ is comparable on~$(0,\nf R2]$ to~$\varrho^{-1}\log^{-3\eta_c}(eR/\varrho)$. In polar coordinates,
\begin{equation*}
	\int_{\cu_0}G(|x|)^{\eta_c}\,dx\leq C\int_0^{\nf R2}\frac{d\varrho}{\varrho\log^{3\eta_c}(eR/\varrho)}<\infty\,,
\end{equation*}
because~$3\eta_c>1$; indeed~$\eta_c=\nf{r^*}2>\nf12$, since~$r^*>r>1$ by~\eqref{e.parameter.facts}. As~$G$ is decreasing,~$u_\eps\leq G(|x|)$, so the left side of~\eqref{e.sharpness.weak.harnack.failure} stays bounded for~$\eta=\eta_c$. The endpoint inequality from Section~\ref{ss.harnack.endpoint}, together with Proposition~\ref{p.sharpness.weak.harnack}, gives the exact range~$0<\eta\leq\eta_c$ in Theorem~\ref{t.harnack}.

\appendix

\section{Weighted spaces and variational properties}
\label{a.weighted.proofs}

This appendix proves Proposition~\ref{p.weighted.calculus} and Lemma~\ref{l.weighted.testing}, compares~$H^1_\a(U)$ with the space of~\cite{AK.HC}, and then proves the three propositions and the subadditivity bounds of Section~\ref{s.harmonic.matrices}. Throughout,~$U\subseteq\Rd$ is a bounded connected Lipschitz domain and the coefficient field satisfies the assumptions of Section~\ref{s.weighted.spaces}, including~\eqref{e.weighted.hypotheses}. The proofs use no result of Sections~\ref{s.triadic.simplices}--\ref{s.sharpness}. We write~$\a^{\nf12}$ for the positive square root, and constants~$C(U)$ depend only on~$U$ and~$d$. For~$N>0$ we write~$T_N(u)\coloneqq\max\{-N,\min\{u,N\}\}$.

\begin{proposition}[Weighted calculus and approximation]
\label{p.weighted.calculus}
Under~\eqref{e.weighted.hypotheses}, the following statements hold.
\begin{enumerate}[label=\textup{(\roman*)}]
\item\label{i.weighted.identification} The space~$H^1_\a(U)$ is a Hilbert space continuously and injectively embedded in~$W^{1,1}(U)$. Its energy is~$\Ener_U(u)=\norm{\a^{\nf12}\nabla u}_{L^2(U)}^2$, and~$\a\nabla u\in L^1(U;\Rd)$. Convergence in~$H^1_\a(U)$ implies convergence of gradients and fluxes in~$L^1$.
\item\label{i.weighted.zero} One has~$\Hzero{\a}(U)\subseteq W^{1,1}_0(U)$; in particular, elements of~$\Hzero{\a}(U)$ have zero trace. For every~$\varphi\in\Hzero{\a}(U)$,
\begin{equation}
	|(\varphi)_U|\leq C(U)\Bigl(\int_U\tr\a^{-1}\Bigr)^{\nf12}\Ener_U(\varphi)^{\nf12}\,,\qquad \int_U\nabla\varphi=0\,.
	\label{e.weighted.coercivity}
\end{equation}
Consequently energy is coercive on~$\Hzero{\a}(U)$, and an element with zero gradient is zero.
\item\label{i.weighted.domains} If~$U'\supseteq U$ is another bounded connected Lipschitz domain on which~\eqref{e.weighted.hypotheses} holds, restriction maps~$H^1_\a(U')$ continuously into~$H^1_\a(U)$, and extension by zero maps~$\Hzero{\a}(U)$ continuously into~$\Hzero{\a}(U')$, preserving energy.
\item\label{i.weighted.chain} For smooth~$\Phi:\R\to\R$ with bounded derivative and~$u\in H^1_\a(U)$, one has~$\Phi(u)\in H^1_\a(U)$ and~$\nabla \Phi(u)=\Phi'(u)\nabla u$. For every~$c\in\R$,~$\indc_{\{u=c\}}\nabla u=0$ almost everywhere. Positive parts and the truncations~$T_N$ are continuous on~$H^1_\a(U)$, do not increase energy, and satisfy
\begin{equation}
	\nabla u_+=\indc_{\{u>0\}}\nabla u\,,\qquad \nabla T_N(u)=\indc_{\{|u|<N\}}\nabla u\,.
	\label{e.weighted.truncation.gradients}
\end{equation}
Moreover~$T_N(u)\to u$ in~$H^1_\a(U)$ as~$N\to\infty$.
\item\label{i.weighted.approximation} Every element of~$H^1_\a(U)$ is a limit of bounded smooth functions. The approximations may be chosen nonnegative if the function is, bounded by~$N$ in absolute value if~$|u|\leq N$, both at once, and compactly supported in~$U$ for elements of~$\Hzero{\a}(U)$.
\item\label{i.weighted.products} Lipschitz functions with compact support in~$U$, and Lipschitz functions on~$\overline U$ vanishing on~$\partial U$, belong to~$\Hzero{\a}(U)$. If~$w\in H^1_\a(U)\cap L^\infty(U)$ and~$\varphi$ is Lipschitz with compact support in~$U$, then~$\varphi w\in\Hzero{\a}(U)$ and
\begin{equation}
	\nabla(\varphi w)=\varphi\nabla w+w\nabla\varphi\,.
	\label{e.weighted.product}
\end{equation}
If~$f,g\in H^1_\a(U)\cap L^\infty(U)$, then~$fg\in H^1_\a(U)$ and~$\nabla(fg)=f\nabla g+g\nabla f$. A bounded element of~$H^1_\a(U)$ vanishing outside a compact subset of~$U$ belongs to~$\Hzero{\a}(U)$. Positive parts and the maps~$T_N$ preserve~$\Hzero{\a}(U)$.
\end{enumerate}
\end{proposition}

\begin{proof}
Most assertions are first checked on smooth functions and then passed to the limit: the~$L^1$ bounds of Step~1 identify the limits, and dominated convergence of the energy densities gives convergence in~$H^1_\a(U)$. Steps~1 and~2 prove~\ref{i.weighted.identification}--\ref{i.weighted.domains}, and Step~3 proves~\ref{i.weighted.chain}. Step~4 constructs the bounded smooth approximations of~\ref{i.weighted.approximation}, Step~5 uses them to prove the assertions of~\ref{i.weighted.products} on Lipschitz functions and products, and Step~6 treats~$\Hzero{\a}(U)$, which completes~\ref{i.weighted.approximation} and~\ref{i.weighted.products}.

\smallskip

\emph{Step 1.} Identification of the completion. We show that the norm dominates the~$W^{1,1}$ norm on smooth functions of finite energy and that an element of the completion is determined by its limit in~$W^{1,1}(U)$. We use the mean-zero Poincar\'e inequality~\cite[Lemma 3.24,~$p=1$]{Guermond.traces},
\begin{equation}
	\norm{v-(v)_U}_{L^1(U)}\leq C(U)\norm{\nabla v}_{L^1(U)}\qquad\mbox{for every }v\in W^{1,1}(U)\,,
	\label{e.ordinary.mean.poincare}
\end{equation}
and the two Cauchy--Schwarz bounds
\begin{equation}
	\norm{\nabla v}_{L^1(U)}\leq\Bigl(\int_U\tr\a^{-1}\Bigr)^{\nf12}\Ener_U(v)^{\nf12}\,,\qquad \norm{\a\nabla v}_{L^1(U)}\leq\Bigl(\int_U\tr\a\Bigr)^{\nf12}\Ener_U(v)^{\nf12}\,,
	\label{e.weighted.gradient.lone}
\end{equation}
obtained by writing~$\nabla v=\a^{-\nf12}(\a^{\nf12}\nabla v)$ and~$\a\nabla v=\a^{\nf12}(\a^{\nf12}\nabla v)$ and using~$|\a^{\mp\nf12}e|^2\leq(\tr\a^{\mp1})|e|^2$. On these functions, the norm~\eqref{e.weighted.norm} therefore dominates the~$W^{1,1}$ norm. It comes from the inner product
\begin{equation*}
	(u,w)\mapsto(u)_U(w)_U+\int_U\nabla u\cdot\a\nabla w\,,
\end{equation*}
which is definite because~$U$ is connected. If a sequence~$(u_j)$ of smooth functions of finite energy is Cauchy, then~$u_j\to u$ in~$W^{1,1}(U)$ and~$\a^{\nf12}\nabla u_j\to g$ in~$L^2$. By~\eqref{e.weighted.gradient.lone},~$\nabla u_j\to\a^{-\nf12}g$ in~$L^1$, so~$g=\a^{\nf12}\nabla u$, and~$\a\nabla u_j\to\a\nabla u$ in~$L^1$. Thus the limit is determined by~$u$: if~$u=0$, both its mean and~$g$ vanish. This proves~\ref{i.weighted.identification}; the bounds~\eqref{e.weighted.gradient.lone} extend to the completion, and so does
\begin{equation}
	\Bigl|\int_U\nabla v\cdot\a\nabla u\Bigr|\leq\Ener_U(v)^{\nf12}\Ener_U(u)^{\nf12}\qquad(u,v\in H^1_\a(U))\,.
	\label{e.weighted.pairing}
\end{equation}

\smallskip

\emph{Step 2.} Zero boundary values and domains. The defining approximations of~$\varphi\in\Hzero{\a}(U)$ converge in~$W^{1,1}(U)$, so~$\varphi\in W^{1,1}_0(U)$ and its trace vanishes~\cite[Theorem 3.10]{Guermond.traces}. The zero-boundary Poincar\'e inequality~$\norm{\varphi}_{L^1(U)}\leq C(U)\norm{\nabla\varphi}_{L^1(U)}$ of~\cite[Lemma 3.27]{Guermond.traces} and~\eqref{e.weighted.gradient.lone} give the first bound in~\eqref{e.weighted.coercivity}, and~$\int_U\nabla\varphi=0$ passes to the limit from smooth compactly supported functions. Hence
\begin{equation*}
	\norm{\varphi}_{H^1_\a(U)}^2\leq\Bigl(1+C(U)^2\int_U\tr\a^{-1}\Bigr)\Ener_U(\varphi)\,.
\end{equation*}
Restrictions to~$U$ of smooth functions of finite energy on~$U'$ are smooth functions of finite energy on~$U$, and~$|(u)_U|\leq|U|^{-1}\norm{u}_{L^1(U')}$ gives continuity of restriction. Zero extensions of~$\varphi_j\in C^\infty_c(U)$ lie in~$C^\infty_c(U')$, keep their energies and have means~$(|U|/|U'|)(\varphi_j)_U$; passing to the limit,
\begin{equation}
	\norm{\widetilde\varphi}_{H^1_\a(U')}^2=\frac{|U|^2}{|U'|^2}|(\varphi)_U|^2+\Ener_U(\varphi)\,,\qquad \nabla\widetilde\varphi=\indc_U\nabla\varphi\,,
	\label{e.weighted.zero.extension}
\end{equation}
where~$\widetilde\varphi$ is the zero extension.

\smallskip

\emph{Step 3.} Chain rule, locality and truncations. Let~$u_j\to u$ with smooth functions~$u_j$ of finite energy. Then~$\Phi(u_j)$ is smooth with finite energy,~$\Phi(u_j)\to \Phi(u)$ in~$L^1$, and
\begin{equation*}
	\Phi'(u_j)\nabla u_j-\Phi'(u)\nabla u=\Phi'(u_j)\nabla(u_j-u)+\bigl(\Phi'(u_j)-\Phi'(u)\bigr)\nabla u\,.
\end{equation*}
After multiplication by~$\a^{\nf12}$ the first term tends to zero in~$L^2$, and the second does so by dominated convergence along almost everywhere convergent subsequences, with dominator~$4\norm{\Phi'}_\infty^2\nabla u\cdot\a\nabla u$. Completeness gives the smooth chain rule. Next, every~$W^{1,1}$ function is absolutely continuous on almost every coordinate line~\cite[Theorem 2.36]{Kinnunen.sobolev}, and a function of one variable has zero derivative almost everywhere on each of its level sets; hence
\begin{equation}
	\indc_{\{u=c\}}\nabla u=0\quad\mbox{almost everywhere, for each }c\in\R\,.
	\label{e.weighted.locality}
\end{equation}
The smooth functions~$y\mapsto\frac12(\sqrt{y^2+\eps^2}+y)$ on~$\R$ converge uniformly to~$y_+$ as~$\eps\to0$, and their derivatives converge to~$\indc_{\{y>0\}}$ away from zero. By~\eqref{e.weighted.locality} and dominated convergence, their compositions with~$u$ converge to~$u_+$ in~$H^1_\a(U)$, which gives the first formula in~\eqref{e.weighted.truncation.gradients}. Continuity of~$u\mapsto u_+$ follows in the same way from
\begin{equation*}
	\nabla(u_j)_+-\nabla u_+=\indc_{\{u_j>0\}}\nabla(u_j-u)+\bigl(\indc_{\{u_j>0\}}-\indc_{\{u>0\}}\bigr)\nabla u\,.
\end{equation*}
The identity~$T_N(u)=-N+(u+N)_+-(u-N)_+$ transfers both facts to~$T_N$, and the gradient formulas give energy contraction. Finally~$|u-T_N(u)|\leq|u|\indc_{\{|u|>N\}}$ and~$\Ener_U(u-T_N(u))=\int_{\{|u|>N\}}\nabla u\cdot\a\nabla u$ tend to zero.

\smallskip

\emph{Step 4.} Bounded smooth approximations. We compose approximations by smooth functions of finite energy with smoothed truncations. Suppose~$|u|\leq N$ and let smooth functions~$u_j$ of finite energy converge to~$u$, so~$T_N(u_j)\to u$ by Step~3. Let~$G_\eps$ be a mollification of~$T_N$ on~$\R$ at scale~$\eps$. Then~$|G_\eps|\leq N$,~$|G_\eps'|\leq1$ and~$G_\eps'\to\indc_{\{|y|<N\}}$ away from~$\pm N$, so locality and dominated convergence give~$G_\eps(u_j)\to T_N(u_j)$ as~$\eps\to0$, and a diagonal choice gives smooth approximations bounded by~$N$. For~$0\leq u\leq N$, use~$y\mapsto\min\{y_+,N\}$ instead of~$T_N$. For general~$u$, approximate~$T_N(u)$ and let~$N\to\infty$; for nonnegative~$u$, use~$u\wedge N$.

\smallskip

\emph{Step 5.} Lipschitz functions and products. Let~$f$ be Lipschitz with compact support in~$U$. Its mollifications belong to~$C^\infty_c(U)$ for small radii and converge uniformly. Their gradients are the mollifications of~$\nabla f$~\cite[Lemma 1.18 and Theorem 1.19]{Kinnunen.sobolev}, which converge to~$\nabla f$ at its Lebesgue points and are bounded by~$\operatorname{Lip}(f)$, since~$|\nabla f|\leq\operatorname{Lip}(f)$ almost everywhere~\cite[Theorem 3.31]{Kinnunen.sobolev}. Dominated convergence against~$4\operatorname{Lip}(f)^2\tr\a$ shows that they converge in~$H^1_\a(U)$, so~$f\in\Hzero{\a}(U)$; mollification preserves nonnegativity and interval bounds containing zero. If~$f$ is instead Lipschitz on~$\overline U$ and zero on~$\partial U$, the compactly supported functions~$(f-\eps)_+-(-f-\eps)_+$ have gradients~$\indc_{\{|f|>\eps\}}\nabla f$ and converge to~$f$ in~$H^1_\a(U)$, by locality and the same dominator. For the product, let~$|w|\leq N$ and take the approximations~$w_j$ of Step~4, bounded by~$N$ and converging almost everywhere. Then~$\varphi w_j\in\Hzero{\a}(U)$ by the first case, and
\begin{equation*}
	\int_U\bigl|\a^{\nf12}\bigl(\nabla(\varphi w_j)-\varphi\nabla w-w\nabla\varphi\bigr)\bigr|^2\leq2\norm{\varphi}_\infty^2\Ener_U(w_j-w)+2\operatorname{Lip}(\varphi)^2\int_U|w_j-w|^2\tr\a\longrightarrow0
\end{equation*}
by dominated convergence, which proves~\eqref{e.weighted.product}. For bounded~$f,g\in H^1_\a(U)$, let~$\Phi$ be smooth with bounded derivative and equal to~$y^2$ on an interval containing the ranges of~$f\pm g$; the chain rule and~$fg=\frac14(\Phi(f+g)-\Phi(f-g))$ give~$fg\in H^1_\a(U)$ with the product rule. If a bounded~$g\in H^1_\a(U)$ vanishes outside a compact~$K\subset U$, choose~$\varphi\in C^\infty_c(U)$ equal to one near~$K$; then~$g=\varphi g\in\Hzero{\a}(U)$ by~\eqref{e.weighted.product}.

\smallskip

\emph{Step 6.} Approximations and truncations with zero boundary values. Let~$\psi_j\in C^\infty_c(U)$ converge to~$u\in\Hzero{\a}(U)$. By Steps~3 and~5, the compactly supported Lipschitz functions~$(\psi_j)_+$ and~$T_N(\psi_j)$ lie in~$\Hzero{\a}(U)$ and converge to~$u_+$ and~$T_N(u)$. If~$u\geq0$ (and~$u\leq N$), mollify~$(\psi_j)_+$ (respectively~$\min\{(\psi_j)_+,N\}$) at a small enough radius; for~$|u|\leq N$ use~$T_N(\psi_j)$. This gives the remaining approximations.
\end{proof}

% Identification with AK.HC, Section 2.1 (e.H1s: completion of smooth functions under (||u||_{L^1}^2+energy)^{1/2};
% zero-boundary closure; solutions tested with C^infty_c). Qualitative only; its constant does not enter any estimate.
\paragraph{The space of~\cite{AK.HC}.} The space of~\cite[Section~2.1]{AK.HC} is the completion of the same class of smooth functions of finite energy under the norm~$(\norm{u}_{L^1(U)}^2+\Ener_U(u))^{\nf12}$; its zero-boundary subspace is the closure of~$C^\infty_c(U)$, and its solutions are defined by the same smooth tests. On smooth functions of finite energy,~\eqref{e.ordinary.mean.poincare} and the first bound in~\eqref{e.weighted.gradient.lone} give
\begin{equation*}
	|(u)_U|\leq|U|^{-1}\norm{u}_{L^1(U)}\,,\qquad \norm{u}_{L^1(U)}\leq|U|\,|(u)_U|+C(U)\Bigl(\int_U\tr\a^{-1}\Bigr)^{\nf12}\Ener_U(u)^{\nf12}\,,
\end{equation*}
so the two norms are equivalent on this class. Hence the two completions coincide; by Step~1 of the proof of Proposition~\ref{p.weighted.calculus}, both are the same subspace of~$W^{1,1}(U)$. The two closures of~$C^\infty_c(U)$ and the two classes of solutions coincide as well. The constant in the second bound depends on~$\int_U\tr\a^{-1}$; the comparison serves only to identify the spaces, and all Hilbert-space arguments use the inner product of~\eqref{e.weighted.norm}.

\paragraph{Testing.} The subsolution inequality~\eqref{e.weighted.subsolution} is imposed only for smooth test functions, whereas the test functions built in Sections~\ref{s.whitney.extension}--\ref{s.harnack} from truncations, powers and harmonic replacements on simplices are merely nonnegative elements of~$\Hzero{\a}(U)$. The next lemma admits all such test functions, shows that truncations of subsolutions are subsolutions, and gives the inequality satisfied by powers of nonnegative supersolutions, which Section~\ref{s.harnack} uses.

\begin{lemma}[Admissible tests, truncations and powers]
\label{l.weighted.testing}
Assume~\eqref{e.weighted.hypotheses}.
\begin{enumerate}[label=\textup{(\roman*)}]
\item\label{i.testing.admissible} If~$v\in\Csub(U)$, then
\begin{equation}
	\int_U\nabla g\cdot\a\nabla v\leq0\qquad\mbox{for every }0\leq g\in\Hzero{\a}(U)\,,
	\label{e.weighted.testing}
\end{equation}
with equality for every~$g\in\Hzero{\a}(U)$ if~$v\in\Csol(U)$. If~$-v\in\Csub(U)$, the inequality holds reversed.
\item\label{i.testing.compact} Every nonnegative~$g\in H^1_\a(U)$ vanishing outside a compact subset of~$U$ belongs to~$\Hzero{\a}(U)$.
\item\label{i.testing.truncation} If~$u\in\Csub(U)$ and~$c\in\R$, then~$(u-c)_+\in\Csub(U)$.
\item\label{i.testing.powers} Let~$u\geq0$ with~$-u\in\Csub(U)$, let~$\eps>0$ and~$w\coloneqq u+\eps$, and let~$m<\nf12$ with~$m\neq0$. Then~$\log w\in H^1_\a(U)$ with~$\nabla\log w=w^{-1}\nabla u$, and~$v\coloneqq w^m$ belongs to~$H^1_\a(U)\cap L^r(U)$, with~$r$ as in~\eqref{e.intro.parameters},~$\nabla v=mw^{m-1}\nabla u$ and
\begin{equation}
	\frac{\nabla v\cdot\a\nabla v}{v}=m^2w^{m-2}\nabla u\cdot\a\nabla u\in L^1(U)\,.
	\label{e.power.source}
\end{equation}
For every bounded nonnegative~$\phi\in H^1_\a(U)$ vanishing outside a compact subset of~$U$,
\begin{equation}
	m\int_U\nabla\phi\cdot\a\nabla v\geq(1-m)\int_U\frac\phi v\,\nabla v\cdot\a\nabla v\,.
	\label{e.power.testing}
\end{equation}
\end{enumerate}
\end{lemma}

\begin{proof}
Assertion~\ref{i.testing.admissible} follows from the approximations in Proposition~\ref{p.weighted.calculus}\ref{i.weighted.approximation} and the continuity~\eqref{e.weighted.pairing}; for a supersolution apply it to~$-v$. For~\ref{i.testing.compact}, if~$g\geq0$ vanishes outside a compact~$K\subset U$, choose~$\varphi\in C^\infty_c(U)$ equal to one near~$K$; then~$g\wedge N=\varphi(g\wedge N)\in\Hzero{\a}(U)$ by~Proposition~\ref{p.weighted.calculus}\ref{i.weighted.products}, and~$g\wedge N\to g$ by~Proposition~\ref{p.weighted.calculus}\ref{i.weighted.chain}.

\smallskip

For~\ref{i.testing.truncation}, fix~$0\leq\varphi\in C^\infty_c(U)$ and a smooth nondecreasing~$\chi_\eps:\R\to[0,1]$, zero on~$(-\infty,c]$ and one on~$[c+\eps,\infty)$. By Proposition~\ref{p.weighted.calculus}\ref{i.weighted.chain} and~Proposition~\ref{p.weighted.calculus}\ref{i.weighted.products}, the function~$\varphi\chi_\eps(u)$ is a nonnegative element of~$\Hzero{\a}(U)$, and testing gives
\begin{equation*}
	0\geq\int_U\chi_\eps(u)\nabla\varphi\cdot\a\nabla u+\int_U\varphi\chi_\eps'(u)\nabla u\cdot\a\nabla u\,.
\end{equation*}
By Proposition~\ref{p.weighted.calculus}\ref{i.weighted.chain},~$(u-c)_+\in H^1_\a(U)$. Discard the nonnegative second term and let~$\eps\downarrow0$, with dominator~$\norm{\nabla\varphi}_\infty|\a\nabla u|\in L^1$. This gives
\begin{equation*}
	\int_U\nabla\varphi\cdot\a\nabla(u-c)_+=\int_U\indc_{\{u>c\}}\nabla\varphi\cdot\a\nabla u\leq0\,.
\end{equation*}

\smallskip

For~\ref{i.testing.powers}, since~$w\geq\eps$, the functions~$y\mapsto\log(y+\eps)$,~$y\mapsto(y+\eps)^m$ and~$y\mapsto(y+\eps)^{m-1}$ on~$[0,\infty)$ have smooth extensions to~$\R$ with bounded derivatives, so Proposition~\ref{p.weighted.calculus}\ref{i.weighted.chain} gives~$\log w,v,w^{m-1}\in H^1_\a(U)$ with the chain rule, and~$0<w^{m-1}\leq\eps^{m-1}$. The identity~\eqref{e.power.source} follows from~$\nabla v=mw^{m-1}\nabla u$, and its right side is at most~$m^2\eps^{m-2}\nabla u\cdot\a\nabla u$. If~$m<0$, then~$v\leq\eps^m$; if~$0<m<\nf12$, then~$0<mr<1$ and~$v^r\leq1+w$, and~$w\in L^1(U)$ by Proposition~\ref{p.weighted.calculus}\ref{i.weighted.identification}. In both cases~$v\in L^r(U)$. For~$\phi$ as in the statement, the function~$w^{m-1}\phi$ is bounded, nonnegative and vanishes outside a compact subset of~$U$, so it belongs to~$\Hzero{\a}(U)$ by Proposition~\ref{p.weighted.calculus}\ref{i.weighted.products}. Its gradient is
\begin{equation*}
	\nabla(w^{m-1}\phi)=w^{m-1}\nabla\phi+(m-1)w^{m-2}\phi\nabla u\,.
\end{equation*}
Since~$-u\in\Csub(U)$,~\ref{i.testing.admissible} gives
\begin{equation*}
	0\leq\int_Uw^{m-1}\nabla\phi\cdot\a\nabla u+(m-1)\int_Uw^{m-2}\phi\,\nabla u\cdot\a\nabla u\,.
\end{equation*}
Multiplying by~$m^2$ and inserting~$mw^{m-1}\nabla u=\nabla v$ and~\eqref{e.power.source} gives~\eqref{e.power.testing}.
\end{proof}

\paragraph{Proofs for Section~\ref{s.harmonic.matrices}.} The remaining proofs are those of Propositions~\ref{p.harmonic.replacement},~\ref{p.upper.dirichlet} and~\ref{p.lower.matrix} and of the bounds~\eqref{e.upper.subadditivity} and~\eqref{e.lower.cube.aggregation}, in this order.

\begin{proof}[Proof of Proposition~\ref{p.harmonic.replacement}]
The form~$(\varphi,\eta)\mapsto\int_U\nabla\eta\cdot\a\nabla\varphi$ is bounded on~$\Hzero{\a}(U)$ by~\eqref{e.weighted.pairing} and coercive by~\eqref{e.weighted.coercivity}. The right side~$\eta\mapsto-\int_U\nabla\eta\cdot\a\nabla g$ is bounded on~$\Hzero{\a}(U)$ by~\eqref{e.weighted.pairing}, since~$g\in H^1_\a(U)$. The Lax--Milgram lemma~\cite[Lemma 25.2]{Ern.Guermond.hilbert} gives a unique~$\varphi\in\Hzero{\a}(U)$ with~$\int_U\nabla\eta\cdot\a\nabla\varphi=-\int_U\nabla\eta\cdot\a\nabla g$ for all~$\eta\in\Hzero{\a}(U)$, and~$H_Ug\coloneqq g+\varphi$ satisfies~\eqref{e.harmonic.orthogonality}. Expanding~$\Ener_U(w)$ around~$H_Ug$ gives~\eqref{e.harmonic.minimization}. Any harmonic representative satisfies~\eqref{e.harmonic.orthogonality} by Lemma~\ref{l.weighted.testing}, hence~\eqref{e.harmonic.minimization}, so it equals~$H_Ug$ by coercivity. Uniqueness implies linearity, and constants are their own representatives.

\smallskip

For the maximum principle let~$N\coloneqq\esssup_Ug$ and~$\varphi_j\in C^\infty_c(U)$ converge to~$H_Ug-g$. The functions~$(g+\varphi_j-N)_+$ converge to~$(H_Ug-N)_+$ by Proposition~\ref{p.weighted.calculus}\ref{i.weighted.chain} and vanish outside~$\supp\varphi_j$, so they belong to~$\Hzero{\a}(U)$ by Lemma~\ref{l.weighted.testing}. Hence~$(H_Ug-N)_+\in\Hzero{\a}(U)$, and~\eqref{e.harmonic.orthogonality} with~\eqref{e.weighted.truncation.gradients} gives~$\Ener_U((H_Ug-N)_+)=0$, so~$(H_Ug-N)_+=0$. Applying this to~$-g$ gives the lower bound.
\end{proof}

\begin{proof}[Proof of Proposition~\ref{p.upper.dirichlet}]
Pointwise,~$-\nabla w\cdot\a\nabla w+2e\cdot\a\nabla w=e\cdot\a e-(\nabla w-e)\cdot\a(\nabla w-e)\leq e\cdot\a e$, which gives the upper bound in~\eqref{e.upper.coefficient.bound}; the zero function is a competitor in both classes, and~$\Csol(U)\subseteq\Csub(U)$ gives the middle inequality. If~$w\in\Csol(U)$, then~$\ell_e-H_U\ell_e\in\Hzero{\a}(U)$ and Lemma~\ref{l.weighted.testing} give~$\fint_Ue\cdot\a\nabla w=\fint_U\nabla H_U\ell_e\cdot\a\nabla w$, so completing the square,
\begin{equation*}
	\fint_U\bigl(-\nabla w\cdot\a\nabla w+2e\cdot\a\nabla w\bigr)=\frac{\Ener_U(H_U\ell_e)}{|U|}-\frac{\Ener_U(w-H_U\ell_e)}{|U|}\,.
\end{equation*}
The supremum is attained at~$H_U\ell_e$, and~\eqref{e.harmonic.minimization} gives the minimum. Since~$e\mapsto H_U\ell_e$ is linear by Proposition~\ref{p.harmonic.replacement}, this is a quadratic form, and~$H_U(\ell_e+c)=H_U\ell_e+c$. It is positive definite: by~\eqref{e.weighted.coercivity},~$(\nabla H_U\ell_e)_U=e$, so zero energy forces~$e=0$.
\end{proof}

\begin{proof}[Proof of Proposition~\ref{p.lower.matrix}]
The mean-zero subspace of~$H^1_\a(U)$ is a Hilbert space for the inner product~$\fint_U\nabla u\cdot\a\nabla w$: on it the induced norm is~$|U|^{-\nf12}$ times the norm~\eqref{e.weighted.norm}, and zero energy forces zero by~\eqref{e.ordinary.mean.poincare}. By Cauchy--Schwarz,~$|e\cdot(\nabla w)_U|\leq(\fint_Ue\cdot\a^{-1}e)^{\nf12}(\Ener_U(w)/|U|)^{\nf12}$, so~$w\mapsto e\cdot(\nabla w)_U$ is bounded, and the Riesz theorem~\cite[Theorem C.24]{Ern.Guermond.hilbert} gives a mean-zero~$u_e$, linear in~$e$, with
\begin{equation}
	\fint_U\nabla w\cdot\a\nabla u_e=e\cdot(\nabla w)_U\qquad\mbox{for every }w\in H^1_\a(U)\,;
	\label{e.lower.riesz}
\end{equation}
the identity extends from mean-zero~$w$ because both sides depend only on~$\nabla w$. Testing with~$w\in C^\infty_c(U)$ shows~$u_e\in\Csol(U)$. Completing the square with~\eqref{e.lower.riesz},
\begin{equation*}
	\fint_U\bigl(-\nabla w\cdot\a\nabla w+2e\cdot\nabla w\bigr)=\frac{\Ener_U(u_e)}{|U|}-\frac{\Ener_U(w-u_e)}{|U|}\,,
\end{equation*}
so both suprema equal~$\Ener_U(u_e)/|U|$, a quadratic form in~$e$. It is positive definite, since~\eqref{e.lower.riesz} with~$w=\ell_e$ gives~$|e|^2=\fint_U\nabla\ell_e\cdot\a\nabla u_e$. Inserting~$cw$ in~\eqref{e.lower.all.functions} and optimizing over~$c\in\R$ gives~\eqref{e.lower.mean.gradient}. Finally,
\begin{equation*}
	-\nabla w\cdot\a\nabla w+2e\cdot\nabla w=e\cdot\a^{-1}e-(\nabla w-\a^{-1}e)\cdot\a(\nabla w-\a^{-1}e)\,.
\end{equation*}
This gives~\eqref{e.lower.coefficient.bound}.
\end{proof}

\begin{proof}[Proof of~\eqref{e.upper.subadditivity} and~\eqref{e.lower.cube.aggregation}]
Fix~$e\in\Rd$. A solution~$w\in\Csol(U)$ restricts to a solution on each~$U_i$ by Proposition~\ref{p.weighted.calculus}\ref{i.weighted.domains}. The integral over~$U$ of the integrand in~\eqref{e.upper.directional.definition} is the sum of its integrals over the~$U_i$, each bounded by~$|U_i|\,e\cdot\a(U_i)e$ by~\eqref{e.upper.dirichlet}. Dividing by~$|U|$ and taking the supremum over~$w$ proves~\eqref{e.upper.subadditivity}. The same argument with~$w\in H^1_\a(U)$ and~\eqref{e.lower.all.functions} proves~\eqref{e.lower.cube.aggregation}. The sums of competitor integrals converge absolutely:~\eqref{e.weighted.gradient.lone} puts both~$\nabla w$ and~$\a\nabla w$ in~$L^1(U)$, while~$\Ener_U(w)<\infty$. The matrix series converge absolutely by~\eqref{e.upper.norm.bound} and~\eqref{e.lower.norm.bound}, since~$|\a(U_i)|\leq(\tr\a)_{U_i}$ and~$|\a_*^{-1}(U_i)|\leq(\tr\a^{-1})_{U_i}$, whose volume-weighted sums are~$\int_U\tr\a$ and~$\int_U\tr\a^{-1}$.
\end{proof}

\section{Averages over simplices or cubes and Gaussian averages}
\label{a.besov}

% Change-list items V9-015 and V9-016. The proof is Step 1, (4)--(6), of notes/positive-coefficients-besov-equivalence-2026-09-28.md, for simplices, with the same argument for cubes; the comparison of Gaussian averages with Littlewood--Paley pieces (Steps 2--3 of that note) is cited, not proved.
Proposition~\ref{p.besov.averages} compares the averages of a positive semidefinite field over the simplices of~$\mathscr T_k$ or over the triadic cubes of side length~$3^{-k}$ with its Gaussian averages at the time~$3^{-2k}$. Lemma~\ref{l.negative.sobolev} bounds~\eqref{e.negative.regularity.norm} by negative Sobolev norms and gives Theorem~\ref{t.coefficient.conditions}\ref{c.negative.sobolev}. Apart from Theorem~\ref{t.coefficient.conditions}\ref{c.negative.sobolev} and Corollary~\ref{c.sobolev.coefficients}, the proofs of the other results of this paper do not use this appendix.

\smallskip

For~$t>0$ let
\begin{equation*}
	G_t(x)\coloneqq(4\pi t)^{-\nf d2}\exp\Bigl(-\frac{|x|^2}{4t}\Bigr)\qquad(x\in\Rd)\,,
\end{equation*}
and for a matrix field~$\mathbf b$ on~$\cu_0$ let~$\widetilde{\mathbf b}$ be its extension by zero to~$\Rd$. For a scalar or matrix field~$F\in L^1(\Rd)$, since~$G_t$ is positive with integral one,~$|G_t*F|\leq G_t*|F|$ pointwise and~$\norm{G_t*F}_{L^p(\Rd)}\leq\norm{F}_{L^p(\Rd)}$ for~$1\leq p\leq\infty$, and~$G_t*F$ is positive semidefinite everywhere if~$F$ is positive semidefinite almost everywhere.

\begin{proposition}[Averages and Gaussian averages]
\label{p.besov.averages}
Let~$1\leq p\leq\infty$, and let~$\mathbf b$ be a matrix field on~$\cu_0$, symmetric and positive semidefinite almost everywhere, with~$|\mathbf b|\in L^1(\cu_0)$. For~$k\geq0$ let~$\mathscr P_k$ be either~$\mathscr T_k$ or the family of the triadic cubes~$z+\cu_{-k}$ with~$z\in3^{-k}\Zd\cap\cu_0$. For~$p=\infty$ the means below are maxima. There is~$C<\infty$, depending only on~$d$, such that for every~$k\geq0$
\begin{equation}
	C^{-1}\biggl(\avsum_{Q\in\mathscr P_k}|(\mathbf b)_Q|^p\biggr)^{\nf1p}\leq\norm{G_{3^{-2k}}*\widetilde{\mathbf b}}_{L^p(\Rd)}\leq C\biggl(\avsum_{Q\in\mathscr P_k}|(\mathbf b)_Q|^p\biggr)^{\nf1p}\,.
	\label{e.besov.heat}
\end{equation}
Consequently, for every~$s>0$, each of the sums
\begin{equation}
	\sum_{k\geq0}3^{-ks}\biggl(\avsum_{Q\in\mathscr P_k}|(\mathbf b)_Q|^p\biggr)^{\nf1{2p}}\qand\sum_{k\geq0}3^{-ks}\norm{G_{3^{-2k}}*\widetilde{\mathbf b}}_{L^p(\Rd)}^{\nf12}
	\label{e.besov.averages}
\end{equation}
is at most~$C^{\nf12}$ times the other.
\end{proposition}

The second sum in~\eqref{e.besov.averages}, squared, is comparable, with constants depending only on~$d$,~$p$ and~$s$, to the quasi-norm of~$\widetilde{\mathbf b}$ in the Besov space~$B^{-2s}_{p,\nf12}(\Rd)$ defined through the Fourier transform. Take~$L=-\Delta$, so that~$e^{-tL}=e^{t\Delta}$ is convolution with~$G_t$. For the negative order~$-2s$, use the derivative-free case~$m=0$ of the heat characterization~\cite[Definition~6.6]{Kerkyacharian.Petrushev}; the condition~$m>-2s$ holds and the derivative factor is the identity. Thus~\cite[Definition~6.1(i) and Theorem~6.7(a)]{Kerkyacharian.Petrushev}, applied to the entries of~$\widetilde{\mathbf b}$, show that this quasi-norm is comparable to
\begin{equation*}
	\norm{G_1*\widetilde{\mathbf b}}_{L^p(\Rd)}+\biggl(\int_0^1t^{\nf s2}\norm{G_t*\widetilde{\mathbf b}}_{L^p(\Rd)}^{\nf12}\,\frac{dt}{t}\biggr)^2\,.
\end{equation*}
By the semigroup property and the contraction above,~$\norm{G_t*\widetilde{\mathbf b}}_{L^p(\Rd)}$ is nonincreasing in~$t$, so the integral over~$[3^{-2k-2},3^{-2k}]$ lies between multiples, depending only on~$s$, of the terms of the second sum in~\eqref{e.besov.averages} with indices~$k$ and~$k+1$. The first term of the display is the square of the term~$k=0$ of that sum, hence at most the square of the sum. Passing between the entries and the operator norm changes the constants only by factors depending on~$d$, since the operator norm of a matrix lies between the largest absolute value of its entries and the sum of the absolute values of its entries.

\begin{proof}[Proof of Proposition~\ref{p.besov.averages}]
Fix~$k\geq0$. The members of~$\mathscr P_k$ have equal volumes, at least~$3^{-kd}/d!$, and diameter~$\sqrt d\,3^{-k}$, and they cover~$\cu_0$ up to a null set, so the arithmetic mean over~$\mathscr P_k$ is the volume-weighted sum, as in~\eqref{e.partition.average}. Let~$Q\in\mathscr P_k$. If~$x,y\in Q$, then~$|x-y|\leq\sqrt d\,3^{-k}$, so~$G_{3^{-2k}}(x-y)\geq(4\pi)^{-\nf d2}e^{-\nf d4}3^{kd}$. Since~$\widetilde{\mathbf b}$ is positive semidefinite, we may drop the integral over~$\Rd\setminus Q$ in~$G_{3^{-2k}}*\widetilde{\mathbf b}(x)$, and~$|Q|\geq3^{-kd}/d!$ gives, as quadratic forms,
\begin{equation*}
	(\mathbf b)_Q\leq d!\,(4\pi)^{\nf d2}e^{\nf d4}\,G_{3^{-2k}}*\widetilde{\mathbf b}(x)\qquad(x\in Q)\,.
\end{equation*}
If also~$z\in Q$, then~$|x-z|^2\leq2|x-y|^2+2d3^{-2k}$ for every~$x\in\Rd$, so that~$G_{3^{-2k}}(x-y)\leq2^{\nf d2}e^{\nf d4}G_{2\cdot3^{-2k}}(x-z)$. We average in~$z\in Q$, multiply by the positive semidefinite matrix~$\mathbf b(y)$, integrate over~$y\in Q$, and sum over~$Q\in\mathscr P_k$. For every~$x\in\Rd$ this gives
\begin{equation*}
	G_{3^{-2k}}*\widetilde{\mathbf b}(x)\leq2^{\nf d2}e^{\nf d4}\sum_{Q\in\mathscr P_k}\bigl(G_{2\cdot3^{-2k}}*\indc_Q\bigr)(x)\,(\mathbf b)_Q\,.
\end{equation*}
The operator norm is monotone on positive semidefinite matrices. The first bound, raised to the power~$p$ and integrated over each~$Q\in\mathscr P_k$, or for~$p=\infty$ its essential supremum over~$Q$, bounds the~$L^p$ mean of~$|(\mathbf b)_Q|$ over~$\mathscr P_k$ by~$d!\,(4\pi)^{\nf d2}e^{\nf d4}\norm{G_{3^{-2k}}*\widetilde{\mathbf b}}_{L^p(\Rd)}$. The second bound, the triangle inequality and the contraction of~$G_{2\cdot3^{-2k}}*$ on~$L^p(\Rd)$, applied to~$\sum_{Q\in\mathscr P_k}|(\mathbf b)_Q|\indc_Q$, give the reverse bound with the factor~$2^{\nf d2}e^{\nf d4}$. This proves~\eqref{e.besov.heat} with~$C=d!\,(4\pi)^{\nf d2}e^{\nf d4}$, which does not depend on~$k$ or~$p$. Taking square roots in~\eqref{e.besov.heat}, multiplying by~$3^{-ks}$ and summing over~$k$ compares the two sums in~\eqref{e.besov.averages}.
\end{proof}

\smallskip

% Adapted from AADKM (arXiv:2601.05690), Lemma 2.1 with q=p; local source paper/coarsegrainedDGMN.tex (SHA-256 af0dd4f4), l.weak.Wsp.bounds.Besov.circ. AADKM tests with step functions on cubes (e.dual.id--e.p-bound), which needs t<1-1/p; here, for positive semidefinite fields, the test functions are Gaussian averages, as in notes/positive-coefficients-besov-equivalence-2026-09-28.md, Section 3, (12), and the cube averages are recovered by Proposition p.besov.averages. Translation: AADKM scale 3^k, k<=0, is 3^{-k} here; its order t is \beta and its s is the present s; its sum of square roots over scales is the square root of e.negative.regularity.norm. On the unit cube the volume-normalized norms of AADKM coincide with e.fractional.seminorm. Change-list items V9-012 and V9-016.
\begin{lemma}[Negative Sobolev norms]
\label{l.negative.sobolev}
Let~$1<p\leq\infty$,~$s>0$ and~$0<\eps\leq2s$, and let~$\mathbf b$ be a matrix field on~$\cu_0$, symmetric and positive semidefinite almost everywhere, with~$|\mathbf b|\in L^1(\cu_0)$. There is~$C<\infty$, depending only on~$d$,~$p$,~$s$ and~$\eps$, such that
\begin{equation*}
	\norm{\mathbf b}_{\mathring{\underline B}^{-2s}_{p,\nf12}(\cu_0)}\leq C\norm{\mathbf b}_{W^{-2s+\eps,p}(\cu_0)}\,.
\end{equation*}
\end{lemma}

In~\cite{AADKM}, for orders~$0<\beta<1$, the negative Sobolev norm~$\norm{\cdot}_{W^{-\beta,p}(\cu_0)}$ of a scalar function is its norm in the dual of~$W^{\beta,p'}(\cu_0)$, where the pairing with a bounded test function is the integral, and matrix fields are taken entrywise. Since the operator norm of a matrix is at most the sum of the absolute values of its entries,~\eqref{e.negative.sobolev.norm} is at most the sum of these norms of the entries of~$\mathbf b$.

\begin{proof}
We test with Gaussian averages, in the notation above. The constants~$C$ below depend only on~$d$,~$p$,~$s$ and~$\eps$. Write~$2s-\eps=m+\alpha$ as in Section~\ref{ss.notation}, let~$0<t\leq1$ and~$g\in L^{p'}(\Rd)$. The function~$G_t*g$ is bounded and smooth, with~$\nabla^j(G_t*g)=(\nabla^jG_t)*g$, and by scaling~$\norm{\nabla^jG_t}_{L^1(\Rd)}\leq Ct^{-\nf j2}$ for~$0\leq j\leq m+1$, so that~$\norm{\nabla^j(G_t*g)}_{L^{p'}(\Rd)}\leq Ct^{-\nf j2}\norm{g}_{L^{p'}(\Rd)}$ for these~$j$.

\smallskip

If~$\alpha>0$, the bounds~$\norm{\nabla^mG_t(\cdot+z)-\nabla^mG_t}_{L^1(\Rd)}\leq\min\{2\norm{\nabla^mG_t}_{L^1(\Rd)},|z|\norm{\nabla^{m+1}G_t}_{L^1(\Rd)}\}$ give, for~$z\in\Rd$,
\begin{equation*}
	\norm{\nabla^m(G_t*g)(\cdot+z)-\nabla^m(G_t*g)}_{L^{p'}(\Rd)}\leq Ct^{-\nf m2}\min\bigl\{1,t^{-\nf12}|z|\bigr\}\norm{g}_{L^{p'}(\Rd)}\,.
\end{equation*}
We write the seminorm~\eqref{e.fractional.seminorm} of~$\nabla^m(G_t*g)$ on~$\Rd$ as an integral over~$z=y-x$, insert this bound and split the radial integral at~$|z|=t^{\nf12}$; the two parts are at most~$Ct^{-\nf{(2s-\eps)p'}2}\norm{g}_{L^{p'}(\Rd)}^{p'}$, because~$0<\alpha<1$. Since~$t\leq1$, so that~$t^{-\nf j2}\leq t^{-\nf{(2s-\eps)}2}$ for~$j\leq m$, and restriction to~$\cu_0$ does not increase the norm, this gives
\begin{equation*}
	\norm{G_t*g}_{W^{2s-\eps,p'}(\cu_0)}\leq Ct^{-\nf{(2s-\eps)}2}\norm{g}_{L^{p'}(\Rd)}\,.
\end{equation*}

\smallskip

Since~$\tr\mathbf b\geq0$ and~$|\tr M|\leq d|M|$ for symmetric matrices~$M$, the definition~\eqref{e.negative.sobolev.norm} gives, for every bounded~$h$,
\begin{equation*}
	\Bigl|\int_{\cu_0}h\,\tr\mathbf b\Bigr|\leq d\norm{\mathbf b}_{W^{-2s+\eps,p}(\cu_0)}\norm{h}_{W^{2s-\eps,p'}(\cu_0)}\,.
\end{equation*}
The function~$G_t*\tr\widetilde{\mathbf b}$ belongs to~$L^p(\Rd)$. Since~$G_t$ is even, Fubini's theorem gives, for~$g\in L^{p'}(\Rd)$,
\begin{equation*}
	\int_{\Rd}(G_t*\tr\widetilde{\mathbf b})\,g=\int_{\cu_0}\tr\mathbf b\,(G_t*g)\,.
\end{equation*}
Use~$h=G_t*g$ in the preceding bound and apply the norm estimate above. Taking the supremum over~$\norm{g}_{L^{p'}(\Rd)}\leq1$ with~$t=3^{-2k}$ gives
\begin{equation}
	\norm{G_{3^{-2k}}*\tr\widetilde{\mathbf b}}_{L^p(\Rd)}\leq C3^{k(2s-\eps)}\norm{\mathbf b}_{W^{-2s+\eps,p}(\cu_0)}\qquad(k\geq0)\,.
	\label{e.negative.sobolev.step}
\end{equation}
The field~$G_{3^{-2k}}*\widetilde{\mathbf b}$ is positive semidefinite with trace~$G_{3^{-2k}}*\tr\widetilde{\mathbf b}$, and~$|M|\leq\tr M$ for positive semidefinite matrices~$M$. Hence~\eqref{e.besov.heat} for the family of cubes and~\eqref{e.negative.sobolev.step} bound the~$L^p$ mean of~$|(\mathbf b)_{z+\cu_{-k}}|$ over~$z\in3^{-k}\Zd\cap\cu_0$ by~$C3^{k(2s-\eps)}\norm{\mathbf b}_{W^{-2s+\eps,p}(\cu_0)}$. Taking square roots, multiplying by~$3^{-ks}$, summing the geometric series~$\sum_k3^{-k\eps/2}=(1-3^{-\nf\eps2})^{-1}$ and squaring proves the lemma, by~\eqref{e.negative.regularity.norm}.
\end{proof}

% Source: writing/vs10/appendix_c_symmetric2.tex, SHA-256 7f97f48e12c5e8493e8cf9a9695663b99f997dd3b89674724e36b8e1670ca77f.
% Proposition and spatial-moment proof retained; the shared finite/countable subadditivity proof is in Appendix A.
\section{Comparison of cubical and simplicial constants}
\label{a.cubical.simplicial}

This appendix proves Proposition~\ref{p.cubical.simplicial.equivalence}. Subadditivity compares each cube with its simplices. For the reverse comparison, we partition each simplex into triadic Whitney cubes. The cubes of side length~$3^{-j}$ inside a simplex of size~$3^{-k}$ occupy a fraction at most~$C3^{-(j-k)}$ of its volume. H\"older's inequality converts this volume bound into a comparison of spatial moments, which we then sum over scales.

Throughout,~\eqref{e.weighted.hypotheses} holds on~$\cu_0$, and~$\mathbf B(U)$ denotes either~$\a(U)$ or~$\a_*^{-1}(U)$.

\begin{lemma}
\label{l.cubical.simplicial.moments}
For~$1\leq p\leq\infty$ and~$k\geq0$, set
\begin{equation*}
	B_{k,p}\coloneqq \biggl(\avsum_{\triangle\in\mathscr T_k} |\mathbf B(\triangle)|^p\biggr)^{\nf1p} \qquad\mbox{and}\qquad \widetilde B_{k,p}\coloneqq \biggl(\avsum_{Q\in\mathscr Q_k} |\mathbf B(Q)|^p\biggr)^{\nf1p}\,.
\end{equation*}
For~$p=\infty$, replace the means by maxima. For every~$1\leq p\leq\infty$ and~$k\geq0$,
\begin{equation}
	\widetilde B_{k,p}\leq B_{k,p} \leq C\sum_{j\geq k} 3^{-(j-k)(1-1/p)}\widetilde B_{j,p}\,,
	\label{e.cubical.simplicial.moments}
\end{equation}
where~$C$ depends only on~$d$.
\end{lemma}

\begin{proof}
Throughout the proof,~$C$ depends only on~$d$. By~\eqref{e.upper.subadditivity}--\eqref{e.lower.cube.aggregation},
\begin{equation*}
	\mathbf B(Q)\leq \sum_{\substack{\triangle\in\mathscr T_k\\\triangle\subseteq Q}} \frac{|\triangle|}{|Q|}\mathbf B(\triangle) \qquad(Q\in\mathscr Q_k)\,.
\end{equation*}
Taking operator norms, applying Jensen's inequality and averaging over~$\mathscr Q_k$ proves~$\widetilde B_{k,p}\leq B_{k,p}$. For~$p=\infty$, take maxima instead.

\smallskip

Fix~$\triangle\in\mathscr T_k$ and take its triadic Whitney decomposition~$\mathscr W(\triangle)$, consisting of the maximal triadic cubes~$Q$ such that~$3\overline Q\subseteq\triangle$. Write~$\mathscr W_j(\triangle)\coloneqq\mathscr W(\triangle)\cap\mathscr Q_j$. Outside the countable union of triadic grid faces, every point of~$\triangle$ lies in an arbitrarily small admissible cube and hence in a maximal one. Nesting of the triadic grid gives disjoint interiors, so these cubes partition~$\triangle$ up to a null set and have~$j\geq k$. The parent of a maximal cube is not admissible, so every~$Q\in\mathscr W_j(\triangle)$ lies in a boundary layer of width~$C3^{-j}$. Since the simplices are dilates and coordinate permutations of a fixed simplex,
\begin{equation}
	\sum_{Q\in\mathscr W_j(\triangle)}\frac{|Q|}{|\triangle|} \leq C3^{-(j-k)}\,.
	\label{e.simplex.whitney.layer}
\end{equation}

For~$j\geq k$, write~$S_j(\triangle)\coloneqq\sum_{Q\in\mathscr W_j(\triangle)}\frac{|Q|}{|\triangle|}|\mathbf B(Q)|$ for the contribution of the~$j$th Whitney layer. Applying~\eqref{e.upper.subadditivity}--\eqref{e.lower.cube.aggregation} to the countable partition and taking operator norms gives
\begin{equation*}
	|\mathbf B(\triangle)|\leq\sum_{j\geq k}S_j(\triangle)\,.
\end{equation*}
For~$1<p<\infty$, H\"older's inequality and~\eqref{e.simplex.whitney.layer} give
\begin{equation*}
	S_j(\triangle)^p\leq C^{p-1}3^{-(j-k)(p-1)}\sum_{Q\in\mathscr W_j(\triangle)}\frac{|Q|}{|\triangle|}|\mathbf B(Q)|^p\,.
\end{equation*}
Each cube belongs to the Whitney decomposition of at most one~$\triangle\in\mathscr T_k$. Summing over the simplices gives
\begin{equation*}
	\sum_{\triangle\in\mathscr T_k}|\triangle|S_j(\triangle)^p\leq C^{p-1}3^{-(j-k)(p-1)}\widetilde B_{j,p}^{\,p}\,.
\end{equation*}
Both spatial means are volume averages. Taking the~$p$th root and applying Minkowski's inequality in the sum over~$j$ proves the second inequality of~\eqref{e.cubical.simplicial.moments}. For~$p=1$, sum directly; for~$p=\infty$, use~\eqref{e.simplex.whitney.layer} and take maxima.
\end{proof}

\begin{proof}[Proof of Proposition~\ref{p.cubical.simplicial.equivalence}]
Taking square roots in~\eqref{e.cubical.simplicial.moments} and using~$(\sum_\ell b_\ell)^{\nf12}\leq\sum_\ell b_\ell^{\nf12}$ for nonnegative~$b_\ell$ gives
\begin{equation*}
	3^{-ks}B_{k,p}^{\nf12} \leq C\sum_{\ell\geq0} 3^{-\ell(\frac12(1-\frac1p)-s)} 3^{-(k+\ell)s}\widetilde B_{k+\ell,p}^{\nf12}\,.
\end{equation*}
For~$1\leq\mathsf m\leq\infty$, the triangle inequality in~$\ell^{\mathsf m}$, the bound of each tail by the full sequence norm and~$s<\nf12(1-\nf1p)$ imply
\begin{equation*}
	\bigl\|\{3^{-ks}B_{k,p}^{\nf12}\}_{k\geq0}\bigr\|_{\ell^{\mathsf m}} \leq\frac{C}{1-3^{s-\frac12(1-\frac1p)}} \bigl\|\{3^{-ks}\widetilde B_{k,p}^{\nf12}\}_{k\geq0} \bigr\|_{\ell^{\mathsf m}}\,.
\end{equation*}
The normalizing factor~$(1-3^{-s\mathsf m})^{1/\mathsf m}$ is the same on both sides for finite~$\mathsf m$. Taking~$\mathbf B(U)=\a(U)$ and squaring proves the upper bound in~\eqref{e.Lambda.cube.simplex.equivalence}. The lower bound follows from~$\widetilde B_{k,p}\leq B_{k,p}$.

Taking~$\mathbf B(U)=\a_*^{-1}(U)$ and replacing~$(s,p)$ by~$(t,q)$ gives the same comparison for~$\lambda_{t,\mathsf m,q}(\cu_0)^{-\nf12}$ and~$\widetilde\lambda_{t,\mathsf m,q}(\cu_0)^{-\nf12}$. Inverting and squaring proves~\eqref{e.lambda.cube.simplex.equivalence}, with the usual extended interpretation when a lower constant is zero.
\end{proof}

\section{Bombieri's lemma}
\label{a.bombieri}

This appendix proves Bombieri's lemma, which the proof of Lemma~\ref{l.crossover} uses in place of the John--Nirenberg inequality. It is the lemma of Bombieri and Giusti~\cite{Bombieri.Giusti} in the form of~\cite[Lemma~A.2]{AADKM}, with the exponent~$q$ there equal to one. Here the radii lie between~$\nf34$ and~$\nf78$, and the function is positive almost everywhere, as holds for~$w^{-p_*}$ in the crossover estimate. The proof bounds~$\int v^a$ using a level set and H\"older's inequality, then iterates over radii.

\begin{lemma}[Bombieri's lemma]
\label{l.bombieri}
Let~$A_1,A_2\geq1$ and~$\kappa>0$, and let~$v$ be measurable on~$\frac78\cu_0$ with~$0<v<\infty$ almost everywhere,~$v\in L^1(\frac78\cu_0)$ and~$\norm{\log v}_{L^1(\frac78\cu_0)}\leq A_1$. Suppose that, for~$0<a<1$ and~$\nf34\leq\rho_1<\rho_2\leq\nf78$,
\begin{equation}
	\norm{v}_{L^1(\rho_1\cu_0)}\leq\bigl(A_2(\rho_2-\rho_1)^{-\kappa}\bigr)^{\nf1a}\norm{v}_{L^a(\rho_2\cu_0)}\,.
	\label{e.bombieri.hypothesis}
\end{equation}
Then~$\norm{v}_{L^1(\frac34\cu_0)}\leq\exp\bigl(C(\kappa)A_1A_2^6\bigr)$, where~$C(\kappa)$ depends only on~$\kappa$.
\end{lemma}

\begin{proof}
For~$\nf34\leq\rho\leq\nf78$, let
\begin{equation*}
	f(\rho)\coloneqq\max\Bigl\{0,\log\int_{\rho\cu_0}v\Bigr\}\,.
\end{equation*}
Since~$v>0$ almost everywhere,~$f$ is well defined. It is nondecreasing in~$\rho$ and bounded by~$f(\nf78)<\infty$.

\smallskip

\emph{Step 1.} A one-step inequality. Fix~$\nf34\leq\rho_1<\rho_2\leq\nf78$, write~$\delta\coloneqq\rho_2-\rho_1$, and suppose first that~$f(\rho_2)>3A_1$, so that~$\int_{\rho_2\cu_0}v=e^{f(\rho_2)}$. Define
\begin{equation*}
	E\coloneqq\Bigl\{x\in\rho_2\cu_0:\log v(x)>\frac{f(\rho_2)}3\Bigr\}\,.
\end{equation*}
Chebyshev's inequality gives~$|E|\leq3A_1/f(\rho_2)$. For~$0<a<1$, splitting the integral of~$v^a$ over~$E$ and its complement and applying H\"older's inequality on~$E$,
\begin{equation*}
	\int_{\rho_2\cu_0}v^a\leq e^{\nf{af(\rho_2)}3}|\rho_2\cu_0|+\biggl(\int_{\rho_2\cu_0}v\biggr)^a|E|^{1-a}\leq e^{\nf{af(\rho_2)}3}+e^{af(\rho_2)}\Bigl(\frac{3A_1}{f(\rho_2)}\Bigr)^{1-a}\,.
\end{equation*}
Choose~$a\in(0,1)$ by
\begin{equation*}
	\frac1a=1+\frac{f(\rho_2)}{\log(f(\rho_2)/(3A_1))}\,.
\end{equation*}
Then~$(1-a)\log\nf{f(\rho_2)}{3A_1}=af(\rho_2)$ and the last term equals one. Hence~$\int_{\rho_2\cu_0}v^a\leq2e^{af(\rho_2)/3}$, so~$\log\norm{v}_{L^a(\rho_2\cu_0)}\leq\nf{f(\rho_2)}3+\nf{\log2}a$, and~\eqref{e.bombieri.hypothesis} gives
\begin{equation*}
	\log\norm{v}_{L^1(\rho_1\cu_0)}\leq\frac{f(\rho_2)}3+\frac1a\log\frac{2A_2}{\delta^\kappa}=f(\rho_2)\biggl(\frac13+\frac{\log(2A_2\delta^{-\kappa})}{\log(f(\rho_2)/(3A_1))}\biggr)+\log\frac{2A_2}{\delta^\kappa}\,.
\end{equation*}
If~$f(\rho_2)\geq3A_1(2A_2)^6\delta^{-6\kappa}$, then
\begin{equation*}
	\log\frac{f(\rho_2)}{3A_1}\geq6\log(2A_2\delta^{-\kappa})\,.
\end{equation*}
The bracket is then at most~$\nf12$, and, since~$\log x\leq x^6$ for~$x=2A_2\delta^{-\kappa}\geq1$, we obtain~$f(\rho_1)\leq\frac12f(\rho_2)+3A_1(2A_2)^6\delta^{-6\kappa}$. Otherwise~$f(\rho_1)\leq f(\rho_2)<3A_1(2A_2)^6\delta^{-6\kappa}$. In both cases
\begin{equation}
	f(\rho_1)\leq\frac12f(\rho_2)+3A_1(2A_2)^6(\rho_2-\rho_1)^{-6\kappa}\,.
	\label{e.bombieri.step}
\end{equation}

\smallskip

\emph{Step 2.} Iteration over radii. For~$j\geq0$, let
\begin{equation*}
	R_j\coloneqq\frac78-\frac18\,2^{-j/(12\kappa)}\,,\qquad R_{j+1}-R_j=\frac18\bigl(1-2^{-1/(12\kappa)}\bigr)2^{-j/(12\kappa)}\,.
\end{equation*}
Then~$R_0=\nf34$. Iterating~\eqref{e.bombieri.step} between~$R_j$ and~$R_{j+1}$ for~$j<n$,
\begin{equation*}
	f\bigl(\tfrac34\bigr)\leq2^{-n}f(R_n)+3A_1(2A_2)^6\biggl(\frac8{1-2^{-1/(12\kappa)}}\biggr)^{6\kappa}\sum_{j<n}2^{-j}2^{\nf j2}\,.
\end{equation*}
The first term tends to zero because~$f\leq f(\nf78)$, and the series is at most~$(1-2^{-1/2})^{-1}$. Since~$\log\norm{v}_{L^1(\frac34\cu_0)}\leq f(\nf34)$, this proves the lemma.
\end{proof}

\subsection*{Acknowledgments}
S.~A. and T.~K. acknowledge support from the European Research Council (ERC) under the European Union's Horizon Europe research and innovation programme, grant agreement number 101200828. T.~K. and A.~T. were supported by the Academy of Finland.

\subsection*{Declaration of generative AI use}
The proof development and the preparation of this paper were assisted by ChatGPT and Claude. The authors take full responsibility for the contents of the paper.

\end{document}